\documentclass[11pt]{article}

\usepackage{amsfonts,amssymb, latexsym, amsmath, amsthm,mathrsfs, verbatim,geometry}
\usepackage{graphics}
\usepackage{slashed}
\usepackage{url,color}
\usepackage{enumerate}
\usepackage{esint}
\usepackage{pifont}
\usepackage{upgreek}
\usepackage{times}
\usepackage{calligra}
\usepackage{bm}

\usepackage{hyperref}
\hypersetup{
    colorlinks,
    citecolor=red,
    filecolor=green,
    linkcolor=blue,
    urlcolor=black
}
\usepackage{marvosym}

\newcommand{\beq}{\begin{equation}}
\newcommand{\eeq}{\end{equation}}
\newcommand{\beqs}{\begin{equation*}}
\newcommand{\eeqs}{\end{equation*}}
\newcommand{\beqa}{\begin{equation}\begin{aligned}}
\newcommand{\eeqa}{\end{aligned}\end{equation}}
\newcommand{\beqas}{\begin{equation*}\begin{aligned}}
\newcommand{\eeqas}{\end{aligned}\end{equation*}}

\def\R{\mathbb R}
\def\C{\mathbb C}
\def\N{\mathbb N}

\def\al{\alpha}

\def\be{\begin{equation}}
\def\ee{\end{equation}}
\def\bea{\begin{eqnarray}}
\def\eea{\end{eqnarray}}
\def\beas{\begin{eqnarray*}}
\def\eeas{\end{eqnarray*}}

\def\a{{\bf a}}
\def\l{\lambda}
\def\la{\lambda}
\def\ga{\gamma}
\def\th{\theta}
\def\pa{\partial }
\def\om{\omega}
\def\dif{\textup{d}}

\def\S{y_\ast}

\def\de{\delta}
\let\Re\relax
\let\Im\relax
\DeclareMathOperator{\Re}{Re}
\DeclareMathOperator{\Im}{Im}

\newcommand*\diff{\mathop{}\!\mathrm{d}} 

\def\a{{\bf a}}

\def\L{\bar{\Lambda}}

\def\l{\lambda}

\def\lv{\left\vert}
\def\rv{\right\vert}

\def\trho{\tilde{\rho}}

\def\tr{\tilde{r}}

\def\eps{\varepsilon}

\def\bcr{\begin{color}{red}}
\def\bcb{\begin{color}{blue}}
\def\ec{\end{color}}
\def\tom{\tilde\omega}
\def\tr{\tilde\rho}

\def\X{\mathcal X}
\def\Y{\mathcal Y}

\def\rhoLP{\widehat{\rho}}
\def\vLP{\widehat{v}}
\def\PLP{\widehat{\Pi}}
\def\omLP{\widehat{\omega}}
\def\wLP{{w}}

\def\what{\hat{w}}

\def\bfI{\mathbf{I}}
\def\bcc{\begin{color}{magenta}}

\usepackage{tikz}
\usetikzlibrary{positioning}
\usepackage{blindtext}
\usepackage{tcolorbox}
\usepackage{graphicx}
\hypersetup{citecolor=blue}
\usepackage[latin1]{inputenc}
\usepackage{times}
\usepackage[T1]{fontenc}
\usepackage{wrapfig}
\usepackage{caption}
\usetikzlibrary{calc}
\usetikzlibrary{arrows.meta}
\usepackage{mdframed}
\usepackage{lipsum}

\def\bp{\widehat{\pa}_z}
\def\bd{\widehat D_z}
\def\bl{\widehat\Delta_z}

\def\bzeta{\widehat\zeta}
\def\H{\mathcal H}
\def\HmZ{\mathcal H^{2m}_{Z}}

\def\HdotmZ{\dot{\mathcal H}^{2m}_{Z}}
\def\GmZ{\dot{\mathcal H}^{2m}_{Z}}
\def\HmZEul{\mathcal H^{2m,\textup{Eul}}_{Z}}
\def\DZ{\mathcal D_Z}
\def\DZodd{\mathcal D_Z^{\text{odd}}}
\def\DZeven{\mathcal D_Z^{\text{even}}}
\def\Dinf{\mathcal D_{\infty}}
\def\Dinfodd{\mathcal D_\infty^{\text{odd}}}
\def\Dinfeven{\mathcal D_\infty^{\text{even}}}

\def\L{\mathcal L}
\def\K{\mathcal K}

\def\bfL{{\bf L}}
\def\bfLEul{{\bf L}^{\textup{Eul}}}
\def\bfN{{\bf N}}
\def\bfP{{\bf P}}
\def\E{\mathcal E}
\def\U{\Phi}
\def\X{\mathcal X}
\def\Y{\mathcal Y}

\def\Phis{\Phi_-}
\def\Phiu{\Phi_+}
\def\tG{\tilde{G}}
\def\tH{\tilde{H}}
\def\a{\mathfrak{a}}

\def\phij{\phi_{2j}}

\def\bD{\widehat{\mathcal D}}
\def\X{\mathcal X}
\def\CLP{C_{\textup{LP}}}
\def\bG{\widehat{G}}

\def\PhiT{\Phi_{\text{in}}^T}
\def\tE{{\mathcal{E}}}
\def\Om{\Omega}
\def\RR{\mathfrak{R}}
\def\Pb{\mathbb P}

\def\sin{s_T}
\def\I{\mathcal I}
\def\tpsi{\widetilde{\psi}}

\def\rhoLPT{\varrho_{\text{LP,T}}}
\def\uLPT{u_{\text{LP,T}}}

\def\sg{{\Lambda_0}}

\def\bct{\begin{color}{teal}}
\def\m{\mathfrak{m}} 
\def\HmZm{\mathcal H^{2\m}_{Z_0}}
\def\JDN{2 [\frac{j}{2}]} 
  
\def\JUP{2 [\frac{j+1}{2}]} 
\def\KUP{2 [\frac{k+1}{2}]}
\def\IUP{2 [\frac{i+1}{2}]}

\newtheorem{theorem}{Theorem}[section]
\newtheorem{definition}[theorem]{Definition}
\newtheorem{proposition}[theorem]{Proposition}

\newtheorem{lemma}[theorem]{Lemma}
\newtheorem{remark}[theorem]{Remark}

\newcounter{dummy}
\makeatletter
\newcommand\myitem[1][]{\item[#1]\refstepcounter{dummy}\def\@currentlabel{#1}}
\makeatother

\allowdisplaybreaks[2]

\numberwithin{equation}{section}

\begin{document}

\title{Nonlinear stability of the Larson-Penston collapse}

\author{Yan Guo\thanks{Division of Applied Mathematics, Brown University, Providence, RI 02912, USA, Email: Yan\_Guo@brown.edu.}, \ Mahir Had\v zi\'c\thanks{Department of Mathematics, University College London, London WC1E 6XA, UK. Email: m.hadzic@ucl.ac.uk.}, \ Juhi Jang\thanks{Department of Mathematics, University of Southern California, Los Angeles, CA 90089, USA, and Korea Institute for Advanced Study, Seoul, Korea.  Email: juhijang@usc.edu.}, \ and Matthew Schrecker\thanks{Department of Mathematics, University of Bath,  Email: mris21@bath.ac.uk}}

\date{}
\maketitle
\abstract{
We prove nonlinear stability of the Larson-Penston family of self-similarly collapsing solutions to the isothermal Euler-Poisson system. Our result applies to radially symmetric perturbations and it is the first full nonlinear stability result for radially imploding compressible flows. At the heart of the proof is the ground state character of the Larson-Penston solution, which exhibits important global monotonicity properties used throughout the proof.

One of the key challenges is the proof of mode-stability for the non self-adjoint spectral problem which arises when linearising the dynamics around the Larson-Penston collapsing solution. To exclude the presence of complex growing modes other than the trivial one associated with time translation symmetry, we use a high-order energy method in low and high frequency regimes, for which the monotonicity properties are crucially exploited, and use rigorous computer-assisted techniques in the intermediate regime. In addition, the maximal dissipativity of the linearised operator is proven on arbitrary large backward light cones emanating from the singular point using the global monotonicity of the Larson-Penston solutions. Such a flexibility in linear analysis also facilitates nonlinear analysis and allows us to identify the exact number of derivatives necessary for the nonlinear stability statement. The proof is based on a two-tier high-order weighted energy method which ties bounds derived from the Duhamel formula to quasilinear top order estimates. To prove global existence we further use the Brouwer fixed point theorem to identify the final collapse time, which suppresses the trivial instability caused by the time-translation symmetry of the system.
}

\tableofcontents

\section{Introduction and the main result}

A fundamental model of a self-gravitating isolated star is given
by the compressible Euler-Poisson system~\cite{Chandrasekhar38}.
The unknowns are the star density $\varrho$, the velocity vector field ${\bf u}$, the pressure $p$, and the gravitational 
potential $\Phi$. We work under the assumption of radial symmetry and therefore assume that the velocity vector field takes the form
${\bf u}(t,{\bf x})=u(t,R)\frac{{\bf x}}{R}$, $R=|{\bf x}|$, where $u(t,R)$ is the scalar-valued radial velocity component. Moreover, in radial symmetry, the gravitational field 
can be expressed as a nonlocal operator of the gas density $\varrho$ through the standard formula
\begin{align}\label{E:FORCE}
\pa_R \Phi(t,R) = \frac{m(t,R)}{R^2}, \ \ m(t,R) & = \int_0^R 4\pi \sigma^2 \varrho(t,\sigma)\,d\sigma,
\end{align}
where we have expressed $\Phi(t,\cdot)$ and $\varrho(t,\cdot)$ as functions of the radial variable $R$. We assume further that the fluid is isothermal, i.e.~the pressure $p$ equals the density $\varrho$.\footnote{If one assumes $p=k\rho$ for some constant $k>0$, by a simple scaling argument one may set $k=1$ without loss of generality.}
The resulting radial isothermal Euler-Poisson system reads
\begin{align}
\pa_t \varrho + \Big(\pa_R+\frac2R\Big)(\varrho u) & = 0, \label{E:ECONT}\\
\varrho \left(\pa_t u + u\pa_R u\right) + \pa_R \varrho + \varrho \frac{m(t,R)}{R^2}
&= 0, \label{E:EMOM}
\end{align}
where~\eqref{E:ECONT} expresses the conservation of mass and~\eqref{E:EMOM} the conservation of momentum. 
To formulate the initial-value problem for~\eqref{E:ECONT}--\eqref{E:EMOM}, we further specify the initial data at time $t=-1$ as
\begin{align}\label{E:EULERINITIAL}
\varrho(-1,\cdot) = \varrho_0, \ \ u(-1,\cdot) = u_0.
\end{align}

Particularly important in astrophysics 
is the question of dynamic stellar collapse driven by density blow-up, which 
is also referred to as stellar implosion. A commonly used tool in the study of implosion is to search for  
{\em self-similar} blow-up solutions, which correspond to an ansatz of the form
\begin{align}
\varrho_{\text{LP},T}(t,R) & = \frac1{2\pi(T-t)^2}\rhoLP\Big(\frac{R}{T-t}\Big), \label{E:LPRHO} \\
u_{\text{LP},T}(t,R) & = \widehat{u}\Big(\frac R{T-t}\Big),\label{E:LPU}
\end{align}
which honours the scaling symmetry of the isothermal Euler-Poisson system~\eqref{E:ECONT}--\eqref{E:EMOM}.

In 1969, in their seminal works, Penston~\cite{Penston1969} and Larson~\cite{Larson1969} independently provided numerical evidence for the existence of a smooth self-similar profile corresponding to an implosion solution of the form~\eqref{E:LPRHO}--\eqref{E:LPU}, referred to as the {\em Larson-Penston} (LP) collapse or solution. In 1977, Hunter~\cite{Hunter77} numerically found  further smooth self-similar implosion profiles of the form~\eqref{E:LPRHO}--\eqref{E:LPU}, raising the question of whether the LP-collapse plays a role analogous to a ``ground state'' in a discrete family of self-similar collapse profiles. See also~\cite{Shu77,Whitworth85} for further numerical discussion of these and other possible self-similar implosions for an isothermal fluid. 

The existence of the LP-collapse was rigorously shown by Guo, Had\v zi\'c, and Jang~\cite{GHJ2021b}, who proved that 
the LP-profile $(\rhoLP,\widehat u)$ is real-analytic as a function of the self-similar variable $y=\frac{R}{T-t}\in[0,\infty)$. 
In this paper, we rigorously address the problem of nonlinear asymptotic stability of the 1-parameter family of the LP collapsing solutions, 
where the blow up time $T\in\mathbb R$ in~\eqref{E:LPRHO}--\eqref{E:LPU} serves as the parameter. 
The main result of this work is the following informally stated theorem:


\begin{theorem}[Informal statement]\label{T:MAININFORMAL}
The 1-parameter family of Larson-Penston collapsing solutions~\eqref{E:LPRHO}--\eqref{E:LPU} is nonlinearly dynamically stable against radial perturbations, in the sense
that small perturbations in suitable weighted  Sobolev spaces  converge to a nearby Larson-Penston solution.
\end{theorem}


A precise formulation of Theorem~\ref{T:MAININFORMAL}, including the definition of stability and the 
specification of function spaces, can be found in Theorem~\ref{T:MAIN}, which is stated in Lagrangian coordinates. For completeness we also provide a version of the theorem stated in Eulerian variables, see Theorem~\ref{T:EULERMAIN}. We emphasise here that the function spaces in the precise statement specification allow for solutions of finite mass and energy and can be thought of as Sobolev spaces $H^k$ with suitable weights at infinity. In fact, we  allow for $k=6$, a feature of our proof discussed further below.  Our data are specified at time $t=-1$ and the blow up occurs at a time $T$, $|T|\ll1$.

Theorem~\ref{T:MAININFORMAL} is the first  result to 
show \emph{full} nonlinear stability for a fluid implosion problem against radial perturbations. In the general context of compressible fluids, the first example of smooth self-similar implosion for the compressible Euler system was given  in the seminal work of Merle, Rapha\"el, Rodnianski, and Szeftel~\cite{MRRS1,MRRS2}, who invented a general framework for the study of (finite codimension) nonlinear stability of such flows.
Various formal and numerical arguments in the physics literature suggested
that the LP-solution may be stable against radial perturbations~\cite{Ori88,Hanawa1997,Ha1998, HaMa2000,Maeda01,Ha2003,Harada03,Ha2004,Brenner98} and our result in particular provides a rigorous justification
for these claims.
We  show that the LP-solution is a local attractor for the dynamics, which makes it consistent with the {\em Similarity Hypothesis} picture found in astrophysics discussions~\cite{Ha2003,Ha2004}.
The known smooth implosion solutions in compressible fluid mechanics~\cite{MRRS1,MRRS2,BuCaGo2025} are known to be finite-codimension stable, but with unknown codimension, and hence
are not known to be stable.

The question of implosion for self-gravitating gases has naturally attracted a lot of attention in the physics literature. The simplest such solutions are obtained when there is no pressure in the system; such gases are referred to as dust. Due to their simplicity, self-gravitating dust profiles can be solved for explicitly in radial symmetry using comoving coordinates. While such solutions played a significant role in building up intuition on the nature of gravitational collapse, both in Newtonian as well as relativistic setting~\cite{Lemaitre1933,To1934,OpSn1939,Ch1984}, they do not account for the key leading order effect associated with gas dynamics: the pressure.

In addition to the above-mentioned works on isothermal self-gravitating fluids, 
in the case of the general adiabatic equation of state $p=\rho^\ga$, $\frac65<\ga<\frac43$, Yahil~\cite{Yahil83} gave numerical evidence for the existence of the analogues of the LP-type collapse. The existence of the Yahil profiles was proved by the authors of this article~\cite{GHJS2022} in the full supercritical range $1<\gamma<\frac43$, exploiting subtle nonlinear invariances of the associated ODE system and developing an ad hoc shooting method. It is well known that smooth implosion is impossible in the subcritical range, $\ga>\frac43$, e.g.,~\cite{Deng02}, while in the critical case, $\ga=\frac43$, a partial decoupling of the equations leads to the expanding and contracting  Goldreich-Weber solutions,~\cite{Goldreich80}.  Recently Sandine~\cite{Sandine2024,Sandine2025} rigorously addressed the ``highly oscillatory'' regime of the Hunter family of solutions, establishing their existence in both the isothermal ($\ga=1$) and polytropic case with $1<\ga<\frac65$. An additional implosion mechanism for the mass-supercritical EP-system ($1<\ga<\frac43$) was introduced in~\cite{GHJ2021a}, wherein the existence of ``dust-like'' imploding 
stars was shown. These solutions are {\em not} self-similar in the sense of~\eqref{E:LPRHO}--\eqref{E:LPU}.  In the fully relativistic setting, Guo, Had\v zi\'c, and Jang~\cite{GHJ2023} have shown the existence of the relativistic LP-solution, originally predicted numerically by Ori and Piran~\cite{OP1990}. This collapse profile solves the Einstein-Euler system and has important causal implications as it leads to the formation of naked singularities from smooth initial data.

In the absence of gravity, the study of self-similar radial solutions to the compressible Euler equations has a long history and goes back to the studies of self-similar shock waves by Guderley~\cite{Guderley42}, von Neumann~\cite{vonNeu47}, Sedov~\cite{Sedov46}, and Taylor~\cite{Taylor41}. Guderley constructed solutions (with numerical methods) to the full Euler equations describing the self-similar implosion of a spherical shock wave onto the origin in finite time, with the strength of the shock and certain fluid variables blowing up at time of collapse. The Guderley shock wave solution can be continued beyond the time of collapse with a further expanding shock wave that is again self-similar. These solutions were only recently constructed rigorously in~\cite{Jang24,Jang25} (and see further references within). The von Neumann--Taylor--Sedov blast wave concerns only an outgoing self-similar spherical shock wave, and has also been the subject of much study.

Smooth implosion solutions to the radial compressible isentropic Euler equations were first rigorously constructed and shown to be finite-codimension stable in a series of breakthrough works of Merle, Rapha\"el, Rodnianski, and Szeftel~\cite{MRRS1,MRRS2,MRRS3}. Whether the implosion profiles from~\cite{MRRS1} are stable is not clear; for a numerical work in this direction see~\cite{Biasi}. The works~\cite{MRRS1,MRRS2} inspired several extensions and refinements: see, for example,~\cite{BuCaGo2025,ShWeWaZh2025}. For a finite codimension stability result for nonradial perturbations see~\cite{CaGoShSt2024}, and for results on vorticity blow up for Euler equations see~\cite{ChCiShVi2024,Ch2024}.

In this work, the most significant obstacle that must be overcome to establish {\em full} stability (by contrast to finite co-dimension stability) occurs at the level of linear analysis. Here the EP system, linearised around the LP profile in suitable self-similar variables, gives rise to an unbounded, non-self-adjoint operator, which requires the introduction of significant new ideas to prove the full mode stability. Such operators occur naturally in many stability problems for hyperbolic PDE, see for example the aforementioned Euler implosion results of~\cite{MRRS2,MRRS3}, singularity formation for wave maps and semilinear wave equations~\cite{Costin17,Do2011,DoSc2016}, hyperbolic Yang-Mills equations~\cite{Glogic22}, or in the stability of the Kerr solution to the vacuum Einstein equations~\cite{Te2020}.  

Proving mode-stability for unbounded non-self-adjoint operators is in general a very challenging problem, as eigenvalues can spread throughout the complex plane and there are few if any abstract spectral techniques that one may employ. Such problems occur very naturally already at the semi-linear level where the background profile is typically explicit; see for example~\cite{Do2011,Glogic22,Te2020} and the very nice concise review by Donninger~\cite{Do2024} of this topic. A lot of effort has been invested in recent decades in understanding mode stability around such explicit profiles. In particular, if 
the profile has certain global analyticity properties and a particular algebraic structure (specifically, if it is a rational function of the self-similar variable), a technique developed by Costin, Donninger, Glogi\'c, and Huang~\cite{CoDoGlHu2016,Costin17, Glogic2018,Glogic22}
can be used to exclude nontrivial growing modes. In the present problem, however, this approach a priori does not work. The solution itself is not explicit, nor it is clear what its {\em global} analyticity properties are (it is only known to be locally real-analytic).

A key achievement of the present study is that we develop a new schematic approach to address the central difficulty of mode stability, specifically the possibility of complex unstable modes. Considering the eigenvalue problem, a second order, complex coefficient, singular ODE, we make the fundamental observation that commuting the ODE with a suitable differential operator induces a damping effect that leads to good sign conditions in the coefficients. This damping is especially effective for eigenvalues with large imaginary part and appears to have a universal character, so that we expect this feature to appear in a wide range of problems in singularity formation, including other fluid implosion problems. This damping allows us to prove energy identities that unify the real and imaginary parts into a single coercive estimate. This estimate is sufficient to exclude unstable eigenvalues with large imaginary part. Excluding eigenvalues with large real part is more straightforward, and leads to the non-existence of eigenvalues with positive real part except in a compact region of the complex plane.
We then employ interval arithmetic, a rigorous form of computer-assisted proof, to exclude potential unstable modes in this remaining compact set.    In the context of the present work, the monotonicity of the LP solution enables us to quantify the order of derivatives required to perform this analysis, but we emphasise that the structural feature on which the argument relies, i.e.~the damping effect, appears to hold for a wide family of problems, and is not specific to LP.

A second key feature of our approach is  another use of the monotonicity of the LP solution  to prove maximal dissipativity of the linearised operator on arbitrarily large backwards cones from the singular point, $r\leq Z_0(T-t)$. As in~\cite{MRRS2}, we develop an approach that ties together a high order energy method that avoids derivative loss with a lower order decay on backwards cones arising from the Duhamel formula and linear stability analysis. Compared with earlier strategies for combining interior and exterior estimates, our technique enables us to get negative powers of $Z_0$ as weights in nonlinear estimates, and so use a unified global energy structure and framework to significantly streamline sections of the nonlinear analysis.

A third novelty arises through our use of Lagrangian variables. Due to the behaviour of the Larson-Penston Lagrangian flow map $\bzeta$ in the far-field, we are forced to incorporate suitable weights in our energy scheme to close nonlinear estimates. A naive approach to selecting these weights falls afoul of the failure of critical and super-critical Hardy-Sobolev inequalities, and so we develop a hierarchy of weights adapted to the growth rates of the Lagrangian flow map and its derivatives in order to be able to close the estimates. This is a generic feature of the Lagrangian framework and we expect it to be adaptable to other problems of this flavour.

Finally, through the use of monotonicity properties of the ground state LP solution we explicitly quantify sufficient regularity required for the nonlinear stability, in contrast to the existing works on self-similar implosion (e.g.~\cite{MRRS2} and others), which may require very high numbers of derivatives. We prove nonlinear stability in a weighted $L^2$-based topology requiring 6 derivatives of $\rho$ and $u$.

We now proceed to provide a detailed breakdown of these ideas and explain the high-end overview of the proof of the main theorem.  

\subsection{Main ideas of the proof and outline}


To address the stability question, we work in Lagrangian (or comoving) coordinates, which
in the context of radial symmetry naturally maps the Euler-Poisson system into a quasilinear 
scalar wave equation satisfied by the flow map function $\eta$. 
 The EP-system then reads
\begin{align}\label{E:ETAINTRO}
\pa_t^2 \eta - \frac{\pa_r^2 \eta}{(\pa_r\eta)^2 } - \frac{2}{\eta} + \frac{ M(r) }{\eta^2} + \frac{\pa_rg}{g}\frac1{\pa_r\eta}=0.
\end{align} 
Here the function $g(r)$ is a smooth, positive function that characterises the initial particle labelling $\eta_0$ at time $t=-1$ from~\eqref{E:EULERINITIAL} via the relation
\begin{align}\label{E:DATAININTRO}
\varrho_0(\eta_0(r)) \eta_0^2 (r)\pa_r \eta_0(r) = g(r), \ r\ge0.
\end{align}
As is commonly the case when studying the stability of self-similar solutions, we pass to 
similarity coordinates $(s,z)$ defined precisely below in~\eqref{E:SCALINGLAGR}  and interpret the LP-profile as a steady state of a suitably rescaled 
Lagrangian formulation of the EP-system, which  reads (see Lemma~\ref{L:SELFSIMLAG} below)
\begin{align}
\pa_s\begin{pmatrix} \zeta \\ \mu\end{pmatrix} = \begin{pmatrix}  \mu - z\pa_z \zeta + \zeta  \\
-z\pa_z \mu  - \pa_z \left( \frac1{\zeta_z}\right) + \frac2{\zeta}-\frac{\tilde M}{\zeta^2}-\frac{\tilde g_z}{\tilde g}\frac1{\zeta_z}\end{pmatrix}, \label{E:ZETAMUINTRO}
\end{align}
where
\[
\zeta(s,z) := (T-t)^{-1}\eta(t,r), \ \ \mu(s,z) := \pa_t\eta(t,r). 
\]
The freedom to specify the function $g(r)=\tilde g(s,z)$ is the comoving gauge freedom in the problem. Via~\eqref{E:DATAININTRO}, it relates the choice of the initial density to the the choice of initial labelling. By choosing $g\in L^1(\mathbb R)$, we ensure that the total initial mass $\int_{0}^\infty \varrho_0(R)R^2 \diff R=\int_0^\infty g(r)\diff r$ is finite, which is a physical requirement on the model, see Remark~\ref{R:FINITEMASS}. On the other hand, in order to 
interpret the LP-solution as a steady state in the Lagrangian formulation, we must make a specific choice of the labelling 
\be\label{E:GLPINTRO}
g_{\text{LP}} = (4\pi)^{-1}\CLP,
\ee
so that the associated  LP flow map $\bzeta$ solves the self-similar steady state equation
\be\label{E:LPLAGINTRO}
 \Big( z^2 - \frac{1}{( \pa_z \bzeta)^2 } \Big) \pa_z^2\bzeta +  \frac{\CLP z}{\bzeta^2} - \frac{2}{\bzeta}=0.
\ee
Already at this stage we see the role of the \emph{sonic point} $z_*$ at which $ z^2 - \frac{1}{( \pa_z \bzeta)^2 }=0$ (compare Definition~\ref{def:sonic} below). 

Condition~\eqref{E:GLPINTRO} is of course incompatible with the requirement of finite total mass. It simply reflects the fact that $\int_{0}^\infty\varrho_{\text{LP}}(t,r) r^2\diff r=\infty$ due to the large tails of $\varrho_{\text{LP}}(t,r)$ which decay like $r^{-2}$ as $r\to\infty$, for any fixed $t$. Therefore, to formulate the stability question, we choose the gauge function $g$ in~\eqref{E:ETAINTRO} to agree with $g_{\text{LP}}$ on a large finite region $r\in[0,r_\ast]$ and then we make it decay sufficiently fast for $r\gg r_\ast$. This ensures on one hand that in a large region which contains the backward sound cone, the LP-solution is indeed a steady state of~\eqref{E:ZETAMUINTRO}, while in the asymptotic region $r\gg r_\ast$ the equation incurs
an effective ``source" term coming from $g$, which however is small and has no bearing on the "interior" dynamics.

The second important feature of the  Lagrangian variables  highlights in a crucial way  intrinsic regularity discrepancies
between the behaviour of the solution at the centre $z=0$ and $z\gg1$, viewed as a function of
the fluid label $z$. Specifically, it is easy to show that 
\[
\bzeta(z)\sim_{z\to0} z^{\frac13}, \ \ \bzeta(z)\sim_{z\to\infty} z.
\]
In fact, the analyticity of the LP-solution at the origin in the Eulerian variables~\cite{GHJ2021b} shows that $\bzeta$ expands analytically in the odd powers of $z^{\frac13}$ near $z=0$  (see Lemma~\ref{SS:SSP}), while
away from $z=0$, the flow map $\bzeta$ is real-analytic as a function of $z$. To handle this fractional regularity, we work in a framework of weighted operators $\bp=z^{\frac23}\pa_z$ and $\bd=\pa_z(z^{\frac23}\cdot)$, see Definition~\ref{D:KEYOP}.

To prove the stability of the LP-solution, 
we consider the perturbation 
\begin{align}
\begin{pmatrix} \th \\ \phi\end{pmatrix} =\begin{pmatrix} \zeta \\ \mu\end{pmatrix}-\begin{pmatrix} \bzeta \\ z\pa_z\bzeta-\bzeta\end{pmatrix}.
\end{align}

{\em Linearised problem.}
In order to show stability we must show that the linearisation around the exact LP-profile in the interior region does not produce any nontrivial growing modes. Linearising
the problem in a region $z\le Z_0$ the linearised dynamics takes the form
\begin{align}\label{E:LININTRO}
\pa_s\Phi = \bfL \Phi, \ \ \Phi = \begin{pmatrix} \th \\ \phi\end{pmatrix}, 
\end{align}
where $\bf L$ is a non-selfadjoint linear operator (see Lemma~\ref{L:GPROPERTIES}). The time-translation symmetry of
the LP-solution gives rise to a trivial growing mode (see Lemma~\ref{L:GMODE})
\begin{align}
\lambda_{\text{trivial}}=1.
\end{align} 
In order to prove the stability of the LP solution, we must therefore show that $\lambda_{\text{trivial}}$ is indeed the only eigenvalue with nonnegative real part. 
This is the crucial and hardest step in the proof of Theorem~\ref{T:MAIN}.
In the literature on stability of self-similar collapsing profiles, this type of linear stability statement is often referred to as {\em mode-stability}, which is a terminology we adopt in this work as well.

It turns out that it is not too hard to prove that there are no unstable eigenvalues $\la$ with $\Re \la>1$; a form of an energy-type identity excludes such modes, see Lemmas~\ref{L:STRIPONLY}--\ref{L:REALGEQONE}.
Therefore all (hypothetical) growing modes must be confined to the strip $\mathcal I = \{\la\in\mathbb C\,\big|\, 0\le \Re\la\le1\}$. To show that $1$ is the only eigenvalue in $\mathcal I$,  
we develop a strategy that combines an ad-hoc high-order energy method and rigorous computer-assisted proofs (by means of interval arithmetic).
Concretely, we devise a divide-and-conquer strategy that splits the strip $\mathcal I$ into three regions:
\begin{enumerate}
\item[(1)] Small frequencies: 
\[
\mathcal I_{\text{low}} : = \{\la\in\I\,\big| \, |\Im\la|\le b_0\};
\]
\item[(2)] High frequencies: 
\[
\mathcal I_{\text{high}} : = \{\la\in\I\,\big| \, |\Im\la|\ge b_1\};
\]
\item[(3)]
Intermediate frequencies
\[
\I_{\text{inter}} : = \{\la\in\I\,\big| \, b_0\le |\Im\la|\le b_1\},
\]
\end{enumerate}
where $0<b_0<b_1$ are suitable control parameters. The proof then runs on contradiction.

The proof of absence of growing modes for small and high frequencies $\mathcal I_{\text{low}}$ and $\I_{\text{high}}$ relies on two key ideas. As the eigenvalue equation is a complex coefficient, second order ODE with singular points, we must handle sign indeterminacies for the both the real and imaginary parts of the equation. 
For example, to exclude a high-frequency growing mode in the region $\I_{\text{high}}$, we first make the fundamental observation in Proposition~\ref{prop:A4B4} that commuting the eigenvalue equation with the differential operator $\big(\pa_y(\pa_y+\frac{2}{y})\big)^m$ creates a gap by shifting the key first order coefficient by $2m\frac{\wLP'}{\wLP}<0$, where $\wLP$ is a natural weight equivalent to the distance to the sonic point. We secondly combine the real and imaginary parts of the natural energy identities to obtain effective energy inequalities that may be used to derive a contradiction. Due to the crucial sign condition, we are able to overcome an indeterminate sign in a top order coefficient in the natural energy identity for the eigenvalue equation, and so to derive an energy inequality of the schematic form:
\beqa
-&\int \chi |D_y\Psi_\l|^2+\int |\Im\l|^2 \chi{H_\l} |\Psi_\l|^2\geq \int \chi \tilde H_\l |\Psi_\l|^2,
\eeqa
where $\chi$ is a suitable weight function, $H_\l$ and $\tilde H_\l$ are coefficient functions of the LP solution, and $\Psi_\l$ is a suitable derivative of a hypothetical eigenfunction. We then prove that $H_\l\leq -c<0$ and $|\tilde H_\l|$ is bounded independently of $|\Im\l|\gg1$ in order to deduce a contradiction. As a feature of the monotonicity of the LP solution, we are in fact able to quantify the $m$ required to obtain this coercive estimate, and we take $m=2$. To make precise the values of $b_1$ and $b_0$ for which we may apply this argument,
we use Interval Arithmetic (IA) to determine that the numerical values $b_0=\frac15$  and $b_1=8$ (determined
in Propositions~\ref{P:NOSMALLIMAG} and~\ref{P:NOLARGEIMAG} respectively) are sufficient to conclude the contradiction.

To exclude the presence of growing modes in the intermediate region $\I_{\text{inter}}$, we must employ a different method. We use
 ODE arguments to show that a hypothetical eigenfunction cannot exist due to global incompatibilities imposed by the values of Frobenius indices at the end-points of the interval $[0,z_*]$. At this stage, we see one of the key roles of the sonic point $z_*$ in the linearised equation. The degeneracy of the linear operator at this sonic point imposes a constraint on regular solutions (in particular on sufficiently smooth eigenfunctions) that restricts the dimension of the potential space of eigenfunctions. More precisely, any eigenfunction must satisfy a particular second order ODE with coefficients that are singular at both the origin and sonic point (see Proposition~\ref{L:EVALUEPROB} below). By studying the local properties of solutions around each of the singular points, we find that the regularity requirements imposed by the function space in which we perform our stability analysis restricts the space of solutions locally around each of these points to a one-dimensional space. By employing an interval arithmetic ODE solver, we are then able to prove that solutions which are regular at the sonic point are singular at the origin and vice versa, thus deducing the non-existence of eigenvalues in $\I_{\text{inter}}$.
  This is the content of Proposition~\ref{P:NOINTERIMAG}. We emphasise that Propositions~\ref{P:NOLARGEIMAG},~\ref{P:NOSMALLIMAG}, and~\ref{P:NOINTERIMAG} rely crucially on Interval Arithmetic techniques.

Interval arithmetic is a rigorous form of computer assisted proof that has seen substantial recent application in the theory of PDE, e.g.,~\cite{BuCaGo2025,Castro20,Chen21,Chen22,Cohen24}, as well as solving a number of important open conjectures, e.g.,~\cite{Gabai03,Hales05,Hass00}. For further references, see the survey~\cite{GS18}. It  replaces point values (which may not be exactly machine representable) by closed intervals whose end-points are representable in order to perform arithmetic operations. By doing so, the result of any arithmetic computation is an interval that is guaranteed to include the true result. In this work, the main package we use is the \verb!VNODE-LP! package for rigorous ODE solving, the website and documentation for which can be found at~\cite{Nedialkov10a,Nedialkov10b}. Details of the implementation of the interval arithmetic strategy are contained in Appendix~\ref{APP:IA}.

\begin{figure}
\begin{center}
\begin{tikzpicture}
\begin{scope}[scale=0.7, transform shape]

\coordinate [] (A) at (-3,7){};

\node at (-3.6,7) {$t=T$};

\coordinate [label=below:$r$] (B) at (9,7){};

\coordinate [label=below:$\text{Sonic line}$] (D) at (-1.8,5){};

\node at (-1.7,5.8) {$z=z_\ast$};

\coordinate [label=below:$\text{Accretive line}$] (E) at (4,5){};

\node at (2.2,5.95) {$z=Z_0\gg z_\ast$};

\coordinate [] (F) at (6.5,5){};

\node at (8.4,6) {Dampening region};

\coordinate [] (H) at (-3,5){};

\node at (-3.7,5) {$t=-1$};

\coordinate [label=left:$t$] (G) at (-3,8.3){};

\draw[very thick] (A)--(9,7);

\draw[very thick] (A)--(-3,4);

\draw[very thick] (A)--(-3,8.5);

\draw[thick] (A)--(D);

\draw[thick] (A)--(E);

\coordinate [label=above:$r_\ast$] (F1) at (6.5,7){};

\draw[thick] (F1)--(F);

\draw[thick,dashed] (H)--(9,5);

\draw[fill=white] (-3,7) circle (3pt);
\end{scope}
\end{tikzpicture}
 \caption{\small Sonic cone, accretive cone, and the dampening region - schematic picture}
 \label{F:LP}
\end{center}
\end{figure}
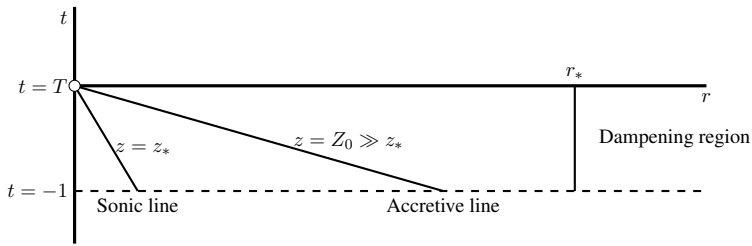

To obtain the corresponding semi-group estimate for the linearised flow we next show that our operator satisfies the
assumptions of the Lumer-Phillips theorem. Concretely, we prove that 
the operator $\bfL$ is maximally accretive
on a suitable high-order Sobolev space in a domain $[0,Z]$ for any choice of $Z>z_\ast$. In doing so, we rely roughly speaking on a
splitting of the type
\[
\bfL = \bfL_0 + \tilde{\bfL},
\]
where $\bfL_0$ is the top order  part of $\bfL$.  After taking derivatives, we show that $\bfL_0$ gains coercivity and it does so on
any arbitrary domain $[0,Z]$ which contains the backward sound cone ($Z>z_\ast$). To prove this we exploit in a fundamental way 
certain monotonicity properties of the LP-profile itself. We contrast this to
the recent series of works on finite-codimension stability of imploding self-similar solutions to compressible
Euler~\cite{MRRS2}, wherein such an accretivity result can only be shown in domains marginally larger than the width of the background sound cone (referred to in the cited work as the repulsivity property of the underlying self-similar profile). See also the development of a global approach in weighted spaces in~\cite{ChCiShVi2024},
as well as~\cite{Glogic24,Glogic25,Kim2024} in the semilinear setting.

The role of the sonic line is fundamental in this analysis. This line marks the backwards acoustical cone from the singularity, and gives the boundary of the domain of dependence of the singular point. It is therefore essential that estimates can be shown across this line to connect the interior behaviour (which directly influences the blowup) with exterior behaviour, where the domain of influence remains strictly away from the singularity. The degeneracy in the PDE system across the sonic line is a manifestation of the fact that as one crosses the sonic line, the vector field $\pa_s$ changes from timelike to spacelike across with respect to the naturally induced acoustical metric. Our choice of variables and exploitation of the monotonicity of the LP solution significantly simplifies the connection problem as well as clarifying the degree of regularity required to prove coercivity in the region $\{0\leq z\leq Z\}$ containing the backwards acoustical cone. 

This coercivity property, together with statements that $\tilde{\bfL}$ is  a  compact perturbation of $\bfL_0$ in suitable spaces, and that
the operator $\bfL$ is mode stable, leads via the Lumer-Phillips theorem to the crucial linear bound
\begin{align}\label{E:LPINTRO}
\|e^{\bfL s}({\bf I}-{\bf P})\Phi\|_{\HmZ} \le e^{-\sg s} \|\Phi\|_{\HmZ}, \ \ \Phi \in \HmZ,
\end{align}
for some $\sg >0$. Here ${\bf P}$ is the projector onto the 1-dimensional eigenspace generated by the trivial eigenvalue 1 and the Hilbert spaces $\HmZ$, defined below in Definition~\ref{D:HILBERT}, can be thought of as equivalent to Sobolev spaces $H^{2m+1}(0,Z)\times H^{2m}(0,Z)$. This is the content of Theorem~\ref{T:LINEARMAIN}
and its proof is contained in Sections~\ref{S:LINEARMODE} and~\ref{S:ACCRETIVITY}.


{\em Nonlinear analysis.}
The proof of nonlinear stability relies on a two-tier high-order energy framework, which is detailed in Sections~\ref{S:DUHAMEL}--\ref{S:POINTWISE}.
To carry it out, we introduce an index $\m\in\mathbb N$ which is used to keep track of the number of derivatives in our function spaces.  
The low-order energy spaces are confined to the so-called ``interior" region $[0,Z_0]$, $Z_0\gg z_\ast$ and correspond to spaces $\HmZm$ introduced in Section~\ref{S:LOW}. The high-order energy spaces contain two more derivatives and are equivalent to suitably weighted Sobolev norms on the {\em whole}
semi-infinite interval $[0,\infty)$. They are given by the energies $\tE_{\le 2(\m+1)}$ introduced in Section~\ref{S:HIGH}.

Section~\ref{S:DUHAMEL} is devoted to the  low-order energy bounds.
The idea is to use the Duhamel formula
\begin{align}\label{E:DUHAMELINTRO}
\Phi(s) = e^{\bfL(s-\sin)} \Phi^T_{in} + \int_{\sin}^s e^{\bfL(s-\sigma)}{\bf N}[\Phi(\sigma)]\diff\sigma,
\end{align}
where $\bfN[\Phi]$ contains the nonlinearity in the problem and $s=s_T=\log(1+T)$ corresponds to the initial time slice $t=-1$.
Formulation~\eqref{E:DUHAMELINTRO} allows us to exploit the semi-group decay~\eqref{E:LPINTRO} modulo two key obstructions.
Firstly, one must mod out the trivial growing mode associated with time-translation invariance and secondly,~\eqref{E:DUHAMELINTRO} features a derivative loss, since the problem is quasi-linear. 
We deal with the first issue by projecting the dynamics into the unstable direction and the stable part, thus separating the two. By interpolating with the top order energy $\tE_{\leq 2(\m+1)}$, we absorb the derivative loss and arrive at a bound of the form
\begin{align}
	\|\Phi\|_{\HmZm}\leq &C \underbrace{\sup_{\sigma\in[\sin,s]}(e^{\Omega \sigma}\tE_{\leq 2(\m+1)}^{\frac12})}_{\text{derivative loss due to quasilinearity}}\sup_{\sigma\in[\sin,s]}(e^{\nu\sg \sigma}\|\Phi\|_{\HmZm}) e^{-(\Omega+\nu\sg ) s}  + \underbrace{e^{-\sg (s-\sin)}\|\Phi^T_{in}\|_{\HmZm}}_{\text{semi-group decay}}
	\notag \\
	&+C\underbrace{e^{s-\sin}\Big\|\mathbf{P}\Big(\Phi^T_{in}+\int_{\sin}^{S_T}e^{-\sigma} \mathbf{N}[\Phi]\,\dif \sigma\Big)\Big\|_{\HmZm}}_{\text{symmetry-induced instability}}. \label{E:LOWBOUNDINTRO}
\end{align}

The purpose of Section~\ref{S:ENERGYBOUNDS} is to estimate  $e^{2\Omega s}\tE_{2j}(s)$, $0\leq j\leq\m+1$. To that end, instead of using~\eqref{E:LININTRO}
we write the perturbation equations in the form (see Lemma~\ref{L:QUASI0})
\begin{align}
\pa_s\begin{pmatrix} \th \\ \phi\end{pmatrix} = \begin{pmatrix} - z\pa_z \theta + \theta+\phi \\
-z\pa_z \phi +  \frac{\bzeta_z^2}{\zeta_z^2} K \theta +\mathfrak S[\th]+ \mathfrak R[\theta]\end{pmatrix}, \label{E:PERTURBINTRO}
\end{align}
where the leading order operator $K$ is given by 
\[
K\th = \bp  \left((\bp\bzeta)^{-2} \bd \theta\right).
\]
The remainder term  $\mathfrak S$ 
corresponds to errors arising from the far-field cutoff and vanishes on the region $r\leq r_*$, so that the effects of asymptotic flattening at the level of the estimates are felt only in the far-away region $r\ge r_\ast\gg1$. The remainder term 
 $\mathfrak R$ contains either lower order linear-in-$(\th,\phi)^\top$ contributions or genuinely nonlinear errors. 

As usual with such quasilinear problems, the idea is to commute~\eqref{E:PERTURBINTRO} with suitable differential operators (in this case 
$(\bp\bd)^{j})$ and seek an energy identity. The quasilinearised formulation~\eqref{E:PERTURBINTRO} shares one important common trait with the linearised problem~\eqref{E:LININTRO}: {after} taking sufficiently many derivatives, the leading order operator produces a {\em damping} term, which leads to a good, stabilising effect at the top order.

In order to close these energy estimates, we must then handle all remainder terms, including the errors arising from commutation (which are typically linear in the unknown, but below top order in derivative count). An essential difficulty arises due to the growth and decay rates of the coefficients arising from the LP solution and its derivatives in the equation (recall $\bzeta\sim z$ in the far-field). This is an intrinsic feature of the Lagrangian variables. The presence of these weights in the system of equations leads to a failure of Hardy-Sobolev inequalities that would allow us to close estimates using standard techniques. 

To estimate the remainders, we therefore introduce two new, key strategies. The first is to devise a suitable hierarchy of weights $\chi_{2j}$ at each order of derivatives in the weighted energy functionals, compare~\eqref{E:CHIJDEF} below. The growth rate at infinity of $\chi_{2j}$ compensates for the expected behaviour of the coefficient multiplying $(\bp\bd)^j\phi$. The second new element of our strategy is to exploit the fact that the global monotonicities of the LP profile allow us to obtain damping in arbitrarily large backwards cones, $z\leq Z_0$. Thanks to the carefully designed weights, we are able to control the main nonlinear term in the exterior region $z\geq Z_0$ (see Proposition~\ref{P:EXT1}) via
\[C\Big(Z_0^{-\frac{\al+1}{2}+\de}\tE_{\leq 2j} + Z_0^{-\frac{\al+1}{2}+2\de j}\tE_{\leq 2j}^{\frac32}\Big)\]
for a suitable $\al>1$ and small $\de$ and, in the interior region $z\leq Z_0$, by
\[ C(Z_0)\|\Phi\|_{\mathcal{H}^{2j-1}_{Z_0}}\tE_{2j}^{\frac12}+C \tE_{\leq 2j}^{\frac32},\]
see Proposition~\ref{P:INT1}.
By exploiting the largeness of $Z_0$, we may absorb the challenging exterior terms onto the left in our energy estimate, compare Proposition~\ref{P:EE2}.

In conclusion, we show an a priori bound of the schematic form
\begin{align}
\pa_s\tE_{\le 2(\m+1)} + C_1\tE_{\le 2(\m+1)} \le C\|\Phi\|_{\HmZm}^2
+C (r_\ast e^{s})^{-2a}  + C \tE^{\frac32}_{\le 2(\m+1)}, \label{E:CALEINTRO}
\end{align}
where the first term on the right corresponds to quadratic, localised, lower order terms, the middle term arises from asymptotic flattening, and the final term is the usual trilinear bound.

The nonlinear bound~\eqref{E:CALEINTRO} is predicated on the availability of a good $L^\infty$ bound on the lower order quantities, such as $\|\frac{\th}{\bzeta}\|_{W^{1,\infty}}$ for example. Such bounds however do not follow directly from Hardy-Sobolev embeddings, and instead require a more delicate argument which relies on integrating the 
solutions backwards along the characteristics. This is a Lagrangian way to capture the finite speed of propagation in the system, which allows us to ``export" information
from inside the backward accretive cone $0\le z\le Z_0$ to outside of it.  A related difficulty arises in the Eulerian setting~\cite{MRRS2}. We address these $L^\infty$ bounds in Section~\ref{S:POINTWISE}.

Finally, in Section~\ref{S:MAINTHEOREM} we prove Theorem~\ref{T:MAIN}. The idea is to combine the two coupled a priori bounds~\eqref{E:LOWBOUNDINTRO} with~\eqref{E:CALEINTRO} to show that the solution exists globally. This however is possible only if the formally exponentially growing term $e^{s-\sin}\Big\|\mathbf{P}\Big(\Phi^T_{in}+\int_{\sin}^{S_T}e^{-\sigma} \mathbf{N}[\Phi]\,\dif \sigma\Big)\Big\|_{\HmZm}$ in~\eqref{E:LOWBOUNDINTRO} is suppressed. Arguing by contradiction, we assume that the maximal existence time $S_T$ is finite for all $|T|$ sufficiently small. 
We then use the Brouwer fixed point theorem, by analogy to 
classical stable-manifold constructions, to show that there exists a choice of $T$ that suppresses the exponential growth by showing that the problematic term above vanishes. For such a choice of $T$ the estimates~\eqref{E:LOWBOUNDINTRO} with~\eqref{E:CALEINTRO} lead to a strict improvement of the a priori bounds, thus extending the maximal existence interval by a finite amount - a contradiction. Details of the Brouwer fixed point argument are given in Proposition~\ref{P:BROUWER}, and the proof of the main theorem immediately thereafter.
The second part of Section~\ref{S:MAINTHEOREM} is devoted to the proof of Theorem~\ref{T:EULERMAIN}, which gives an interpretation of the stability result in Eulerian coordinates.


\medskip

{\bf Acknowledgments.}
Y. Guo acknowledges the support of the NSF grant DMS-2405051.
M. Had\v zi\'c acknowledges the support of the EPSRC Early Career Fellowship EP/S02218X/1.
J. Jang acknowledges the support of the NSF grant DMS-2306910.
M. Schrecker acknowledges the support of the EPSRC Post-doctoral Research Fellowship EP/W001888/1.



\section{Precise formulation of stability }


\subsection{Self-similar formulation and the Larson-Penston collapse}


Since the isothermal Euler-Poisson (EP) system~\eqref{E:ECONT}--\eqref{E:EMOM} is invariant under the unique rescaling 
\begin{align}
\varrho \mapsto  \lambda^{-2}\varrho (\frac t\l, \frac R\l), \ \ 
u  \mapsto u (\frac t\l, \frac R\l),\label{E:SCALING}
\end{align}
it is natural to formulate the stability problem in self-similar coordinates. 
We  define the self-similar coordinates $(s,y)$ through the requirement 
\begin{align}
\frac{ds}{dt} = \frac1{T-t}, \ \ y = \frac R{T-t},
\end{align} 
which leads to 
\begin{align}\label{E:SST}
s(t) := - \log (T-t).
\end{align}
Here the particular choice~\eqref{E:SST} normalises the initial $t=-1$ time slice (compare~\eqref{E:EULERINITIAL}) to correspond to the self-similar time 
\be\label{E:SINITIAL}
\sin=-\log(1+T).
\ee
We now introduce the self-similar density and modified velocity $(\rho,v)$ via
\begin{align}
\varrho(t,R) = \frac1{2\pi}(T-t)^{-2}\rho(s,y), \ \ u(t,R) = v(s,y) - y.
\end{align}
The EP-system~\eqref{E:ECONT}--\eqref{E:EMOM} transforms into
\begin{align}
\pa_s\rho +  D_y(\rho v) - \rho & = 0, \label{E:CONTSS}\\
\rho\left(\pa_sv+ v\pa_y (v-y)\right) +  \pa_y\rho + 2\frac{\rho M}{y^2} & = 0, \label{E:MOMSS}
\end{align}
where $M[\rho](s,y) := \int_0^y z^2 \rho(s,z)\diff z$ and $D_y : = \frac1{y^2}\pa_y\left(y^2 \cdot\right) = \pa_y + \frac2y$ is the 3-d divergence operator in radial symmetry.

The Larson-Penston (from now on LP) solution is a  steady state of~\eqref{E:CONTSS}--\eqref{E:MOMSS}, whose existence was proved in~\cite{GHJ2021b}. For convenience, we state this as a theorem.
\begin{theorem}[{\cite[Theorem 1.2]{GHJ2021b}}]\label{T:LP}
There exists a steady (i.e.~$s$-independent) solution of~\eqref{E:CONTSS}--\eqref{E:MOMSS}, the LP solution, denoted
\be\label{E:LPSOL}
(\rhoLP(y),\vLP(y)) = (\rhoLP(y),y\omLP(y)), \ \ y\in[0,\infty),
\ee
where $\rhoLP,\omLP$ are real-analytic functions on $[0,\infty)$ such that 
\begin{align}
&\rhoLP(0)>\frac13, \ \omLP(0)=\frac13 ,\\
& \lim_{y\to \infty}\omLP(y)=1, \ \lim_{y\to\infty}(y^2\rhoLP(y))\in(0,\infty).
\end{align}
The LP solution admits a unique sonic point $y_*\in(2,3)$ such that $\vLP^2(y_*)=1$. Moreover, it enjoys the crucial monotonicity properties
\begin{align}
\rhoLP'(y)<0,\quad \omLP'(y)>0,\quad y>0.\label{E:LPMONOTONICITY}
\end{align}
\end{theorem} 
We note that the monotonicity property~\eqref{E:LPMONOTONICITY} was proved in~\cite[Remark 4.16]{GHJS2022}.

\begin{remark}
The exact LP-solution has infinite total mass and energy. However, by the finite-speed of propagation property, it is only the region inside the backward acoustical cone that affects the implosion (compare Figure~\ref{F:LP} below), while the global quantities, such as the total mass and the total energy reflect the $y\to\infty$ tail behaviour of the solution. 
For this reason, the solutions described in Theorem~\ref{T:MAININFORMAL} feature data that are suitably flattened  at spatial infinity to ensure that the total mass and energy are finite; see Section~\ref{S:PRECISE}. 
\end{remark}


\subsection{Lagrangian formulation}


Problem~\eqref{E:CONTSS}--\eqref{E:MOMSS} can be viewed as the self-similar Eulerian formulation of the EP-evolution. We next explain the Lagrangian formulation of the stability problem for the LP-solution, which is at the heart of this paper. 
In moving to the Lagrangian formulation, we are motivated by the observation that the Lagrangian coordinates offer an immediate reformulation of the radially symmetric Euler-Poisson flow as
a system of quasilinear wave equations. In fact,
the reinterpretation of the background LP-solution~\eqref{E:LPSOL} itself in the comoving picture offers further insights into the nature of the implosion collapse. 

For a given smooth, strictly positive function $g:[0,\infty)\to(0,\infty)$, we consider a choice of initial particle labelling $\eta_0:[0,\infty)\to[0,\infty)$ satisfying
$\eta_0(0)=0$ and 
\begin{align}\label{E:DATAIN}
\varrho_0(\eta_0(r)) \eta_0^2 (r)\pa_r \eta_0(r) = g(r), \ r\ge0.
\end{align}
The freedom to specify $g$ is the gauge freedom associated with comoving coordinates.
We next provide the comoving formulation of the EP-system.


\begin{lemma}\label{L:SELFSIMLAG}
Let $(\varrho,u)$ be a classical solution to~\eqref{E:ECONT}--\eqref{E:EMOM} and let $\eta_0:[0,\infty)\to[0,\infty)$ be a choice of the initial particle labelling satisfying~\eqref{E:DATAIN}.
Then the flow map $\eta(t,r)$, which by definition solves the ordinary differential equation
\begin{align}
\pa_t\eta(t,r) & = u(t,\eta(t,r)), \\
\eta(-1,r) & = \eta_0(r),
\end{align}
solves the following partial differential equation:
\begin{align}
\pa_t^2 \eta - \frac{\pa_r^2 \eta}{(\pa_r\eta)^2 } - \frac{2}{\eta} + \frac{ M(r) }{\eta^2} + \frac{\pa_rg}{g}\frac1{\pa_r\eta}=0. \label{E:EPLAGR}
\end{align}
Moreover, upon introducing the self-similar change of variables
\begin{align}\label{E:SCALINGLAGR}
\frac{ds}{dt}= \frac{1}{T-t}, \ z= \frac{r}{T-t}, \quad  \zeta(s,z) = \frac{\eta(t,r)}{T-t},
\end{align}
the rescaled flow map $\zeta(s,\cdot)$ solves 
\begin{align}
\zeta_{ss} +2z\zeta_{s z} +\left(z^2-\frac{1}{\zeta_z^2}\right)\zeta_{zz} -\zeta_s + \frac{\tilde M(s,z)}{\zeta^2} - \frac2{\zeta}+\frac{\tilde g_z}{\tilde g}\frac1{\zeta_z} = 0,\label{E:ZETADAMPENED}
\end{align}
where
\begin{align}\label{E:MTILDEDEF}
\tilde M(s,z) : = \frac1{T-t} M(r) = \frac{4\pi}{T-t} \int_0^r g(r')\diff r' = 4\pi \int_0^{e^sz}\tilde g(s,z) \diff z, \ \ \tilde g(s,z) := g(r).
\end{align}
Now setting
\begin{align}\label{E:BIGPSIDEF}
\mu : = \zeta_s+\Lambda\zeta-\zeta,
\end{align}
then the pair $(\zeta,\mu)$ solves the following first order system
\begin{align}
\zeta_s & = \mu - \Lambda \zeta + \zeta,  \label{E:NLZETA}\\
\mu_s & = -\Lambda \mu  - \pa_z \left( \frac1{\zeta_z}\right) + \frac2{\zeta}-\frac{\tilde M}{\zeta^2}-\frac{\tilde g_z}{\tilde g}\frac1{\zeta_z}, \label{E:NLPSI}
\end{align}
where we have defined the scaling operator
\beq\label{E:LAMBDADEF}
\Lambda : = z\pa_z.
\eeq
\end{lemma}


\begin{proof}
Equation~\eqref{E:ZETADAMPENED} is the standard derivation of the Lagrangian formulation of the problem, see for example~\cite{Jang14}. The system~\eqref{E:NLZETA}--\eqref{E:NLPSI} is a first-order reformulation of~\eqref{E:ZETADAMPENED}.
\end{proof}


\begin{remark}
We note that 
\begin{align}
\eta(t,r) & = (T-t)\zeta(s,\frac{r}{T-t}), \notag\\
\pa_t\eta(t,r) & = - \zeta (s,z) + \zeta_s + \Lambda\zeta = \mu(s,z),
\end{align}
where we recall~\eqref{E:BIGPSIDEF}. Therefore $\mu$ simply corresponds to the particle velocity along the flow lines.
\end{remark}


\begin{remark}
Strictly speaking, for any choice of the labelling gauge function $g$ we obtain a different system of PDE~\eqref{E:NLZETA}--\eqref{E:NLPSI}. The freedom to specify  
$g$ is crucial in addressing the stability problem.
\end{remark}


We next explain how the LP-solution embeds into the comoving description and in particular what the associated gauge $g$ is. 


\begin{lemma}[Mapping between the Lagrangian LP-representation to the Eulerian LP-representation]\label{L:LPLAGRANGIAN}
In the Lagrangian variables, there exists a constant $\CLP>0$ such that the  LP-solution $\bzeta=\bzeta(z)$ corresponds to steady states of the equation~\eqref{E:ZETADAMPENED}, 
with the particular choice of the initial labelling
\begin{align}
g(r) = g_{\textup{LP}}(r) \equiv \frac{\CLP}{4\pi}.
\end{align}
In this case, $[0,\infty)\ni z\mapsto\bzeta(z)\in[0,\infty)$ solves
\be\label{E:LPLAG}
 \Big( z^2 - \frac{1}{( \pa_z \bzeta)^2 } \Big) \pa_z^2\bzeta +  \frac{\CLP z}{\bzeta^2} - \frac{2}{\bzeta}=0.
\ee
Moreover
the pair $(\rhoLP,\vLP)$ given by 
\begin{align}\label{E:LPPROP}
\rhoLP(y) : = \frac{\CLP z}{2y^2 \vLP}= \frac{\CLP}{2\bzeta^2\pa_z\bzeta}, \ \ \vLP(y) : = z\pa_z \bzeta,
\end{align}
with
\begin{align}
y: = \bzeta(z),
\end{align}
is a steady state solution of~\eqref{E:CONTSS}--\eqref{E:MOMSS}.
\end{lemma}


\begin{proof}
For $y$, $\rhoLP$ and $\vLP$ as in the statement, 
 using $z\pa_z = \vLP \pa_y$, we may write \eqref{E:LPLAG}  as 
\be
(\vLP^2 -1) (\vLP\pa_y \vLP - \vLP ) +2 \vLP^2 ( \rhoLP\vLP - \frac{1}{y})=0,
\ee
from which we see that 
\be\label{eq:v}
\pa_y \vLP = 1 - \frac{2\vLP}{y} + \frac{2\vLP^2 (\rhoLP -\omLP )}{1-\vLP^2},
\ee 
where we recall from~\eqref{E:LPSOL} that 
$
\omLP = \frac{\vLP }{y}. 
$ 
We also see that $\pa_y (y^2 \vLP \rhoLP) = \frac{\CLP}{2}\pa_y z = \frac{\CLP z}{2\vLP} = \frac{\CLP}{2\bzeta_z} = y^2\rhoLP $ from which we deduce that 
\be\label{E:RHOSS}
\pa_y \rhoLP = - \frac{2\vLP \rhoLP (\rhoLP -\omLP )}{1-\vLP^2}.
\ee
Equations~\eqref{E:RHOSS} and~\eqref{eq:v} correspond to the self-similar formulation of the EP-system (see~\cite{GHJ2021b} and~\eqref{E:RHOLP}--\eqref{E:OMLP}).
\end{proof}

As shown in~\cite{GHJ2021b}, there exists a unique value $y_\ast=\bzeta(z_\ast)\in(2,3)$ such that $z_\ast^2( \pa_z\bzeta(z_\ast))^2=1$. 
This is the sonic point which corresponds to the boundary of the backward acoustical cone (see Fig.~\ref{F:LP}). We continue to refer to its Lagrangian label $z_\ast$
as the sonic point, as in the following definition.

\begin{definition}[Sonic point]\label{def:sonic}
The unique value of $y_\ast\in(0,\infty)$ such that $\vLP(y_\ast)=1$ is called the {\em sonic point}.
By slight abuse of notation, we also refer to its Lagrangian label $z_\ast=\bzeta^{-1}(y_\ast)$ as the sonic point, as no confusion can arise.
Note that $z_\ast$ is also characterised as the unique value such that $z_\ast^2( \pa_z\bzeta(z_\ast))^2=1$. It follows that $z^2\pa_z\bzeta(z)^2<1$ for $z\in(0,z_*)$ and  $z^2\pa_z\bzeta(z)^2>1$ for $z>z_*$.
\end{definition}


\subsection{Function spaces and Energy Functionals}


As shown in Lemmas~\ref{L:ONETHIRD}--\ref{L:ZETABAR}, the self-similar LP flow map $\bzeta$ is not smooth in $z$, but only in $z^\frac13$ near $z=0$. 
Therefore, in order to study the dynamics of the flow map $\zeta$ near $\bzeta$, we must honour this intrinsic regularity 
at $z=0$.  In particular, the Eulerian coordinate $y\approx z^{\frac13}$ near the origin, and so the Eulerian gradient and divergence operators scale like $z^{\frac23}\pa_z$ and $\pa_z(z^{\frac23}\cdot)$, respectively. 
 This suggests a use of weighted differential operators in the $z$-variable to adjust to the $C^{0,\frac13}$-regularity at the origin and motivates the following definition.


\begin{definition}[Key differential operators]\label{D:KEYOP}
We define the operators $\bp$ and $\bd$
\begin{align}\label{E:DIVGRADDEF}
&\bp: = z^{\frac23}\pa_z, \ \ \bd: = \pa_z\left(z^{\frac23} \cdot \right).
\end{align}
Recalling the analogy to the physical space gradient and divergence operators, we let
\beq\label{E:BDELTA}
\bl:= \bd\bp.
\eeq
Then, for any $k\in\mathbb N$, we let 
\begin{align}
& \bD_0:=  \text{{\em Id}}, \ \ \bD^1:=\bd, \ \ \bD^2 := \bp\bd, \notag\\
& \bD^k:= 
\begin{cases} 
(\bp\bd)^{\frac k2} & \ k \ \text{ is even,} \\
\bd \bD^{k-1} & \ k \ \text{ is odd}.
\end{cases}\label{DEF:D^k}
\end{align}
\end{definition}


If a function $u(z)$, $u:[0,\infty)\to\mathbb C$ is viewed as a function of $y=\bzeta(z)$, we use the notation
$
\tilde u(y) =  u(z).
$
To any $Z>z_\ast$, we associate the space of functions
\begin{align}\label{E:DZDEF}
\DZ: = \left\{\begin{pmatrix}  \theta \\ \phi \end{pmatrix} : [0,Z]\to \mathbb C^2 \, \Big| \, y\mapsto \begin{pmatrix} \tilde \theta(y) \\ \tilde\phi(y) \end{pmatrix} \in C^\infty([0,\bzeta(Z)];\C)^2  \right\}.
\end{align}
A subspace of $\DZ$ consisting of functions whose even Taylor coefficients at $y=0$ vanish is denoted by $\DZodd$, i.e.
\begin{align}\label{E:DZODD}
\DZodd: = \left\{\begin{pmatrix}  \theta \\ \phi \end{pmatrix} \in \DZ \, \Big| \, \tilde \theta^{(2k)}(0)=\tilde\phi^{(2k)}(0)=0, \ k\in\mathbb N_0\right\}.
\end{align}
Analogously, we define the even version
\begin{align}\label{E:DZEVEN}
\DZeven: = \left\{\begin{pmatrix}  \theta \\ \phi \end{pmatrix} \in \DZ \, \Big| \, \tilde \theta^{(2k-1)}(0)=\tilde\phi^{(2k-1)}(0)=0, \ k\in\mathbb N\right\}.
\end{align}
Finally, in a  slight abuse of notation, we set 
\beq\label{E:DINFDEF}
\Dinf=\cap_{Z>z_*}\DZ,\quad \Dinfodd = \cap_{Z>z_*}\DZodd,\quad \Dinfeven=\cap_{Z>z_*}\DZeven.
\eeq


\begin{remark}
In particular, by~\eqref{E:BARZETAREGULARITY}, $\bzeta\in\Dinfodd$.
We observe that the operators $\bp,\bd$ defined in~\eqref{E:DIVGRADDEF} are adapted to the regularity classes $\DZodd,\DZeven$ in that they reverse parity. In particular, for any $u\in \DZodd$, we have $\bD^2 u \in \DZodd$.
\end{remark}


In the remainder of the paper we will make frequent use of the following  weight function.
\begin{align}\label{E:GGDEF}
\bG(z) := \frac1{(\bp \bzeta)^2},
\end{align}
where
 $\bzeta$ is the LP-flow map
given by Lemma~\ref{L:LPLAGRANGIAN}.


\begin{lemma}\label{L:GBAR}
The weight function $\bG(z)\in\Dinfeven$ is strictly positive and satisfies the following:
\begin{align}
&\bG(z)=9+O_{z\to0}(z^{\frac23}), \qquad \bG(z)=O_{z\to\infty}(z^{-\frac43}),\label{E:GBDS}\end{align}
and
$z^{\frac23}-\bG(z) <0$ for  $z<z_*$, $z^{\frac23}-\bG(z)>0$ for  $z>z_*$.
Moreover, 
 the derivatives of $\bG$ satisfy
\begin{align}
&\pa_z\bG(z)\leq 0, \ \ \ z>0,\label{E:DZG}\\
&\Big|\frac{\bp\bG}{\sqrt{\bG}}\Big|= O(z^{-1}), \ \ \ z\to\infty. \label{E:DZGBARFF}
\end{align}
\end{lemma}


\begin{proof}
The proof can be found in Appendix~\ref{SS:SSP}.
\end{proof}

From~\eqref{E:DZGBARFF}, given $\a>0$, we define $Z_*>z_*$ such that 
\beq\label{E:DGSQRTG}
\Big|\frac{\bp\bG}{\sqrt{\bG}}\Big|\leq\a<\frac23, \text{ for }z\geq Z_*.
\eeq
In the sequel, for precision, we fix $\a=\frac1{24}$.


Additionally, we require a further weight function. To this end, for each $0\leq m\leq 2M$, we take $\tilde g_{m}(z)$ as the solution to
\beq
\frac{\tilde g_{m}'(z)}{\tilde g_{m}(z)}=mz^{-\frac23}\frac{\bp \bG}{\bG},\qquad \tilde g_{m}(0)=1,
\eeq
which exists globally in $z$. It is straightforward to see from Lemma~\ref{L:GBAR} that $\tilde g_{m}(z)>0$ for all $z\geq 0$ and $\tilde g_{m}'(z)<0$ for $z>0$. However, we only wish to make use of this weight inside the backwards acoustic cone from the origin, and so we define a constant $\bar g_{m}>0$ and a new, smooth function $g_{m}(z)$ such that $g_{m}'(z)\leq 0$ for all $z\geq 0$ and
\beq\label{DEF:gm}
g_{m}(z)=\begin{cases}
\tilde g_{m}(z), & z\in[0,Z_*],\\
\bar g_{m}, & z\geq 2Z_*
\end{cases}
\eeq
is such that 
\beq\label{E:gmprime}
\Big\|z^{\frac23}\sqrt{\bG}\frac{g_{m}'}{m g_{m}}\Big\|_{L^\infty([Z_*,\infty))}\leq \a=\frac1{24},
\eeq 
where we recall~\eqref{E:DGSQRTG}.


\subsubsection{Low order energy space}\label{S:LOW}


We first introduce our low order energy spaces. These are defined on a finite interval $[0,Z]$ and are  versions of the usual Sobolev spaces, appropriately translated into the $z$ coordinates, and weighted in order to be well adapted to the linear analysis.

Let $Z>z_\ast$ be fixed.
For any $j,m\in\N_0$, to any two pairs $\begin{pmatrix}  \theta_1 \\  \phi_1 \end{pmatrix}$, $\begin{pmatrix}  \theta_2 \\  \phi_2\end{pmatrix} $ from the set $\DZodd$ (recall~\eqref{E:DZODD}) we associate 
the homogeneous complex inner product.
\begin{align}
\left(\begin{pmatrix}  \theta_1 \\  \phi_1 \end{pmatrix}\, ,\, \begin{pmatrix}  \theta_2 \\  \phi_2 \end{pmatrix}\right)_{\dot{\mathcal H}_{j,Z}^{m}}
:=\int_0^{Z} \bD^{m+1}  \theta_1 \overline{\bD^{m+1}  \theta_2} \bG(z) g_j(z) \diff z + \int_0^{Z}  \bD^m\phi_1 \overline{ \bD^m\phi_2} g_j(z) \diff z,
\end{align}
where we recall Definition~\ref{D:KEYOP}.
We then define the corresponding inhomogenous inner product
\begin{align}\label{E:INNERPROD}
\left(\begin{pmatrix} \theta_1 \\  \phi_1 \end{pmatrix}\, ,\, \begin{pmatrix} \theta_2 \\  \phi_2 \end{pmatrix}\right)_{\mathcal{H}^m_{j,Z}}
:= 
\left(\begin{pmatrix} \theta_1 \\  \phi_1 \end{pmatrix}\, ,\, \begin{pmatrix} \theta_2 \\  \phi_2 \end{pmatrix}\right)_{\dot{\H}_{j,Z}^m}
+  \beta  \left(\begin{pmatrix} \theta_1 \\  \phi_1 \end{pmatrix}\, ,\, \begin{pmatrix} \theta_2 \\  \phi_2 \end{pmatrix}\right)_{\dot{\H}^0_{j,Z}}
\end{align}
for a constant $\beta>0$  
sufficiently small, depending on $m$, to be specified later. 
Particular importance will be played by the spaces $\H_{m,Z}^m$, wherein the number of derivatives and the strength of the weight are mutually dependent.

\begin{definition}[Hilbert space for the linear analysis]\label{D:HILBERT}
We define the Hilbert space $\H^m_{m,Z}$ 
as the completion of $\DZodd$ with respect to the the norm $\|\cdot\|_{\H^m_{m,Z}}$ induced by the $\H^m_{m,Z}$ inner product .
\end{definition}


\begin{remark}
Our spectral analysis is most naturally carried out in complex-valued spaces $\H^m_{m,Z}$, as the underlying linear operator is not self-adjoint. However, 
it is clear from the definition of the linearised operator (see~\eqref{E:BOLDLDEF} below) that if initial data are real valued so is the resulting solution. 
\end{remark}

It is clear from the definition of the space  $\H^m_{m,Z}$ 
and the properties of $\bG$ that we have the following equivalence to the Sobolev space 
\beq\label{E:HSOBOLEVEQUIV}
\H^m_{m,Z}\cong H_{\textup{odd}}^{m+1} \times H_{\textup{odd}}^{m}
\eeq where $H_{\textup{odd}}^{m}= H_{\textup{odd}}^{m}([0,Y], y^2dy) $ is the Hilbert space generated by 3-d radial divergence and gradient in $y$ coordinates $(y\simeq z^\frac13)$ respecting the parity condition as in \eqref{E:DZODD}.


\begin{remark}
Given a fixed $m\in\N$, it is also clear that for all $k,j\in\mathbb N$, the low order norms with different weights are equivalent:
\[
\|\cdot\|_{\mathcal{H}^m_{j,Z}}\cong\|\cdot\|_{\mathcal{H}^m_{k,Z}},
\]
where the implicit constants depend on $m,Z,k,j$.
\end{remark}


We next introduce the low regularity index $\m$ which will be used to quantify regularity of the solution in the backward accretive cone $[0,Z]$.


\begin{definition}[Low regularity index $\m$]
We now take  $\m\in\N$ to be any natural number such that
\beq\label{E:MCONDITION}
\frac76-\m(\frac23-\a)<0,
\eeq
where $\a=\frac1{24}$ is the constant from~\eqref{E:DGSQRTG}. We note that it is both necessary and sufficient to take $\m\geq 2$.
\end{definition}


\begin{remark}
We shall see below that the constant $\beta>0$ from~\eqref{E:INNERPROD} will be chosen sufficiently small depending on this $\m$, see~Theorem~\ref{T:LINEARMAIN}.
\end{remark}


\subsubsection{High order energies and pointwise norm}\label{S:HIGH}

The central role of the low order Hilbert space is to enable us to develop a  linear stability theory in arbitrarily large backwards cones from the singular point. In order to develop the linear theory into a full nonlinear stability result, we shall fix a $Z_0>0$ and apply the linear theory on the backwards cone $z\leq Z_0$.  
Then to close estimates at the top order and prove the full nonlinear stability, we work with a weighted energy space defined on the whole half-line $z\geq 0$, incorporating both the weights $g_j$ and growing factors in the far-field. In order to address the quasilinear nature of the problem, we replace the weight $\bG=(\bp\bzeta)^{-2}$ used in the low order spaces by the full quantity $(\bp\zeta)^{-2}$. Moreover, we introduce growing weights in the far-field to handle the failure of critical Hardy-Sobolev embeddings and finally include a further large constant ($\kappa$ defined below) in order to bridge between interior and exterior regions. 
 
 We define the higher order energy spaces in terms of a regularity index $M>\m$, where we recall~\eqref{E:MCONDITION}. As will be seen below, especially in Section~\ref{S:ENERGYBOUNDS}, many of our estimates work for any (fixed) choice of $M>\m$. However, for precision, in our main Theorem~\ref{T:MAIN}, we take $M=\m+1$.

We therefore define a family of weight functions as follows. For a constant $c\in(0,\frac23)$ to be specified later and $j\in\N_0$, we set
\begin{align}\label{E:CHIJDEF}
\chi_{2j}(z) :=\begin{cases}
\kappa(1+z)^{-\al}, & \ \ j=0,\\
g_{2j}(z)\big(\kappa+(1+z)^{2cj-\al}\big), &\ \ j\in\mathbb N,
\end{cases}
\end{align} 
where we recall $g_{2j}$ was defined above in~\eqref{DEF:gm} and where the constants $\kappa=\kappa(Z_0)>0$ and $\al>1$ satisfy the following.
\begin{itemize}
\myitem[(a1)]\label{item:a1} $2c-\al>0$.
\myitem[(a2)]\label{item:a2} The constants $c<\frac23$ and $\al>0$ are such that $\de\leq \frac1{2M}$ where 
\be\label{def:delta}
\delta := \frac{2}{3}- c + \frac{\alpha-1}{2}.
\ee
In particular, we have $4\de M<\al+1$.
\myitem[(a3)]\label{item:a3} $\kappa=\kappa_0(1+Z_0)^{2cM-\al}$, where $\kappa_0\gg (\frac23-\a)^{-1}$, where $\a=\frac1{24}$ is defined in~\eqref{E:DGSQRTG}.
\end{itemize}
We now define the weighted energy functional
\begin{align}
\E_{2j}=\E_{2j}[\U]& :=\int \chi_{2j} \bigg( \frac{(\bD^{2j+1}\theta)^2}{(\bp\zeta)^2} +(\bD^{2j}\phi)^2 \bigg) \diff z, \ j\in\mathbb N_0. \label{E:TILDEEDEF}
\end{align}

For a fixed $M>0$, we define the high-order energies up to order $2M$ by setting, for any $j\in\N_0$, $j\leq M$,
\begin{align}\label{E:TEDEF}
\tE_{\le 2j}&: =  \sum_{\ell=0}^j\tE_{2\ell}[\U].
\end{align}
For simplicity, we also define the total top order energy $\tE_{\leq 2M}=\tE$.

We define further, for $s\geq \sin$ (recall~\eqref{E:SINITIAL}),
the pointwise norm of the flow
\begin{align}\label{E:PDEF}
\Pb[\Phi](s)=\Pb(s):= \|\frac{\th(s,\cdot)}{\bzeta}\|_{L^\infty} +\|\frac{\bp\theta(s,\cdot)}{\bp\bzeta}\|_{L^\infty} +\|\frac{z\pa_z^2\theta(s,\cdot)}{\bzeta_z}\|_{L^\infty}+\|\frac{\pa_s\bp\theta(s,\cdot)}{\bp\bzeta}\|_{L^\infty}.
\end{align}

 Under the  assumption that the pointwise norm $\Pb[\Phi]\leq \frac14$, the weights $(\bp\bzeta)^{-2}$ and $(\bp\zeta)^{-2}$ are equivalent, and hence we may make the norm bounds, for $m\leq M$,
\beq\label{E:HMZTE}
\Big\|\begin{pmatrix} \theta \\ \phi \end{pmatrix}\Big\|_{\HmZ}^2\leq C(m,Z)\tE_{\le 2 m}.
\eeq
 In order to provide a function space associated to the top order energy functional, we define the Hilbert space $\mathfrak H$ to be the completion of 
$\mathcal{D}^{\textup{odd}}_\infty$ (recall~\eqref{E:DINFDEF}) with respect to the norm $\|\cdot\|_{\mathfrak{H}}$ defined by
\beq\label{E:FRAKHDEF}
\Big\|\begin{pmatrix} \theta \\ \phi \end{pmatrix}\Big\|_{\mathfrak{H}}^2=\sum_{j=0}^M\int \chi_{2j} \bigg( \frac{(\bD^{2j+1}\theta)^2}{(\bp\bzeta)^2} +(\bD^{2j}\phi)^2 \bigg) \diff z.
\eeq
Again, it is clear that if $\Pb[\Phi]\leq \frac14$, then $\Phi\in\mathfrak{H}$ is equivalent to $\tE[\Phi]<\infty$.

The high-order energy $\tE$, as we shall see in Section~\ref{S:ENERGYBOUNDS}, stems from the quasilinear nature of the problem.
On the other hand, the quantity $\Pb(s)$ gives the pointwise control of the Lagrangian flow trajectories, which is necessary to prove global existence in the $s$-variables, as well as to interpret the global stability in $(s,z)$-variables as the stability of the LP solution.


\subsection{Precise statement of the main theorem}\label{S:PRECISE}


To state the main theorem, we first clarify the meaning of stability by explaining the exact meaning of asymptotic data flattening, the high-order energy norms, and the initial data.

{\em Asymptotic flattening.}
We take the index $M\in\N$ as in Section~\ref{S:HIGH}, recalling that in the sequel, this will be fixed to $M=\m+1$.
To asymptotically flatten the data and thereby enforce finite total mass and total energy assumption, we consider a family of smooth functions $g:[0,\infty)\to(0,\infty)$ 
which satisfy the following {\em tail asymptotic} properties:
\begin{enumerate}
\myitem[(g1)] $g:[0,\infty)\to(0,\infty)$ is a strictly positive $C^{2M}$-function and there exists a sufficiently large constant $r_\ast>0$, to be specified later, such that:\label{item:g1}
\[
g(r) = \frac{\CLP}{4\pi}, \ \ r\le r_\ast,
\]
where the constant $\CLP$ was introduced in Lemma~\ref{L:LPLAGRANGIAN};
\myitem[(g2)]\label{item:g2} there exists a constant $\ga>0$ such that, for all $0\le k\le 2M$,
$$|g^{(k)}(r)|\le \frac \ga{(1+r)^{2+k}}, \qquad \Big|\frac{\dif^k}{\dif r^k}\big(\frac{g'(r)}{g(r)}\big)\Big|\leq \frac{\gamma}{(1+r)^{k+1}}. $$ 
\end{enumerate}

{\em Background profile.}
By Lemma~\ref{L:LPLAGRANGIAN}, the flow map corresponding to the LP-family of solutions~\eqref{E:LPRHO}--\eqref{E:LPU} is given by the 1-parameter family 
\begin{align}
T \mapsto \eta_{\text{LP},T}(t,r):=(T-t)\bzeta(\frac{r}{T-t}),
\end{align}
where $\bzeta$ is the universal LP-profile given by Lemma~\ref{L:LPLAGRANGIAN}.
In the first order formulation~\eqref{E:NLZETA}--\eqref{E:NLPSI}, the LP-solution corresponds to the pair
\begin{align}
\begin{pmatrix}\bzeta \\ \widehat\mu \end{pmatrix} := \begin{pmatrix}\bzeta \\ \Lambda\bzeta-\bzeta\end{pmatrix} = \begin{pmatrix} \eta_{\text{LP},0}(-1,\cdot) \\ \pa_t\eta_{\text{LP},0}(-1,\cdot)\end{pmatrix}.
\end{align}

We are now in a position to state the main result of this work.

 
\begin{theorem}[Main theorem]\label{T:MAIN} 
There exist $\bar C>0$, $\tilde\eps_0>0$, $\eps_0>0$, $\m\in\N$, $\sg >0$, $Z_0>0$ such that the following statement is true. For any $\tilde\eps\in[0,\tilde\eps_0]$,
let $\begin{pmatrix} \zeta_0(\cdot) \\ \mu_0(\cdot)\end{pmatrix}$ be a profile such that 
\begin{align} \label{IC0}
\| \tilde\Phi_0 \|^2_{\HmZm}<\frac{\tilde\eps}{2},\ \  \tE_{\leq 2(\m+1)}[\tilde\Phi_0] <\frac{\eps_0}{4} \ \text{ and } \  \Pb[\tilde\Phi_0]&<\frac{\sqrt{\eps_0}}{4}, 
\end{align} 
 where
 \begin{align}
\tilde\Phi_0 
&:=\begin{pmatrix}\zeta_0 - \bzeta \\ \mu_0-\widehat\mu \end{pmatrix} \label{E:PROFILE}
\end{align}
and the constants $\kappa_0>0$, $c\in(0,\frac23)$ and $\al>1$ (appearing in the definition of $\tE$~\eqref{E:CHIJDEF}--\eqref{E:TEDEF}) are chosen depending on $\eps_0,\m$ so that they satisfy assumptions~\ref{item:a1}--\ref{item:a3}. Moreover, let $r_*>1$, depending on $\eps_0,Z_0$, be such that the flattening function $g$ (see~\eqref{E:DATAIN}) satisfies assumptions {\em \ref{item:g1}--\ref{item:g2}}.

Then there exists a constant $C_0$ and a final time $T$, $0\le |T|\le C_0\sqrt{\tilde\eps}$, such that the solution to the  initial  value problem~\eqref{E:NLZETA}--\eqref{E:NLPSI} with initial data  
\begin{align}
 \begin{pmatrix}\theta \\ \phi\end{pmatrix}(\sin,z) & =\begin{pmatrix}\theta_0 \\ \phi_0\end{pmatrix}   = \PhiT : = \begin{pmatrix}(T+1)^{-1} \zeta_0( r(T+1))  \\ \mu_0(r(T+1))  \end{pmatrix} - \begin{pmatrix} \bzeta \\ \widehat\mu \end{pmatrix}\label{E:PROFILE1}
\end{align}
exists and is global with respect to the self-similar time $s$ defined through 
\begin{align}\label{E:SST2}
s(t) :=  - \log (T-t), \ \ -1\le t<T.
\end{align}
Moreover, given $\nu\in(0,1)$, there exists $\Om>0$ such that the following bounds hold 
\begin{align}
\sup_{s\in[\sin,\infty)} e^{2\nu \sg s} \| \Phi (s) \|^2_{\HmZm} \le 2\bar C\tilde\eps, \ 
\sup_{s\in[\sin,\infty)} e^{2\Omega s}\tE_{\leq 2(\m+1)} [\Phi ](s) \le \frac{\eps_0}{2}, \ \ \sup_{s\in[\sin,\infty)} \Pb[\Phi] (s) \le \bar C \sqrt{\eps_0}. \label{E:MAINBOUND}
\end{align}
\end{theorem}


We recall from~\eqref{E:MCONDITION} that $\m=2$ is sufficient, which results in a total of $7$ derivatives on $\zeta$ and $6$ on $\mu$. In fact the theorem will hold for any $\m\geq 2$. We refer the reader to Remark~\ref{R:ENHANCEDREG} for further discussion of these regularity requirements.

\begin{remark}
The choice of the 1-parameter family of the initial data profiles $\Phi_{\text{in}}^T $ is natural and is 
designed to capture the trivial instability induced by the time translation invariance of the problem. 
We shall later see (cf.~Proposition~\ref{P:BROUWER}) that the time $T$ is obtained through the Brouwer fixed point argument, reminiscent of the classical
stable manifold constructions in dynamical systems. Similar topological arguments have been implemented in both semilinear and quasilinear contexts (see for example~\cite{Glogic22,Do2024,MRRS2}.
\end{remark}


\begin{remark}[Consistency of the choice of initial data]\label{R:CONS}
It is necessary to understand the effect of the $T$-modulation on the size of the initial data. We write, for any profile $\Xi=(\zeta,\mu)^\top$,
\beq
\Xi^T:=\begin{pmatrix}(T+1)^{-1} \zeta( r(T+1))  \\ \mu({r}({T+1}))  \end{pmatrix}.
\eeq
We then decompose
\begin{align}
\PhiT = \Xi_0^T-\Xi_{\text{LP}}^T + \Xi_{\text{LP}}^T-\Xi_{\text{LP}}, \ \ \qquad \Xi_0:=\begin{pmatrix} \zeta_0 \\ \mu_0\end{pmatrix}, \ \  \Xi_{\text{LP}}:=\begin{pmatrix}\bzeta \\ \widehat\mu\end{pmatrix}.\label{E:XIDEF}
\end{align}
 From the smoothness of the LP profile, it is then straightforward to see that, for $0\le|T|\ll1$ sufficiently small (depending on $\m$, $Z_0$), 
\beq\label{E:IDMODULATE}
\|\Xi_0^T-\Xi_{\textup{LP}}^T\|^2_{\HmZm}= \|\tilde{\Phi}_0^T\|^2_{\HmZm}\leq \frac32\tilde\eps\quad \text{ and } \quad\|\Xi_0^T-\Xi_{\textup{LP}}^T\|^2_{\H^0_{2\m,Z_0}}\leq\frac{3}{2\beta}\tilde\eps,
\eeq
where we recall~\eqref{E:INNERPROD}.
Moreover, there exists $C_1(\m,Z_0)>0$ such that 
\[\|\Xi_{\text{LP}}^T-\Xi_{\text{LP}}\|^2_{\HmZm}\leq C_1T^2.\]
In order to close our Brouwer fixed-point argument at the end, we take  $T_0=C_0\sqrt{\tilde\eps}>0$ such that
\beq\label{E:T0DEF}
\frac{2}{\sqrt{\beta(\m,Z_0)}}\|\Gamma\|^{-1}_{\H^0_{2\m,Z_0}}\sqrt{\tilde\eps}\leq \frac12 T_0,
\eeq
where we the trivial growing mode $\Gamma$ is given in~\eqref{E:GM}.
Let the constant 
 \begin{align}\label{E:IDEST0}
 \bar C(\m,Z_0)=4(C_1C_0^2+2),
 \end{align} 
 so that for all $T\in[-T_0,T_0]$,
\beq\label{E:IDEST}
e^{2\nu\sg \sin}\|\PhiT\|^2_{\HmZm}\leq \frac12\bar{C}\tilde\eps.
\eeq
Below, we shall take $\eps_0=\widehat{C}\tilde\eps_0$ where $\widehat{C}(\m,Z_0)\gg\bar C$.
By continuity with respect to $T$, we also have, for $T_0$ sufficiently small (i.e.~$\tilde\eps$ sufficiently small),
\beq\label{E:INITIALDATA}
e^{2\Omega \sin}\tE_{\leq 2(\m+1)}[\PhiT] <\frac{3\eps_0}{8} \ \text{ and } \  \Pb[\PhiT]<\frac{\sqrt{\eps_0}}{2}.
\eeq
Bounds~\eqref{E:IDEST}--\eqref{E:INITIALDATA} in particular show that our choice of initial data is consistent with Theorem~\ref{T:MAIN}
in the sense that~\eqref{E:MAINBOUND} holds at the initial time $s=\sin$ with a strict inequality.
\end{remark}


\begin{remark}[On constants $\sg $, $\nu$, and $\Om$]The constant $\sg \in(0,1)$ will be defined below in Theorem~\ref{T:LINEARMAIN} and is determined via the decay rate of the semi-group associated to the linearised operator. The constant $\nu\in(0,1)$ may be chosen arbitrarily and is simply used to obtain a slightly slower decay rate than $\sg $.  The constant $\Om$ must be taken more carefully to satisfy 
\beq\label{E:OMEGABDS}
0<\Om\leq\min\{\frac13\nu\sg ,\frac{\al-1}{4}\}.
\eeq
At this stage, we fix $\nu\in(0,1)$ and, assuming $\sg $ to be given as in Theorem~\ref{T:LINEARMAIN}, take $\Om$ satisfying~\eqref{E:OMEGABDS}. These constants will be considered fixed from henceforth.
\end{remark}


\begin{remark}
In the original time variable $t$ the problem is initiated at the time slice $t=-1$; we therefore describe the dynamics on a time-interval of length $T+1$, $|T|\ll1$.
\end{remark}

To set up the global existence argument, we define $S_T>0$ to be the maximal existence time such that the low order norm satisfies the growth and smallness bound
\beq\label{E:STDEF}
S_T:=\sup \{s>\sin\,|\,e^{2\nu\sg s}\|\Phi\|_{\HmZm}^2\leq\bar C\tilde \eps\}
\eeq
for a suitable $\tilde \eps$.
We shall show that there exists a $T$, $|T|\ll1$, $0<\tE(\sin)\ll1$ sufficiently small, and $r_\ast\gg1$ sufficiently large, such that $S_T=\infty$.
To facilitate the analysis, for some $\eps_0>0$, $C_*>0$, we introduce our main a priori assumption: there exists some $S>0$ such that, for $s\in[\sin,S]$, we have the estimates
\begin{align}
e^{2\Omega s}\tE_{\leq 2(\m+1)}\le\eps_0, \label{E:APRIORIMAIN}\\
\Pb(s)\le C_*(Z_0)\sqrt{\eps_0},  \label{E:APRIORIMAINPT}
\end{align}
which will be  used in our a priori estimates in Sections~\ref{S:DUHAMEL},~\ref{S:ENERGYBOUNDS}, and~\ref{S:POINTWISE}.

\begin{remark}
Note that the $T$-dependence is contained in the choice of data~\eqref{E:PROFILE1} and through the choice of the initial time $\sin$~\eqref{E:SINITIAL}. The equation does not depend on $T$.
\end{remark}

 We remark for completeness that the local well-posedness of the self-similar Euler-Poisson system~\eqref{E:NLZETA}--\eqref{E:NLPSI} in the space $\mathfrak{H}$ under the assumption $\Pb[\Phi]\leq\frac14$ follows from the standard hyperbolic theory and the a priori estimates that we shall establish below in Sections~\ref{S:ENERGYBOUNDS}--\ref{S:POINTWISE}. We omit the  details.\\


\noindent{\em Eulerian description.}
With $g$ and $\eta_0:[0,\infty)\to[0,\infty)$ as in the statement of Theorem~\ref{T:MAIN}
the Eulerian initial density is given through the relation~\eqref{E:DATAIN}.
Moreover, for any $t\ge-1$ we have the fundamental identity
\begin{align}
\varrho(t,\eta(t,r))& =  e^{2s} \frac{\tilde g(s,z)}{\zeta(s,z)^2\pa_z\zeta(s,z)}, \label{E:RHOEULERIAN}
\end{align}
which follows from the relation $\pa_t\big(\varrho\circ\eta \eta^2\pa_r\eta\big) = 0$,~\eqref{E:DATAIN}, and~\eqref{E:SCALINGLAGR}.
If we denote 
$
R = \eta(t,r)
$
we may use a fixed spatial scale $Z_0/2$, related to the slope of the accretivity cone, to read off the behaviour of the Eulerian density. In particular, from~\eqref{E:RHOEULERIAN},
the largeness of $Z_0/2\gg1$, and the asymptotic relations $\eta(t,r)_{r\to0}\sim r^{\frac13}$ and $\eta(t,r)_{r\to\infty}\sim r$, it is easy to see that
\begin{enumerate}
\item in the backward cone $\{R\leq (T-t)\frac{Z_0}{2}\}$ we have $\varrho(t,R)\sim \frac1{(T-t)^2}$,
\item in the intermediate region $\{(T-t)Z_0/2\le R\le Z_0/2\}$ we have $\varrho(t,R)\sim \frac1{R^2}$,
\item and in the far-field (dampened) region $\{R\ge Z_0/2\}$ we have $\varrho(t,R)\sim \frac{g(R)}{R^2}$.
\end{enumerate}

To make the above statements precise we introduce the space-time distance
\begin{align}
d_T(t,R) = |T-t| + R, \ \  (t,R)\in[-1,T)\times\mathbb R_+,
\end{align} 
which will allow us to quantitatively describe the errors triggered by the perturbation of the LP-solution.


\begin{theorem}[LP-stability in Eulerian variables]\label{T:EULERMAIN}
Let $T\in\mathbb R$, $|T|\ll1$ be given as in Theorem~\ref{T:MAIN} and let $[-1,T)\ni (\varrho,u)(t,\cdot)$ be the unique solution to the initial value problem~\eqref{E:ECONT}--\eqref{E:EULERINITIAL}. Then there
exists a $C^1$-map $[-1,T)\times [0,\infty)\ni (t,R)\mapsto f(t,R)\ge0$  such that $f(t,0)=0$ for all $t\in[-1,T)$ and
\begin{align}\label{E:FBOUND}
\big|\frac{f(t,R)}{R}-1\big| + \big|R\pa_R\big(\frac{f(t,R)}{R}\big)\big| \le C \sqrt{\eps_0} \ d_T(t,R)^\Om,
\end{align}
so that the following statements hold.
\begin{enumerate}
\item[(a)]({\em Implosion})
There exists a constant $h>0$, independent of $Z_0$, and constants $0<C_1<C_2<\infty$ such that, for any $(t,R)\in[-1,T)\times [0,hZ_0]$, we have
\begin{align}\label{E:SANDWICH}
C_1\rhoLPT(t,f(t,R)) \le \varrho(t,R)\le C_2\rhoLPT(t,f(t,R)).
\end{align}
In particular, $C_1 \rhoLP(0)(T-t)^{-2} \le \varrho(t,0)\le C_2\rhoLP(0)(T-t)^{-2}$  for all $t\in[-1,T)$ and the density blows up at the centre as $t\to T^-$.
\item[(b)]({\em Stability in cylinders and backward cones})
Let $R\in[0,hZ_0]$ be given. There exists a constant $C>0$ such that
\begin{align}
\big|\frac{\varrho-\rhoLPT\circ f}{\rhoLPT\circ f}\big| +\big|R\pa_R\big(\frac{\varrho-\rhoLPT\circ f}{\rhoLPT\circ f}\big)\big|  \le C \sqrt{\eps_0} \ d_T(t,R)^\Om.
\label{E:DOM}
\end{align}
In particular, there exists a constant $c>0$ such that for any fixed $0<R\ll1$ we have the quantitative bound
\begin{align}\label{E:DOM2}
\sup_{t\in[T-\frac{R}{c},T)}\big|\frac{\varrho-\rhoLPT\circ f}{\rhoLPT\circ f}\big| \le C\sqrt{\eps_0} R^\Om.
\end{align}
\item[(c)]({\em Velocity is uniformly bounded and stable})
Assume additionally that $\|u_0-\uLPT\|_{L^\infty} + \|R\pa_R(u_0-\uLPT)\|_{L^\infty}\le \sqrt{\eps_0}$.
Then there exists a $C>0$ such that 
\begin{align}
\sup_{(t,R)\in[-1,T)\times[0,\infty)}|u(t,R)| \le C.
\end{align}
In fact the velocity is uniformly close to the LP-velocity profile in scaling-invariant norms:
\begin{align}
&\sup_{(t,R)\in[-1,T)\times[0,\infty)}|u(t,R)- \uLPT\circ f(t,R)|  \notag \\
&+\sup_{(t,R)\in[-1,T)\times[0,\infty)}|R\pa_R\big(u(t,R)- \uLPT\circ f(t,R)\big)|  \le C\sqrt{\eps_0}
\end{align}
\end{enumerate}
\end{theorem}

The proof of the theorem is a corollary of Theorem~\ref{T:MAIN} and is presented in Section~\ref{S:MAINTHEOREM}.

\begin{remark}
We note that the right-hand side of the  bound~\eqref{E:DOM} is proportional to $(T-t)^\Om$ and $R^\Om$ in the interior of the 
backward cone $\frac{R}{T-t}\le Z_0/2$ and its exterior respectively. It in particular captures the transfer of information from backward cones to 
cylinders $\{R\le Z_0/2\}$ making the space-time distance $d_T$ a scale-invariant measure of stability of the LP-solution as $R\to0$, see~\eqref{E:DOM2}. 
\end{remark}

\begin{remark}[Finite mass and energy]\label{R:FINITEMASS}
We note that by~\eqref{E:DATAIN}, $M(r) =4\pi \int_0^r \varrho_0 \eta_0^2 \pa_r\eta_0 \diff \tilde r = 4\pi \int_0^{\eta_0(r)} \varrho_0(R)R^2\diff R$, and therefore $M(r)$ corresponds to the $4\pi$ multiple of the local mass contained in the ball of radius $\eta_0(r)$. 
In particular, as a consequence of $(g2)$ we have 
\[
M[\rho]=\lim_{r\to\infty}M(r) =4\pi \int_0^\infty g(\tilde r)\diff \tilde r<\infty,
\] 
thus ascertaining that the initial data for our problem
have finite total mass. Similarly, it is easy to check that the total energy 
\begin{align}
E[\rho,u] = \int_0^\infty \left(\frac12 \rho u^2 +\rho\log\rho-\rho - |\pa_R\Phi|^2\right) \, R^2\diff R
\end{align}
of our initial data is finite.
\end{remark}

\subsubsection{Selection of parameters}

As there are several parameters involved in the statement of Theorem~\ref{T:MAIN}, the definitions of the spaces $\HmZ$, and the energy norms $\tE$, for the convenience of the reader, we collect here the relevant dependencies of these constants.

\begin{itemize}

\item {\em Regularity.} The regularity index $\m$ is determined directly from the constraint~\eqref{E:MCONDITION} and may be fixed to be 2. We let $M=\m+1$ (although it can be any other integer larger than $\m$).

\item {\em Weight powers.}
We introduce the constants $c\in(0,\frac23)$ and $\al>1$, depending only on $\m$, satisfying~\eqref{def:delta}. 

\item {\em Accretivity.} The slope of the accretivity cone $Z_0$ is taken large, depending on $\frac23-c$, $\al-1$, and estimates on the LP solution.
The constant $\beta$ of~\eqref{E:INNERPROD} is taken sufficiently small as $\beta(\m,Z_0)$, using Theorem~\ref{T:LINEARMAIN}.

\item  The constant $\kappa_0$ is taken sufficiently large, using properties of the LP solution in Proposition~\ref{P:ENERGYGLOBAL}.

\item {\em High-order energy.}
We let $\kappa=\kappa(Z_0,\kappa_0,c, M, \al)$ to be given by~\ref{item:a3}. We then define the energy norms $\tE$ in~\eqref{E:CHIJDEF}--\eqref{E:TEDEF}. 

\item {\em Rates of decay.}
We choose $\nu\in(0,1)$ arbitrarily, determine $\sg >0$ from Theorem~\ref{T:LINEARMAIN} (maximal accretivity theory), and subsequently choose $\Om$ satisfying~\eqref{E:OMEGABDS}. All these constants appear in~\eqref{E:MAINBOUND} in the statement of Theorem~\ref{T:MAIN}.

\item {\em Pointwise norms.}
The constant $C_*(Z_0)$ taken in~\eqref{E:APRIORIMAIN} is defined in Lemma~\ref{L:STPROPERTIES}, depending on $Z_0$ and $\Omega$.

\item {\em Smallness.}
The small constant $\eps_0>0$ in Theorem~\ref{T:MAIN} is chosen sufficiently small, depending on $Z_0$.
We set $\tilde\eps_0=\widehat{C}^{-1}\eps_0$ and $T_0=C_0\sqrt{\tilde\eps}$ for suitable $\widehat{C},C_0$, depending only on $Z_0$,$\m$, where we recall $T_0$
from Remark~\ref{R:CONS}.

\item {\em Asymptotic flattening.}
The constant $r_*$ in assumptions~\ref{item:g1}--\ref{item:g2} is taken to depend on $\eps_0$.
\end{itemize}


\section{Linearised operator and mode stability}\label{S:LINEARMODE}


The goal of this section is to prove mode-stability of the linearised operator. To that end we first introduce the associated function-analytic framework  and then define the linearised operator 
before stating the mode stability in Theorem~\ref{T:SPECTRAL} below.

We introduce the perturbed quantities:
\begin{align}\label{E:THDEF}
\th : = \zeta - \bzeta, \quad \phi : = \mu - \widehat\mu = \th_s + z\pa_z \th - \th
\end{align}
and let 
\begin{align}
\Phi : = \begin{pmatrix} \th \\ \phi \end{pmatrix}
\end{align}
denote the perturbed pair.


\subsection{Formulation of the linearised problem}


Our goal is to study the linearised problem on a finite, but sufficiently large $z$-domain $[0,Z]$. Instead of relying on the quasilinearised formulation of Lemma~\ref{L:QUASI0},
which is carefully formulated to avoid derivative loss in our top-order energy estimates in Section~\ref{S:ENERGYBOUNDS}, we state the linearised problem.

We introduce the scaling operator
\begin{align}
\Lambda : = z\pa_z.
\end{align}
and the key second order operator
\begin{align}
K \theta & : =  \bp  \left(\bG \bd \theta\right), \label{E:KDEF}
\end{align}
where we recall the weight function $\bG$ from~\eqref{E:GGDEF}. Now a direct computation reveals
\begin{align}
\pa_z \Big( \frac{\pa_z\theta}{\bzeta_z^2}  \Big)&=K\theta +\mathcal V_1 \theta,\quad \text{ where }\quad\mathcal V_1  = \frac29 z^{-\frac23} \bG + \frac43 z^{-\frac13} \bp^2 \bzeta \bG^\frac32 .\label{E:KIDENTITY}
\end{align}

In order to linearise the problem, we will commonly need to handle expressions such as $(\bzeta+\th)^{-2}$. To compress notation, we define notation for a recurring nonlinearity by 
\beq
N_1(\bzeta,\theta) : = \bzeta^{-2} \Big(\frac{\theta}{\bzeta}\Big)^2 \frac{3+\frac{2\theta}{\bzeta}}{\Big(1+\frac{\theta}{\bzeta}\Big)^2},\label{E:NONEDEF}
\eeq
so that, for example,
\begin{align}
\frac1{\zeta^2}=\frac1{\bzeta^2} - \frac{2}{\bzeta^3} \theta + N_1(\bzeta,\theta).
\end{align}

\begin{lemma}[Reformulation of the problem]\label{L:GPROPERTIES}
Let $(\zeta,\mu)$ be formally a solution of~\eqref{E:NLZETA}--\eqref{E:NLPSI} and let $(\th,\phi)$ be given via~\eqref{E:THDEF}.
Then, for $z\leq r_*$,  $(\theta,\phi)$ formally satisfies the following nonlinear problem
\begin{align}\label{E:NONLINEARFORM}
\pa_s\Phi ={\bf L} \Phi + {\bf N}[\Phi], \ \ \Phi = \begin{pmatrix} \theta \\ \phi \end{pmatrix},
\end{align}
where
\begin{align}
\bf L\begin{pmatrix} \theta \\ \phi \end{pmatrix}  & :=  \begin{pmatrix}  \phi - \Lambda\theta +\theta \\  -\Lambda \phi + K\theta + \mathcal V\theta  \end{pmatrix},\label{E:BOLDLDEF}\\
 {\bf N}\begin{pmatrix} \theta \\ \phi \end{pmatrix} &: = \begin{pmatrix} 0 \\ \mathcal N[\theta] \end{pmatrix}.\label{E:BOLDNDEF}
\end{align}
Here, the potential is defined by
\begin{align}\label{E:POTDEF}
\mathcal V : = \frac29 z^{-\frac23} \bG - 2\bzeta^{-2} + \frac43 z^{-\frac13} \bp^2\bzeta \bG^{\frac32}+ 2\CLP  z \bzeta^{-3}.
\end{align}
and 
\beqa
\mathcal N[\theta] &: =
-\left(\frac{2}{\bzeta_z^3}\theta_z-N_1(\bzeta_z,\theta_z)\right)\theta_{zz} +  N_1(\bzeta_z,\theta_z)\bzeta_{zz} 
 - \CLP z N_1(\bzeta,\theta)+ \frac{2\theta^2}{\bzeta^2(\bzeta+\theta)}  \label{E:NDEF}
\eeqa
Moreover, $\mathcal V\in\DZeven$ and 
 $\bfL :  D({\bf L})\subset \HmZ\to\HmZ$ where
\begin{align}
D(\bfL) := \left\{\Phi \in\HmZ\,| \, \bfL\Phi \in\HmZ \right\}.
\end{align}
In particular the associated linearised dynamics takes the form
\begin{align}\label{E:LINEAR1}
\pa_s\Phi = \bfL \Phi.
\end{align}
\end{lemma}

\begin{proof}
We start with the simple expansions from~\eqref{E:NONEDEF},
\begin{align}
\frac1{\zeta^2}=\frac1{\bzeta^2} - \frac{2}{\bzeta^3} \theta + N_1(\bzeta,\theta),\qquad \frac1{\zeta_z^2}=\frac1{\bzeta_z^2} - \frac{2}{\bzeta_z^3} \theta_z + N_1(\bzeta_z,\theta_z),
\end{align}
and note as well
\begin{align}
\frac1\zeta  & =  \frac1{\bzeta} - \frac1{\bzeta^2}\theta + \frac{\theta^2}{\bzeta^2(\bzeta+\theta)}, \ \ 
\frac1{\zeta_z}   =  \frac1{\bzeta_z} - \frac1{\bzeta_z^2}\theta_z + \frac{\theta_z^2}{\bzeta_z^2(\bzeta_z+\theta_z)}. 
\end{align}
We plug these back into~\eqref{E:ZETADAMPENED}
and obtain
\begin{align}
0 &= \theta_{ss} +2\Lambda\theta_{s}  + (z^2 - \frac1{\bzeta_z^2})\theta_{zz} + \frac{2\bzeta_{zz}}{\bzeta_{z}^3}\theta_z + \frac{2}{\bzeta_z^3}\theta_z\theta_{zz}
- N_1(\bzeta_z,\theta_z)(\bzeta_{zz}+\theta_{zz}) \notag\\
& \ \ \ \ - \theta_s - \frac{2\tilde M(s,z)}{\bzeta^3}\theta + \tilde M(s,z) N_1(\bzeta,\theta) 
+\frac2{\bzeta^2}\theta - \frac{2\theta^2}{\bzeta^2(\bzeta+\theta)} + (\log \tilde g)_z\frac1{(\bzeta+\theta)_z} \notag\\
& \ \ \ \ +(z^2-\frac1{\bzeta_z^2})\bzeta_{zz}  + \frac{\tilde M(s,z)}{\bzeta^2}- \frac2{\bzeta} \notag\\
& = \theta_{ss} +2\Lambda\theta_{s} + \pa_z ((z^2 - \frac1{\bzeta_z^2}) \theta_z) - \left(\theta_s + 2\Lambda \theta \right) 
+\frac2{\bzeta^2} \left(1- \frac{\tilde M(s,z)}{\bzeta} \right)\theta \notag\\
& \ \ \ \ + \left(\frac{2}{\bzeta_z^3}\theta_z- N_1(\bzeta_z,\theta_z)\right)\theta_{zz} -  N_1(\bzeta_z,\theta_z)\bzeta_{zz} \notag\\
& \ \ \ \ + \tilde M(s,z) N_1(\bzeta,\theta)- \frac{2\theta^2}{\bzeta^2(\bzeta+\theta)}+ (\log \tilde g)_z\frac1{(\bzeta+\theta)_z}
+ \frac{\tilde M(s,z)-\CLP z}{\bzeta^2}.\label{E:NONLINEARFLOW}
\end{align}
By our choice of $g$ satisfying~\ref{item:g1}, we have $g(r)=1$ for $r\le r_\ast$. In self-similar 
coordinates $(s,z)$, the set 
$r\le r_\ast$ is characterised by the requirement $z\le r_\ast e^s$.
Therefore for
$z \le r_\ast$,
we have 
\be\label{E:SCF}
g(s,z)=g(ze^{-s})=1, \ \ s\ge 0.
\ee 
It follows that in the region $z\le r_\ast$ equation~\eqref{E:NONLINEARFLOW}
takes on the simpler form
\begin{align}
0&=\theta_{ss} +2\Lambda\theta_{s} + \Lambda^2\theta -\pa_z ( \frac1{\bzeta_z^2} \theta_z) - \left(\theta_s + \Lambda \theta \right) 
+\frac2{\bzeta^2} \left(1- \frac{\CLP z}{\bzeta} \right)\theta - \mathcal N[\theta],\label{E:INTERIORFLOW}
\end{align}
where the ``interior" nonlinearity $\mathcal N[\theta]$ takes the form~\eqref{E:NDEF} and we have noted $\pa_z(z^2\pa_z\th)=\Lambda^2\th+\Lambda\th$. Now rearranging this equation, using~\eqref{E:KIDENTITY} and recalling the relations $\phi_s=\th_{ss}+\Lambda\th_s-\th_s$ and $\Lambda\phi=\Lambda\th_s+\Lambda^2\th-\Lambda\th$, it is a simple exercise to verify~\eqref{E:NONLINEARFORM}.

It remains  only to prove  $\mathcal V\in\DZeven$. 
We further decompose $\mathcal V$ into 
$$
\frac43 z^{-\frac13} \bp^2\bzeta \bG^{\frac32}+ 2\CLP z \bzeta^{-3}
$$
which is clearly in $\DZeven$, and the remaining terms,
$$ \frac29 z^{-\frac23} \bG - 2\bzeta^{-2}.$$
This contribution is slightly trickier as each entry formally expands like $z^{-\frac23}$ to the leading order. However, since $\bG=9 + O_{z\to0}(z^{\frac23})$ and $\bzeta^{-2} =z^{-\frac23}\big( 1+ O(z^{\frac23})\big)$ from Lemma~\ref{L:GBAR} and~\eqref{E:BARZETAREGULARITY}, we see that the
leading order singular terms cancel and $\mathcal V\in\DZeven$.
\end{proof}

Since the family of LP-solutions contains a 1-parameter freedom in specifying the final blow-up time $T$, this 
generates a trivial instability direction for the linearised problem, induced by the time-translation symmetry of the 
problem.


\begin{lemma}[The symmetry-induced growing mode]\label{L:GMODE}
The vector
\begin{align}\label{E:GM}
\Gamma : = \begin{pmatrix} \bzeta - \Lambda \bzeta \\ \Lambda \bzeta - \Lambda^2\bzeta \end{pmatrix}
\end{align}
is an eigenfunction of the operator $\bfL$ associated to eigenvalue $\l=1$. Moreover,
\begin{align}\label{E:GAMMAFORMULA}
\Gamma= -\frac d{dT} \Phi_{\text{LP}}^T\Big|_{T=0}, 
\end{align}
where
\begin{align}
 \Phi_{\text{LP}}^T(\cdot) = 
\begin{pmatrix}
(T+1)^{-1} \bzeta(\cdot(T+1)) \\
\widehat\mu(\cdot(T+1))
\end{pmatrix}.
\end{align}
\end{lemma}


\begin{proof}
The proof is a direct calculation which relies crucially on the LP-equation~\eqref{E:LPLAG}.
A straightforward application of the chain rule shows that~\eqref{E:GM} is derived from~\eqref{E:GAMMAFORMULA}. To see that this is indeed an eigenmode with eigenvalue 1, we use~\eqref{E:KIDENTITY} in~\eqref{E:BOLDLDEF} to write the eigenvalue equation in the form 
\beq\label{E:EMODE1}
\bfL\begin{pmatrix}
\theta_1\\
\phi_1
\end{pmatrix}
=\begin{pmatrix}
\phi_1 -\Lambda\theta_1 +\theta_1 \\
-\Lambda \phi_1+\pa_z\Big(\frac{1}{\bzeta_z^2}\pa_z\theta_1\Big) -\Big(\frac{2}{\bzeta^2}\big(1-\frac{\CLP z}{\bzeta}\big)\theta_1\Big)
\end{pmatrix} 
= \begin{pmatrix}
\theta_1\\
\phi_1
\end{pmatrix}.
\eeq
From~\eqref{E:GM}, it is obvious that the first component of this equation is satisfied for the claimed eigenfunction. To verify the second component, we note the simple identity
\beqa
-\Lambda^2\bzeta+\Lambda^3\bzeta+\pa_z\Big(\frac{1}{\bzeta_z^2}\pa_z\big(\bzeta-\Lambda\bzeta\big)\Big)=\Lambda\Big((z^2-\frac{1}{\bzeta^2})\pa^2_z\bzeta\Big)-\frac{1}{\bzeta^2_z}\pa^2_z\bzeta.
\eeqa
Thus, inserting the definition $\phi_1=\Lambda\bzeta-\Lambda^2\bzeta$ into  the second component of the left hand side of~\eqref{E:EMODE1} and  using the LP equation~\eqref{E:LPLAG} on each term, we rearrange to find
\beqa
\Big(\bfL\begin{pmatrix}
\theta_1\\
\phi_1
\end{pmatrix}\Big)_2=-\Lambda\Big(\frac{\CLP z}{\bzeta^2}-\frac{2}{\bzeta}\Big)-\frac{2}{\bzeta}\Big(1-\frac{\CLP z}{\bzeta}\Big)\big(\bzeta-\Lambda\bzeta\big)-z^2\pa_z^2\bzeta -\frac{\CLP z}{\bzeta^2}+\frac{2}{\bzeta}=\Lambda\bzeta-\Lambda^2\bzeta
\eeqa
as required.
\end{proof}

One of the central results of the present work is to establish the full linear stability of the operator $\bfL$. This is the content of the next theorem, the proof of which will occupy the remainder of Section~\ref{S:LINEARMODE} and all of Section~\ref{S:ACCRETIVITY}. We remind the reader that the Hilbert spaces $\HmZ$ depend on a small constant $\beta$, as in~\eqref{E:INNERPROD}.

\begin{theorem}\label{T:LINEARMAIN}
There exists  $\m\in\N$ such that, for all $m\geq \m$ and all $Z>z_*$, there exists $\beta>0$ such that  the operator $\bfL:D(\bfL)\subset \HmZ\to\HmZ$ is a closed operator with the following properties. There exists $\sg >0$ such that the spectrum satisfies
\beq
\sigma(\bfL)\cap\{\l\in\C\,|\,\Re\l\geq -\sg \}=\{1\},
\eeq
where $1$ is a simple eigenvalue.
$\bfL$ generates a quasicontraction semi-group with the growth estimate, for all $\Phi\in\H^{2\m}_Z$,
\beq
\left\|e^{s\bfL}\Phi \right\|_{\HmZ} \le e^{ s}\left\|\Phi\right\|_{\HmZ}.
\eeq
Moreover, letting $\bfP$ denote the Riesz projection onto the eigenvalue $1$, the semi-group satisfies $ e^{s\bfL}\bfP = \bfP e^{s\bfL} = e^s \bfP$ and, for all $\Phi\in\HmZ$,
\begin{align}
\left\|e^{s\bfL}\left(\mathbf{I}-\bfP\right) \Phi \right\|_{\HmZ} \le&\, e^{-\sg s}\left\|\Phi\right\|_{\HmZ},\label{E:STABLEBOUND}\\
\|\bfP e^{s\bfL}\Phi\|_{\H^{2\m}_Z} \le&\, e^s \|\Phi\|_{\HmZ}.\label{E:PPROPERTIES}
\end{align}
\end{theorem}

The eigenvalue $\l=1$ is, as advertised above, a trivial mode corresponding to time translation, not to meaningful instability. In Lemma~\ref{L:GROWINGPROJ} we will establish certain useful properties of the Riesz projection (defined in the standard way in~\eqref{E:RIESZDEF}) associated to this eigenvalue. The regularity parameter $\m$ is chosen as in~\eqref{E:MCONDITION}.

As discussed in the  introduction,  in order to prove this theorem, we wish to apply the Lumer-Phillips theorem, and so we require two key ingredients: maximal accretivity of a suitable compact perturbation of the operator $\bfL$, and the full mode stability of $\bfL$ itself. As the essential spectrum is unchanged by compact perturbations, these two results will be combined to give Theorem~\ref{T:LINEARMAIN}. The construction of the compact perturbation and its maximal accretivity is deferred to Section~\ref{S:ACCRETIVITY}. In this section, we establish the mode stability of $\bfL$, stated in the following theorem.

\begin{theorem}[Mode stability]\label{T:SPECTRAL}
For any $Z>z_\ast$  and $m\geq 2$   the operator $\bfL:D(\bfL)\subset \HmZ\to\HmZ$ is mode stable, i.e. it has exactly one eigenvalue with non-negative real part given by  $\l=1$. 
\end{theorem}

For the remainder of this section, we will continue to work with the standing assumption $m\geq 2$ and suppose we have chosen a $Z>z_*$. This will be further strengthened in Section~\ref{S:ACCRETIVITY}.


\subsection{Formulations of the eigenvalue problem}

For the purpose of proving mode-stability, it is more convenient to work with the Eulerian linearisation
around the LP-solution. The equivalence of the two linearisations is stated and proved in Appendix~\ref{A:EULLAGEQUIV}, from which we recall the definition of the operator $\bfL^{\text{Eul}}$ from Lemma~\ref{L:EULERIANLIN}.

We begin by stating the central eigenfunction equation, satisfied for any eigenvalue $\l\in\C\setminus\{1\}$.

\begin{lemma}\label{L:EVALUEPROB}
Let $\l\in\C\setminus\{1\}$ be an eigenvalue of $\bfL$  on $\HmZ$, $m\geq 2$. Then $\l$ is also an eigenvalue of $\bfLEul$ on $\HmZEul$, defined as
\beq\label{E:HMZEULDEF}
\HmZEul:=\bigg\{\begin{pmatrix}
\psi \\ P
\end{pmatrix}:[0,\bzeta(Z)]\to \C^2\,\Big|\,\begin{pmatrix}
\psi \\ P
\end{pmatrix} =\begin{pmatrix}\th\circ \bzeta^{-1} \\ \phi\circ \bzeta^{-1}\end{pmatrix} \text{ for some } \begin{pmatrix}\th \\ \phi\end{pmatrix}\in\HmZ\bigg\}
\eeq  and there exists $\big(\psi_\la,P_\l\big)^\top\in\HmZEul$ satisfying $P_\l=(1-\l)\psi_\l$ and 
\beq\label{E:PSILA}
\pa_yD_y\psi_\la + \Big(\frac{\wLP'+2(1-\la)\vLP}{\wLP}-\frac{\rhoLP'}{\rhoLP}\Big)D_y\psi_\la+\frac{2\rhoLP+2(1-\l)\vLP'-(2-\l)(1-\l)}{\wLP}\psi_\la=0.
\eeq
Here we recall the Eulerian sonic point $y_\ast$ from Definition~\ref{def:sonic} and we have defined the key weight 
\beq\label{E:WDEF}
\wLP(y):=1-\vLP^2(y)=1-y^2\omLP^2,
\eeq
which acts as a signed distance to the sonic point; that is, $\wLP$ is strictly positive on $[0,y_*)$, vanishes at $y_*$ with non-vanishing slope, and is strictly negative on $(y_*,\infty]$. 
\end{lemma}

\begin{proof}
This follows directly from Lemma~\ref{L:LAGEUL} and the definition of $\bfLEul$ from~\eqref{E:BFLEULDEF}.
\end{proof}


As the eigenvalue problem is stated primarily in terms of the function $\psi_\la$, we will frequently abuse notation and simply refer to $\psi_\l\in\HmZEul$ instead of making reference to the pair $(\psi_\l,P_\l)$.

In order to study the properties of potential eigenvalues in the right half-plane (of the complex plane), it is essential to understand the regularity requirements imposed by the singular eigenvalue ODE,~\eqref{E:PSILA}. As this is a linear equation of regular type (in the Frobenius sense), we therefore identify the space of linearly independent solutions via a Frobenius analysis in the following lemma.


\begin{lemma}\label{lemma:Frobenius}
Suppose $\l\in\C\setminus\{1\}$ with $\Re\l\geq 0$ is an eigenvalue of $\bfL$ and $\psi_\la\in\HmZEul$ satisfies the Eulerian eigenvalue equation~\eqref{E:PSILA}. Then $\psi_\la(y)$ is real-analytic as a function of $y$ on the closed interval $[0,y_*]$ and there exist two sequences of complex coefficients $(A_j)$ and $(B_j)$ with $A_0,B_0\neq 0$ such that $\psi_\l$ satisfies the power series expansions
\begin{align}
\psi_\la(y)=&\,y\sum_{j=0}^\infty A_j y^j,  && y\ll1,\label{E:EFUNORIGIN}\\
\psi_\la(y)=&\,\sum_{j=0}^\infty B_j (y-y_*)^j,  && |y-y_*|\ll1.\label{E:EFUNSONIC}
\end{align}
\end{lemma}

\begin{proof}
A direct calculation shows that~\eqref{E:PSILA} can equivalently be written as
\beq\label{eq:phila}
 \psi_\la''+V_1(y)\psi_\la'+V_2(y)\psi_\la=0,
 \eeq
 where $V_1$ and $V_2$ are given by
\beqa\label{def:V1V2}
 V_1=&\,\frac{2}{y}-\frac{\rhoLP'}{\rhoLP}+\frac{\wLP'}{\wLP}+\frac{2(1-\la)\vLP}{\wLP},\\
 V_2=&\,-\frac{2}{y^2}-\frac{2\rhoLP'}{y\rhoLP}+\frac{2\wLP'}{y\wLP}+\frac{4(1-\la)\vLP}{y\wLP}+\frac{2\rhoLP}{\wLP}+\frac{2(1-\la)\vLP_y}{\wLP}-\frac{(2-\la)(1-\la)}{\wLP}.
 \eeqa 
 It is then a straightforward exercise to deduce that the Frobenius indices at the origin are $1$ and $-2$. Thus, solutions are of the form
$$\psi_\la(y)=y\sum_{k=0}^\infty A_ky^k\quad \text{ or } \quad \psi_\la(y)= C_1\log y\sum_{k=0}^\infty A_ky^k+y^{-2}\sum_{k=0}^\infty B_ky^k.$$
The regularity hypothesis $\psi_\la\in\HmZEul$ with $m\geq 2$ rules out the second solution, leaving us just with the first possibility:
$$\psi_\la(y)=y\sum_{k=0}^\infty A_ky^k.$$
On the other hand, to compute the Frobenius indices at the sonic point, $y=y_*$, we first take the shifted weight $\tilde w(y)=\wLP(y)(y-y_*)^{-1}$ and verify that the indicial equation  is 
\beqas
0=&\,r(r-1)+\Big(\wLP'(y_*)\frac{1}{\tilde w(y_*)}+\frac{2(1-\la)\vLP(y_*)}{\tilde w(y_*)}\Big)r
=r\Big(r-\frac{1-\la}{1-\frac{1}{y_*}}\Big),
\eeqas
where we have used the identities~\eqref{eq:order012coeffs} for the expansions of the LP solution at $y_*$. This quadratic has roots 
$$r_1=0,\qquad r_2=r_2(\la)=\frac{1-\la}{1-\frac{1}{y_*}}.$$
In the case that $\l\neq \frac{1}{y_*}$, as $\l\neq 1$ and $\Re\l\geq 0$, it is clear that $\Re r_2(\l)<2$, $\Re r_2(\l)\neq 0,1$, and so we conclude the claimed power series expansion of the solution, again from the regularity requirement $\psi_\l\in\HmZEul$, $m\geq 2$. In the exceptional case  $\l=1-\frac{1}{y_*}$, we have $r_2(\l)=1$ and the independent solutions are
$$\psi_\la^{(1)}=\sum_{k=1}^\infty B_k (y-y_*)^k,\qquad \psi_\la^{(2)}=C_1\log (y-y_*)\sum_{k=1}^\infty B_k (y-y_*)^k+\sum_{k=0}^\infty A_k (y-y_*)^k,$$
where we take a suitable branch of the complex logarithm and we assume also that $B_1=1$. A direct computation then reveals that either $C_1\neq 0$ (violating the regularity requirement $\psi_\l\in\HmZEul$) or $A_1=C_1=0$.
\end{proof}

To exclude eigenvalues of the operator $\mathbf{L}$ in different regions of the half-plane $\Re\la\geq0$, we will rely on different, but equivalent, formulations of the eigenvalue problem.


\subsection{Exclusion of eigenvalues with $\Re\la> 1$}


Our first goal is to prove that $\mathbf L$ has no eigenvalues with real part greater than or equal to 1. To this end, we employ the supersymmetric reformulation (for related ideas see, for example, \cite{Costin17}).


\begin{lemma}\label{L:SUPERSYMM}
Suppose $\l\in\C\setminus\{1\}$ is an eigenvalue of $\bfL$ on $\HmZ$ and define $\psi_\l$ as the analytic solution to~\eqref{E:PSILA}.
There exists a smooth function $\mathcal{A}(y)>0$ such that
\beq\label{def:flambda}
f_\la(y):=\frac{y \wLP(y)^{\frac12-\frac12\frac{1-\la}{1-\frac1{y_*}}}}{\sqrt{\mathcal A}\sqrt{\rhoLP}}
\eeq
solves $2(\log f_\l)'(y)=V_1(y)$, where we recall $V_1$ was defined in \eqref{def:V1V2}. Then the function
\beq\label{def:vlam}
v_\la(y)=f_\la(y)\psi_\la(y)
\eeq
satisfies the normal form equation
\beq\label{eq:supersym}
-v_{\la}''+\widetilde{V}v_\la=-\frac{(1-\la)^2}{\wLP^2} v_\la,
\eeq
where
\beqa\label{def:Vtilde}
\widetilde{V}=\frac{2}{y^2}+\big(\frac{\rhoLP'}{\rhoLP}\big)^2-\frac14\Big(\frac{\rhoLP'}{\rhoLP}+\frac{\wLP'}{\wLP}\Big)^2-\frac{\rhoLP''}{2\rhoLP}+\frac{\wLP''}{2\wLP}+\frac{\rhoLP'}{y\rhoLP}-\frac{\wLP'}{y\wLP}-\frac{2\rhoLP}{\wLP}.
\eeqa
\end{lemma}

\begin{proof}
We first note from~\eqref{def:V1V2} and the obvious identity $w'=-2\vLP\vLP'$ that 
\beqas
V_1(y)=\frac{2}{y}-\frac{\rhoLP'}{\rhoLP}+\frac{\wLP'}{\wLP}-\frac{(1-\la)\wLP'}{\wLP\vLP'}
\eeqas
Observing that $\vLP'=\omLP+y\omLP'\geq \frac13>0$ and $\vLP'(y_*)=1-\frac{1}{y_*}$, the existence of the claimed function $\mathcal{A}$ follows easily, so that 
$2\frac{f_\la'}{f_\la}= V_1(y)$.
Thus, expanding derivatives,
\beqas
-v_{\la}''=&\,\big(V_1-2\frac{f_\la'}{f_\la}\big)v_\la'+\big(2\frac{(f_\la')^2}{f_\la^2}-\frac{f_\la'}{f_\la}V_1+V_2-\frac{f_\la''}{f_\la}\big)v_\la=\big(V_2-\frac{f_\la''}{f_\la}\big)v_{\la}.
\eeqas
In order to simplify the right hand side, we note $\frac{f_\la''}{f_\la}=\frac12V_1'+\frac14V_1^2$ and rearrange to obtain
\beqa
\frac{f_\la''}{f_\la}-V_2=&\,\frac{2}{y^2}+\big(\frac{\rhoLP'}{\rhoLP}\big)^2-\frac14\Big(\frac{\rhoLP'}{\rhoLP}+\frac{\wLP'}{\wLP}\Big)^2-\frac{\rhoLP''}{2\rhoLP}+\frac{\wLP''}{2\wLP}+\frac{\rhoLP'}{y\rhoLP}-\frac{\wLP'}{y\wLP}-\frac{2\rhoLP}{\wLP}\\
&+\frac{1-\la}{\wLP^2}\Big(-\vLP'\wLP-\frac{2\vLP \wLP}{y}-\frac{\rho'\vLP \wLP}{\rhoLP}+(1-\la)\vLP^2+(2-\la)\wLP\Big).
\eeqa
The first line on the right hand side  clearly matches the definition of $\widetilde{V}$ in \eqref{def:Vtilde}. To simplify the remainder, we first recall from~\eqref{E:RHOLP}--\eqref{E:OMLP} that $\omLP'=\frac{1-3\omLP}{y}-\frac{\omLP}{\rhoLP}\rhoLP'$, so that we easily verify
\beqas
-\vLP'\wLP-\frac{2\vLP \wLP}{y}-\frac{\rho'\vLP \wLP}{\rhoLP}+(1-\la)\vLP^2+(2-\la)\wLP=\wLP\Big(-\vLP'-\frac{2\vLP}{y}-\frac{\rho'\vLP}{\rhoLP}+1\Big)+(1-\la)=1-\l.
\eeqas
 Hence 
 $$\frac{f_\la''}{f_\la}-V_2=\widetilde{V}+\frac{(1-\la)^2}{\wLP^2}.$$
 This concludes the proof.
\end{proof}

\begin{lemma}\label{L:STRIPONLY}
Suppose $\la$ is an eigenvalue of $\mathbf L$ on $\HmZ$ with $\Re\la>1$. Then $\la\in\R$.
\end{lemma}

\begin{proof}
From Lemma~\ref{L:SUPERSYMM}, as $\l$ is an eigenvalue of $\bfL$, there exists a solution $v_\la$ to the supersymmetric equation~\eqref{eq:supersym} of the form~\eqref{def:vlam}, where the corresponding $\psi_\la$ is analytic by Lemma~\ref{lemma:Frobenius}. We therefore see by definition that, close to the sonic point $y_*$, 
\[|v_\l| +|y-y_*||v_\l'|\leq C|y-y_*|^{\frac12-\Re \frac{1-\la}{1-\frac{1}{y_*}}}\]
while $|\widetilde V|\leq C|y-y_*|^{-1}$. We therefore observe that, as $\Re\l>1$, $|v_\l'|^2$, $\widetilde V|v_\l|^2$ and $\wLP^{-1}|v_\l|^2$ are all integrable on $[0,y_*]$.

 From the super-symmetric form of the eigenfunction equation~\eqref{eq:supersym}, we multiply  by the complex conjugate, $\overline{v_\la}$, and integrate by parts to deduce the energy identity
$$\int_0^{y_*}\big||v_\la'|^2|+\widetilde{V}|v_\la|^2\dif y=-(1-\la)^2\int_0^{y_*}\frac{|v_\la|^2}{\wLP}\dif y,$$
where the integrability of each term is verified from Lemma~\ref{lemma:Frobenius}, the formula~\eqref{def:vlam} and the condition $\Re\la>1$, as is the vanishing of the boundary term arising from integration by parts. 
As the integrands are all real, we conclude that $(\la-1)^2$ must also be real and so, as $\Re\la>1$, therefore $\la\in\R$.
\end{proof}

\begin{lemma}\label{L:REALGEQONE}
There exist no eigenvalues $\la$ of $\mathbf L$ on $\HmZ$ with $\Re\la>1$.
\end{lemma}

\begin{proof}
 From  Lemma~\ref{L:STRIPONLY}, we know that any such eigenvalue must be real. Suppose for a contradiction that $\la>1$ is an eigenvalue. We define $v_1$ via the formula~\eqref{def:vlam} with $\la=1$ and the eigenfunction $\psi_1$ associated to this time-translation mode, see Lemma~\ref{L:GROWINGMODEEUL}. We then derive from~\eqref{eq:supersym} the quotient form equation for $q_\l:=v_\la v_1^{-1}$,
\beq\label{eq:q}
-v_1q_\l''-2v_1'q_\l'=-\frac{(1-\la)^2}{\wLP^2} v_1 q_\l.
\eeq
Multiplying the left hand side by $v_1 q_\l$ and integrating by parts over $(0,y_*)$, we arrive at
\beq
\int_0^{y_*}-v_1q_\l''v_1q_\l\,\dif y=\int_0^{y_*}2v_1v_1'q_\l'q_\l+v_1^2(q_\l')^2\dif y,
\eeq 
where the integrands are all integrable and the boundary terms at $y=0,y_*$ vanish due to Lemma~\ref{lemma:Frobenius},~\eqref{def:vlam} and $\la>1$.
Thus,
 multiplying \eqref{eq:q} by $v_1q_\l$ and integrating, we find
\beqas
\int_0^{y_*}v_1^2(q_\l')^2\dif y=\int_0^{y_*}-\frac{(1-\la)^2}{\wLP^2} (v_1q_\l)^2\,\dif y,
\eeqas
and hence $q_\l\equiv 0$.
\end{proof}


\subsection{Exclusion of eigenvalues at high frequencies ($\Re\la\in[0, 1]$, $\Im\la\gg 1$)}\label{S:LARGEIMAG}


The strategy we employ to exclude eigenvalues in the strip with large imaginary part is based on a high order energy argument for the eigenvalue equation. We therefore begin by deriving a suitable equation for fourth order derivatives of any potential eigenfunctions.

Throughout this section and the next, to isolate real and complex arguments, we set
\beq
1-\la=a+ib.
\eeq

\begin{proposition}\label{prop:A4B4}
Let $\la\in\C$ and suppose $\psi_\la\in\HmZEul$ is a solution to equation~\eqref{E:PSILA}. The function 
\beq\label{def:P}
\Psi_\l:=(\pa_y D_y)^2\psi_\la
\eeq
satisfies the equation
\beq\label{eq:Peqn}
\pa_yD_y \Psi_\l+A^{(4)} D_y\Psi_\l+B^{(4)} \Psi_\l=0,
\eeq
where 
\beq\label{E:A4DEF}
A^{(4)}=\frac{5\wLP'+2a\vLP+2ib\vLP}{\wLP}-\frac{\rhoLP'}{\rhoLP}-D^{(4)},
\eeq
where $D^{(4)}$ is smooth, vanishes at $y=0$,
\beq\label{ineq:D42ests}
D^{(4)}_1:=\Re D^{(4)}=O(b^{-2}),\quad  D^{(4)}_2:=\Im D^{(4)}=O(b^{-1}),
\eeq
and, setting $b_1=\Re(\wLP B)$, $b_2=\Im(\wLP B)$, we have a constant $C>0$, independent of $b$, such that
\beqa\label{eq:b4top}
\big|b^{(4)}_1(y)-b^2\big|\leq &\,C,\\
\big|b^{(4)}_2(y)-b(10\vLP'(y)-2a-1)\big|\leq &\, Cb^{-1},
\eeqa
uniformly in $y\in[0,y_*]$.
\end{proposition}

\begin{proof}
We first derive the main identity by commuting derivatives onto the eigenvalue ODE,~\eqref{E:PSILA}. First set
\beqa\label{E:AUPPERZERO}
A^{(0)}=\frac{\wLP'+2a\vLP+2ib\vLP}{\wLP}-\frac{\rhoLP'}{\rhoLP}
\eeqa
and
\beqa\label{def:b0}
&B^{(0)}=\frac{b^{(0)}}{v},\qquad &&b^{(0)}=b^{(0)}_1+ib^{(0)}_2,\\
&b^{(0)}_1=b^2-a^2-a+2\rhoLP+2a\vLP',\qquad &&b^{(0)}_2=b(2\vLP'-2a-1),
\eeqa
so that
\eqref{E:PSILA} becomes
\beq
\pa_yD_y\psi_\la+A^{(0)}D_y\psi_\la+B^{(0)}\psi_\la=0.
\eeq
Note that as $\rhoLP$ and $\omLP$ are even (\emph{cf}.~Appendix~\ref{SS:SSP}) , $A^{(0)}$ is odd-in-$y$ and $B^{(0)}$ is even-in-$y$.

Provided $B^{(0)}\neq 0$, we may apply  $\pa_yD_y$ to obtain the ODE
$$\pa_yD_y(\pa_y D_y\psi)+A^{(2)}D_y(\pa_yD_y\psi)+B^{(2)}\pa_yD_y\psi=0,$$
where
\begin{align}
A^{(1)}=&\,A^{(0)}-\frac{(B^{(0)})'}{B^{(0)}},\quad B^{(1)}=B^{(0)}+D_y A^{(0)}-\frac{(B^{(0)})'}{B^{(0)}}A^{(0)},\label{E:A1B1DEF}\\
A^{(2)}=&\,A^{(1)}-\frac{(B^{(1)})'}{B^{(1)}},\quad B^{(2)}=-\frac{2}{y}A^{(1)}+B^{(1)}+\pa_y A^{(1)}-\frac{(B^{(1)})'}{B^{(1)}}A^{(1)}.\label{E:A2B2DEF}
\end{align}
Observe from the parity considerations that each $A^{(j)}$ is odd-in-$y$ and each $B^{(j)}$ is even-in-$y$. Moreover, it is clear that
\[ -\frac{(B^{(0)})'}{B^{(0)}} = \frac{\wLP'}{\wLP}-d^{(0)},\qquad d^{(0)}=-\frac{(b^{(0)})'}{b^{(0)}},\]
where $b^2\Re(d^{(0)})$ and $b\Im(d^{(0)})$ are bounded independent of $b\geq 1$ in $C^4$.  This gives us from~\eqref{E:A1B1DEF}
\[ A^{(1)}=\frac{2\wLP'+2a\vLP+2ib\vLP}{\wLP}-\frac{\rhoLP'}{\rhoLP}-D^{(1)},\qquad D^{(1)}=d^{(0)},\]
and 
\beqas
B^{(1)}=\frac{b^{(1)}}{\wLP}=&\,\frac{b^{(0)}}{\wLP}+\frac{\wLP''+2a\vLP'+2ib\vLP'}{\wLP}-\Big(\frac{\rhoLP'}{\rhoLP}\Big)'+\frac{2}{y}\Big(\frac{\wLP'+2a\vLP+2ib\vLP}{\wLP}-\frac{\rhoLP'}{\rhoLP}\Big)\\
&-\frac{\wLP'}{\wLP}\frac{\rhoLP'}{\rhoLP}-d^{(0)}\frac{\wLP'+2a\vLP+2ib\vLP}{\wLP}+d^{(0)}\frac{\rhoLP'}{\rhoLP},
\eeqas
where we have noted that the most singular terms of order $\wLP^{-2}$ appearing in $B^{(1)}$ from~\eqref{E:A1B1DEF} have cancelled. We see that $b^{(1)}$ is a smooth, even function and satisfies
\beqas
|\Re(b^{(1)})-b^2|\leq C,\qquad \big|\Im(b^{(1)})-b(4\vLP'+4\omLP-2a-1)\big|\leq Cb^{-1}.
\eeqas
  Inserting these expressions into~\eqref{E:A2B2DEF}, we set 
  \beqs
  d^{(1)}=\frac{(b^{(1)})'}{b^{(1)}},\qquad D^{(2)}=d^{(0)}+d^{(1)}
  \eeqs and easily find
  \beqs
A^{(2)}=\frac{3\wLP'+2a\vLP+2ib\vLP}{\wLP}-\frac{\rhoLP'}{\rhoLP}-D^{(2)},
\eeqs
where again, $b^2\Re(d^{(1)})$ and $b\Im(d^{(1)})$ are uniformly bounded independent of $b\geq 1$ in $C^4$ and 
\beqas
B^{(2)}=\frac{b^{(2)}}{\wLP},\qquad \big|\Re(b^{(2)})-b^2\big|\leq C,\quad \big|\Im(b^{(2)})-b(6\vLP'-2a-1)\big|\leq Cb^{-1}.
\eeqas

Iterating this procedure again, we arrive at $A^{(4)}$, $B^{(4)}$ with the claimed structure and solving~\eqref{eq:Peqn}. Note that $B^{(j)}\neq 0$ for each $j$ for sufficiently large $b$ by the claimed estimates. 
\end{proof}

Based on the ODE~\eqref{eq:Peqn}, in the next proposition, we will derive the key energy identity that excludes eigenvalues with large imaginary part. 
As a preliminary step, we introduce the function
\beq\label{def:U}
U(y)=-\frac{2\vLP(y)(\vLP'(y)-\vLP'(y_*))}{\wLP(y)}>0\text{ is smooth and bounded.}
\eeq
Defining the constants
\beq
a_*=\vLP'(y_*)=1-\frac{1}{y_*}\quad\text{ and }\quad\al_*=a-2a_*,
\eeq
 we may write
\beq\label{E:AFOURDEF}
A^{(4)}=\frac{3\wLP'+2\al_*\vLP+2ib\vLP}{\wLP}+2U-\frac{\rhoLP'}{\rhoLP}-D^{(4)}_1-iD^{(4)}_2. 
\eeq
Since $2<y_\ast<3$, we have the crucial sign condition $\al_* = a - 2a_\ast<0$.

\begin{lemma}\label{L:LARGEBENERGY}
We define a weight function $$\chi=\frac{1}{\rho}\exp\big(\int_0^y (2U-D^{(4)}_1)\big)$$ such that
\beq\label{def:chi}
\frac{\chi'}{\chi}=-\frac{\rhoLP'}{\rhoLP}-D^{(4)}_1+2U.
\eeq
Suppose $\psi_\la\in\HmZEul$ is a solution to equation~\eqref{E:PSILA} and let $\Psi_\l=(\pa_y D_y)^2\psi_\la$ be the corresponding solution to the ODE~\eqref{eq:Peqn}.
We write $\Psi_\l=\Psi_{\l,1}+i\Psi_{\l,2}$ . Then $\Psi_\l$ satisfies the energy identity
\beqa\label{eq:Penergy}
-&\int_0^{y_*}\wLP^3\chi |D_y\Psi_\l|^2y^2\dif y+\int_0^{y_*}\wLP^2 \chi{H_\l} |\Psi_\l|^2y^2\,\dif y\\
&=\int_0^{y_*}\wLP^3\chi D^{(4)}_2\Big(D_y(\Psi_{\l,1},\Psi_{\l,2})\cdot((\Psi_{\l,2},-\Psi_{\l,1})+\frac{b}{\al_*}(\Psi_{\l,1},\Psi_{\l,2}))\Big)y^2\dif y,
\eeqa
where
\beq\label{E:HPSI}
H_\l(y)=\big(\al_*+\frac{b^2}{\al_*}\big)\big(-2\vLP'+1+\frac{4\vLP^2\vLP'}{\wLP}+\vLP D^{(4)}_1-2\vLP U\big)+b^{(4)}_1+\frac{b}{\al_*}b^{(4)}_2.
\eeq
\end{lemma}

\begin{proof}
Throughout the proof, to simplify notation, we drop the subscript $\la$. We let $\al_*=a-2a_*$ and introduce the weight $\wLP^3\chi$ so that \eqref{eq:Peqn} becomes
\beq
\pa_y(\wLP^3\chi D_y\Psi)+\wLP^2\chi\big(2\al_*\vLP+2ib\vLP-i\wLP D^{(4)}_2\big)D_y\Psi+\wLP^2 \chi\big(b^{(4)}_1+ib^{(4)}_2)\Psi=0.
\eeq
We note the identity
\beq
D_y\Psi\overline{\Psi}=\frac{1}{2}\pa_y\big(|\Psi|^2\big)+\frac{2}{y}|\Psi|^2+i\{\Psi_2,\Psi_1\}.
\eeq
Testing against the function $\overline{\Psi}$ in the weighted space $L^2_{y^2}$ and integrating by parts in the first term, we have
\beqa
-\int_0^{y_*}&\wLP^3\chi|D_y\Psi|^2y^2\dif y+\int_0^{y_*}\wLP^2 \chi(2\al_*\vLP+2ib\vLP-i\wLP D^{(4)}_2)\big(\frac{1}{2}\pa_y\big(|\Psi|^2\big)+\frac{2}{y}|\Psi|^2+i\{\Psi_2,\Psi_1\}\big)y^2\dif y\\
&+\int_0^{y_*}\wLP^2\chi\big(b^{(4)}_1+ib^{(4)}_2)|\Psi|^2y^2\dif y =0.
\eeqa
Here we have used the fact that $D_y\Psi$ and $\Psi$ are both bounded at the origin to see that both boundary terms vanish.

Taking real and imaginary parts,
\begin{align}
-&\int_0^{y_*}\wLP^3\chi|D_y\Psi|^2y^2\dif y+\int_0^{y_*}\wLP^2 \chi\Big(\al_*\vLP\big(\pa_y\big(|\Psi|^2\big)+\frac{4}{y}|\Psi|^2\big)-2b\vLP\{\Psi_2,\Psi_1\}\Big)y^2\dif y\notag\\
&+\int_0^{y_*}\wLP^3 \chi D^{(4)}_2\{\Psi_2,\Psi_1\}y^2\dif y+\int_0^{y_*}\wLP^2\chi b^{(4)}_1|\Psi|^2y^2\dif y=0 ,\label{eq:Penergyreal}\\
&\int_0^{y_*}\wLP^2 \chi b\vLP\big(\pa_y\big(|\Psi|^2\big)+\frac{4}{y}|\Psi|^2\big)y^2\dif y+\int_0^{y_*}\wLP^2\chi b^{(4)}_2|\Psi|^2y^2\dif y\notag\\
&-\int_0^{y_*}\wLP^3 \chi D^{(4)}_2\big(\frac12\pa_y\big(|\Psi|^2\big)+\frac{2}{y}|\Psi|^2\big)y^2\dif y  =-\int_0^{y_*}2\wLP^2 \chi\al_*\vLP\{\Psi_2,\Psi_1\}y^2\dif y.\label{eq:Penergyim}
\end{align}
Substituting \eqref{eq:Penergyim} for the Poisson bracket term in \eqref{eq:Penergyreal}, we have
\beqa\label{E:ENERGYRECOMBINE}
0=&-\int_0^{y_*}\wLP^3\chi|D_y\Psi|^2y^2\dif y+\int_0^{y_*}\wLP^2 \chi\al_*\vLP\big(\pa_y\big(|\Psi|^2\big)+\frac{4}{y}|\Psi|^2\big)y^2\dif y\\
&+\int_0^{y_*}\wLP^3 \chi D^{(4)}_2\{\Psi_2,\Psi_1\}y^2\dif y+\int_0^{y_*}\wLP^2\chi b^{(4)}_1|\Psi|^2y^2\dif y\\
&+\frac{b}{\al_*}\Big(\int_0^{y_*}\wLP^2 \chi b\vLP\big(\pa_y\big(|\Psi|^2\big)+\frac{4}{y}|\Psi|^2\big)y^2\dif y+\int_0^{y_*}\wLP^2\chi b^{(4)}_2|\Psi|^2y^2\dif y\\
&-\int_0^{y_*}\wLP^3 \chi D^{(4)}_2\big(\frac12\pa_y\big(|\Psi|^2\big)+\frac{2}{y}|\Psi|^2\big)y^2\dif y\Big)\\
=&\,-\int_0^{y_*}\wLP^3\chi |D_y\Psi|^2y^2\dif y+\int_0^{y_*}\wLP^2 \chi \big(\al_*+\frac{b^2}{\al_*}\big)\vLP\big(\pa_y\big(|\Psi|^2\big)+\frac{4}{y}|\Psi|^2\big)y^2\dif y\\
&+\int_0^{y_*}\wLP^2\chi \big(b^{(4)}_1+\frac{b}{\al_*}b^{(4)}_2\big)|\Psi|^2y^2\dif y+\int_0^{y_*}\wLP^3 \chi D^{(4)}_2\{\Psi_2,\Psi_1\}y^2\dif y\\
&-\frac{b}{\al_*}\int_0^{y_*}\wLP^3 \chi D^{(4)}_2\big(\frac12\pa_y\big(|\Psi|^2\big)+\frac{2}{y}|\Psi|^2\big)y^2\dif y\Big).
\eeqa
Now note that from the LP equations~\eqref{E:RHOLP}--\eqref{E:OMLP}, we have $(\vLP\rhoLP)'=(1-2\omLP)\rhoLP$ so that $-\frac{\vLP\rhoLP'}{\rhoLP}=\vLP'+2\omLP-1$ and hence
\beqa
-\pa_y\Big(y^2\wLP^2 \chi \vLP \Big)=&\,-y^2\wLP^2\chi \Big(\frac{2\vLP}{y}-\frac{4\vLP^2\vLP'}{\wLP}-\frac{\vLP\rhoLP'}{\rhoLP}-\vLP D^{(4)}_1+2\vLP U+\vLP'\Big)\\
=\,&y^2\wLP^2\chi \big(-4\omLP-2\vLP'+1+\frac{4\vLP^2\vLP'}{\wLP}+\vLP D^{(4)}_1-2\vLP U\big),
\eeqa
where we have also used \eqref{def:chi} for the quantity $\frac{\chi '}{\chi }$.

Therefore, integrating by parts in the first term containing $\pa_y(|\Psi|^2)$ (noting again that all boundary terms vanish) yields
\beqa
0=&\,-\int_0^{y_*}\wLP^3\chi |D_y\Psi|^2y^2\dif y+\int_0^{y_*}\wLP^2 \chi \big(\al_*+\frac{b^2}{\al_*}\big)4\omLP |\Psi|^2\big)y^2\dif y\\
&+\int_0^{y_*}\wLP^2 \chi \big(\al_*+\frac{b^2}{\al_*}\big)\big(-4\omLP-2\vLP'+1+\frac{4\vLP^2\vLP'}{\wLP}+\vLP D^{(4)}_1-2\vLP U\big)|\Psi|^2y^2\dif y\\
&+\int_0^{y_*}\wLP^2\chi \big(b^{(4)}_1+\frac{b}{\al_*}b^{(4)}_2\big)|\Psi|^2y^2\dif y+\int_0^{y_*}\wLP^3 \chi D^{(4)}_2\{\Psi_2,\Psi_1\}y^2\dif y\\
&-\frac{b}{\al_*}\int_0^{y_*}\wLP^3 \chi D^{(4)}_2\big(\frac12\pa_y\big(|\Psi|^2\big)+\frac{2}{y}|\Psi|^2\big)y^2\dif y\Big).
\eeqa
Noting now that the $4\omLP$ terms exactly cancel, we re-group terms as
\beqa\label{eq:fourthenergyfull}
-&\int_0^{y_*}\wLP^3\chi |D_y\Psi|^2y^2\dif y\\
&+\int_0^{y_*}\wLP^2 \chi \Big(\big(\al_*+\frac{b^2}{\al_*}\big)\big(-2\vLP'+1+\frac{4\vLP^2\vLP'}{\wLP}+\vLP D^{(4)}_1-2\vLP U\big)+b^{(4)}_1+\frac{b}{\al_*}b^{(4)}_2\Big)|\Psi|^2y^2\dif y\\
=&-\int_0^{y_*}\wLP^3 \chi D^{(4)}_2\Big(\{\Psi_2,\Psi_1\}-\frac{b}{\al_*}\big(\frac12\pa_y\big(|\Psi|^2\big)+\frac{2}{y}|\Psi|^2\big)\Big)y^2\dif y,
\eeqa
and observe that the second line is exactly 
$$\int_0^{y_*}\wLP^2 \chi {H_\l} |\Psi|^2y^2\,\dif y$$
as required.
Considering now the right hand side of this energy identity, we  see that
$$\{\Psi_2,\Psi_1\}=\pa_y\Psi_2\Psi_1-\pa_y\Psi_1\Psi_2=(D_y\Psi_2)\Psi_1-\frac{2}{y}\Psi_2\Psi_1-D_y\Psi_1\Psi_2+\frac{2}{y}\Psi_2\Psi_1=-(D_y\Psi)\cdot \Psi^\perp,$$
where we have treated $\Psi=(\Psi_1,\Psi_2)$, $D_y\Psi=(D_y\Psi_1,D_y\Psi_2)$ as vectors in $\R^2$ to take the perpendicular vector, while
$$\frac12\pa_y\big(|\Psi|^2\big)+\frac{2}{y}|\Psi|^2=\pa_y\Psi_1\Psi_1+\pa_y\Psi_2\Psi_2+\frac{2}{y}(\Psi_1\Psi_1+\Psi_2\Psi_2)=D_y\Psi\cdot\Psi.$$
Thus 
\beqa
-\int_0^{y_*}&\wLP^3\chi D^{(4)}_2\Big(\{\Psi_2,\Psi_1\}-\frac{b}{\al_*}\big(\frac12\pa_y\big(|\Psi|^2\big)+\frac{2}{y}|\Psi|^2\big)\Big)y^2\dif y\\
=&\,\int_0^{y_*}\wLP^3\chi D^{(4)}_2\Big(D_y\Psi\cdot(\Psi^\perp+\frac{b}{\al_*}\Psi)\Big)y^2\dif y
\eeqa
as required.
\end{proof}

\begin{remark}
A simple estimate on the right hand side of \eqref{eq:Penergy} shows 
\beqa
&\int_0^{y_*}\wLP^3\chi D^{(4)}_2\Big(D_y\Psi\cdot(\Psi^\perp+\frac{b}{\al_*}\Psi)\Big)y^2\dif y\\
&\geq -\int_0^{y_*}\wLP^3\chi |D_y\Psi|^2y^2\dif y-\frac{1+\frac{b^2}{\al_*^2}}{4}\int_0^{y_*}\wLP^3\chi |D^{(4)}_2|^2|\Psi|^2y^2\dif y.
\eeqa
So to exclude   $\la=1-a-ib$  as an eigenvalue, it is sufficient to estimate
\beq
{H_\l}(y)+\frac{1+\frac{b^2}{\al_*^2}}{4}\wLP|D^{(4)}_2|^2
\eeq
and show this is negative.
\end{remark}


\begin{remark}
From \eqref{eq:b4top} and \eqref{ineq:D42ests}, we identify the top order terms with respect to $b$ in ${H_\l}$ as
\beqa
{H_\l}+\frac{1+\frac{b^2}{\al_*^2}}{4}\wLP|D^{(4)}_2|^2=&\,\frac{b^2}{\al_*}\Big(-2\vLP'+1+\frac{4\vLP^2\vLP'}{\wLP}+\vLP D^{(4)}_1-2\vLP U+\al_*+10\vLP'-2a-1\Big)+O(1)\\
=&\,\frac{b^2}{\al_*}\Big(\frac{4\vLP^2\vLP'}{\wLP}-a-2a_*+8\vLP'-2\vLP U\Big)+O_{b\to\infty}(1).
\eeqa
Analysing the leading order coefficient, we first observe that, as $y_*\in[2,3]$, $a_*\leq \frac23$, while also $a\in[0,1]$ and $\vLP'=y\omLP'+\omLP\geq \frac13$, so that
\beqas
-a-2a_*+8\vLP'\geq -1 - \frac{4}{3} +\frac{8}{3} = \frac{1}{3},\\
\frac{4\vLP^2\vLP'}{\wLP}-2\vLP U = \frac{4\vLP^2}{\wLP} (2\vLP' - \vLP'(y_*)),
\eeqas
where $2\vLP' - \vLP'(y_*) \ge 2 \frac{1}{3} -\frac{2}{3} =0 $. Thus,
$$\frac{4\vLP^2\vLP'}{\wLP}-a-2a_*+8\vLP'-2\vLP U \geq\frac13.$$
Hence, as $\al_*=a-2a_*<0$ for all $a\in[0,1]$, the energy estimate provides a contradiction provided $b$ is sufficiently large. 
\end{remark}

\begin{proposition}\label{P:NOLARGEIMAG}
There exist no eigenvalues $\la=1-a-ib$ of $\bfL$ on $\HmZ$ with $a\in[0,1]$ and $b\geq 8$.
\end{proposition}

\begin{proof}
Suppose for a contradiction that there exists such an eigenvalue. Then, moving to the Eulerian framework, we obtain a function $\psi_\la\in\HmZEul$ solving~\eqref{eq:phila} and therefore satisfying the expansions~\eqref{E:EFUNORIGIN}--\eqref{E:EFUNSONIC} of Lemma~\ref{lemma:Frobenius}. Defining $\Psi_\l=(\pa_yD_y)^2\psi_\la$, from Proposition~\ref{prop:A4B4}, we obtain that $\Psi_\l$ satisfies~\eqref{eq:Peqn} and hence, from Lemma~\ref{L:LARGEBENERGY}, $\Psi_\l$ satisfies the energy identity~\eqref{eq:Penergy}. Now applying  the interval arithmetic Lemma~\ref{L:HLAMBDASIGN}, as described in Appendix~\ref{S:ENERGYIA},  we obtain that, for $b\geq 8$, $a\in[0,1]$, 
\beq
{H_\l}(y)+\frac{1+\frac{b^2}{\al_*^2}}{4}\wLP|D^{(4)}_2|^2<0
\eeq
for all $y\in[0,y_*]$. We therefore conclude that $\Psi_\l\equiv 0$, and hence $D_y\pa_yD_y\psi_\la=c_1\in\R$, implying, as $\psi_\la\in\HmZEul$, that $\psi_\l$ is solving this ODE yields $\pa_yD_y\psi_\la= \frac{c_1y}{3}$, and hence $D_y\psi_\la = \frac{c_1y^2}{6}+c_2$ and $\psi_\la =\frac{c_1 y^3}{30}+\frac{c_2 y}{3}$. However, it is simple to see that no function of this form is an exact solution of the eigenvalue equation~\eqref{eq:phila}, and hence we deduce a contradiction.
\end{proof}



\subsection{Exclusion of eigenvalues at low frequencies ($\Re\la\in[0, 1]$, $\Im\la\ll 1$)}\label{S:SMALLIMAG}

Although we have now restricted the region in which potential unstable eigenvalues of the operator $\bfL$ may lie to a given compact set ($\Re\l\in[0,1]$, $|\Im\l|\leq 8$), in order to simplify the implementation of interval arithmetic, we further eliminate the possibility of eigenvalues close to the real line (except for the trivial, growing mode). In order to eliminate the obstruction caused by the growing mode, we work now with the following quotient form of the equation.

The quotient $Q_\la:=\frac{\psi_\la}{g_1}$, where $g_1$ is the Eulerian eigenfunction for the trivial mode $\la=1$ (see Lemma~\ref{L:GROWINGMODEEUL}), satisfies
\beq\label{eq:Qlambda}
Q_\la''+W_1Q_\la'+W_2Q_\la=0,
\eeq
where
\beqa\label{def:W1W2}
W_1=&\,\frac{4}{y}+\frac{\wLP'+2(1-\la)\vLP}{\wLP}+\frac{\rhoLP'}{\rhoLP}-\frac{2\omLP'}{1-\omLP},\\
W_2=&\,\frac{1}{\wLP}\Big(\la(1-\la)+2(1-\la)\omLP\frac{1-\vLP'}{1-\omLP}\Big).
\eeqa
The proof of this identity is a direct computation based on~\eqref{eq:phila}.

Based on this identity, we  prove the following proposition.

\begin{proposition}\label{P:NOSMALLIMAG}
There do not exist any eigenvalues $\l=1-a-ib$ of $\bfL$ on $\HmZ$ such that $a\in[0,1]$ and $|b|\leq \frac15$ except the trivial eigenvalue $\l=1$.
\end{proposition}

\begin{proof}
The proof proceeds in several steps, similar to Section~\ref{S:LARGEIMAG}. However, for convenience, we do not respect the gradient/divergence structure, but instead simply derive an equation for the derivatives of the quotient $Q_\l$. We suppose for a contradiction that there exists an such an eigenvalue $\l$, and hence that there is a non-trivial function $Q_\l$ solving~\eqref{eq:Qlambda}. By Lemma~\ref{lemma:Frobenius}, such a $Q_\l$ is necessarily analytic.

\emph{Step 1:} First, we define the function
\beq\label{E:PLADEF}
P_\l:=Q_\l''.
\eeq
Then $P_\l$ satisfies the equation
\beq\label{eq:Pequation}
P_\l''+\widetilde{A}^{(2)}P_\l'+\widetilde{B}^{(2)}P_\l=0,
\eeq
where
\beq\label{def:A2QB2Q}
\widetilde{A}^{(2)}=\frac{6}{y}+\frac{\wLP'+2(a-2a_*)\vLP+2ib\vLP}{\wLP}+(\widetilde{a}_1^{(2)}+i\widetilde{a}_2^{(2)}),\quad \widetilde{B}^{(2)}=\frac{ \widetilde{b}_1^{(2)} +i\widetilde{b}_2^{(2)}  }{\wLP},
\eeq
for some smooth, odd, real-valued functions $\widetilde{a}_j^{(2)}$ and even $\widetilde{b}_j^{(2)}$, $j=1,2$ satisfying
\beqa
\widetilde{a}_1^{(2)}=&\,\frac{\rhoLP'}{\rhoLP}-\frac{2\omLP'}{1-\omLP}+2U-\widetilde{d}^{(0)}_1-\widetilde{d}^{(1)}_1,\qquad \widetilde{a}_2^{(2)}=-\widetilde{d}^{(0)}_2-\widetilde{d}^{(1)}_2,
\eeqa
where $U$ was defined above in~\eqref{def:U} and
the precise form of the remaining coefficients is contained in Appendix~\ref{S:PEQCOEFFS}, which provides the proof of this claim.

\emph{Step 2:} We next derive an energy identity for the quantity $P_\l$.
We define the constant $\al_* = a-2a_*$ and a weight $\widetilde{\chi}\geq0$ satisfying
\beq\label{def:GP}
\frac{\widetilde{\chi}'}{\widetilde{\chi}}=\frac6y+\widetilde{a}_1^{(2)}.
\eeq
Then  $P$ satisfies the energy-type inequality
\beqa\label{ineq:Penergyorigin}
\Big(\al_*+\frac{b^2}{\al_*}\Big)\widetilde{\chi}(y_*)|P(y_*)|^2+\int_0^{y_*}\widetilde{\chi}\Big(\widetilde{H}_\l + \frac{1+\frac{b^2}{\al_*^2}}{4}\wLP|\widetilde{a}^{(2)}_2|^2\Big)|P|^2\,\dif y\geq 0,
\eeqa
where
\beqa\label{def:HP}
\widetilde{H}_\l=\big(\al_*+\frac{b^2}{\al_*}\big)\Big(-4\omLP-1+\frac{2\omLP'\vLP}{1-\omLP}-2\vLP U+\vLP(\widetilde{d}^{(0)}_1+\widetilde{d}^{(1)}_1)\Big)+\widetilde{b}^{(2)}_1+\frac{b}{\al_*}\widetilde{b}^{(2)}_2.
\eeqa
To prove~\eqref{ineq:Penergyorigin}, we begin by multiplying the equation \eqref{eq:Pequation} by $\wLP\widetilde{\chi}$ and using \eqref{def:GP} to combine terms as (dropping the subscript $\l$ for notational convenience)
\beq
(\wLP\widetilde{\chi}P')'+\big(2\al_*\vLP+2ib\vLP+i\wLP\widetilde{a}^{(2)}_2\big)\widetilde{\chi}P'+(\widetilde{b}^{(2)}_1+i\widetilde{b}^{(2)}_2)\widetilde{\chi}P=0.
\eeq
Then, following a similar argument to the proof of Lemma~\ref{L:LARGEBENERGY}, we test the equation with $\overline{P}$, integrate and, after splitting into real and imaginary parts to eliminate the Poisson bracket terms, with a little work, we arrive at an identity analogous to~\eqref{E:ENERGYRECOMBINE}:
\beqa\label{eq:Penergyintegralidentity}
{}&-\int_0^{y_*}\wLP\widetilde{\chi}|P'|^2\,\dif y+\int_0^{y_*}\widetilde{\chi}\big(\al_*+\frac{b^2}{\al_*}\big)\vLP\pa_y(|P|^2)\,\dif y+\int_0^{y_*}\widetilde{\chi}\big(\widetilde{b}^{(2)}_1+\frac{b}{\al_*}\widetilde{b}^{(2)}_2\big)|P|^2\,\dif y\\
&-\int_0^{y_*}\wLP\widetilde{\chi}\widetilde{a}^{(2)}_2\big(\{P_2,P_1\}-\frac12\frac{b}{\al_*}\pa_y(|P|^2)\big)\,\dif y=0.
\eeqa
We now observe that 
\beqa
-\pa_y(\vLP \widetilde{\chi})=&\,-\widetilde{\chi}\Big(\frac{6\vLP}{y}+\frac{\rhoLP'\vLP}{\rhoLP}-\frac{2\omLP'\vLP}{1-\omLP}+2\vLP U-\vLP(\widetilde{d}^{(0)}_1+\widetilde{d}^{(1)}_1)+\vLP'\Big)\\
=&\,-\widetilde{\chi}\Big(4\omLP+1-\frac{2\omLP'\vLP}{1-\omLP}+2\vLP U-\vLP(\widetilde{d}^{(0)}_1+\widetilde{d}^{(1)}_1)\Big),
\eeqa
where we have used the identity $\frac{\rhoLP'\vLP}{\rhoLP}=-\vLP'-2\omLP+1$.

In addition, we may consider $P=(P_1,P_2)$ as a vector in $\R^2$ and regroup the terms
\beqa
\{P_2,P_1\}-\frac12\frac{b}{\al_*}\pa_y(|P|^2)=-(P'\cdot P^\perp+\frac{b}{\al_*}P'\cdot P\big)=-P'\cdot(P^\perp+\frac{b}{\al_*}P).
\eeqa
Integrating by parts in the first $\pa_y(|P|^2)$ term of \eqref{eq:Penergyintegralidentity}, we use both of these identities, along with $\vLP(y_*)=1$ in the boundary term,  and get
\beqa
&-\int_0^{y_*}\wLP\widetilde{\chi}|P'|^2\,\dif y+\int_0^{y_*}\widetilde{\chi}\big(\al_*+\frac{b^2}{\al_*}\big)\Big(-4\omLP-1+\frac{2\omLP'\vLP}{1-\omLP}-2\vLP U+\vLP(\widetilde{d}^{(0)}_1+\widetilde{d}^{(1)}_1)\Big)|P|^2\,\dif y\\
&+\Big(\al_*+\frac{b^2}{\al_*}\Big)\widetilde{\chi}(y_*)|P(y_*)|^2+\int_0^{y_*}\widetilde{\chi}\big(\widetilde{b}^{(2)}_1+\frac{b}{\al_*}\widetilde{b}^{(2)}_2\big)|P|^2\,\dif y\\
=&-\int_0^{y_*}\wLP\widetilde{\chi}\widetilde{a}^{(2)}_2\big(P'\cdot(P^\perp+\frac{b}{\al_*}P)\big)\,\dif y\\
\geq&\,-\int_0^{y_*}\wLP\widetilde{\chi}|P'|^2\,\dif y-\frac{1+\frac{b^2}{\al_*^2}}{4}\int_0^{y_*}\wLP\widetilde{\chi}|\widetilde{a}^{(2)}_2|^2|P|^2\,\dif y.
\eeqa
Cancelling the weighted derivative terms, we obtain the claimed inequality~\eqref{ineq:Penergyorigin}.

\emph{Step 3:} Finally, we apply the interval arithmetic Lemma~\ref{L:HLAMBDATILDE} to verify that there exists a constant $c_1>0$ such that, for all  $a\in[0,1]$, $b\in[0,\frac15]$,  the following inequality holds:
\beqa
\widetilde{H}_\l + \frac{1+\frac{b^2}{\al_*^2}}{4}\wLP|\widetilde{a}^{(2)}_2|^2\leq -c_1<0,\ \ \ \ y\in[0,y_*].
\eeqa
We therefore conclude from~\eqref{ineq:Penergyorigin} that $P_\l\equiv 0$ and hence $Q_\l$ is an affine function. However, it is straightforward to see that, for $\l\neq 1$,~\eqref{eq:Qlambda} cannot be solved by an affine function. We therefore arrive at the desired contradiction.
\end{proof}

\subsection{Exclusion of eigenvalues in the intermediate region}


The final step to conclude the mode stability of the operator $\bfL$ is to eliminate the possibility of eigenvalues with imaginary part between $\frac15$ and $8$. The previous sections were essentially analytic and relied on interval arithmetic only to  provide explicit numerical bounds, rather than simply proving the existence of sufficient constants. In contrast, the argument of this section relies heavily and directly on the implementation of interval arithmetic (and is the reason we need the exact constants in Propositions~\ref{P:NOLARGEIMAG} and~\ref{P:NOSMALLIMAG}).

\begin{proposition}\label{P:NOINTERIMAG}
The operator $\bfL$ has no eigenvalues $\l=1-a-ib$ with $a\in[0,1]$ and $b\in[\frac15,8]$. 
\end{proposition}

We remark before continuing that combining this proposition with  Propositions~\ref{P:NOLARGEIMAG} and~\ref{P:NOSMALLIMAG}, we conclude the proof of Theorem~\ref{T:SPECTRAL}.

\begin{proof}
The proof proceeds as follows. Given a candidate eigenvalue $\l=1-a-ib$ as in the statement of the proposition, we recall from Lemma~\ref{lemma:Frobenius} that an eigenfunction is rigidly constrained (up to taking a constant complex scalar multiple) to obey the analytic expansions~\eqref{E:EFUNORIGIN}--\eqref{E:EFUNSONIC} at the origin and sonic point.  We first apply Lemma~\ref{L:EFUNCTIONTAYLOR}(i) to construct the regular solution around the sonic point by Taylor expansion with precise error bounds and then employ interval arithmetic to solve the eigenfunction ODE~\eqref{eq:phila} backwards from the sonic point $y_*$ to $y=0.01$, giving us $(\psi_\l^r,(\psi_\l^r)')$. We then compare this to the solution obtained directly by Taylor expansion around the origin from Lemma~\ref{L:EFUNCTIONTAYLOR}(ii) to obtain $(\psi_\l^l,(\psi_\l^l)')$. We verify that these two complex vectors are not (complex) linearly dependent, thus showing that there exists no regular solution of~\eqref{eq:phila} defined on the whole of $[0,y_*]$, and hence $\l$ is not an eigenvalue. This is executed using the function \verb!intermediate_evalue_excluder! in the attached code file.
\end{proof}


\subsection{Surjectivity}

In Section~\ref{S:ACCRETIVITY}, specifically Theorem~\ref{T:LUPH}, in order to show the maximal accretivity of a compact perturbation $\mathscr{L}_m$ of $\bfL$ and so apply the Lumer-Phillips theorem, we must show the surjectivity of $\mathscr{L}_m-\l\mathbf{I}$ for all real $\l$ sufficiently large. In preparation for this, we here prove the equivalent surjectivity statement for the full operator $\bfL$.

\begin{lemma}[Surjectivity of $\bf L - \l \bfI$ for $\l$ sufficiently large]\label{L:ONTO}
There exists a $\l_0\in\mathbb R$, such that for any $\la\geq\la_0$,  $\begin{pmatrix} f_1 \\ f_2 \end{pmatrix}\in \HmZ$, there exists a unique 
$\begin{pmatrix} \theta \\ \phi\end{pmatrix}\in \HmZ$ such that 
\begin{align}\label{E:ONTO}
\left(\mathbf{L}  - \l \bfI\right) \begin{pmatrix} \theta \\ \phi\end{pmatrix} = \begin{pmatrix} f_1 \\ f_2\end{pmatrix}.
\end{align}
\end{lemma}

\begin{proof}
{\bf Step 1.} We start by showing that for any $\begin{pmatrix} f_1 \\ f_2 \end{pmatrix}\in \DZodd$ there exists a unique $\begin{pmatrix} \theta \\ \phi\end{pmatrix}\in  \HmZ$ so that~\eqref{E:ONTO} is true. A simple density  argument then allows us to infer the claim for $\begin{pmatrix} f_1 \\ f_2 \end{pmatrix}\in \HmZ$. 

Suppose $(f_1,f_2)^\top\in\DZodd$. We begin by observing that the surjectivity problem for $\bfL-\la \bfI$ is equivalent to the corresponding surjectivity problem defined for the linearised Eulerian operator $\bfL^{\text{Eul}}$
by Lemma~\ref{L:LAGEUL}. In particular, from Lemma~\ref{L:LAGEUL}(ii) and~\eqref{E:BFLEULDEF}, $\psi(y)=\bar\rho(y) \theta(z)$ ($y=\bzeta(z)$) solves
\begin{align}
\pa_y&\,\Big(\big(1-\vLP^2\big)D_y\psi\Big)-2\big(\vLP(\omLP-\rhoLP)-(1-\l)\vLP\big)D_y\psi+\Big(2\rhoLP+2(1-\l)(\vLP'-1)+\l(1-\l)\Big)\psi\notag\\
=&\,\rhoLP\big(\vLP\pa_y f+\l f+g\big).\label{E:SURJLAG}
\end{align}
Now just like in~\eqref{eq:phila} we may re-write the equation~\eqref{E:SURJLAG} as
 \beq\label{eq:psisource}
 \psi''+V_1(y)\psi'+V_2(y)\psi=(1-\vLP^2)^{-1}H,
 \eeq
 where $V_1$, $V_2$ are given above in~\eqref{def:V1V2},
  the source term 
 \beq
 H(y)=\rhoLP\big(\vLP\pa_y f_1+\l f_1+f_2\big)
 \eeq
 is smooth and $O(y)$ at the origin from the assumption $f,g\in \DZodd$. Note that we used the ``prime" notation instead of $\pa_y$. From Lemma~\ref{lemma:Frobenius},
 the homogeneous equation
 \beq\label{eq:philahomog}
 \psi''+V_1(y)\psi'+V_2(y)\psi=0
 \eeq
 admits solutions $\psi_0$ and $\psi_1$ to the homogeneous equation~\eqref{eq:philahomog} which are analytic at $y=0$ and $y=y_*$, respectively. Observe that if $\l>1$, then $\psi_0$ and $\psi_1$ must be linearly independent, else $\l$ would be an eigenvalue for the operator $\mathbf{L}^{\text{Eul}}$, in contradiction to Lemma~\ref{L:REALGEQONE}. We henceforth assume $\la>1$.
 
 The Wronskian $W(\psi_0,\psi_1)$ of $\psi_0$ and $\psi_1$  then solves the equation $W'=-V_1 W$, i.e.,
 $$\big(\log W\big)'=-\frac2y+\frac{\rhoLP'}{\rhoLP}-\frac{\wLP'}{\wLP}+\frac{(1-\la)\wLP'}{\wLP\vLP'},$$
 where we have used $\wLP'=-2\vLP\vLP'$. This is solved by
 \beq\label{eq:Wronskian}
 W(\psi_0,\psi_1)=\mathcal A(y)\frac{\rhoLP}{y^2}\wLP(y)^{-1+\frac{(1-\la)}{1-\frac1{y_*}}},
 \eeq
 where $\mathcal A(y)>0$ is the smooth function defined in  Lemma~\ref{L:SUPERSYMM}. 
 
 Therefore, from the method of variation of constants, we find a solution to the inhomogeneous problem~\eqref{eq:psisource} given by
 \begin{align}
 \psi(y)=&\,-\psi_0(y)\int_y^{y_*}\frac{s^2\psi_1(s)H(s)}{W(\psi_0,\psi_1)(s)\wLP(s)}\,\dif s-\psi_1(y)\int_0^y\frac{s^2\psi_0(s)H(s)}{W(\psi_0,\psi_1)(s)\wLP(s)}\,\dif s\notag\\
 =&\,-\psi_0(y)\int_y^{y_*}\frac{s^2\psi_1(s)H(s)\wLP(s)^{\frac{\la-1}{1-\frac1{y_*}}}}{\mathcal A(s)\rhoLP(s)}\,\dif s-\psi_1(y)\int_0^y\frac{s^2\psi_0(s)H(s)\wLP(s)^{\frac{\la-1}{1-\frac1{y_*}}}}{\mathcal A(s)\rhoLP(s)}\,\dif s.\label{E:PSIVARCONSTS}
 \end{align}
 Now, from standard ODE theory and the Frobenius indices computed in Lemma~\ref{lemma:Frobenius}, we observe that there exists a function $\psi_2(y)$, smooth on $(0,\infty)$ and locally analytic around $y_*$, such that $\wLP(y)^{-\frac{\la-1}{1-\frac1{y_*}}}\psi_2(y)$ is a solution to~\eqref{eq:philahomog} which is linearly independent from $\psi_1$. Therefore, there exist constants $c_1,c_2$ such that 
 $$\psi_0(y)=c_1\psi_1(y)+\frac{c_2\psi_2(y)}{\wLP(y)^{\frac{\la-1}{1-\frac1{y_*}}}}.$$
 Clearly  $c_2\neq 0$ as $\psi_0$ and $\psi_1$ are linearly independent. 
 Thus the second integral in~\eqref{E:PSIVARCONSTS} converges to a finite value $\mathcal{I}$ as $y\to y_*$.
Thus, rearranging~\eqref{E:PSIVARCONSTS}, we arrive at
 \beqa
 \psi(y)=-c_2\frac{\psi_2(y)}{\wLP(y)^{\frac{\la-1}{1-\frac1{y_*}}}}\int_y^{y_*}\frac{s^2\psi_1(s)H(s)\wLP(s)^{\frac{\la-1}{1-\frac1{y_*}}}}{\mathcal A(s)\rhoLP(s)}\,\dif s-\mathcal{I}\psi_1(y)+c_2\psi_1(y)\int_y^{y_*}\frac{s^2\psi_2(s)H(s)}{\mathcal A(s)\rhoLP(s)}\,\dif s.
 \eeqa
 From the analyticity of $\psi_1(y)$ and $\psi_2(y)$  around $y_*$ and the smoothness of $H(y)$, it is now easy to see that $\psi(y)$ extends smoothly up to $y_*$. It is then simple to see that, as $\psi_1$, $\psi_2(y)|\wLP(y)|^{-\frac{\la-1}{1-\frac1{y_*}}}$ is still an analytic basis for solutions to~\eqref{eq:philahomog} for $y>y_*$, the solution extends smoothly beyond $y_*$ also, by replacing $\wLP$ with $|\wLP|$ everywhere in the formula above. 
 
 We have therefore shown that the solution $\psi(y)$ as defined in~\eqref{E:PSIVARCONSTS} is smooth on $(0,\bzeta(Z))$ (the domain of definition and smoothness for $H$; recall here that $Z$ is as in the definition of the space $\HmZ$). In order to check that $\psi(y)$ is also smooth up to and including the origin, $y=0$, we return to~\eqref{E:PSIVARCONSTS} and observe that, as the Frobenius indices at $0$ are $1$ and $-2$, $\psi_0$ is analytic and $O(y)$ as $y\to0$, while $$|\psi_1^{(j)}(y)|\leq Cy^{-2-j}, \text{ for }j=0,1,2,\quad \text{ and }\qquad \Big|\frac{\dif}{\dif y}\big(\frac{\psi_1(y)}{y}\big)\Big|\leq Cy^{-4}.$$
 Thus, inspecting~\eqref{E:PSIVARCONSTS}, we see easily that as $H(y)=O(y)$ and $H$ is smooth, $|D_y\psi(y)|$ is bounded near $y=0$, and $|\pa_yD_y\psi(y)|\leq Cy$ (recall $D_y=\pa_y+\frac2y$ is the 3d divergence operator in the radial coordinates). Thus, shifting from the radial $y$ variable to the full-space $\R^3$, we have that $D_y\psi\in H^1(B_\delta(0))$ for some $\delta\in(0,y_*)$. 
 
 On the other hand, applying the divergence $D_y$ to~\eqref{E:SURJLAG} and rearranging, we directly see that for $\la$ sufficiently large so that $$2\rhoLP(y) +2(1-\la)(\vLP'(y)-1)+\la(1-\la)<0,\ \ \ y\in[0,y_*],$$ 
 we have that $D_y\psi$ solves a uniformly elliptic equation on $B_\delta(0)\subset \R^3$. (We crucially use here that $\rhoLP$, $\vLP$ etc are representatives of radial functions, as is $D_yH$ by the choice of the function space $\DZodd$ for $f$ and $g$ and properties of $\bzeta$.) Hence elliptic regularity theory guarantees that $D_y\psi$ is in fact smooth at $0$ also. 
 We have therefore shown the existence of $\l_0>1$ such that for all $\la\geq\la_0$ and $(f_1,f_2)^\top\in \DZodd$, there is a smooth solution to~\eqref{E:ONTO}. 
 \end{proof}


\subsection{Time translation mode properties}

We define the standard Riesz projection onto the trivial time-translation mode $\l=1$ by 
\begin{align}\label{E:RIESZDEF}
\bfP=\frac{1}{2\pi i}\int_{\ga}(\bfL-\l\bfI)^{-1}\,\dif \l,
\end{align}
where $\ga$ is any closed, positively-oriented circle around 1 with radius less than $1$. Then, by Theorem~\ref{T:LINEARMAIN}, the only spectral point inside the contour region is the single eigenvalue $1$.

\begin{lemma}\label{L:GROWINGPROJ}
The growing mode $\Gamma$ and Riesz projector  satisfy the following properties:
\begin{itemize}
\item[(i)] The eigenvalue $\la=1$ is a simple eigenvalue of multiplicity one.
\item[(ii)] The projection $\bfP$ satisfies $\textup{rg}\,\bfP=\langle\Gamma\rangle$.
\end{itemize}
\end{lemma}

\begin{proof}
The proof follows the same lines as~\cite[Proposition 3.4]{Glogic22} and so we provide only a sketch outline here. To prove (i), we use the Lagrangian-Eulerian equivalence  Lemma~\ref{L:LAGEUL}(ii) and~\eqref{E:BFLEULDEF} to observe that if $(\th,\phi)^\top\in\textup{ker}\,(\bfL-\bfI)$, then the corresponding Eulerian function $\psi(y)=\rhoLP(y)\th(z)$ satisfies
\beq
\psi''(y)+\Big(\frac2y-\frac{\rhoLP'}{\rhoLP}+\frac{\wLP'}{\wLP}\Big)\psi'+\Big(-\frac2{y^2}-\frac{2\rhoLP'}{y\rhoLP}+\frac{2\wLP'}{y\wLP}+\frac{2\rhoLP}{\wLP}\Big)\psi=0.
\eeq
Clearly the eigenfunction $g_1(y)=y\rhoLP(1-\omLP)$ is a solution to this problem (compare Lemma~\eqref{L:GROWINGMODEEUL}), and by reduction of order, we find a second, linearly independent solution
\beq
g_2(y)=g_1(y)\int_{y}^{y_*}\big(\tilde y^{3}\wLP(\tilde y)(1-\omLP(\tilde y))\big)^{-1}\,\dif \tilde y.
\eeq
 One checks directly that $g_2(y)\simeq y^{-1}$ as $y\to 0$, $g_2(y)\simeq \log(y_*-y)$ as $y\to y_*$. As this is a linearly independent solution and is clearly not in  $\HmZEul$, the proof of (i) is complete.

As in~\cite[Proposition 3.4]{Glogic22}, to show (ii), it is sufficient to show that there does not exist $(\th,\phi)^\top\in\HmZ$ such that
\[ (\bfL-\bfI)(\th,\phi)^\top=\Gamma.\]
Assuming for a contradiction that such a pair exists, we argue again in the Eulerian framework and find that this is equivalent to the existence of $\psi(y)$ such that
\beq
\psi''(y)+\Big(\frac2y-\frac{\rhoLP'}{\rhoLP}+\frac{\wLP'}{\wLP}\Big)\psi'+\Big(-\frac2{y^2}-\frac{2\rhoLP'}{y\rhoLP}+\frac{2\wLP'}{y\wLP}+\frac{2\rhoLP}{\wLP}\Big)\psi=\frac{\rhoLP}{\wLP}(\vLP\pa_yg_1+g_1).
\eeq
As in~\eqref{eq:Wronskian}, we see easily that the Wronskian of the two independent solutions is given by 
\[W(g_1,g_2)=\frac{\rhoLP}{y^2\wLP}.\]
Applying the variation of parameters method, there exist constants $c_1,c_2$ such that the solution $\psi$ is of the form
\beq
\psi(y)=c_1g_1(y)+c_2g_2(y)+g_2(y)\int_0^y \tilde y^2g_1(\vLP\pa_yg_1+g_1)\,\dif \tilde y-g_1(y)\int_0^y\tilde y^2g_2(\vLP\pa_yg_1+g_1)\,\dif \tilde y.
\eeq
One checks from the asymptotics of $g_1$ and $g_2$ that both integrals are well-defined near $y=0$. Moreover, as $g_2\simeq y^{-1}$ near $y=0$, we must have $c_2=0$. From the asymptotic property $g_2(y)\simeq \log(y_*-y)$ near $y_*$, we see that the final term on the right remains bounded near the sonic point. However, from the same property, we find that $\psi$ blows up logarithmically near the sonic point unless the integral
\[ \int_0^{y_*} y^2g_1(\vLP\pa_yg_1+g_1)\,\dif y\]
vanishes. But from~\eqref{E:GROWINGMODEPROP}, the integrand is strictly positive, and hence we arrive at the desired contradiction.
\end{proof}


\section{Maximal accretivity}\label{S:ACCRETIVITY}


The purpose of this section is to complement the mode stability result of Theorem~\ref{T:SPECTRAL} with an analysis of the remaining portion of the spectrum of $\bfL$. To this end, we will identify a leading order operator that captures the top order behaviour of $\bfL$ and enables us to prove control of the essential spectrum. This operator will be shown to be a compact perturbation of $\bfL$, and hence enables us to apply standard spectral theoretic results and the Lumer-Phillips theorem in order to conclude the proof of Theorem~\ref{T:LINEARMAIN}. 


\subsection{Differential operator properties}

In order to handle lower order terms arising from the application of $\bD^{2m}$ to $\bfL$, it is convenient to be able to invert $\bD^2$. As we work in radial symmetry, it is particularly straightforward to write down the inverse for this operator directly. We recall that we are working with functions defined on the interval $z\in[0,Z]$ for some $Z>z_*$ fixed.

\begin{definition}[Inversion of $\bD^2$]
To any  $f\in L^2(0,Z)$, we associate the inverse
\begin{align}\label{E:BLINVERSE}
\bD^{-2} f (z) : =  z^{-\frac23}\int_0^z\int_0^{\bar z} \tilde z^{-\frac23}  f(\tilde z)\, \diff \tilde z\,d\bar z,
\end{align}
which is normalised by the requirement $\bD^{-2}f(0)=0$. Accordingly, we define, for $m\geq 2$,
\begin{align}
\bD^{-2m} := \bD^{-2}\bD^{-2(m-1)}.
\end{align}
\end{definition}

It is easy to see that $\bD^2 (\bD^{-2}f) = f$. We next prove
some boundedness properties of the operator $\bD^{-2}$ that will be used often throughout this section.


\begin{lemma}\label{L:INVERSE}
Let $m\in\N$. There exists $C>0$ such that for any $f\in L^2(0,Z)$, $z\in[0,Z]$,
\begin{align}
\lv \bD^{-2m}f(z) \rv & \leq C z^{\frac23 m-\frac16} \|\frac f{z^{\frac13}}\|_{L^2(0,Z)}
\leq C z^{\frac23 m-\frac16}\|\bd f\|_{L^2(0,Z)}, \label{E:I1}\\
\|\bD^{-2m} f\|_{L^2(0,Z)} & \leq C Z^{\frac 23m+\frac13}\|\frac f{z^{\frac13}}\|_{L^2(0,Z)} \leq C Z^{\frac 23m+\frac13} \|\bd f\|_{L^2(0,Z)}. \label{E:I2}
\end{align}
\end{lemma}


\begin{proof}
Note that the bound~\eqref{E:I2} is a simple consequence of~\eqref{E:I1}. We clearly have the bound
\begin{align}
\lv\int_0^{\bar z} \tilde z^{-\frac23}  f(\tilde z)\, \diff \tilde z \rv
\le \|\frac f{z^{\frac13}}\|_{L^2(0,\bar z)} \|z^{-\frac13}\|_{L^2(0,\bar z)} \leq C \|\frac f{z^{\frac13}}\|_{L^2(0,\bar z)} \bar z^{\frac16}.
\end{align}
Therefore  from~\eqref{E:BLINVERSE} we have
\begin{align}
\lv \bD^{-2}f(z)\rv \le z^{-\frac23} \|\frac f{z^{\frac13}}\|_{L^2(0,z)}  \int_0^z \bar z^{\frac16}\,d\bar z \leq C z^{\frac12} \|\frac f{z^{\frac13}}\|_{L^2(0,z)},\label{E:INVERSEPOINTWISE}
\end{align}
which proves~\eqref{E:I1} when $m=1$.
We now use~\eqref{E:INVERSEPOINTWISE} to infer 
$\lv \bD^{-2}f(z)\rv \le z^{\frac12} \|\frac f{z^{\frac13}}\|_{L^2(0,Z)}$ for any $z\in(0,Z]$. Therefore, for any $m\ge2$, noting $|\bD^{-2}f(z)|\leq \bD^{-2}|f|(z)$,
\begin{align*}
\lv \bD^{-2m}f(z) \rv \leq C  \|\frac f{z^{\frac13}}\|_{L^2(0,Z)} \bD^{-2(m-1)} (z^{\frac12}) \leq C z^{\frac23 m-\frac16} \|\frac f{z^{\frac13}}\|_{L^2(0,Z)}
\leq C z^{\frac23 m-\frac16}\|\bd f\|_{L^2(0,Z)},
\end{align*}
where we have used~\eqref{E:HARDYREFINED} with $\beta=0$ in the last line.
\end{proof}

\subsubsection{Commutation properties}


As we work in the space $\HmZ$, it is important to understand the commutation between $\bfL$ and the operator $\bD^{2m}$. To this end, we first establish a number of commutation relations and define some convenient families of operators. We recall the operator definitions~\eqref{E:DIVGRADDEF} and~\eqref{DEF:D^k}.

It is easy to check that formally the following commutation properties hold.
\begin{align}
\bp \Lambda = &\,\Lambda\bp + \frac13 \bp ,\label{E:GAIN1}\\
\bd \Lambda =&\, \Lambda \bd +\frac13 \bd,\label{E:GAIN2}\\
\bp\bd F = &\,\bd \bp F - \frac29 z^{-\frac23}F, \ \ F\in \DZ.
\end{align}
We also note that, from~\eqref{DEF:D^k} and~\eqref{E:KDEF}, 
\begin{align}\label{E:KDEF2}
K\theta = \bp  \left(\bG \bd \theta\right)  =\bG \bD^2 \theta + \bp \bG \, \bd \theta.
\end{align}
Properties~\eqref{E:GAIN1}--\eqref{E:GAIN2} encode the dissipativity gain that will become transparent in Proposition~\ref{P:DISSIP}. 
They in particular give the commutator relations (recall the definition~\eqref{E:BDELTA})
\be\label{E:ACC1}
[\bl, \Lambda] =\frac23 \bl, \ \ [\bD^2,\Lambda] = \frac23 \bD^2.
\ee


In order to prove the desired maximal accretivity properties we need to carefully describe the commutation properties between the 
leading order elliptic operator $K$ and $\bD^2$.
To that end we introduce the following algebra of operators.
We let the trivial operator classes $\mathcal{X}_0=\mathcal{Y}_0=\{I\}$ contain only the identity operator. Then,
\begin{align}
\mathcal X_{2j} : =&\, \left\{\mathcal P =\bp \mathcal R_j  \dots \bp \mathcal R_1 \,|\, (\mathcal R_1,\dots, \mathcal R_j)\in \{\bd, z^{-\frac13}\}^j\right\}, \ j\in\mathbb N, \label{E:XTWOJ}\\
\mathcal X_{2j-1} : =&\, \left\{\mathcal P =\mathcal R_{j-1} \bp \mathcal R_j  \dots \bp \mathcal R_1\,|\,(\mathcal R_1,\dots, \mathcal R_{j+1})\in \{\bd, z^{-\frac13}\}^j\right\}, \ j\in\mathbb N.\label{E:XTWOJPLUSONE}
\end{align}
Similarly, we define the concatenated operator classes
\begin{align}
\mathcal Y_{2j} : =&\, \left\{\mathcal P = \mathcal R_j \bp  \dots \mathcal R_1\bp \,|\, (\mathcal R_1,\dots, \mathcal R_j)\in \{\bd, z^{-\frac13}\}^j\right\}, \ j\in\mathbb N, \\
\mathcal Y_{2j-1} : =&\, \left\{\mathcal P =\bp \mathcal R_{j-1} \bp \dots \mathcal R_1 \bp\,|\, (\mathcal R_1,\dots, \mathcal R_{j+1})\in \{\bd, z^{-\frac13}\}^j\right\}, \ j\in\mathbb N.
\end{align}
In short we write
\begin{align}
\mathcal Y_{2j} = \mathcal X_{2j-1}\bp, \ \ \mathcal Y_{2j+1} = \mathcal X_{2j}\bp.
\end{align}

It is clear from the definitions that the operator algebras are designed to act particularly nicely on $\DZodd$ and $\DZeven$, respectively. In particular, 
let $f\in\DZeven$, $g\in\DZodd$, and take $P\in\mathcal{Y}_j$, $Q\in\mathcal{X}_j$. Then
\beqs
Pf\in\begin{cases}
\DZodd, & j\text{ odd},\\
\DZeven, & j\text{ even},
\end{cases}\qquad
Qg\in\begin{cases}
\DZeven, & j\text{ odd},\\
\DZodd, & j\text{ even}.
\end{cases}
\eeqs

For the purposes of estimates in the far-field (i.e.~when $z\gg1$), it is often convenient to have the following characterisation of these operators in terms of standard $\pa_z$ derivatives. This is the content of the next lemma.

\begin{lemma}\label{L:XJDECOMP}
For $Q\in\X_j$ or $Q\in\Y_j$, there exist two collections of real constants $c_\ell^Q$ and $\tilde c_\ell^Q$ such that
\[
Q=\sum_{\ell=0}^j c_\ell^Q z^{-\frac{j-\ell}{3}}\bD^\ell=\sum_{\ell=0}^j\tilde c_\ell^Qz^{\ell-\frac{j}{3}}\pa_z^\ell .
\] 
In particular, such formulae hold for the operators $\bD^j$. Moreover, we have the reverse identities
\be\label{EquivID}
\pa_z^k f = \frac{ \bD^k f }{z^{\frac23 k}} + \sum_{i=1}^{k} c_{ki} \frac{\bD^{k-i} f}{z^{\frac23 k+\frac{i}{3}}} .
\ee
\end{lemma}


As a consequence of these identities, it is trivial to prove the following.

\begin{lemma}\label{L:PGBOUND}
Let $k\in\N$ and $P\in \Y_k$.  We define
\beq\label{E:GTILDEDEF}
\tG(z) : = \frac{z^{\frac13}}{\bzeta(z)}.
\eeq
Then  $\tG\in\Dinfeven$, and there exists $C>0$, depending on $k$, such that, for all $z\geq 0$,
\begin{align}\label{E:PGBOUND}
\lv P\bG(z) \rv \leq C (1+z)^{-\frac{4+k}{3}},\ \ \ \  \lv P\sqrt{\bG(z)} \rv \leq C (1+z)^{-\frac{2+k}{3}}\ \ \ \  \lv P\tilde G(z) \rv \leq C (1+z)^{-\frac{2+k}{3}}.
\end{align}
\end{lemma}


\begin{proof}
We recall~\eqref{E:GGDEF} and~\eqref{E:GTILDEDEF}. Given the expansion of $\bzeta$ from Lemma~\ref{L:ZETABAR},~\eqref{E:PGBOUND} follows easily.
\end{proof}


We introduce a convenient product and chain rule for the operators in the algebras $\mathcal{X}_j$ and $\mathcal{Y}_j$.
\begin{lemma}\label{L:XYPRODCHAIN}
Let  $j\in\N$, $f,g\in \Dinf$. Then, for any $Q\in\mathcal{X}_j$ and $P\in\mathcal{Y}_j$, there exist families of constants such that
\begin{align}
Q(fg)=\sum_{k=0}^j\sum_{Q_1\in\mathcal{X}_k,\,P_2\in\mathcal{Y}_{j-k}}c_{Q_1P_2}Q_1fP_2g,\label{E:XJPROD}\\
P(fg)=\sum_{k=0}^j\sum_{P_1\in\mathcal{Y}_k,\,P_2\in\mathcal{Y}_{j-k}}c_{P_1P_2}P_1fP_2g.\label{E:YJPROD}
\end{align}
Moreover, if $f\in C^j([-\frac12,\frac12])$, $g\in C^j$, then, for $P\in\mathcal{Y}_j$, the quantity $P\big(f(g(z))\big)$ expands as a linear combination of terms of the form 
\beq\label{E:YJCHAIN}
f^{(k)}(g(z))P_{j_1}g\cdots P_{j_k}g,
\eeq
for some $j_1,\ldots,j_k\in\N$ such that
\[
1\leq k\leq j,\ \ P_{j_n}\in\mathcal{Y}_{j_n},\ \ n\in\{1,\ldots,k\},\ \ j_1+\cdots+j_k=j.
\]
\end{lemma}

\begin{proof}
The proofs of~\eqref{E:XJPROD}--\eqref{E:YJPROD} follow from a simple inductive argument and the usual Leibniz rule, noting also that $\bp=\bd-\frac23 z^{-\frac13}$. 

The proof of~\eqref{E:YJCHAIN} is similarly simple and relies on~\eqref{E:XJPROD}--\eqref{E:YJPROD}, where we note the constraint $k\geq 1$ in the formula occurs as $P\in\mathcal{Y}_j$, so that $P$ applies at least one $\bp$ directly to $f(g(z))$.
\end{proof}



\begin{lemma}\label{L:KEYCOMM}
Let $\theta\in\DZodd$, $m\in\N$. We then have the following commutation property:
\begin{align}
\bD^{2m} K \theta =&\, K \bD^{2m}\theta + 2m \bp \bG \bD^{2m+1}\theta + \mathcal R_{2m}\theta,\label{E:H2jCOMM}
\end{align}
where
\begin{align}\label{E:REMAINDER}
\mathcal R_{2m} \theta = \sum_{\ell=1}^{2m} \sum_{Q\in \mathcal X_\ell, P\in \mathcal Y_{2m+2-\ell}} c^\ell_{PQ} P\bG Q \theta.
\end{align}
\end{lemma}

\begin{proof}
We begin in the case $m=1$. From the Leibniz rule
\be\label{E:LEIBNIZ}
\bd(uv) = u \bp v + \bd u v
\ee
we easily conclude
\begin{align}
\bd K\theta & = \bd (\bG  \bD^2 \theta + \bp \bG   \bd \theta) \notag\\
& = \bp \bG  \bD^2 \theta + \bG  \bd\bD^2\theta + \bl \bG  \bd \theta + \bp \bG  \bD^2 \theta \notag\\
& = \bG  \bD^3 \theta + 2\bp \bG  \bD^2 \theta + \bl \bG  \bd \theta.
\end{align}
We now apply $\bp$ to the above identity and obtain
\begin{align}
\bD^2 K \theta & = \bp\left(\bG  \bD^3 \theta + 2\bp \bG  \bD^2 \theta + \bl \bG  \bd \theta\right) \notag\\
& = K \bD^2\theta + 2 \bp^2 \bG  \bD^2 \theta + 2\bp \bG  \bp\bD^2 \theta + \bp\bl \bG  \bd \theta + \bl \bG  \bD^2\theta.
\end{align}
We now use the relation $\bp u = \bd u- \frac23 z^{-\frac13}u$ to rewrite
\beq\label{E:GRAD^2}
\bp^2 \bG  = \bl \bG  - \frac23 z^{-\frac13}\bp \bG , \ \ \bp\bD^2\theta = \bD^3\theta - \frac23 z^{-\frac13} \bD^2\theta
\eeq
and finally obtain the identity
\begin{align}
\bD^2 K \theta = K \bD^2\theta + 2\bp \bG  \bD^3 \theta + \left(3\bl \bG  -\frac83z^{-\frac13}\bp \bG \right) \bD^2\theta +\bp\bl \bG  \bd\theta,\label{E:HCOMM}
\end{align}
which is precisely~\eqref{E:H2jCOMM} in the case $m=1$. In particular
\be\label{E:RONEFORMULA}
\mathcal R_2 \theta =   \Big(3\bl \bG  -\frac83z^{-\frac13}\bp \bG \Big) \bD^2\theta +\bp\bl \bG  \bd\theta.
\ee
We now proceed by induction on $m$. 
Assume that~\eqref{E:H2jCOMM} holds for some $m\in\mathbb N$. 
By the inductive hypothesis and~\eqref{E:HCOMM}, we deduce
\beqa
\bD^{2m+2} K\theta - K \bD^{2m+2}\theta & = \bD^2\big(\bD^{2m}K\theta - K\bD^{2m}\theta\big)+2\bp\bG\bD^{2m+3}\theta + \mathcal R_2\bD^{2m}\theta\\
& = \bD^2\big(2m\bp\bG\bD^{2m+1}\theta\big)+\bD^2\mathcal R_{2m}\theta+2\bp\bG\bD^{2m+3}\theta + \mathcal R_2\bD^{2m}\theta.
\eeqa
Expanding the first term on the right using~\eqref{E:GRAD^2}, we find 
\beqa
\bD^2\big(\bp\bG\bD^{2m+1}\theta\big) &=\bp\bG\bD^{2m+3}\theta +\Big(2\bl\bG -\frac23 
z^{-\frac13}\bp\bG\Big)\bD^{2m+2}\theta + \bp\bl\bG\bD^{2m+1}\theta.
\eeqa  
Thus, we have obtained
\beqa
\bD^{2m+2} K\theta - K \bD^{2m+2}\theta & = (2m+2)\bp\bG\bD^{2m+3}\theta+  \mathcal{R}_{2m+2}\theta, 
\eeqa
where 
\beq 
\mathcal R_{2m+2}\theta=\mathcal R_2\bD^{2m}\theta + \bD^{2}\mathcal R_{2m}\theta +2m \Big(2\bl\bG -\frac23 z^{-\frac13}\bp\bG\Big)\bD^{2m+2}\theta +2m  \bp\bl\bG\bD^{2m+1}\theta 
\eeq 
is easily seen to be of the form~\eqref{E:REMAINDER} with the aid of~\eqref{E:GRAD^2}.
\end{proof}


\subsection{Dissipativity}


Our aim is to identify a compact perturbation of $\bfL$ for which we can show maximal dissipativity on the space $\HmZ$. In this section, we demonstrate that the leading order dynamics of $\bfL$, suitably defined, induce dissipation in the semi-norm $\dot{\H}^{2m}_{Z}$. To show this, we commute $\bfL$ with the operator $\bD^{2m}$, applying the commutation identities of the previous section, and derive the key Proposition~\ref{P:DISSIP}. 

We now commute~\eqref{E:LINEAR1} with $\bD^{2m}$, $m\in\mathbb N$, recalling~\eqref{E:BOLDLDEF}.
To do that, we recall the crucial commutation property~\eqref{E:ACC1}, which gives the identity
\begin{align}
[\bD^{2m},\Lambda] : = \bD^{2m}\Lambda - \Lambda \bD^{2m} = \frac{2m}{3}\bD^{2m}.
\end{align}
Therefore, applying $\bD^{2m}$ to~\eqref{E:LINEAR1}, assuming $\theta\in\DZodd$, and denoting 
\begin{align}
\theta_{2m}:=\bD^{2m}\theta, \ \ \phi_{2m} : = \bD^{2m}\phi,
\end{align}
we obtain the system
\begin{align}
\pa_s\theta_{2m} & = \phi_{2m} - \Lambda\theta_{2m} - \frac{2m-3}3 \theta_{2m}, \label{E:SIMPLETHETAM}\\
\pa_s\phi_{2m} & = -\Lambda \phi_{2m} -\frac{2m}3\phi_{2m}+ K\theta_{2m} + 2m \bp \bG \bd \theta_{2m} + \mathcal R_{2m} \theta+\bD^{2m}\left(\mathcal V\theta\right), \label{E:SIMPLEPSIM}
\end{align}
where we have used Lemma~\ref{L:KEYCOMM}. This naturally leads to the operator  $A_{2m}:D(A_{2m})\subset \H^0_{2m,Z}\to \H^0_{2m,Z}$, where $D(A_{2m})=\bD^{2m}D(\bfL)$,  defined by 
\begin{align}
\label{E:AMDEF}
A_{2m} \begin{pmatrix} \theta \\ \phi \end{pmatrix}: = \begin{pmatrix} \phi- \Lambda\theta - \frac{2m-3}3 \theta \\ -\Lambda\phi -
\frac{2m}3\phi+ K\theta + 2m \bp \bG  \bd \theta  \end{pmatrix}.
\end{align}
Our strategy to prove Theorem~\ref{T:SPECTRAL} relies in part on the dissipative properties of the operator $A_{2m}$, which effectively captures the highest order behaviour of the linearised operator $\mathbf{L}$. Before proceeding, we therefore establish the crucial dissipation estimate for this leading order contribution. As a consequence of this estimate, we are also able to identify the key regularity parameter $\m$ that will play a crucial role in the remainder of this paper.


\begin{proposition}\label{P:DISSIP}
There exists an  $\m\in\mathbb N$  sufficiently large and a constant $k_0>0$ so that the following dissipativity bound holds:  for each  $m \ge \m$,
\begin{align}
\textup{Re} \left( \begin{pmatrix} \theta \\ \phi \end{pmatrix}\,,\, A_{2m} \begin{pmatrix} \theta \\ \phi \end{pmatrix}\right)_{\dot{\H}^0_{2m,Z}}
 \le - k_0 m \left\| \begin{pmatrix} \theta \\ \phi \end{pmatrix}\right\|_{\dot{\H}^0_{2m,Z}}^2, \ \ \begin{pmatrix} \theta \\ \phi \end{pmatrix}\in D(A_{2m}).
 \end{align}
\end{proposition}


\begin{proof}
 To prove the key dissipation estimate, we work with $(\th,\phi)^\top\in\DZodd$, and apply a simple density argument at the end to conclude the inequality also for $(\th,\phi)^\top\in D(A_{2m})$. 

Let us evaluate 
\begin{align}
 \left( \begin{pmatrix} \theta \\ \phi \end{pmatrix}\,,\, A_{2m} \begin{pmatrix} \theta \\ \phi \end{pmatrix}\right)_{\dot{\H}^0_{2m,Z}} 
& = \int_0^{Z}  \Big({\bd\theta}\overline{\bd \phi}  - \bd\theta\overline{\bd \Lambda \theta}  - \frac{2m-3}{3}|\bd\theta|^2\Big) \bG(z) g_{2m}(z)  \diff z \notag\\
& \ \ \ \  + \int_0^{Z} \Big(- \phi\overline{\Lambda\phi}  - \frac{2m}{3}|\phi|^2 + \phi \overline{{K\theta}}+ 2m \bp \bG  \phi\overline{\bd \theta}  \Big) g_{2m}(z)  \diff z. \label{E:D0}
\end{align}
It is straightforward to check that the following general formula holds:
\begin{align}
\text{Re} \int_0^Z \overline{\Lambda \phi} \phi g(z)\diff z = -\frac12 \int_0^Z |\phi|^2 \pa_z\left(z g(z)\right)\diff z + \frac12\left(|\phi|^2 zg(z)\right)\Big|^Z_0.\label{eq:LambdaIBP}
\end{align}
In particular
\begin{align}
{}&- \text{Re} \int_0^Z  \bd \theta \overline{\bd \Lambda \theta} \bG(z) g_{2m}(z)  \diff z  =
 -\text{Re} \int_0^Z  \bd\theta\overline{\Lambda \bd\theta}  \bG(z) g_{2m}(z)  \diff z - \frac13 \int_0^Z  |\bd \theta|^2 \bG(z)  g_{2m}(z) \diff z \notag\\
 & =\frac12 \int_0^Z  |\bd \theta|^2 \pa_z\left(\bG(z) g_{2m}(z)  z\right)\diff z - \frac12|\bd \theta(Z)|^2 \bG(Z)  g_{2m}(Z) Z
 - \frac13 \int_0^Z  |\bd \theta|^2 \bG(z) g_{2m}(z)  \diff z \notag\\
 & = \frac16 \int_0^Z  |\bd \theta|^2 \bG(z)  g_{2m}(z) \diff z + \frac12 \int_0^Z  |\bd \theta|^2 \big(\bp \bG(z) z^{\frac13} g_{2m}(z)+\bG(z) g_{2m}'(z) z \big) \diff z - \frac12|\bd \theta(Z)|^2 \bG(Z)  g_{2m}(Z) Z.
\end{align}
Similarly,
\begin{align}
- \text{Re} \int_0^Z  \phi\overline{\Lambda \phi}    g_{2m}(z) \diff z &
= \frac12 \int_0^Z |\phi|^2 \big(g_{2m}(z)+zg_{2m}'(z)\big)  \diff z- \frac12 |\phi(Z)|^2 g_{2m}(Z)  Z.
\end{align}
We next observe the key calculation
\begin{align}
& \int_0^{Z}  \Big({\bd\theta}\overline{\bd \phi} \bG(z) g_{2m}(z)  + \phi\overline{K\theta}  g_{2m}(z)  \Big)\diff z  = \int_0^{Z}  {\bd\theta}\overline{\bd \phi}\bG(z)  g_{2m}(z) +  \bp \left(\bG \overline{\bd \theta}\right) \phi  g_{2m}(z) \diff z   \notag\\
& =  \int_0^{Z}  \Big(\bd\theta\overline{\bd \phi} \bG(z) g_{2m}(z)  -  \bd \phi \overline{\bd\theta} \bG(z)  g_{2m}(z)- z^{\frac23}\phi\overline{\bd\theta} \bG g_{2m}'(z)  \Big)\diff z  + z^{\frac23} \bG \overline{\bd\theta}  \phi  g_{2m}   \Big|^Z_0,
\end{align}
and therefore
\begin{align}
&\text{Re}\bigg( \int_0^{Z} \bd\theta \overline{\bd \phi} \bG(z) g_{2m}(z)  + \overline{K\theta} \phi  g_{2m}(z) \diff z  \bigg)\notag \\
&\quad= Z^{\frac23} \bG(Z)  g_{2m}(Z) \text{Re}\left(\bd\theta (Z) \bar\phi(Z)\right)- \int_0^Zz^{\frac23}\text{Re}
\big(\phi\overline{\bd\theta}\big) \bG g_{2m}'(z)  \diff z. \label{E:D1}
\end{align}
From~\eqref{E:D0}--\eqref{E:D1} we conclude that 
\begin{align}
& \text{Re}\left( \begin{pmatrix} \theta \\ \phi \end{pmatrix}\,,\, A_{2m} \begin{pmatrix} \theta \\ \phi \end{pmatrix}\right)_{\dot{\H}^0_{2m,Z}} \notag\\
&= \left(\frac76-\frac{2m}{3}\right)\int_0^Z |\bd\theta|^2 \bG g_{2m}(z) \,\diff z +\left(\frac12-\frac{2m}{3}\right)\int_0^Z |\phi|^2 g_{2m}(z) \diff z\\
&\ \ \ \ + \frac12 \int_0^Z |\phi|^2zg_{2m}'(z)\diff z   + 2m\int_0^Z \bp \bG \text{Re}\left(\overline{\bd \theta} \phi\right)  g_{2m}(z)  \, dz \notag\\
&\ \ \ \ +  \frac12 \int_0^Z  |\bd \theta|^2 \big(\bp \bG(z) z^{\frac13}g_{2m}(z)+\bG(z) g_{2m}'(z) z \big) \diff z - \int_0^Zz^{\frac23}\text{Re}
\big(\phi\overline{\bd\theta}\big) \bG g_{2m}'(z)  \diff z\notag\\
& \ \ \ \ -\frac12 |\bd\theta(Z)|^2 \bG(Z) g_{2m}(Z) Z -\frac12 |\phi(Z)|^2  g_{2m}(Z) Z + Z^{\frac23} \bG(Z) g_{2m}(Z)  \text{Re}\left(\bd\theta(Z) \overline{\phi(Z)}\right) \label{E:D2}.
\end{align}
We now recall the definition of $g_{2m}$ from~\eqref{DEF:gm}, so that  
\beqa
{}&\bigg|2m\int_0^Z \bp \bG \text{Re}\left(\overline{\bd \theta} \phi\right)  g_{2m}(z)  \, dz- \int_0^Zz^{\frac23}\text{Re}
\big(\phi\overline{\bd\theta}\big) \bG g_{2m}'(z)  \diff z\bigg|\\
&\leq\int_{Z_*}^Z\Big|2m\frac{\bp \bG}{\sqrt{\bG}}-z^{\frac23}\sqrt{\bG}\frac{g_{2m}'}{g_{2m}}\Big|\big|\text{Re}
\big(\phi\overline{\bd\theta}\big) \big|{\sqrt{\bG}}g_{2m}(z)\,\dif z\\
&\leq m\a\bigg(\int_{Z_*}^Z|\bd \theta|^2\bG g_{2m}\,\dif z+\int_{Z_*}^Z|\phi|^2 g_{2m}\,\dif z\bigg),
\eeqa
where we have used that $\bp\bG,g_{2m}'\leq 0$ to estimate, for $z\geq Z_*$,
\beq\label{E:BADTERMDIFF}
 \frac12\Big|2m\frac{\bp \bG}{\sqrt{\bG}}-z^{\frac23}\sqrt{\bG}\frac{g_{2m}'}{g_{2m}}\Big|\leq \max\Big\{m\Big\|\frac{\bp \bG}{\sqrt{\bG}}\Big\|_{L^\infty([Z_*,\infty)},\Big\|z^{\frac23}\sqrt{\bG}\frac{g_{2m}'}{2g_{2m}}\Big\|_{L^\infty([Z_*,\infty))}\Big\}\leq m \a
 \eeq  
by~\eqref{E:DGSQRTG} and~\eqref{E:gmprime}.

As a consequence, from~\eqref{E:D2}, using also $\bp\bG\leq 0$ and $g_{2m}'\leq 0$, we obtain
\beqa\label{E:D3}
& \text{Re}\left( \begin{pmatrix} \theta \\ \phi \end{pmatrix}\,,\, A_{2m} \begin{pmatrix} \theta \\ \phi \end{pmatrix}\right)_{\dot{\H}^0_{2m,Z}} \\
&\leq \left(\frac76-m\big(\frac{2}{3}-\a\big)\right)\int_0^Z |\bd\theta|^2 \bG g_{2m}(z) \,\diff z +\left(\frac12-m\big(\frac{2}{3}-\a\big)\right)\int_0^Z |\phi|^2 g_{2m}(z) \diff z\\
& \ \ \ \ -\frac12 |\bd\theta(Z)|^2 \bG(Z) g_{2m}(Z) Z -\frac12 |\phi(Z)|^2 Zg_{2m}(Z) + Z^{\frac23} \bG(Z) g_{2m}(Z)  \text{Re}\left(\bd\theta(Z) \overline{\phi(Z)}\right). 
\eeqa 
Moreover,
\begin{align}
&\Big(-\frac12 |\bd\theta(Z)|^2 \bG(Z)Z -\frac12 |\phi(Z)|^2 Z + Z^{\frac23} \bG(Z) \text{Re}\left(\bd\theta(Z) \overline{\phi(Z)}\right)\Big)g_{2m}(Z) \notag\\
&
\le\Big( -\frac12 |\bd\theta(Z)|^2 \bG(Z)Z-\frac12 |\phi(Z)|^2 Z +\frac12  \bG(Z) \left(Z|\bd\theta(Z)|^2+Z^{\frac13}|\phi(Z)|^2\right) \Big)g_{2m}(Z)\notag\\
& =\Big(-  \frac12 Z^{\frac13}|\phi(Z)|^2 \left(Z^{\frac23}-\bG(Z)\right) \Big)g_{2m}(Z)\notag\\
& \le 0, \label{E:D4}
\end{align}
where we have used the condition $Z\ge z_\ast$, which implies $Z^{\frac23}-\bG(Z)\ge0$ by Lemma~\ref{L:GBAR}. From~\eqref{E:D3}--\eqref{E:D4} we then infer that
\beqa
& \text{Re}\left( \begin{pmatrix} \theta \\ \phi \end{pmatrix}\,,\, A_{2m} \begin{pmatrix} \theta \\ \phi \end{pmatrix}\right)_{\dot{\H}^0_{2m,Z}} \notag\\
&\leq \left(\frac76-m\big(\frac{2}{3}-\a\big)\right)\int_0^Z |\bd\theta|^2 \bG g_{2m}(z) \,\diff z +\left(\frac12-m\big(\frac{2}{3}-\a\big)\right)\int_0^Z |\phi|^2 g_{2m}(z) \diff z.
\eeqa 
Choosing $\m$ as in~\eqref{E:MCONDITION}, this concludes the proof.
\end{proof}

\begin{remark}
We note that we crucially require $Z>z_\ast$. The dissipativity bound is correct for {\em any} such $Z>z_\ast$, which is a consequence of the ground state character of the LP-solution, encoded in the property $\bp\bG<0$ stated in Lemma \ref{L:GBAR}.  

We also observe that the requirement on $\m$ is that $\frac76-\m(\frac23-\a)<0$, as advertised in~\eqref{E:MCONDITION}.
\end{remark}


\subsection{Maximal dissipativity}


In order to use the above Proposition~\ref{P:DISSIP} to obtain spectral information for $\mathbf{L}$, we must now establish the connection between $\bD^{2m}\mathbf{L}$ and $A_{2m}$. From the right hand side of~\eqref{E:SIMPLEPSIM}, we see that the principal difference arises in the lower order terms $\mathcal{R}_{2m}$ and $\bD^{2m}(\mathcal V\theta)$.  In order to control such terms, we employ the inverse operator $\bD^{-2}$, defined above in~\eqref{E:BLINVERSE}. This leads us to the next definition.


\begin{definition}\label{D:LONELTWO}
For each $m\in\N$, we introduce the operators $\mathcal L_m : D(\mathcal L_m)\subset \HmZ\to\HmZ$, and $\mathcal K_m:\HmZ\to\HmZ$ through
\begin{align}
{\mathcal K}_m\color{black}  \begin{pmatrix} \theta \\ \phi \end{pmatrix} & : = \begin{pmatrix} 0 \\  \bD^{-2m}\mathcal R_{2m}\theta 
+\mathcal V \theta\color{black} \end{pmatrix} , \\
\mathcal L_m   & : = \bfL - \mathcal K_m,\label{D:Lm}
\end{align}
where we recall the definitions of $\mathcal{V}$ and $\mathcal{R}_{2m}$ from~\eqref{E:POTDEF} and~\eqref{E:REMAINDER}.
\end{definition}

By~\eqref{E:SIMPLETHETAM}--\eqref{E:AMDEF} and the above definition, we clearly have
\begin{align}
\bD^{2m} \L_m = A_{2m} \bD^{2m}.\label{E:DMAM}
\end{align}
The following lemma states the elementary fact that the operator $\mathcal K_m:\HmZ\to\HmZ\subset \H^0_{2m,Z}$ is indeed bounded viewed as an operator into $\H^0_{2m,Z}$.
 
\begin{lemma}
Let $m\in\N$. There exists a constant $C_{\K_m}>0$ such that the operator $\K_m : \HmZ\to \HmZ\subset \H^0_{2m,Z}$ satisfies the bound
\begin{align}
\bigg\| \K_m  \begin{pmatrix} \theta \\ \phi \end{pmatrix} \bigg\|_{\H^0_{2m,Z}} \le C_{\K_m} \left(\|\bd\bD^{2m}\theta\|_{L^2}+\|\bd \theta\|_{L^2}\right), \ \ 
\begin{pmatrix} \theta \\ \phi \end{pmatrix} \in \HmZ.
\label{E:LTWOBOUND}
\end{align}
As a consequence, there exists a constant $C_{\L_m}>0$ such that 
\begin{align}
\bigg\| \L_m  \begin{pmatrix} \theta \\ \phi \end{pmatrix} \bigg\|_{\H^0_{2m,Z}} \le C_{\L_m} \bigg(\bigg\|\bD^{2m}\begin{pmatrix} \theta \\ \phi \end{pmatrix}\bigg\|_{\H^0_{2m,Z}}
+\bigg\| \begin{pmatrix} \theta \\ \phi \end{pmatrix}\bigg\|_{\H^0_{2m,Z}}\bigg). \label{E:LONEBOUND}
\end{align}
In particular, both operators $\L_m,\K_m : D(\mathcal L_m)\subset \HmZ \to \mathcal H^0_{2m,Z}$, (viewed as operators from $\HmZ$ to $\H^0_{2m,Z}$) are bounded. 
\end{lemma}


\begin{remark}\label{R:ZDEPENDENCE}
We note that the constants above depend on the choice of both $m$ and $Z$ and grow large as $Z$ gets larger. 
\end{remark}


\begin{proof}
By Lemma~\ref{L:INVERSE} we have
\begin{align}
\left\|\K_m \begin{pmatrix} \theta \\ \phi \end{pmatrix} \right\|_{\H^0_{2m,Z}} 
\le C\big( \|\bD^{-2m}\mathcal R_{2m}\theta\|_{L^2} + \|\mathcal V \theta\|_{L^2} \big) \le C\big( \|\bd \mathcal R_{2m}\theta \|_{L^2}+ \|\bd \theta\|_{L^2} \big),
\end{align}
 where we have also used~\eqref{E:HARDYREFINED} and recall that $\mathcal V\in \DZeven$ is bounded, as are $g_{2m}$, $\bG$ from Lemma~\ref{L:GBAR}. From~\eqref{E:REMAINDER}, we then have 
\begin{align}
\bd \mathcal R_{2m}\theta   =&\, \bd\bigg(\sum_{j=1}^{2m} \sum_{Q\in \mathcal X_j, P\in \mathcal Y_{2m+2-j}} c^j_{PQ} P\bG Q \theta\bigg) \notag \\
 =&\, \sum_{k=1}^{m} \sum_{Q\in \mathcal X_{2k}, P\in \mathcal Y_{2m+2-2k}} c^{2k}_{PQ} \left(  P\bG  \bd Q \theta +\bp P\bG  Q \theta \right) 
\\& +  \sum_{k=1}^{m} \sum_{Q\in \mathcal X_{2k-1}, P\in \mathcal Y_{2m+3-2k}} c^{2k-1}_{PQ} \left(  P\bG  \bp Q \theta+\bd P\bG   Q \theta \right). \notag
\end{align} 
Since, for any $k\in\{1,\dots,m\}$, $\bd Q \in \X_{2k+1}$ if $Q\in \X_{2k}$ and similarly $\bp Q\in \X_{2k}$ if $Q\in \X_{2k-1}$, it is obvious from the above
that 
\begin{align}
\|\bd \mathcal R_{2m}\theta \|_{L^2} \le C \sum_{k=1}^{2m+1} \sum_{Q\in \X_k}\|Q\theta\|_{L^2} \le C \left(\|\bd\bD^{2m}\theta\|_{L^2}+\|\bd \theta\|_{L^2}\right),
\end{align}
where we have used Lemma~\ref{L:EQUIVALENCE} in the last line. This implies the bound on $\left\| \K_m  \begin{pmatrix} \theta \\ \phi \end{pmatrix} \right\|_{\H^0_{2m,Z}}$. Since $\L_m = \bfL-\K_m$ and $\bfL$ clearly satisfies a bound analogous to~\eqref{E:LONEBOUND} as $m\geq 1$, the claim follows.
\end{proof}

We now sharpen our understanding of the operator $\K_m$. Due to the smoothing 
induced by the operator $\bD^{-2}$, we have the following lemma:


\begin{lemma}\label{L:L2COMPACT}
Let $m\in\N$. The operator $\K_m:\HmZ\to\HmZ$ is a bounded compact operator.
\end{lemma}


\begin{proof}
By definition, for any $( \theta,\phi)^\top\in \HmZ$, we have $\bD^{2m}\bD^{-2m}\mathcal R_{2m} \theta=\mathcal R_{2m} \theta$.
From the regularity properties $ \bG,\mathcal V\in\DZeven$ from Lemmas~\ref{L:GBAR} and~\ref{L:GPROPERTIES} and formula~\eqref{E:REMAINDER}, it easily follows that 
\begin{align}
\left\|\K_m\begin{pmatrix}  \theta \\ \phi \end{pmatrix}\right\|_{\HmZ}
& \leq  \left\|\mathcal R_{2m} \theta\right\|_{L^2}  + \|\bD^{2m}\big(\mathcal V \theta\big) \|_{L^2}  +  \beta 
 \left\|\K_m\begin{pmatrix}  \theta \\ \phi \end{pmatrix}\right\|_{\H^0_{2m,Z}} \notag\\
& \le C \sum_{j=1}^{2m}\sum_{Q\in X_j}\|Q \theta\|_{L^2}  +  \beta 
C_{\K_m} \left(\|\bd\bD^{2m} \theta\|_{L^2}+\|\bd  \theta\|_{L^2}\right) \notag\\
& \le C  \left(\|\bd\bD^{2m} \theta\|_{L^2}+\|\bd  \theta\|_{L^2}\right), 
\end{align}
which implies the boundedness. Note all $L^2$ norms are taken to be $L^2(0,Z)$. Now assume that $\begin{pmatrix}  \theta_n \\ \phi_n \end{pmatrix}\rightharpoonup 0$ weakly  in $\HmZ$ as $n\to\infty$ for some sequence $\begin{pmatrix}  \theta_n \\ \phi_n \end{pmatrix}\subset \HmZ$. Due to the  compact embedding $\HmZ\hookrightarrow H^{2m}\times H^{2m-1}$ deduced from~\eqref{E:HSOBOLEVEQUIV}, it is easily seen that $\bD^{2m}(\bD^{-2m}\mathcal R_{2m} \theta_n)=\mathcal R_{2m} \theta_n\to 0$ in $L^2(0,Z)$, that $\bD^{2m}(\mathcal V \theta_n)\to 0$ in $L^2(0,Z)$, and hence that  $\K_m \begin{pmatrix}  \theta_n \\ \phi_n \end{pmatrix}\to 0$ in $\HmZ$.  
 Therefore the operator $\K_m:\HmZ\to\HmZ$ is also compact.
\end{proof}


Observe from Proposition~\ref{P:DISSIP} and~\eqref{E:DMAM} that the operator $\mathcal{L}_m$ satisfies good dissipativity properties with respect to the \emph{semi-norm} in $\HdotmZ$. In order to obtain from this an operator with good dissipativity with respect to the full norm in $\HmZ$, we make the following classical argument.

\begin{lemma}\label{L:KTILDE}
Let $m\in\N$. There exists a compact operator $\tilde{\K}_m:\HmZ\to\HmZ$ such that 
\begin{align}\label{E:KTILDE}
\left( \tilde{\K}_m\begin{pmatrix} \Theta \\ \Phi \end{pmatrix}\, ,\, \begin{pmatrix}  \theta \\ \phi \end{pmatrix}\right)_{\HmZ}
= \left(\begin{pmatrix} \Theta \\ \Phi \end{pmatrix}\, , \, \begin{pmatrix}  \theta \\ \phi \end{pmatrix}\right)_{\mathcal H^0_{2m,Z}}, \ \ 
\begin{pmatrix}  \theta \\ \phi \end{pmatrix}\in \HmZ.
\end{align}
\end{lemma}

\begin{proof}
We notice that  $\HmZ$  
embeds compactly into $\mathcal H^0_{2m,Z}$. For any given $\begin{pmatrix} \Theta \\ \Phi \end{pmatrix}\in\H^0_{2m,Z}$, following a standard compactness argument, we consider the linear form $\HmZ\ni  \begin{pmatrix}  \theta \\ \phi \end{pmatrix} \mapsto 
\left(\begin{pmatrix} \Theta \\ \Phi \end{pmatrix}\, , \, \begin{pmatrix}  \theta \\ \phi \end{pmatrix}\right)_{\mathcal H^0_{2m,Z}}$, which is clearly a bounded, anti-linear map from $\HmZ$ to $\mathbb C$. Therefore, by the Riesz representation theorem, there exists a unique $\tilde{\K}_m\begin{pmatrix} \Theta \\ \Phi \end{pmatrix}\in \HmZ$ such that 
\begin{align}
\left( \tilde{\K}_m\begin{pmatrix} \Theta \\ \Phi \end{pmatrix}\, ,\, \begin{pmatrix}  \theta \\ \phi \end{pmatrix}\right)_{\HmZ}
= \left(\begin{pmatrix} \Theta \\ \Phi \end{pmatrix}\, , \, \begin{pmatrix}  \theta \\ \phi \end{pmatrix}\right)_{\mathcal H^0_{2m,Z}}, \ \ 
\begin{pmatrix}  \theta \\ \phi \end{pmatrix}\in \HmZ.
\end{align}
The map $\tilde{\K}_m:\mathcal H^0_{2m,Z}\to \HmZ$ is a bounded linear map. It is not hard to see that $\tilde\K_m$ is symmetric
by
\begin{align}
\left(\tilde{\K}_m\begin{pmatrix} \Theta_1 \\ \Phi_1 \end{pmatrix}, \begin{pmatrix} \Theta_2 \\ \Phi_2 \end{pmatrix} \right)_{\mathcal H^0_{2m,Z}} 
= \overline{\left(\begin{pmatrix} \Theta_2 \\ \Phi_2 \end{pmatrix} ,\tilde{\K}_m\begin{pmatrix} \Theta_1 \\ \Phi_1 \end{pmatrix}\right)_{\mathcal H^0_{2m,Z}}}
= \overline{\left( \tilde{\K}_m \begin{pmatrix} \Theta_2 \\ \Phi_2 \end{pmatrix} \, , \, \tilde{\K}_m\begin{pmatrix} \Theta_1 \\ \Phi_1 \end{pmatrix}\right)_{\HmZ}}, \label{E:SYMMETRY}
\end{align}
and therefore self-adjoint. Note that due to the compact embedding $\HmZ\hookrightarrow \mathcal H^0_{2m,Z}$, considering $\tilde{\K}_m$ as a map into $\mathcal H^0_{2m,Z}$, it
is in fact compact, since it is a composition of a continuous and a compact map. 

Our goal is to show that the restriction $\tilde{\K}_m\Big|_{\HmZ}:\HmZ\to \HmZ$ is a compact, self-adjoint operator.  
To that end, suppose $\begin{pmatrix} \Theta_n \\ \Phi_n \end{pmatrix} \rightharpoonup {\bf 0}$ weakly in $\HmZ$, then clearly 
$\begin{pmatrix} \Theta_n \\ \Phi_n \end{pmatrix} \to {\bf 0}$ in $\mathcal H^0_{2m,Z}$ due to the compact embedding  $\HmZ\hookrightarrow \mathcal H^0_{2m,Z}$. 
Moreover, by~\eqref{E:SYMMETRY}
\begin{align}
\left\|\tilde{\K}_m\begin{pmatrix} \Theta_n \\ \Phi_n \end{pmatrix} \right\|_{\HmZ}^2
= \left(\begin{pmatrix} \Theta_n \\ \Phi_n \end{pmatrix}, \tilde{\K}_m\begin{pmatrix} \Theta_n \\ \Phi_n \end{pmatrix} \right)_{\mathcal H^0_{2m,Z}} 
\le \left\|\tilde{\K}_m\begin{pmatrix} \Theta_n \\ \Phi_n \end{pmatrix} \right\|_{\mathcal H^0_{2m,Z}} \left\|\begin{pmatrix} \Theta_n \\ \Phi_n \end{pmatrix} \right\|_{\mathcal H^0_{2m,Z}}, 
\end{align}
which clearly converges to $0$. Therefore $\tilde{\K}_m\Big|_{\HmZ}:\HmZ\to \HmZ$ is also compact.
\end{proof}

We finally have the pieces in place to define the maximally dissipative operator $\mathscr L_m$ that we will use to derive spectral properties of the original operator $\bfL$. We now consider the operator, for each $m\geq \m$ and $Z>z_*$,
\begin{align}\label{E:LSCR}
\mathscr L_m  := \L_m - \tilde{\K}_m=\bfL-\K_m-\tilde{\K}_m : D(\bfL)\to \HmZ,
\end{align}
where we recall~\eqref{D:Lm}. We recast~\eqref{E:LINEAR1} in the form
\begin{align}
\pa_s\begin{pmatrix} \theta \\ \phi \end{pmatrix} = \mathscr L_m \begin{pmatrix} \theta \\ \phi \end{pmatrix} + \left(\K_m + \tilde{\K}_m  \right)\begin{pmatrix} \theta \\ \phi \end{pmatrix}.
\end{align}


Before stating the main dissipativity estimate on the operator $\mathscr{L}_m$, we remind the reader of the definition of the norm~\eqref{E:INNERPROD} and its dependence on the constant $\beta>0$ and on the end-point $Z$. We also recall the regularity index $\m\in\N$ defined in Proposition~\ref{P:DISSIP}. 

\begin{proposition}[Dissipativity of $\mathscr L_m$]\label{P:LDISSIP}
Let $m\in\N$ be such that $m\geq \m$ and let $Z>z_*$. Then there exist $\beta>0$ and $\sg >0$, depending on $m$ and $Z$, such that the operator $\mathscr L_m  : D(\bfL)\subset \HmZ \to \HmZ$ 
satisfies the dissipativity bound
\begin{align}
\textup{Re}\left(  \begin{pmatrix}  \theta \\ \phi \end{pmatrix}\,,\, \mathscr L_m\begin{pmatrix}  \theta \\ \phi \end{pmatrix}\right)_{\HmZ}
\le - 2\sg \left\|\begin{pmatrix}  \theta \\ \phi \end{pmatrix} \right\|_{\HmZ}^2, \ \ \begin{pmatrix}  \theta \\ \phi \end{pmatrix}\in D(\bfL).
\end{align}
\end{proposition}


\begin{proof}
Observe that for any $\begin{pmatrix}  \theta \\ \phi \end{pmatrix}\in D(\bfL)$, from~\eqref{E:DMAM},~\eqref{E:LONEBOUND}, and Proposition~\ref{P:DISSIP},
\begin{align}
{}&\text{Re}\left(  \begin{pmatrix}  \theta \\ \phi \end{pmatrix}\,,\, \mathscr L_m\begin{pmatrix}  \theta \\ \phi \end{pmatrix}\right)_{\HmZ} \notag\\
&=   \text{Re} \left( \bD^{2m}\begin{pmatrix}  \theta \\ \phi \end{pmatrix}\,,\, A_{2m}\bD^{2m} \begin{pmatrix}  \theta \\ \phi \end{pmatrix}\right)_{\dot{\H}^0_{2m,Z}}
+\beta 
 \text{Re} \left( \begin{pmatrix}  \theta \\ \phi \end{pmatrix}\,,\, \L_m\begin{pmatrix}  \theta \\ \phi \end{pmatrix}\right)_{\dot{\H}^0_{2m,Z}}  - \text{Re} \left( \begin{pmatrix}  \theta \\ \phi \end{pmatrix}\,,\,\tilde{\K}_m\begin{pmatrix}  \theta \\ \phi \end{pmatrix}\right)_{\HmZ} \notag\\
& \le  - k _0 m \left\|\begin{pmatrix}  \theta \\ \phi \end{pmatrix} \right\|_{\GmZ}^2
 -  \left\|  \begin{pmatrix}  \theta \\ \phi \end{pmatrix}\right\|_{\dot{\H}^0_{2m,Z}}^2 
 + \beta 
 \left\| \L_m \begin{pmatrix}  \theta \\ \phi \end{pmatrix}\right\|_{\dot{\H}^0_{2m,Z}}  \left\| \begin{pmatrix}  \theta \\ \phi \end{pmatrix}\right\|_{\dot{\H}^0_{2m,Z}}  \notag\\
 & \le - k _0 m \left\|\begin{pmatrix}  \theta \\ \phi \end{pmatrix} \right\|_{\GmZ}^2
 -  \left\|  \begin{pmatrix}  \theta \\ \phi \end{pmatrix}\right\|_{\dot{\H}^0_{2m,Z}}^2 
+ \beta 
C_{\L_m}  \bigg( \left\|\begin{pmatrix}  \theta \\ \phi \end{pmatrix} \right\|_{\GmZ}^2
+ \left\|  \begin{pmatrix}  \theta \\ \phi \end{pmatrix}\right\|_{\dot{\H}^0_{2m,Z}}^2\bigg) \notag\\
& \le - 2\sg  \left\|\begin{pmatrix}  \theta \\ \phi \end{pmatrix} \right\|_{\HmZ}^2,
\end{align}
for some $\sg >0$ and $\beta$ 
chosen sufficiently small, depending on $C_{\L_m}$ and hence on $m$ and $Z$. 
\end{proof}


To prove that $\mathscr{L}_m$ is maximally dissipative, it remains only to prove that $\mathscr{L}_m-\l \bfI$ is surjective for $\l$ sufficiently large (recall the value $\l_0$  defined in Lemma~\ref{L:ONTO}).
 
\begin{proposition}[Surjectivity of $\mathscr L_m - \l \bfI$ for $\l$ sufficiently large]\label{P:ONTO}
Let $m\in \N$, $m\geq \m$. There exists $\l_1>1$ such that, for any $\l\geq\l_1$ and any $\begin{pmatrix} f \\ g \end{pmatrix}\in \HmZ$, there exists a unique 
$\begin{pmatrix} \theta \\ \phi\end{pmatrix}\in \HmZ$ such that 
\begin{align}
\left(\mathscr L_m - \l \bfI\right) \begin{pmatrix} \theta \\ \phi\end{pmatrix} = \begin{pmatrix} f \\ g\end{pmatrix}.
\end{align}
\end{proposition}


\begin{proof}
We rewrite
\[
\mathscr L_m = \bfL - (\K_m+\tilde{\K}_m),
\]
where we note that $\K_m,\tilde{\K}_m:\HmZ \to \HmZ$ are compact operators by Lemmas~\ref{L:L2COMPACT} and~\ref{L:KTILDE}.
We now use the identity for $\la>1$,
\begin{align}\label{E:IDENTITYSIMPLE}
\bfL - (\K_m+\tilde{\K}_m)-\l \bfI = (\bfL-\l \bfI)\left(\bfI - (\bfL-\l \bfI)^{-1}(\K_m+\tilde\K_m)\right).
\end{align}

We next claim that, for $\l$ sufficiently large, the resolvent operator $(\mathbf{L} - \l \bfI)^{-1}:\HmZ\to\HmZ$  is well-defined  and satisfies the bound
\begin{align}\label{E:RESOLVENTBOUND}
\|(\mathbf{L} - \l \bfI)^{-1}\|_{\L(\HmZ,\HmZ)}\le \frac 2{\l}.
\end{align}
To see this, we first observe by the standard argument that, for $\la$ sufficiently large, using Proposition~\ref{P:LDISSIP} and the boundedness of $\K_m+\tilde{\K}_m$,
 \beqa
 \textup{Re}\left( -(\bfL-\la \bfI) \begin{pmatrix} \theta \\ \phi \end{pmatrix} ,  \begin{pmatrix} \theta \\ \phi \end{pmatrix} \right)_{\HmZ}= \textup{Re}\left( -(\mathscr L_m+\K_m+\tilde{\K}_m-\la \bfI) \begin{pmatrix} \theta \\ \phi \end{pmatrix} ,  \begin{pmatrix} \theta \\ \phi \end{pmatrix} \right)_{\HmZ}\\
 \geq 2\sg \|\begin{pmatrix} \theta \\ \phi \end{pmatrix}\|_{\HmZ}^2-C\|\begin{pmatrix} \theta \\ \phi \end{pmatrix}\|_{\HmZ}^2+\la\|\begin{pmatrix} \theta \\ \phi \end{pmatrix}\|_{\HmZ}^2\geq \frac{\la}{2}\|\begin{pmatrix} \theta \\ \phi \end{pmatrix}\|_{\HmZ}^2.
 \eeqa
 In particular, this implies the injectivity of $(\bfL-\la \bfI)$. By Lemma~\ref{L:ONTO}, $(\bfL-\la \bfI)$ is also surjective provided $\l\geq \l_0$, and hence the resolvent is well-defined and satisfies the estimate~\eqref{E:RESOLVENTBOUND}.

By~\eqref{E:RESOLVENTBOUND}, for $\l$ sufficiently large, we have $\|(\bfL-\l \bfI)^{-1}(\K_m+\tilde\K_m)\|_{\L(\HmZ,\HmZ)}<\frac12$. Moreover $(\bfL-\l \bfI)^{-1}(\K_m+\tilde\K_m)$ is compact and therefore  the 
operator $\bfI - (\bfL-\l \bfI)^{-1}(\K_m+\tilde\K_m)$ is invertible. It then follows from~\eqref{E:IDENTITYSIMPLE} and~\eqref{E:RESOLVENTBOUND} that there exists $\l_1\geq\l_0$ such that $\bfL - (\K_m+\tilde{\K}_m)-\l \bfI$ is invertible for $\l\geq\l_1$. This concludes the proof.
\end{proof}

\begin{theorem}[Lumer-Phillips applied to $\mathscr L_m=\bfL - (\K_m+\tilde{\K}_m)$]\label{T:LUPH}
Let $m\geq \m$, $Z>z_*$. There exist $\beta>0$ and an $\sg >0$ such that the operator $\mathscr L_m : D(\L_m)\subset  \HmZ \to \HmZ $ generates a quasicontraction semigroup  $e^{s\mathscr L_m}_{s\geq 0}$ such that 
\begin{align}
\left\|e^{s \mathscr L_m} \begin{pmatrix} \theta \\ \phi \end{pmatrix}\right\|_{\HmZ} \leq e^{-2\sg s}\left\|\begin{pmatrix} \theta \\ \phi \end{pmatrix}\right\|_{\HmZ}
\end{align}
\end{theorem}


\begin{proof}
The statement follows as a direct consequence of the Lumer-Phillips Theorem, see, e.g.,~\cite[Theorem 12.22]{Renardy04}.
\end{proof}


\begin{proof}[Proof of Theorem~\ref{T:LINEARMAIN}]
We recall the decomposition from~\eqref{E:LSCR},
\begin{align}
\bfL= \mathscr L_m + \K_m+\tilde{\K}_m,
\end{align}
where $\K_m+\tilde{\K}_m$ is a compact operator on $\HmZ$. Thus, from Theorem~\ref{THM:COMPACTPERT}, we immediately find that $\bfL$ generates a strongly continuous semigroup. From Theorem~\ref{THM:COMPACTPERT}(i), for any $\epsilon\in(0,\om)$, the set
\beqs
\sigma(\bfL)\cap\{\l\in\C\,|\,\Re\l\geq -2\sg \}.
\eeqs
contains only the single eigenvalue $\l=1$ associated to the time-translation mode, where we have also applied Theorem~\ref{T:SPECTRAL}. Thus, from  Theorem~\ref{THM:COMPACTPERT}(iii), 
\begin{align}
\left\|e^{s\bfL}\left(\mathbf{I}-\bfP\right) \begin{pmatrix} \theta \\ \phi \end{pmatrix} \right\|_{\HmZ} \le e^{-\sg s}\left\|\begin{pmatrix} \theta \\ \phi \end{pmatrix}\right\|_{\HmZ}, \ \ \begin{pmatrix} \theta \\ \phi \end{pmatrix}\in \HmZ,
\end{align}
with $\sg >0$ provided by Theorem~\ref{T:LUPH}. Moreover, by the standard theory of Riesz projections, the projection $\bfP$ commutes with $\bfL$ so that $ e^{s\bfL}\bfP = e^s \bfP$ and, by applying Lemma~\ref{L:GROWINGPROJ}, we have
\begin{align}
\|e^{s\bfL}\bfP \Phi\|_{\HmZ} \le e^s \|\Phi\|_{\HmZ}, \ \ \Phi\in\HmZ.
\end{align}
This concludes the proof.
\end{proof}


\section{Lower order bounds via Duhamel principle}\label{S:DUHAMEL}


Having established the linear stability of the operator $\bfL$ in Theorem~\ref{T:LINEARMAIN}, we are now in a position to turn to the full, nonlinear problem. As discussed in the introduction, due to the quasilinear character of the problem, we need to establish two kinds of nonlinear estimates: first, exponential decay of a lower-order norm using the semi-group properties of the linearised flow (but with derivative loss); second, top order energy estimates without derivative loss. The purpose of this section is to establish the former estimates.

For precision, we now fix the regularity exponent $m$ to be  $\m\in\N$ as in Theorem~\ref{T:LINEARMAIN}, satisfying~\eqref{E:MCONDITION}. We also take $Z=Z_0$ to be determined later, and let $\beta=\beta(\m,Z_0)>0$ as in Theorem~\ref{T:LINEARMAIN}. We will assume throughout that $\eps_0$ will be taken sufficiently small, depending on $Z_0$, so that for any constant $C(Z_0)$ appearing, we assume $C(Z_0)\eps_0\leq \frac14$ always. In particular, from~\eqref{E:APRIORIMAINPT}, we always retain the norm bound~\eqref{E:HMZTE}.

Starting from~\eqref{E:NONLINEARFORM}, we employ the Duhamel formula to write the solution as
\begin{align}\label{E:DUHAMEL}
\Phi(s) = e^{ \bfL(s-\sin) } \PhiT + \int_{\sin}^{s} e^{\bfL(s-\sigma)}\bfN[\Phi](\sigma) \,d\sigma,
\end{align}
where we recall~\eqref{E:BOLDNDEF} and~\eqref{E:NDEF}. We remind the reader that $T\in\mathbb R$ parametrises the 1-parameter family of initial data $\PhiT$ via~\eqref{E:PROFILE1} and the initial self-similar time $\sin$ by~\eqref{E:SINITIAL}.

We next project the dynamics into the stable part and the growing mode induced by time-translation symmetry of the problem.
Recalling the Riesz projection,~\eqref{E:RIESZDEF}, we introduce the splitting
\begin{align}\label{E:PLUSMINUS}
\Phi = (\bfI-\bfP)\Phi + \bfP\Phi =: \Phis + \Phiu.
\end{align}
From~\eqref{E:DUHAMEL} and the semigroup bound~\eqref{E:STABLEBOUND} of Theorem~\ref{T:LINEARMAIN}, we obtain 
\begin{align}\label{E:MINUSBOUND}
\|\Phis\|_{\HmZm} \le e^{-\sg (s-\sin) }\|\PhiT\|_{\HmZm} + \int_{\sin}^{s} e^{-\sg (s-\sigma)}\|\bfN[\Phi](\sigma)\|_{\HmZm} \,d\sigma.
\end{align}
When we project into the unstable 1-dimensional subspace, we obtain in turn from~\eqref{E:PPROPERTIES} the Duhamel
formula for $\Phiu$:
\begin{align}\label{E:PHIU}
\Phiu(s) = e^{s-\sin}\bfP\PhiT + \int_{\sin}^{s} e^{s-\sigma} \bfP\bfN[\Phi](\sigma)\,d\sigma,
\end{align}
which is equivalently rewritten in the form
\begin{align}\label{E:STABLEMFD}
\Phiu(s) & = e^{s-\sin}\bfP\bigg(\PhiT +\int_{\sin}^{S_T} e^{\sin-\sigma} \bfN[\Phi](\sigma)\diff \sigma\bigg) - \int_{s}^{S_T} e^{s-\sigma} \bfP\bfN[\Phi](\sigma)\diff \sigma ,
\end{align}
 where $S_T$ is the maximal time defined in~\eqref{E:STDEF}.


\subsection{Statement of the low-order a priori bound}


Our goal is to use these Duhamel formulae to prove an a priori bound on the $\HmZ$ norm of the solution $\Phi$ that encodes decay induced by the linear semi-group, but at the cost of a derivative loss. In order to make this precise, we recall the constants $\nu\in(0,1)$ and $\Om>0$ satisfying~\eqref{E:OMEGABDS}.

\begin{proposition}\label{P:EE12}
Let $\Phi$ be a unique solution to~\eqref{E:NONLINEARFORM}  with initial data~\eqref{E:PROFILE1}. 
Assume the a priori bounds~\eqref{E:APRIORIMAIN}--\eqref{E:APRIORIMAINPT} and recall the definition~\eqref{E:STDEF} of the maximal time $S_T$.
Then for any $s\in [\sin,S_T)$ the following bound holds 
\beqa
	e^{\nu\sg s}\|\Phi\|_{\HmZm}\leq &\,C\sup_{\sigma\in[\sin,s]}(e^{\Omega \sigma}\tE_{\leq 2(\m+1)}^{\frac12})\sup_{\sigma\in[\sin,s]}(e^{\nu\sg \sigma}\|\Phi\|_{\HmZm}) e^{-\Omega s}  + e^{-(1-\nu)\sg s}\|\Phi^T_{in}\|_{\HmZm}\\
	&+e^{(1+\nu\sg )s}\Big\|\mathbf{P}\Big(\Phi^T_{in}+\int_{\sin}^{S_T}e^{\sin-\sigma} \mathbf{N}[\Phi]\,\dif \sigma\Big)\Big\|_{\HmZm}.
	\eeqa
\end{proposition}


\begin{remark}
Proposition~\ref{P:EE12} identifies the term
\be\label{E:PERRON}
\bfP\bigg( \PhiT +\int_{\sin}^{S_T} e^{\sin-\sigma} \bfN[\Phi](\sigma)\diff \sigma\bigg)
\ee
as the key obstacle to closing the estimates globally in $s$. This is reminiscent of the classical (un)stable manifold arguments and we shall show in 
Section~\ref{S:MAINTHEOREM} that there exists a choice of $T$, $|T|\ll1$ such that~\eqref{E:PERRON} vanishes.
\end{remark}


\subsection{Nonlinear estimates}


Before proving Proposition~\ref{P:EE12}, we first carefully analyse the algebraic structure of the nonlinearity $\bfN[\Phi]$ near the origin $z=0$. Our key lemma states that $\mathcal N[\theta]$ is suitably regular near $z=0$.


\begin{lemma}
Let $\mathcal N[\theta]$ be the nonlinearity defined in~\eqref{E:NDEF} and let $(\theta,\phi)^\top\in \HmZm$. 
Let
\begin{align}\label{E:GHTILDEDEF}
\tH(z):=z^2\bzeta_{zz}(z)+\CLP z^{\frac13}\tG(z)^2,
\end{align}
where we recall $\tG$ from~\eqref{E:GTILDEDEF}. We note that $\tH\in\DZodd$ is real analytic as a function of $z^{\frac13}$ near $z=0$. 
Then the following identity holds:
\begin{align}
-\color{black}\mathcal N[\theta] &=\mathcal N_1[\theta] + \mathcal N_2[\theta]+\mathcal N_3[\theta]+\mathcal N_4[\theta], \label{E:NDECOMP}
\end{align}
where
\begin{align}
\mathcal N_1[\theta]& :=\sqrt{\bG}\, f_1(\sqrt{\bG}\bp\theta)\, \sqrt{\bG}\bp\theta \bp\big( \sqrt{\bG}\bp \theta\big), \ \  f_1(x) : = \frac{2+x}{(1+x)^2},
\label{E:N1}\\
\mathcal N_2[\theta] &:= \tH\, f_2(\sqrt{\bG}\bp\theta) \, (\sqrt{\bG}\bp\theta)^2, \ \ f_2(x):=-\frac{1}{(1+x)}, \label{E:N2}\\
\mathcal N_3[\theta]&:=\CLP z^{\frac13} \tG^2 \, f_3(\tG \frac{\theta}{z^{\frac13}}) \, (\tG \frac{\theta}{z^{\frac13}})^2, \ \ f_3(x) := \frac{3+2x}{(1+x)^2}, 
\label{E:N3} \\
\mathcal N_4[\theta]& := \sqrt{\bG}\bp \big(\tG \frac{\theta}{z^{\frac13}}\big) f_4(\tG \frac{\theta}{z^{\frac13}},\sqrt{\bG}\bp\theta), \ \ 
f_4(x,y) : = 2\frac{(1+x)(1+y)-1}{(1+x)(1+y)}.
\label{E:N4}
\end{align}
\end{lemma}


\begin{proof}
We observe first that, due to~\eqref{E:NONEDEF},
\begin{align}
\frac{2}{\bzeta_z^3}\theta_z - N_1(\bzeta_z,\theta_z) = \frac1{\bzeta_z^2}\frac{\theta_z}{\bzeta_z} \frac{2+\frac{\theta_z}{\bzeta_z}}{\big(1+\frac{\theta_z}{\bzeta_z}\big)^2}.
\end{align}
Upon rewriting $\theta_{zz}$ in the form $$\theta_{zz}=\bzeta_z\pa_z\big(\frac{\theta_z}{\bzeta_z}\big)+ \frac{\theta_z}{\bzeta_z}\bzeta_{zz},$$
after a brief algebraic manipulation, we have 
\begin{align}
&\Big(\frac{2}{\bzeta_z^3}\theta_z - N_1(\bzeta_z,\theta_z)\Big) \theta_{zz} - N_1(\bzeta_z,\theta_z)\bzeta_{zz}
=\frac1{\bzeta_z} \frac{\theta_z}{\bzeta_z}\frac{2+\frac{\theta_z}{\bzeta_z}}{\big(1+\frac{\theta_z}{\bzeta_z}\big)^2} \pa_z\big( \frac{\theta_z}{\bzeta_z}\big)
- \frac1{\bzeta_z^2} \big(\frac{\theta_z}{\bzeta_z}\big)^2\frac1{1+\frac{\theta_z}{\bzeta_z}}\bzeta_{zz}. 
\end{align}
Recalling the steady state equation~\eqref{E:LPLAG}, we have $\frac{\bzeta_{zz}}{\bzeta_z^2}= -\frac2{\bzeta}+ z^2\bzeta_{zz}+\frac{\CLP z}{\bzeta^2}$, where we note that the term $-\frac2{\bzeta}$ is the most singular one near $z=0$. Using this and~\eqref{E:NDEF} we may therefore rewrite $\mathcal N[\theta]$ in the form
\begin{align}
-\color{black}\mathcal N[\theta] &=\frac1{\bzeta_z} \frac{\theta_z}{\bzeta_z}\frac{2+\frac{\theta_z}{\bzeta_z}}{\big(1+\frac{\theta_z}{\bzeta_z}\big)^2} \pa_z\big( \frac{\theta_z}{\bzeta_z}\big)- \big(\frac{\theta_z}{\bzeta_z}\big)^2\frac1{1+\frac{\theta_z}{\bzeta_z}}\left(-\frac2{\bzeta}+ z^2\bzeta_{zz}+\frac{\CLP z}{\bzeta^2}\right) \notag\\
& \ \ \ \ + \CLP z \frac1{\bzeta^2}\Big(\frac{\theta}{\bzeta}\Big)^2\frac{3+2\frac{\theta}{\bzeta}}{\big(1+\frac{\theta}{\bzeta}\big)^2}- \frac{2\theta^2}{\bzeta^2(\bzeta+\theta)}. \label{E:NINTER1}
\end{align}
The leading order singular contributions cancel out as follows.
\begin{align}
\big(\frac{\theta_z}{\bzeta_z}\big)^2\frac1{1+\frac{\theta_z}{\bzeta_z}}\frac2{\bzeta} - \frac{2\theta^2}{\bzeta^2(\bzeta+\theta)}
& = -\frac2{\bzeta}\Big[ \frac{\frac{\theta^2}{\bzeta^2}}{1+\frac{\theta}{\bzeta}} - \frac{\frac{\theta_z^2}{\bzeta_z^2}}{1+\frac{\theta_z}{\bzeta_z}}\Big] \notag\\
& = \frac2{\bzeta_z}\big(\frac{\theta}{\bzeta}\big)_z \frac{(1+\frac{\theta}{\bzeta})(1+\frac{\theta_z}{\bzeta_z})-1}{(1+\frac{\theta}{\bzeta})(1+\frac{\theta_z}{\bzeta_z})},\label{E:NINTER2}
\end{align}
where the last identity follows from a straightforward algebraic manipulation.
We replace every occurrence of the operator $\frac1{\bzeta_z}\pa_z$ by $\frac1{\bp\bzeta}\bp=\sqrt{\bG}\bp$. 
Using~\eqref{E:GHTILDEDEF} and~\eqref{E:NINTER1}--\eqref{E:NINTER2}, this then leads to the claimed identity.
\end{proof}


The next lemma provides the key nonlinear bound, which captures the loss-of-derivative associated with the quasi-linear nature of the problem.


\begin{lemma}\label{L:NBOUND}
Let $\Phi=\begin{pmatrix}\theta\\\phi\end{pmatrix}\in\mathcal H^{2\m+1}_{Z_0}$ be such that $\|\Phi\|_{\HmZm}\le 1$.
Then there exists a constant $C=C(\m,Z_0)>0$ such that 
\begin{align}
\|\bfN[\Phi]\|_{\HmZm} \le C\|\Phi\|_{\mathcal H^{2\m+1}_{Z_0}}\, \|\Phi\|_{\HmZm}.
\end{align}
In particular,
\beq\label{E:CRUDENONLINEAR0}
\|\bfN[\Phi]\|_{\HmZm}  \le C \|\Phi\|_{\HmZm}\tE_{\leq 2(\m+1)}^{\frac12}. 
\eeq
On the other hand, if $\Phi_1,\Phi_2\in\mathcal{H}^{2\m+1}_{Z_0}$ both satisfy  $\|\Phi_j\|_{\HmZm}\le 1$, $j=1,2$, then
 \beq\label{E:NPHIDIFF}
 \|\mathbf{N}[\Phi_1]-\mathbf{N}[\Phi_2]\|_{\HmZm}\leq \|\Phi_1-\Phi_2\|_{\mathcal{H}^{2\m+1}_{Z_0}}(\|\Phi_1\|_{\HmZm}+\|\Phi_2\|_{\HmZm}).
 \eeq 
\end{lemma}


\begin{proof}
It follows from the assumption~\eqref{E:APRIORIMAINPT} that there exists $C>0$, depending only on $\m$, such that for all $k,\ell\in\{0,\ldots,2\m\}$, 
\beq
\big|f_1^{(k)}(\sqrt{\bG}\bp\theta)\big|+\big|f_2^{(k)}(\sqrt{\bG}\bp\theta)\big|+\big|f_3^{(k)}(\tG \frac{\theta}{z^{\frac13}})\big|+\big|(\pa_x^k\pa_y^\ell f_4)(\tG \frac{\theta}{z^{\frac13}},\sqrt{\bG}\bp\theta)\big|\leq C,\label{E:FKDERIVBDS}
\eeq
where $f_j^{(k)}$ refers to the $k$-th derivative of $f_j$ and $\pa_x^k\pa_y^\ell f_4$ is defined in the obvious way.

An inductive argument based on Lemma~\ref{L:XYPRODCHAIN} shows that 
for any $\ell\in\mathbb N$ we can express $\bD^{2\ell}\mathcal N_1[\theta]$ as a finite linear combination of the expressions of the form
\begin{align}
f_1^{(k)}(\sqrt{\bG} \bp\theta) P_{j_0}(\sqrt{\bG}) P_{j_1}(\sqrt{\bG} \bp\theta) P_{j_2}(\sqrt{\bG} \bp\theta) \dots P_{j_{k+2}}(\sqrt{\bG} \bp\theta) \label{E:CLN1}
\end{align}
for some $j_0,\dots,j_{k+2}\in\mathbb N_0$ where 
\[
0\le k \le 2\ell, \ \ P_{j_n}\in \mathcal Y_{j_n},  \ \ n\in\{0,1,\dots, k+2\}, \ \ \ j_0+j_1+\dots j_{k+2} = 2\ell+1.
\]
Similarly, 
we can express $\bD^{2\ell}\mathcal N_2[\theta]$ as a finite linear combination of the expressions of the form
\begin{align}
f_2^{(k)}(\sqrt{\bG} \bp\theta) Q_{j_0}(\tH) P_{j_1}(\sqrt{\bG} \bp\theta) P_{j_2}(\sqrt{\bG} \bp\theta) \dots P_{j_{k+2}}(\sqrt{\bG} \bp\theta)
\end{align}
for some $j_0,\dots,j_{k+1}\in\mathbb N_0$ where 
\[
0\le k \le 2\ell, \ \ P_{j_n}\in \mathcal Y_{j_n},  \ \ n\in\{1,\dots, k+2\}, \ \ Q_{j_0}\in \mathcal X_{j_0},  \ \ \ j_0+j_1+\dots j_{k+2} = 2\ell.
\]
Similarly, 
we can express $\bD^{2\ell}\mathcal N_3[\theta]$ as a finite linear combination of the expressions of the form
\begin{align}
f_3^{(k)}(\sqrt{\bG} \bp\theta) Q_{j_0}(z^{\frac13}\tG^2) P_{j_1}(\tG \frac{\theta}{z^{\frac13}}) P_{j_2}(\tG \frac{\theta}{z^{\frac13}}) \dots P_{j_{k+2}}(\tG \frac{\theta}{z^{\frac13}})
\end{align}
for some $j_0,\dots,j_{k+2}\in\mathbb N_0$ where 
\[
0\le k \le 2\ell, \ \ P_{j_n}\in \mathcal Y_{j_n},  \ \ n\in\{1,\dots, k+2\}, \ \ Q_{j_0}\in \mathcal X_{j_0},  \ \ \ j_0+j_1+\dots j_{k+2} = 2\ell.
\]
Finally, the expression $\bD^{2\ell}\mathcal N_4[\theta]$ can be written as a finite linear combination of the expressions of the form
\begin{align}
&\pa_x^{k_1}\pa_y^{k_2}f_4(\tG\frac{\theta}{z^{\frac13}}, \sqrt{\bG} \bp\theta) P_{j_0}(\sqrt{\bG}) P_{j_1}(\tG\frac{\theta}{z^{\frac13}}) P_{j_2}(\tG\frac{\theta}{z^{\frac13}}) \dots P_{j_{k_1+1}}(\tG\frac{\theta}{z^{\frac13}}) \times \notag\\
& \qquad \times P_{i_1}(\sqrt{\bG} \bp\theta) P_{i_2}(\sqrt{\bG} \bp\theta) \dots P_{i_{k_2}}(\sqrt{\bG} \bp\theta), \label{E:CLN4}
\end{align}
for some $j_0,\dots,j_{k_1+1}, i_1,\dots, i_{k_2}\in\mathbb N\cup\{0\}$ where 
\[
0\le k_1+k_2 \le 2\ell, \ \ P_{j_n}\in \mathcal Y_{j_n},  \ \ n\in\{0,1,\dots, k_1+1\}, \ \ P_{i_n}\in \mathcal Y_{i_n},  \ \ n\in\{1,\dots, k_2\},
\]
and $j_0+\dots+j_{k_1+1}+ i_1+\dots+i_{k_2}=2\ell+1$. Since,  $f_4(0,0)=0$, to the leading order in $(x,y)$ it is linear near $(x,y)=(0,0)$, and so the expression~\eqref{E:CLN4} is at least quadratic. 

We next observe that for any $j\in\mathbb N$, $j\leq 2\m+1$, and any $P\in\mathcal Y_j$, from Lemmas~\ref{L:PGBOUND}--\ref{L:XYPRODCHAIN}, we may estimate
\begin{align}\label{E:PEXP1}
\Big|P(\sqrt{\bG} \bp\theta)\Big|+\Big|P(\tG \frac{\theta}{z^{\frac13}})\Big| \leq C \sum_{k=0}^{j} \sum_{Q\in \mathcal X_{k+1}, R\in \mathcal Y_{j-k}}  \big|R (\sqrt{\bG})\big| \big|Q \theta\big| \leq C(\m,Z_0) \sum_{k=0}^{j} \sum_{Q\in \mathcal X_{k+1}}  \big|Q \theta\big|.
\end{align}
It follows in particular, using Lemma~\ref{L:EQUIVALENCE}, that for any $P\in \mathcal Y_{j}$,
\begin{align}\label{E:PBOUNDS}
\|P(\sqrt{\bG} \bp\theta)\|_{L^2(0,Z_0)} + \|P(\tG \frac{\theta}{z^{\frac13}})\|_{L^2(0,Z_0)} \le C \|\begin{pmatrix} \theta \\ 0\end{pmatrix}\|_{\mathcal H^{j}_{Z_0}}.
\end{align}

We now consider the quantity $\bD^{2m}\mathcal N[\theta]$. Returning  to~\eqref{E:CLN1}--\eqref{E:CLN4}, we observe that in each summand there is at most one term of 
the general form $P_j(\sqrt G\bp\theta)$ or $P_j(\tG  \frac{\theta}{z^{\frac13}})$ with order $j\ge \m+1$ and note that the maximum
order  of any such  operator  is $2\m+1$. Therefore, when estimating the $L^2(0,Z_0)$ norm of~\eqref{E:CLN1}--\eqref{E:CLN4}, our goal is to 
bound the top order derivative in $L^2$-norm and the lower order derivatives in $L^\infty$-norm. To make this explicit, we consider a typical term in~\eqref{E:CLN1}, and suppose without loss of generality that $ P_{j_{k+2}}\in\mathcal{Y}_j$ for $\m+1\leq j\leq 2\m+1$, all other $P_{j_{\ell}}$ have order $\leq \m$.  Then by~\eqref{E:PBOUNDS},
\beq\label{E:PJK+2}
\big\|P_{j_{k+2}}(\sqrt{\bG} \bp\theta)\big\|_{L^2(0,Z_0)}\leq C \|\begin{pmatrix} \theta \\ \phi \end{pmatrix}\|_{\mathcal H^{2\m+1}_{Z_0}}.
\eeq
From~\eqref{E:PEXP1} and the Hardy-Sobolev embeddings~\eqref{E:XKHARDY}, we find for $P\in\mathcal{Y}_j$, $j\leq \m$,
\begin{align}
\|P(\sqrt{\bG}\bp\theta)\|_{L^\infty(0,Z_0)} +\|P(\tG \frac{\theta}{z^{\frac13}})\|_{L^\infty(0,Z_0)} \le C \sum_{k=0}^j\sum_{Q\in\mathcal{X}_{k+1}}\|Q\theta\|_{L^\infty(0,Z_0)}\leq C\big\|\begin{pmatrix}\theta\\0\end{pmatrix}\big\|_{\mathcal H^{j+2}_{Z_0}}.\label{E:PLOWBDS}
\end{align}
 Thus, combining~\eqref{E:FKDERIVBDS},~\eqref{E:PJK+2}, and~\eqref{E:PLOWBDS}, we have
\beqa
{}&\Big\|f_1^{(k)}(\sqrt{\bG} \bp\theta) P_{j_0}(\sqrt{\bG}) P_{j_1}(\sqrt{\bG} \bp\theta) P_{j_2}(\sqrt{\bG} \bp\theta) \dots P_{j_{k+2}}(\sqrt{\bG} \bp\theta) \Big\|_{L^2(0,Z_0)}\\
&\leq C\big\|P_{j_{k+2}}(\sqrt{\bG} \bp\theta)\big\|_{L^2(0,Z_0)}\prod_{i=1}^{k+1}\big\|P_{j_{i}}(\sqrt{\bG} \bp\theta)\big\|_{L^\infty(0,Z_0)}\\
&\leq C \Big\|\begin{pmatrix} \theta \\ \phi \end{pmatrix}\Big\|_{\mathcal H^{2\m+1}_{Z_0}} \Big\|\begin{pmatrix} \theta \\ \phi \end{pmatrix}\Big\|_{\mathcal H^{\m+2}_{Z_0}}^{k+1}.
\eeqa
Arguing similarly for the other terms~\eqref{E:CLN1}--\eqref{E:CLN4} arising in~\eqref{E:NDECOMP}, we
 conclude that 
\begin{align}
\|\bD^{2m} \big(\mathcal N[\theta]\big)\|_{L^2(0,Z_0)} \le C \big\|\begin{pmatrix}\theta\\0\end{pmatrix}\big\|_{\mathcal H^{2\m+1}_{Z_0}} p( \big\|\begin{pmatrix}\theta\\0\end{pmatrix}\big\|_{\mathcal H^{\m+2}_{Z_0}}), \label{E:MATHCALNBOUND}
\end{align}
where $p(x)$ is a polynomial such that $p(0)=0$. The claim now follows easily from~\eqref{E:MATHCALNBOUND} and~\eqref{E:BOLDNDEF}.
 Since  $\|\Phi\|_{\H^{\m+2}_{Z_0}}\le C \|\Phi\|_{\H^{2\m}_{Z_0}}$ as $\m\geq 2$ and  $\|\Phi\|_{\H^{\m+2}_{Z_0}}\le C \tE_{\leq 2(\m+1)}^{\frac12}$  from~\eqref{E:HMZTE},  estimate~\eqref{E:CRUDENONLINEAR0} follows directly.

Finally, the estimate~\eqref{E:NPHIDIFF} follows from essentially the same argument, now tracking the difference of $\theta_1-\theta_2$ throughout.
\end{proof}


\subsection{Proof of the low-order a priori bound}


{\em Proof of Proposition~\ref{P:EE12}}.
 From Lemma~\ref{L:NBOUND}, for any $s\in[\sin,S_T)$ we have
\begin{align}
\Big\|&\,\int_{\sin}^{s} e^{-\sg (s-\sigma)}\bfN[\Phi] \diff \sigma\Big\|_{\HmZm} 
\le C \int_{\sin}^{s} e^{-\sg (s-\sigma)} \|\Phi\|_{\HmZm}\tE_{\leq 2(\m+1)}^{\frac12} \diff \sigma \notag\\
& \le C \sup_{\sigma\in[\sin,s]} (e^{\Om \sigma  }\tE_{\leq 2(\m+1)}^{\frac12})\sup_{\sigma\in[\sin,s]} (e^{\nu\sg \sigma  }\|\Phi\|_{\HmZm})\int_{\sin}^{s} e^{-\sg (s-\sigma)}  e^{-(\nu\sg +\Om) \sigma  }\diff \sigma \notag\\
& \le C  \sup_{\sigma\in[\sin,s]} (e^{\Om \sigma  }\tE_{\leq 2(\m+1)}^{\frac12})\sup_{\sigma\in[\sin,s]} (e^{\nu\sg \sigma  }\|\Phi\|_{\HmZm})e^{-(\nu\sg +\Om) s } ,
\label{E:NBOUNDCRUDE00}
\end{align}
 where we have used  the assumption $\nu\sg +\Om<\sg $. A similar argument yields
\beq
\Big\|\int_{s}^{S_T} e^{s-\sigma} \bfN[\Phi](\sigma)\diff \sigma\Big\|_{\HmZm}\leq C  \sup_{\sigma\in[\sin,s]} (e^{\Om \sigma  }\tE_{\leq 2(\m+1)}^{\frac12})\sup_{\sigma\in[\sin,s]} (e^{\nu\sg \sigma  }\|\Phi\|_{\HmZm})e^{-(\nu\sg +\Om) s } .
\eeq
We feed these bounds into~\eqref{E:MINUSBOUND} and~\eqref{E:STABLEMFD}  respectively to obtain 
\begin{align}
e^{\nu\sg s }\|\Phis(s)\|_{\HmZm} &\le e^{-(1-\nu)\sg s }\|\PhiT\|_{\HmZm} + C  \sup_{\sigma\in[\sin,s]} (e^{\Om \sigma  }\tE_{\leq 2(\m+1)}^{\frac12})\sup_{\sigma\in[\sin,s]} (e^{\nu\sg \sigma  }\|\Phi\|_{\HmZm})e^{-\Om s } ,\notag\\
e^{\nu\sg s }\|\Phiu(s)\|_{\HmZm} &  
 \le  e^{(1+\nu\sg ) s }\bigg\|\bfP\bigg(\PhiT +\int_{\sin}^{S_T} e^{\sin-\sigma} \bfN[\Phi](\sigma)\diff \sigma\bigg)\bigg\|_{\HmZm} \notag\\
 &\ \ \ + C  \sup_{\sigma\in[\sin,s]} (e^{\Om \sigma  }\tE_{\leq 2(\m+1)}^{\frac12})\sup_{\sigma\in[\sin,s]} (e^{\nu\sg \sigma  }\|\Phi\|_{\HmZm})e^{-\Om s } .
\end{align}
By adding the two bounds above and using the a priori bound~\eqref{E:APRIORIMAIN} 
we obtain
\begin{align}
 e^{\nu\sg s  }\|\Phi\|_{\HmZm} & \le C  \sup_{\sigma\in[\sin,s]} (e^{\Om \sigma  }\tE_{\leq 2(\m+1)}^{\frac12})\sup_{\sigma\in[\sin,s]} (e^{\nu\sg \sigma  }\|\Phi\|_{\HmZm})e^{-\Om s } + e^{-(1-\nu)\sg s }\|\PhiT\|_{\HmZm} \notag\\
& \ \ \ \ +e^{(1+\nu\sg ) s }\bigg\|\bfP\bigg(\PhiT +\int_{\sin}^{S_T} e^{\sin-\sigma} \bfN[\Phi](\sigma)\diff \sigma\bigg)\bigg\|_{\HmZm}.
\end{align}
This concludes the proof of Proposition~\ref{P:EE12}.


\section{High-order energy bounds} \label{S:ENERGYBOUNDS}


As explained in the introduction, the second key ingredient to closing the nonlinear estimates 
is the weighted $L^2$-based energy bounds that fully take the quasilinear 
nature of the problem into account. In contrast to the essentially semi-linear 
approach of Section~\ref{S:DUHAMEL}, the nonlinear bounds of this section experience no derivative loss. The proofs of these propositions will rely on exploiting the precise quasilinear structure of the Euler-Poisson system. We therefore reformulate the perturbation equations in a suitable quasilinearised form in the following lemma.

\begin{lemma}[Perturbation problem - quasilinearised version]\label{L:QUASI0}
Let $(\zeta,\mu)$ be formally a classical solution of~\eqref{E:NLZETA}--\eqref{E:NLPSI} and let $(\th,\phi)$ be given via~\eqref{E:THDEF}. Then  the pair $(\theta,\phi)$ solves the following first order system
\begin{align}
\theta_s & = - \Lambda \theta + \theta+\phi,  \label{E:NLTHETA0}
 \\
\phi_s & = -\Lambda \phi +  \frac{\bzeta_z^2}{\zeta_z^2} K \theta +\frac{\CLP z-\tilde M}{\zeta^2}-\frac{\tilde g_z}{\tilde g}\frac1{\zeta_z} + \mathfrak R[\theta], \label{E:NLPHI0}
\end{align}
where we recall the operator $K$ and potential $\mathcal V_1$ from~\eqref{E:KDEF} and~\eqref{E:KIDENTITY}, and set
\begin{align} 
\mathfrak R[\theta] &=  \frac{\bzeta_z^2}{\zeta_z^2} \mathcal V_1 \th  - \frac{2\th}{\zeta\bzeta}   + \frac{\CLP z\theta(2\bzeta+\th)}{\zeta^2\bzeta^2}   
 -  \frac{\bzeta_{zz}}{\bzeta_z^2 \zeta_z^2} (\pa_z \theta)^2. \label{E:MATHFRAKRDEF}
\end{align}
\end{lemma}

\begin{proof}
Plugging the ansatz $\zeta = \bzeta+\th$, $\mu=\widehat\mu+\phi$ in~\eqref{E:NLZETA}--\eqref{E:NLPSI}, and using the steady state equation~\eqref{E:LPLAG} we obtain
\begin{align}
\theta_s & = - \Lambda \theta + \theta+\phi,  \label{E:NLTHETAAUX} \\
\phi_s & =
 -\Lambda \phi  + \pa_z \Big( \frac{\pa_z\theta}{\zeta_z\bzeta_z}  \Big) -\frac{2\theta}{\zeta\bzeta} + \CLP z \frac{\theta(2\bzeta + \theta)}{\zeta^2\bzeta^2} +\frac{\CLP z-\tilde M}{\zeta^2}-\frac{\tilde g_z}{\tilde g}\frac1{\zeta_z}.\label{E:NLPSI0}
\end{align}
We next quasi-linearise the second term on the right-hand side of~\eqref{E:NLPSI0}. 
\begin{align*}
\pa_z \Big( \frac{\pa_z\theta}{\zeta_z\bzeta_z}  \Big) &= \pa_z \Big( \frac{\bzeta_z}{\zeta_z}\frac{\pa_z\theta}{\bzeta_z^2}  \Big) = \frac{\bzeta_z}{\zeta_z} \pa_z \Big( \frac{\pa_z\theta}{\bzeta_z^2}  \Big) + \frac{\bzeta_{zz}}{\zeta_z} \frac{\pa_z\theta}{\bzeta_z^2}  -\frac{\zeta_{zz}}{\zeta_z^2} \frac{\pa_z\theta}{\bzeta_z} \\
& = \frac{\bzeta_z}{\zeta_z} \pa_z \Big( \frac{\pa_z\theta}{\bzeta_z^2}  \Big) + \frac{\bzeta_{zz}}{\zeta_z^2\bzeta_z^2} (\pa_z\theta)^2    -\frac{1}{\bzeta_z\zeta_z^2} \pa_z\theta\pa_z \Big( \frac{\bzeta_z^2}{\bzeta_z^2} \pa_z\theta\Big)  \\
&= \frac{\bzeta_z^2}{\zeta_z^2} \pa_z \Big( \frac{\pa_z\theta}{\bzeta_z^2}  \Big) - \frac{\bzeta_{zz}}{\bzeta_z^2 \zeta_z^2} (\pa_z \theta)^2 \\
&=  \frac{\bzeta_z^2}{\zeta_z^2} \big( K \theta +\big( \frac29 z^{-\frac23} \bG + \frac43 z^{-\frac13} \bp^2 \bzeta \bG^\frac32\big) \theta\big) -  \frac{\bzeta_{zz}}{\bzeta_z^2 \zeta_z^2} (\pa_z \theta)^2,
\end{align*}
where we recall the operator $K$ from~\eqref{E:KDEF}--\eqref{E:KIDENTITY}. Together with~\eqref{E:NLPSI0}, this completes the proof.
\end{proof}

The goal of this section is the following a priori bounds on the top order energy functional $\tE_{\leq 2(\m+1)}$, which we recall is defined in~\eqref{E:TILDEEDEF} using the constants $Z_0$, $\kappa$, $c$, $\al$ satisfying~\ref{item:a1}--\ref{item:a3}. In order to state our assumptions, we recall the definition of the Hilbert space $\mathfrak{H}$ from~\eqref{E:FRAKHDEF}. We also recall that the constants $\eps_0$ and $C_*(Z_0)$ are such that $\eps_0 C_*(Z_0)\leq\frac14$ so that, whenever~\eqref{E:APRIORIMAINPT} holds, we have $\Pb[\Phi]\leq\frac14$.

\begin{proposition}\label{P:EE2}
Let $\Phi=(\th,\phi)^\top\in L^\infty((\sin,S);\mathfrak{H})$ be the unique solution to~\eqref{E:NLTHETA0}--\eqref{E:NLPHI0} with initial data~\eqref{E:PROFILE1} up to any time $S>\sin$ such that the a priori bounds~\eqref{E:APRIORIMAIN}--\eqref{E:APRIORIMAINPT} hold on $s\in[\sin,S]$.
Then there exist constants $c_3=\min\{\frac1{24},\frac{1}{4}(\frac23-c)\}>0$, $c_4=\frac{\al-1}{4}>0$ such that, for $\kappa_0,Z_0$  sufficiently large  and $\eps_0$ sufficiently small, for any $s\in [\sin,S]$, the following bound holds:
\begin{align}
\frac12\pa_s\tE_{2j}(s)  + (c_3 j +c_4)\tE_{2j} (s) \le &\,C(Z_0)\|\Phi\|_{\HmZm}^2\chi_{j\leq \m}+ \sqrt{\kappa}(r_*e^s)^{-\frac12}\tE_{\leq 2j}^{\frac12}\label{E:APRIORIHIGH}\\
&+\Big(C(r_*e^s)^{-\frac23}+C Z_0^{-\frac12}\Big)\tE_{\leq 2j}+C(Z_0)\tE_{\leq 2j}^{\frac32}+C(Z_0)\|\Phi\|_{\mathcal{H}^{2j-1}_{Z_0}}\tE_{2j}^{\frac12}.\notag
\end{align}
\end{proposition}

\begin{proposition}\label{P:HMZDIFF}
Suppose $\Phi_1=(\th_1,\phi_1)^\top$ and $\Phi_2=(\th_2,\phi_2)^\top$ are both solutions to~\eqref{E:NLTHETA0}--\eqref{E:NLPHI0} with initial data~\eqref{E:PROFILE1} for values $T_1$, $T_2$, as in Proposition~\ref{P:EE2}.
Assume the a priori bounds~\eqref{E:APRIORIMAIN}--\eqref{E:APRIORIMAINPT} hold for each solution on $[\overline\sin,S]$, where $\overline\sin=\max_{i=1,2}\sin(T_i)$. Then, for $s_1, s_2 \in[\overline\sin,S]$, 
\beq
\| (\Phi_1-\Phi_2)(s_2) \|_{\HmZm}  \leq  e^{C(Z_0) (s_2-s_1) } \|(\Phi_1-\Phi_2)(s_1)\|_{\HmZm} . 
\eeq 
\end{proposition}


\subsection{High-order energy identities}


Our first step towards proving the main energy estimate, Proposition~\ref{P:EE2}, is to establish a general weighted energy-type identity for solutions $(\th,\phi)^\top$ of~\eqref{E:NLTHETA0}--\eqref{E:NLPHI0} satisfying the assumptions of Proposition~\ref{P:EE2}. 

We recall the high order regularity index $M=\m+1$, as in Section~\ref{S:HIGH}. To simplify notation, we will retain the notation $M$ throughout this section and set, for each $j\in\N$, $j\leq M$, 
\beq
\theta_{2j}:=\bD^{2j}\theta, \ \ \phij : = \bD^{2j}\phi.
\eeq

\begin{lemma}\label{L:HIGHENERGY0}
Let  $(\th,\phi)^\top$ satisfy the assumptions of Proposition~\ref{P:EE2}, $j\in\N$, $j\leq M$. Then the following high-order energy identity holds:
\begin{align}
& \frac12\frac{d}{ds}\int \chi_{2j} \Big( \frac{(\bd\theta_{2j})^2}{(\bp\zeta)^2} +\phij^2 \Big) \notag \\
&= -\frac{4j-3}{6}\int \chi_{2j} \Big(\frac{(\bd\theta_{2j})^2}{(\bp\zeta)^2} + \phij^2\Big)+ 2j \int \frac{\chi_{2j}\bp \bG}{\sqrt{\bG}} \frac{ \bd \theta_{2j} }{\bp \zeta }\phij 
-\int \chi_{2j} \frac{(\bd\theta_{2j})^2}{(\bp\zeta)^2} \frac{\pa_s\bp\theta}{\bp\zeta}\notag \\
&\ \ \  \
  + \frac12\int \Lambda \chi_{2j}  \Big( \frac{(\bd\theta_{2j})^2}{(\bp\zeta)^2} +\phij^2 \Big)
   - \int  \frac{\bp \chi_{2j}}{(\bp\zeta)^2} \bd \theta_{2j} \phij  - \int \chi_{2j} \frac{(\bd\theta_{2j})^2}{(\bp\zeta)^2}\frac{z\pa_{zz}\zeta}{\pa_z\zeta}  \notag\\
   &\ \ \ \ - \int \chi_{2j} \bp(\frac{\bzeta_z^2}{\zeta_z^2}) \bG  \bd \theta_{2j} \phi_{2j} -2j \int \chi_{2j} \frac{ \pa_z\theta }{\zeta_z} \frac{\bp \bG }{\sqrt{\bG }} \frac{ \bd \theta_{2j} }{\bp \zeta }\phi_{2j}\notag\\
& \ \ \ \ + \int \chi_{2j}\bD^{2j} \Big(\frac{\CLP z-\tilde M}{\zeta^2}-\frac{\tilde g_z}{\tilde g}\frac1{\zeta_z}  \Big)    \phij  + \int \chi_{2j} \Big(\frac{\bzeta_z^2}{\zeta_z^2} \mathcal R_{2j}\theta + \mathcal N_{2j}[\theta] + \bD^{2j}\RR[\th]\Big) \phij,   \label{E:MAGENTA}
\end{align}
where we recall the weights $\chi_{2j}$ from~\eqref{E:CHIJDEF}, the definitions~\eqref{E:REMAINDER} and~\eqref{E:MATHFRAKRDEF},
and we set
\begin{align}\label{E:NTWOMDEF}
\mathcal N_{2j}[\th]:= \bD^{2j} ( \frac{\bzeta_z^2}{\zeta_z^2} K \theta) -  \frac{\bzeta_z^2}{\zeta_z^2}\bD^{2j} K \theta. 
\end{align}
\end{lemma}


\begin{proof}
Applying $\bD^{2j}$ to~\eqref{E:NLTHETA0}--\eqref{E:NLPHI0}, assuming $\theta\in\DZodd$,  
we obtain the system
\begin{align}
\pa_s \bd \theta_{2j}  =&\, \bd \phi_{2j} - \Lambda\bd \theta_{2j} - \frac{2j-2}3 \bd \theta_{2j}, \label{E:SIMPLETHETAM00}\\
\pa_s\phi_{2j}  =&\, -\Lambda \phi_{2j} -\frac{2j}3\phi_{2j}+ \frac{\bzeta_z^2}{\zeta_z^2}  \left( K\theta_{2j} + 2j \bp \bG \bd \theta_{2j} + \mathcal R_{2j} \theta\right)+\mathcal N_{2j}[\theta] \notag \\
& +\bD^{2j} \Big(\frac{\CLP z-\tilde M}{\zeta^2}-\frac{\tilde g_z}{\tilde g}\frac1{\zeta_z}  \Big) + \bD^{2j}\RR[\th], \label{E:SIMPLEPSIM00}
\end{align}
where we have used~\eqref{E:ACC1}, Lemma~\ref{L:KEYCOMM},~\eqref{E:MATHFRAKRDEF}, and~\eqref{E:NTWOMDEF}. 
We recall  $\chi_{2j}( z) \ge 0$, defined in~\eqref{E:CHIJDEF}, is smooth. We multiply \eqref{E:SIMPLETHETAM00}--\eqref{E:SIMPLEPSIM00} by $\chi_{2j} \frac{1}{(\bp \zeta)^2}\bd\theta_{2j} $ and $\chi_{2j} \phi_{2j}  $ respectively and integrate to obtain 
\begin{align}
&\int \chi_{2j} \Big( \frac{1}{(\bp \zeta)^2}\bd\theta_{2j} \pa_s \bd \theta_{2j} + \phi_{2j} \pa_s\phi_{2j}\Big) \label{comp10}
= \\
&- \int \chi_{2j} \frac{1}{(\bp \zeta)^2} \Lambda\bd \theta_{2j} \bd \theta_{2j}  -  \tfrac{2j-2}3 \int \chi_{2j} \frac{1}{(\bp \zeta)^2} (\bd \theta_{2j})^2 - \int  \chi_{2j} \Lambda\phi_{2j} \phi_{2j}  -\tfrac{2j}{3} \int \chi_{2j} \phi_{2j}^2 \label{comp20} \\
&+ \int \chi_{2j}  \bd \phi_{2j}  \frac{1}{(\bp \zeta)^2}\bd\theta_{2j} + \int \chi_{2j} \phi_{2j}  \frac{\bzeta_z^2}{\zeta_z^2}  K\theta_{2j} +2j \int \chi_{2j}  \frac{\bzeta_z^2}{\zeta_z^2} \bp \bG \bd \theta_{2j} \phi_{2j} \label{comp30} \\
& + \int \chi_{2j} \frac{\bzeta_z^2}{\zeta_z^2} \mathcal R_{2j}\theta  \phi_{2j}  + \int \chi_{2j} \mathcal N_{2j}[\theta] \phi_{2j} \notag\\ 
& + \int \chi_{2j}\bD^{2j} \Big(\frac{\CLP z-\tilde M}{\zeta^2}-\frac{\tilde g_z}{\tilde g}\frac1{\zeta_z}  \Big)    \phi_{2j} \diff z  +\int \chi_{2j} \bD^{2j}\RR[\th] \phi_{2j} \diff z. \label{E:ENERGYLAST}
\end{align}
We first rewrite the first three lines: 
\begin{align}
\eqref{comp10}= \frac12\frac{d}{ds}\int \chi_{2j} \bigg( \frac{(\bd\theta_{2j})^2}{(\bp\zeta)^2} +\phi_{2j}^2 \bigg)
+ \int \chi_{2j} \frac{(\bd\theta_{2j})^2}{(\bp\zeta)^2} \frac{\pa_s\bp\theta }{\bp\zeta}.
\end{align}
Next, integrating by parts,
\begin{align}
\eqref{comp20}&= -\frac{4j-7}{6}\int \chi_{2j} \frac{(\bd\theta_{2j})^2}{(\bp\zeta)^2}  - \frac{4j-3}{6}\int \chi_{2j} \phi_{2j}^2 + \frac12\int \Lambda \chi_{2j}  \bigg([ \frac{(\bd\theta_{2j})^2}{(\bp\zeta)^2} +\phi_{2j}^2 \bigg)  \notag\\
&\quad  - \int \chi_{2j} \frac{(\bd\theta_{2j})^2}{(\bp\zeta)^2}\frac{\Lambda\bp\zeta}{\bp\zeta}. \label{E:COMP210}
\end{align}
We note that 
\begin{align}
\frac{\Lambda\bp\zeta}{\bp\zeta}
= \frac{z^{\frac13}\pa_z(z^{\frac23}\pa_z\zeta)}{\pa_z\zeta} =\frac23 + \frac{z\pa_{zz}\zeta}{\pa_z\zeta}. \label{E:COERCIVE10}
\end{align}
We may therefore rewrite~\eqref{E:COMP210} in the form
\begin{align}
\eqref{comp20}&= -\frac{4j-3}{6}\int \chi_{2j} \bigg(\frac{(\bd\theta_{2j})^2}{(\bp\zeta)^2}+\chi_{2j} \phi_{2j}^2\bigg) + \frac12\int \Lambda \chi_{2j}  \bigg(\frac{(\bd\theta_{2j})^2}{(\bp\zeta)^2} +\phi_{2j}^2 \bigg)  \notag\\
&\quad  - \int \chi_{2j} \frac{(\bd\theta_{2j})^2}{(\bp\zeta)^2} \frac{z\pa_{zz}\zeta}{\pa_z\zeta}. \label{E:COMP22}
\end{align}
The first two terms of \eqref{comp30} combine to give
\begin{align}
\eqref{comp30}_1&= \int \chi_{2j}  \bd \phi_{2j}  \frac{1}{(\bp \zeta)^2}\bd\theta_{2j} + \int \chi_{2j} \phi_{2j}  \frac{\bzeta_z^2}{\zeta_z^2}  K\theta_{2j} \notag\\
&= - \int \chi_{2j} \bp(\frac{\bzeta_z^2}{\zeta_z^2}) \bG  \bd \theta_{2j} \phi_{2j}  - \int \bp \chi_{2j} \frac{1}{(\bp\zeta)^2} \bd \theta_{2j} \phi_{2j}, 
\end{align}
and we rewrite the third term of \eqref{comp30} 
\begin{align}
2j \int \chi_{2j}  \frac{\bzeta_z^2}{\zeta_z^2} \bp \bG  \bd \theta_{2j} \phi_{2j} = 2j \int \chi_{2j}  \frac{\bp \bG }{\sqrt{\bG }} \frac{ \bd \theta_{2j} }{\bp \zeta }\phi_{2j} -2j \int \chi_{2j} \frac{ \pa_z\theta }{\zeta_z} \frac{\bp \bG }{\sqrt{\bG }} \frac{ \bd \theta_{2j} }{\bp \zeta }\phi_{2j}.
\end{align}
The identity~\eqref{E:MAGENTA} now follows easily from the above.
\end{proof}


We now specialise the energy identity from Lemma~\ref{L:HIGHENERGY0} to the cases $j=0$ and $1\leq j\leq M$, which gives us the starting point for the high-order energy bounds. We first prove the energy estimate for the zero order energy, for which the weight function has a slightly different form (see~\eqref{E:CHIJDEF}).

\begin{proposition}[Basic energy inequality for $\tE_{0}$]\label{P:ZEROENERGYGLOBAL}
Let $(\th,\phi)^\top$ satisfy the assumptions of Proposition~\ref{P:EE2}. Then,  for  $Z_0>0$ sufficiently large, $\eps_0$ sufficiently small, the modified energy $\tE_{0}(s)$ defined in~\eqref{E:TILDEEDEF} satisfies, for all $s\in(\sin,S)$,
 the following a priori energy bound:
\begin{align}
&\frac12\pa_s\tE_{0}(s)  + c_4\tE_{0} (s) \le C \kappa\|\Phi\|_{\H^{0}_{2\m,Z_0}}^2
  + \int \chi_{0}  \Big(\frac{\CLP z-\tilde M}{\zeta^2}-\frac{\tilde g_z}{\tilde g}\frac1{\zeta_z}  \Big)    \phi \diff z .\label{Energy-zero}
\end{align}
\end{proposition}

\begin{proof}
We take $j=0$ in Lemma~\ref{L:HIGHENERGY0} and, for ease of notation, we set $X = \frac{1}{(\bp\zeta)} \bd \theta$, $Y=\phi$. 
We first calculate
\[
\pa_z\chi_{0}=\kappa\frac{-\al}{1+z}(1+z)^{-\al} = - \frac{\al}{1+z}\chi_0
\]
to obtain
\begin{align}
& \frac12\frac{d}{ds}\int \chi_{0} \big( X^2 + Y^2 \big) \notag \\
&=  \frac12\int \chi_{0} \big( X^2 + Y^2 \big)  - \int \chi_{0}   X^2 \Big(\frac{z\pa_{zz}\zeta}{\pa_z\zeta} + \frac{\pa_s\bp\theta }{\bp\zeta}\Big) 
 + \int \frac{-\alpha}{2}\frac{z}{1+z}\chi_{0}   \Big( X^2 + Y^2 - 2  \frac{\sqrt{\bG }}{z^{\frac13}}\Big(1-\frac{\bp\theta}{\bp\zeta}\Big) XY \Big)  \notag \\
 &\ \ \ \ - \int \chi_{0} \bp(\frac{\bzeta_z^2}{\zeta_z^2}) \bG  (\bp\zeta) XY 
+ \int \chi_{0} \Big(\frac{\CLP z-\tilde M}{\zeta^2}-\frac{\tilde g_z}{\tilde g}\frac1{\zeta_z}  \Big)    Y \diff z.\label{E:ZEROENERGYFIRST}
\end{align}
Given a $0<d\ll1$ sufficiently small, there exists $Z_0$ sufficiently large, from Lemma~\ref{L:GBAR} and~\eqref{E:APRIORIMAINPT}, such that
\begin{align}
(1-d) (X^2+Y^2)\le X^2+Y^2-2 \frac{\sqrt{\bG }}{z^{\frac13}}\Big(1-\frac{\bp\theta}{\bp\zeta}\Big)  XY  \le (1+d)(X^2+Y^2), \ \ z\ge Z_0, 
\end{align}
and for some $e_3 >0$,  independent of $Z_0$,
\begin{align}
 \frac{z}{z+1} \Big| X^2+Y^2-2 \frac{\sqrt{\bG }}{z^{\frac13}} \Big(1-\frac{\bp\theta}{\bp\zeta}\Big) XY \Big|  \le e_3 (X^2+Y^2), \ \ z\in(0, Z_0). 
 \end{align}
Thus, we make the estimate
\beqa\label{E:ZEROENERGY2}
&\int \frac{\alpha}{2}\frac{z}{1+z}\chi_{0}   \bigg( X^2 + Y^2 - 2  \frac{\sqrt{\bG }}{z^{\frac13}}\Big(1-\frac{\bp\theta}{\bp\zeta}\Big) XY \bigg)\\
&\leq \int_0^{Z_0}\frac{\alpha}{2}\chi_{0}e_3(X^2+Y^2) + \int_{Z_0}^\infty\chi_{0} \Big(-\frac{\alpha}{2}(1-d)\Big)(X^2+Y^2).
\eeqa
Next, we expand $\frac{z\pa_{zz}\zeta}{\pa_z\zeta}=\frac{z\pa_{zz}\bzeta}{\pa_z\bzeta}-  \frac{z\pa_{zz}\bzeta}{\pa_z\bzeta} \frac{\pa_z\theta}{\pa_z\zeta} +\frac{z\pa_z^2\theta}{\pa_z\zeta}$ in order to estimate
\beqa\label{E:ZEROENERGY3}
&- \int \chi_{0}   X^2 \Big(\frac{z\pa_{zz}\zeta}{\pa_z\zeta}  + \frac{\pa_s\bp\theta }{\bp\zeta}\Big)
 \leq \int \chi_{0}   X^2 \Big(\frac23\mathbf{1}_{z\leq Z_0}+\frac{C}{Z_0}\mathbf{1}_{z\geq Z_0}+C_1C_*\sqrt{\eps_0} \Big),
\eeqa
where we have used the lower bound $\frac{z\pa_{zz}\bzeta}{\pa_z\bzeta} \ge -\frac23$, and $\frac{z\pa_{zz}\bzeta}{\pa_z\bzeta} = O(\frac1z) $ for $z\gg1$, both of which follow from~\eqref{eq:zdzzZeta}, and the pointwise a priori estimate~\eqref{E:APRIORIMAINPT}. We note that the constant $C_1>0$ depends only on the global bounds of the LP solution, $\bzeta$. 

Applying~\eqref{E:PBOUNDS}, we easily check that $|\bp\big(\frac{\bzeta_z^2}{\zeta_z^2}\big)|\leq C\|\Phi\|_{\mathcal{H}^2_{2\m,Z_0}}\leq C
\sqrt{\eps_0}$ for $z\leq Z_0$, while the pointwise bound~\eqref{E:APRIORIMAINPT} gives $|\bp\big(\frac{\bzeta_z^2}{\zeta_z^2}\big)|\leq C_1C_*\sqrt{\eps_0}$, and hence
\beq\label{E:ZEROENERGY4}
\bigg|\int \chi_{0} \bp(\frac{\bzeta_z^2}{\zeta_z^2}) \bG \bd\theta\phi\bigg|\leq C(Z_0)\sqrt{\eps_0}\int \chi_0 (X^2+Y^2).
\eeq

Combining~\eqref{E:ZEROENERGY2}--\eqref{E:ZEROENERGY3} in~\eqref{E:ZEROENERGYFIRST}, we have obtained
\begin{align}
& \frac12\frac{d}{ds}\int \chi_{0} \big( X^2 + Y^2 \big) \notag \\
&\leq \int_0^{Z_0} \chi_{0}    \big(\frac12 +\frac{\al e_3}{2} +\frac23 +C(Z_0)\sqrt{\eps_0}\big) ( X^2 + Y^2)\label{E:ZEROENERGYFIRST1}\\
&
+ \int_{Z_0}^\infty \chi_{0}  \Big(-\frac{\alpha}{2}(1-d)+\frac12 +\frac{C}{Z_0} +C(Z_0)\sqrt{\eps_0}\Big) ( X^2 + Y^2)+ \int \chi_{0}   \Big(\frac{\CLP z-\tilde M}{\zeta^2}-\frac{\tilde g_z}{\tilde g}\frac1{\zeta_z}  \Big)    Y \diff z.\label{E:ZEROENERGYEXTERIOR1}
\end{align}
We may take $d$ sufficiently small, then $Z_0\gg1$ sufficiently large (depending on $\al-1$), and finally $\eps_0$ sufficiently small (depending on $\al,Z_0$) so that
\beq
0< c_4=\frac{\al-1}{4}\leq \frac{\alpha}{2}(1-d)-\frac12 -\frac{C}{Z_0} -C(Z_0)\sqrt{\eps_0},
\eeq
and hence~\eqref{E:ZEROENERGYEXTERIOR1} provides a coercive estimate in the far-field region. To control the interior contribution in~\eqref{E:ZEROENERGYFIRST1}, we bound 
\beqa
\int_0^{Z_0} \chi_{0}    \big(\frac12 +\frac{\al e_3}{2} +\frac23 +C(Z_0)\sqrt{\eps_0}\big) ( X^2 + Y^2)\leq C\kappa\|\Phi\|_{\mathcal{H}^{0}_{2\m,Z_0}}^2.
\eeqa
This concludes the proof.
\end{proof}

\begin{proposition}[Basic energy inequality for $\tE_{2j}$]\label{P:ENERGYGLOBAL}
Let $(\th, \phi)^\top$ be as in Proposition~\ref{P:EE2}. Then there exists $J\in\mathbb N$  with $J\le \m$  such that for $\kappa_0,Z_0$ sufficiently large and $\eps_0$ sufficiently small, the modified energy $\tE_{2j}(s)$ defined in~\eqref{E:TILDEEDEF} satisfies
 the following a priori energy bound: for $1\le j \le M$,
\begin{align}
&\frac12\pa_s\tE_{2j}(s)  + (c_3 j +c_4)\tE_{2j} (s) \le C \kappa\|\Phi\|_{\H^{2J}_{2\m,Z_0}}^2 \bm{1}_{j\le J}
 \notag\\
&\quad  + \int \chi_{2j} \bD^{2j} \Big(\frac{\CLP z-\tilde M}{\zeta^2}-\frac{\tilde g_z}{\tilde g}\frac1{\zeta_z}  \Big)    \phi_{2j} \diff z \notag\\
&\quad  + \int \chi_{2j} \Big(\frac{\bzeta_z^2}{\zeta_z^2} \mathcal R_{2j}\theta + \mathcal N_{2j}[\theta] + \bD^{2j}\RR[\th]\Big) \phij \diff z,  \label{Energy-high}
\end{align}
where $\bm{1}_{j\le J}$ is the characteristic function of the set $\{1,2,\dots,J\}\subset \mathbb N$.
\end{proposition}

\begin{proof}
We use Lemma~\ref{L:HIGHENERGY0} and for ease of notation we set $X = \frac{1}{(\bp\zeta)} \bd \theta_{2j}$, $Y=\phi_{2j}$. 
We first calculate
\[
\pa_z\chi_{2j}=\frac{g_{2j}'}{g_{2j}}\chi_{2j}+\frac{2cj-\al}{1+z}g_{2j}(1+z)^{2cj-\al}.
\]
 obtain
\begin{align}
& \frac12\frac{d}{ds}\int \chi_{2j} \big( X^2 + Y^2 \big) \notag \\
&= -\frac{2j}{3}\int \chi_{2j}  \bigg( X^2 + Y^2-3  \frac{\bp \bG }{\sqrt{\bG }}\big(1-\frac{\bp\theta}{\bp\zeta}\big) XY  \bigg)  + \frac12\int \chi_{2j} \big( X^2 + Y^2 \big)  - \int \chi_{2j}   X^2 \Big(\frac{z\pa_{zz}\zeta}{\pa_z\zeta}  + \frac{\pa_s\bp\theta }{\bp\zeta}\Big)\notag \\
&\ \ \   + \int \Big(\frac12 z\frac{g_{2j}'(z)}{g_{2j}(z)}+\frac{2cj-\alpha}{2}\frac{z}{1+z}\frac{(1+z)^{2cj-\al}}{\kappa+(1+z)^{2cj-\al}}\Big)\chi_{2j}   \bigg(X^2 + Y^2 - 2  \frac{\sqrt{\bG }}{z^{\frac13}}\Big(1-\frac{\bp\theta}{\bp\zeta}\Big) XY \bigg)  \notag \\
 &\ \ \ \ - \int \chi_{2j} \bp(\frac{\bzeta_z^2}{\zeta_z^2}) \bG  (\bp\zeta) XY \notag\\
& \ \ \   
+ \int \chi_{2j} \bD^{2j} \Big(\frac{\CLP z-\tilde M}{\zeta^2}-\frac{\tilde g_z}{\tilde g}\frac1{\zeta_z}  \Big)    Y \diff z
+ \int \chi_{2j} \Big(\frac{\bzeta_z^2}{\zeta_z^2} \mathcal R_{2j}\theta + \mathcal N_{2j}[\theta] + \bD^{2j}\RR[\th]\Big) Y \diff z. \label{E:BASICENERGYFIRST}
\end{align}

We begin by treating the cross-terms. First, we make the estimate
\beqa\label{E:BASICENERGY1}
{}& \int \frac12 z\frac{g_{2j}'(z)}{g_{2j}(z)}\chi_{2j}  ( X^2 + Y^2) +
\int\chi_{2j}\Big(2j\frac{\bp\bG}{\sqrt{\bG}}-z^{\frac23}\frac{g_{2j}'}{g_{2j}}\sqrt{\bG}\Big)\big(1-\frac{\bp\theta}{\bp\zeta}\big)XY\\
&\leq \int\chi_{2j} \big(\a j\mathbf{1}_{z\leq Z_0}+\frac{C}{Z_0}j\mathbf{1}_{z\geq Z_0}\big)(1+CC_*\sqrt{\eps_0})(X^2+Y^2),
\eeqa
where we have used $g_{2j}'\leq0$ to eliminate the first term on the left,~\eqref{E:gmprime} to estimate the second term on the left as in~\eqref{E:BADTERMDIFF} for $z\leq Z_0$, used $|\frac{\bp\bG}{\sqrt{\bG}}|\leq Cz^{-1}$ for $z\geq Z_0$ by Lemma~\ref{L:GBAR}, and employed the a priori bound~\eqref{E:APRIORIMAINPT} to control $\frac{\bp\theta}{\bp\zeta}$.

For a $0<d\ll1$ sufficiently small, and $Z_0$ sufficiently large, from Lemma~\ref{L:GBAR} and~\eqref{E:APRIORIMAINPT} we have 
\begin{align}
(1-d) (X^2+Y^2)\le X^2+Y^2-2 \frac{\sqrt{\bG }}{z^{\frac13}}\Big(1-\frac{\bp\theta}{\bp\zeta}\Big)  XY  \le (1+d)(X^2+Y^2), \ \ z\ge Z_0, 
\end{align}
and for some $e_3 >0$, depending only on the LP solution (recall we assume $C_*(Z_0)\eps_0\leq\frac14$, so that~\eqref{E:APRIORIMAINPT} yields $ \big|\frac{\bp\theta}{\bp\zeta}\big|\leq\frac14$),
\begin{align}
 \frac{z}{z+1} \Big| X^2+Y^2-2 \frac{\sqrt{\bG }}{z^{\frac13}} \Big(1-\frac{\bp\theta}{\bp\zeta}\Big) XY \Big|  \le e_3 (X^2+Y^2), \ \ z\in(0, Z_0). 
 \end{align}
Thus, recalling $\kappa=\kappa_0(1+Z_0)^{2cM-\al}$ from assumption~\ref{item:a3}, 
\beqa\label{E:BASICENERGY2}
&\int \Big(\frac{2cj-\alpha}{2}\frac{z}{1+z}\frac{(1+z)^{2cj-\al}}{\kappa+(1+z)^{2cj-\al}}\Big)\chi_{2j}   \bigg( X^2 + Y^2 - 2  \frac{\sqrt{\bG }}{z^{\frac13}}\Big(1-\frac{\bp\theta}{\bp\zeta}\Big) XY \bigg)\\
&\leq \int_0^{Z_0}\frac{2cj-\alpha}{2} 
\frac{(1+Z_0)^{2cj-\al}}{\kappa+(1+Z_0)^{2cj-\al}} e_3 \chi_{2j}(X^2+Y^2) 
+ \int_{Z_0}^\infty \frac{2cj-\alpha}{2} (1+d) \chi_{2j} 
(X^2+Y^2)\\
&\leq \int \bigg( (c j-\frac{\al}{2})\frac{e_3}{\kappa_0}\mathbf{1}_{z\leq Z_0} + \Big((cj -\frac{\alpha}{2})(1+d)\Big)\mathbf{1}_{z\geq Z_0}\bigg)\chi_{2j}(X^2+Y^2),
\eeqa
where we have used $2cj -\alpha >0$ for $j \ge1$ due to assumption \ref{item:a1}. 

Exactly as in~\eqref{E:ZEROENERGY3}--\eqref{E:ZEROENERGY4} in the previous proof, we obtain 
\beqa\label{E:BASICENERGY3}
&- \int \chi_{2j}   X^2 \Big(\frac{z\pa_{zz}\zeta}{\pa_z\zeta}  + \frac{\pa_s\bp\theta }{\bp\zeta}\Big)  -\int \chi_{2j} \bp(\frac{\bzeta_z^2}{\zeta_z^2}) \bG(\bp\zeta) XY \\
&\quad \leq 
 \int \chi_{2j}   (X^2 +Y^2) \Big(\frac23\mathbf{1}_{z\leq Z_0}+\frac{C}{Z_0}\mathbf{1}_{z\geq Z_0}+C(Z_0)\sqrt{\eps_0}\Big),
\eeqa 
where we have used the lower bound $\frac{z\pa_{zz}\bzeta}{\pa_z\bzeta} \ge -\frac23$, and $\frac{z\pa_{zz}\bzeta}{\pa_z\bzeta} = O(\frac1z) $ for $z\gg1$, both of which follow from~\eqref{eq:zdzzZeta}, and the pointwise a priori estimate~\eqref{E:APRIORIMAINPT}.

Combining~\eqref{E:BASICENERGY1}--\eqref{E:BASICENERGY3} in~\eqref{E:BASICENERGYFIRST}, we have obtained  
\begin{align}
& \frac12\frac{d}{ds}\int \chi_{2j} \big( X^2 + Y^2 \big) \notag \\
&\le \int_0^{Z_0} \chi_{2j}  \big( j\big(-\frac{2}{3} + \a +\frac{e_3}{\kappa_0} c\big)  +\frac12 - \frac{e_3\al}{2\kappa_0} +\frac23 +C(Z_0)\sqrt{\eps_0}\big) ( X^2 + Y^2)\label{E:ENERGYINTERIOR1}\\
& \ \ \  + \int_{Z_0}^\infty \chi_{2j}  \Big(j\big(-\frac{2}{3}+c(1+d)+\frac{C}{Z_0}\big) +\big(-\frac{\alpha}{2}(1+d)+\frac12 +\frac{C}{Z_0} +C(Z_0)\sqrt{\eps_0}\big)\Big) ( X^2 + Y^2)\label{E:ENERGYEXTERIOR1}\\
& \ \ \   
+ \int \chi_{2j}  \bD^{2j} \bigg(\frac{\CLP z-\tilde M}{\zeta^2}-\frac{\tilde g_z}{\tilde g}\frac1{\zeta_z}  \bigg)    Y \diff z
+ \int \chi_{2j} \bigg(\frac{\bzeta_z^2}{\zeta_z^2} \mathcal R_{2j}\theta + \mathcal N_{2j}[\theta] + \bD^{2j}\RR[\th]\bigg) Y \diff z. \notag
\end{align} 
Thus, it is clear that for $c<\frac23$ and $\al>1$, we may first take $\kappa_0$ large enough, then take $d$ sufficiently small, then $Z_0\gg1$ sufficiently large, then finally $\eps_0$ sufficiently small so that, with $c_3$, $c_4$ as defined in the statement of  Proposition~\ref{P:EE2},
\beq
c_3\leq \frac12\min\big\{\frac{2}{3} - \a -\frac{e_3}{\kappa_0} c,\frac{2}{3}-c(1+d)-\frac{C}{Z_0}\big\}, \qquad c_4\leq \frac{\alpha}{2}(1+d)-\frac12 -\frac{C}{Z_0} -C(Z_0)\sqrt{\eps_0},
\eeq
and hence~\eqref{E:ENERGYEXTERIOR1} provides a coercive estimate in the far-field region. Moreover, from the choice of $\kappa_0$ and~\eqref{E:MCONDITION}, it is clear that, taking $\kappa_0$ large and $\eps_0$ small, depending on $Z_0$, there exists $J\leq\m$ such that 
$$j\big(-\frac{2}{3} + \a+\frac{e_3}{\kappa_0} c\big)  +\frac76 -\frac{e_3\al}{2\kappa_0} +C(Z_0)\sqrt{\eps_0}\begin{cases}
\geq -(c_3 j + c_4), j\leq J,\\
<-(c_3 j+c_4), j> J.
\end{cases}$$
This then also  controls the interior contribution in~\eqref{E:ENERGYINTERIOR1}.\footnote{For the convenience of the reader, we note that, as $\a=\frac1{24}$, $-2(\frac23-\a)+\frac76=-\frac1{12}$, $-3(\frac23-\a)+\frac76=-\frac{17}{24}$.}

It is clear from the definition of $c_3$, $c_4$ that such a $J\leq \m$ exists due to~\eqref{E:MCONDITION}. For $j\leq J$, we bound 
\beqa
{}&\int_0^{Z_0} \chi_{2j}  \Big(j\big(-\frac{2}{3} + \a +\frac{e_3}{\kappa_0} c\big)  +\frac76 -\frac{e_3\al}{2\kappa_0} +C(Z_0)\sqrt{\eps_0}\Big) ( X^2 + Y^2)\leq C\kappa\|\Phi\|_{\mathcal{H}^{2J}_{2\m,Z_0}}^2,
\eeqa
where we recall the definition~\ref{item:a3} of $\kappa$. For $j>J$, we recall the definition of $c_3$ and $c_4$ and employ the coercive estimate to move this contribution onto the left. This concludes the proof.
\end{proof}


\subsection{Technical lemmas}\label{S:TECHNICAL}


In order to prove the top order energy estimates, Propositions~\ref{P:EE2},~\ref{P:HMZDIFF}, we must close the necessary estimates on the nonlinear terms arising on the right hand side of the energy inequalities in Propositions~\ref{P:ZEROENERGYGLOBAL} and Proposition~\ref{P:ENERGYGLOBAL}. In preparation for these proofs, we establish a number of $L^2$- and $L^\infty$-type estimates adapted to the weighted structure of our energy spaces (compare the definition~\eqref{E:TILDEEDEF} of $\tE_{2j}$). As the weights $\chi_{2j}$, defined in~\eqref{E:CHIJDEF}, have both a constant term $\kappa$ and a growing weight $(1+z)^{2cj-\al}$, we treat  these contributions separately in the following.

We note here for the convenience of the reader that we will regularly employ the integer part (or floor) function $[\cdot]$ in the index of the energies $\tE_{\leq 2j}$ due to the fact that we have defined these energy functionals only for even orders.

 Throughout this section, we make the standing assumption 
\beq\label{E:STANDING} \Phi=\begin{pmatrix} \th \\ \phi \end{pmatrix} \in \mathfrak{H},\qquad \tE_{\leq 2M}[\Phi]<\infty, \qquad \Pb[\Phi]\leq\frac14.
\eeq


\subsubsection{$L^2$- and $L^\infty$-type estimates for growing weights}

We begin by establishing weighted $L^2$-type estimates for the differential operators $Q\in\mathcal{X}_j$ (compare~\eqref{E:XTWOJ}--\eqref{E:XTWOJPLUSONE}).

\begin{lemma}\label{L:QBOUNDS}
Let $Q\in \mathcal X_j$, $1\le j \le 2M$, where we recall the definitions~\eqref{E:XTWOJ}--\eqref{E:XTWOJPLUSONE}. Then there exists $C>0$, independent of $Z_0>z_*$, such that
if $j$ is odd we have
\begin{align}\label{E:JODD}
\int_{Z_0}^\infty (1+z)^{(j-1)c-\al-\frac43}|Q\th|^2\diff z \le C  \tE_{\leq 2 [\frac{j}{2}]} ,
\end{align}
and if $j$ is even the following slightly weaker bound holds:
\begin{align}\label{E:JEVEN}
\int_{Z_0}^\infty (1+z)^{(j-2)c-\al-\frac43}|Q\th|^2\diff z \le C  \tE_{\leq \JDN}.
\end{align}
\end{lemma}


\begin{proof}
By Lemma~\ref{L:XJDECOMP}, we express $Q\th$ as a linear combination of the products of
the form $z^{-\frac{j-\ell}{3}}\bD^\ell\th$, where $0\le\ell\le j$.
Note further that
\begin{align*}
\int_{Z_0}^\infty (1+z)^{(j-1)c-\al-\frac43}z^{-\frac{2(j-\ell)}{3}}|\bD^\ell\th|^2\diff z &= \int_{Z_0}^\infty (1+z)^{(j-\ell)(c-\frac23)+c(\ell-1)-\al-\frac43}|\bD^\ell\th|^2\diff z.
\end{align*}
If $\ell=2\ell'+1$ is odd, then 
\begin{align*}
\int_{Z_0}^\infty (1+z)^{(j-\ell)(c-\frac23)+c(\ell-1)-\al-\frac43}|\bD^\ell\th|^2\diff z\leq\int_{Z_0}^\infty (1+z)^{2\ell' c-\al}\frac{|\bd\bD^{2\ell'}\th|^2}{z^{\frac43}}\diff z \le C \tilde\E_{2\ell'},
\end{align*}
where we have used $c<\frac23$ in the first inequality and the definition~\eqref{E:TILDEEDEF} of $\tE$ in the second inequality.
If however $\ell=2\ell'$ is even, then the parity of $j$ makes a difference.
Assuming that $j$ is odd we then always have $j-\ell\ge1$. In that case
\begin{align}
  \int_{Z_0}^\infty (1+z)^{(j-\ell)(c-\frac23)+c(\ell-1)-\al-\frac43}|\bD^\ell\th|^2\diff z& =\int_{Z_0}^\infty (1+z)^{(j-\ell-1)(c-\frac23)+2\ell' c-\al-2}|\bD^{2\ell'}\th|^2\diff z  \notag\\
& \leq C  \tE_{2\ell'}, \notag
\end{align}
where we have used Lemma~\ref{L:LARGEZ} in the last bound. If $j$ is even, then in the case $\ell=2\ell'$ we can prove 
by same token the weaker bound~\eqref{E:JEVEN} using Lemma~\ref{L:LARGEZ}.
\end{proof}


To facilitate nonlinear estimates, we present $L^\infty$ estimates on $(Z_0,\infty)$. It is convenient to estimate unweighted derivatives $\pa_z^k$ and to bound them by our energy built on weighted derivatives $\bD^k\theta$, recalling the identity~\eqref{EquivID} that tells us that,  on $(Z_0,\infty)$, $\pa_z^k f$ and $ \frac{ \bD^k f }{z^{\frac23 k}}$ are equivalent modulo lower order terms. Due to the failure of the critical Hardy inequality at infinity and the growing weights of the energy functionals, we carefully derive some auxiliary weights $z^{a_k}$ in the following lemma and employ them to produce  $L^\infty$ estimates of $z^{a_k}\pa_z^k\theta$ for $0\le k\le 2M$.

\begin{lemma}\label{L:weightedthetaHardy} Let $0\le k\le 2M$ be given, and 
let 
\be\label{DEF:ak}
a_k := k (\frac{c}{2}+\frac23) -\frac12-\frac{\alpha}{2} - \frac12(\frac23-c) \mathbf{1}_{k\in 2\mathbb N +1}.
\ee
Then there exists $C>0$, independent of $Z_0>z_*$, such that for any $(\theta,\phi)^\top$ satisfying~\eqref{E:STANDING} (so that  $\lim_{z\to \infty} z^{a_k} \pa_z^k\theta(z)=0$), we have
\begin{align}\label{Linfty1}
\|z^{a_k} \pa_z^{k}\theta\|_{L^\infty(Z_0,\infty)}\leq C \Big(\|z^{a_k-\frac12} \pa_z^{k}\theta\|_{L^2(Z_0,\infty)}+ \|z^{a_k+\frac12} \pa_z^{k+1}\theta\|_{L^2(Z_0,\infty)}\Big) \leq C   \tE_{\le \KUP}^\frac12 . 
\end{align}
We also obtain the more precise top order estimate, for $j=1,\ldots,M$,
\beqa\label{E:DAMPENINGTOPORDER}
{}&\big|\pa_z^{2j+1}\theta\big|\leq C\Big( z^{-\frac23(2j+1)}|\bD^{2j+1}\theta|+\sum_{k=0}^{2j}z^{k-2j-1}|\pa_z^k\theta|\Big)
\leq C \Big(z^{-\frac23(2j+1)}|\bD^{2j+1}\th| + z^{-a_{2j}-1}\tE_{\leq 2j}^\frac12\Big).
\eeqa
\end{lemma}

\begin{proof} For any $Z\in(Z_0,\infty)$, the first inequality follows from 
\[
\begin{split}
-( Z^{a_k} \pa_z^{k}\theta(Z) )^2 &= 2 \int_Z^\infty (z^{a_k} \pa_z^{k}\theta) ( z^{a_k} \pa_z^{k+1}\theta+ a_k z^{a_k-1} \pa_z^{k}\theta)\dif z\\
&= 2 \int_Z^\infty z^{2a_k} \pa_z^{k}\theta  \pa_z^{k+1}\theta dz+2a_k \int  z^{2a_k-1} (\pa_z^{k}\theta)^2\dif z.
\end{split}
\]
To show the energy bound, we consider the cases $k$ is even or $k$ is odd separately and in each case we bound the terms independently. As the arguments are similar for each case, we show only the estimate for $\|z^{a_k+\frac12} \pa_z^{k+1}\theta\|_{L^2(Z,\infty)}^2$ in the case $k$ is even.   Let $k=2j$ so that $a_k = 2j (\frac{c}{2}+\frac23) -\frac12-\frac{\alpha}{2}$. Then using \eqref{EquivID} and recalling $c<\frac23$,
\begin{align*}
&\|z^{a_k+\frac12} \pa_z^{k+1}\theta\|_{L^2(Z,\infty)}^2 = \int_Z^\infty z^{2j (c+\frac43) -\alpha} |  \pa_z^{2j+1}\theta |^2  dz \leq C   \sum_{i=0}^{2j+1} \int_Z^\infty z^{2jc -\alpha-\frac43-\frac{2i}{3} }(\bD^{2j+1-i} \theta)^2 dz   \\
&\leq C \sum_{\ell=0}^j \int_Z^\infty z^{2jc-\al-\frac43-\frac43(j-\ell)}|\bD^{2\ell+1}\th|^2dz + C\sum_{\ell=0}^j \int_Z^\infty z^{2jc-\al-\frac43-\frac23(2(j-\ell)+1)}|\bD^{2\ell}\th|^2dz\\
&\leq C  \sum_{\ell =0}^j \int_Z^\infty z^{2c\ell  -\alpha -\frac43 } (\bd \bD^{2\ell}\theta )^2 dz +
 C\sum_{\ell =0}^{j }  
\int_Z^\infty z^{2c \ell-\alpha -2} |\bD^{2\ell}\theta|^2 dz \\
&\leq C \sum_{\ell=0}^j \int_0^\infty (1+z)^{2c \ell-\alpha } \frac{ (\bd\bD^{2\ell}\theta)^2}{(\bp\bzeta)^2} dz \leq C \tE_{2j},
\end{align*}
where, in the last inequality, we have bounded the second summation by applying~\eqref{E:theta2} when $2c\ell-\al>\frac73$ and applying~\eqref{E:theta} when $2c\ell-\al<\frac73$.

The remaining estimates follow similarly. This concludes the proof of~\eqref{Linfty1}.

To show~\eqref{E:DAMPENINGTOPORDER}, we apply Lemma \ref{L:XJDECOMP},~\eqref{Linfty1}, and~\eqref{E:AI1}, to estimate 
\beqas
{}&\big|\pa_z^{2j+1}\theta\big|\leq C z^{-\frac23(2j+1)}|\bD^{2j+1}\theta|+C\sum_{k=0}^{2j}z^{k-2j-1}|\pa_z^k\theta|\\
&\leq C z^{-\frac23(2j+1)}|\bD^{2j+1}\theta|+C\sum_{k=0}^{2j}z^{k-2j-1-a_k}\tE_{\leq 2j}^{\frac12}
\leq C z^{-\frac23(2j+1)}|\bD^{2j+1}\th| + Cz^{-a_{2j}-1}\tE_{\leq 2j}^\frac12.
\eeqas
\end{proof}

\begin{remark}\label{R:DELTA} We observe that $a_1$ in \eqref{Linfty1} is 
\be a_1= c-\frac23- \frac{\alpha-1}{2}  = -\delta ,
\ee
where $\delta>0$ is given in~\eqref{def:delta}.
\end{remark}
We note the following simple inequalities for the exponents $a_i$:
\begin{align}
&a_{i}-(i-j)\leq a_{j}, && \text{ for }j\leq i\leq 2M,\label{E:AI1}\\
&\ell-\frac{2j}{3}-1-a_{\ell+1}\leq -cj+\de-1, &&\text{ for }\ell\leq \min\{2j,2M-1\}.\label{E:AI2}
\end{align}

\begin{lemma}\label{L:THETAZNONLINEAREST}
Let $f\in C^{2j}([-\frac12,\frac12])$, $j\leq M$, and assume that~\eqref{E:STANDING} holds. 
Then, for any $k=1,\ldots,2j$, there exists a constant $C>0$, independent of $Z_0>z_*$, such that for all $z\geq Z_0$,
\beqa\label{E:THETAZNONLINEAREST}
\Big|\pa_z^k f\Big(\frac{\theta_z}{\bzeta_z}\Big)\Big|\leq C\begin{cases}
z^{-a_{k+1}}  \tE^{\frac12}_{\le 2[\frac{k+2}{2}] }
+z^{-(\frac{c}{2}+\frac23-\de)k} \tE_{\le 2[\frac{k+1}{2}] },
& 1\leq k\leq 2j-1,\\
z^{-\frac23(2j+1)}|\bD^{2j+1}\theta|+z^{-a_{2j}-1}\tE_{\leq 2j}^{\frac12}+z^{-(\frac{c}{2}+\frac23-\de)2j}\tE_{\leq 2j}, & k=2j,
\end{cases}
\eeqa
and
\be\label{E:THETAZNONLINEAREST2}
\Big|\pa_z^k f\Big(\frac{\theta}{\bzeta}\Big)\Big| \leq C\Big(\tE_{\leq 2[\frac{k+1}{2}]}^{\frac12} z^{-a_k -1}  +  \tE_{\leq 2[\frac{k}{2}]} z^{-(\frac{c}{2} +\frac23-\delta)k }\Big).
\ee 
If, in addition, $f'(0)=0$, then we obtain the improved estimates
\begin{align}
\Big|\pa_z^k f\Big(\frac{\theta_z}{\bzeta_z}\Big)\Big|\leq&\, Cz^{-a_1-a_{k+1}}\tE_{\le 2[\frac{k+2}{2}] }+Cz^{-(\frac{c}{2}+\frac23-\de)\ell}\tE_{\le 2[\frac{k+1}{2}] }\notag\\
&+C\Big(z^{-a_{2j}-a_1-1}\tE_{\leq 2j}+z^{-a_1-\frac23(2j+1)}\tE_2^{\frac12}|\bD^{2j+1}\theta|\Big)\mathbf{1}_{k=2j},\label{E:THETAZNONLINEAREST3}\\
\Big|\pa_z^k f\Big(\frac{\theta}{\bzeta}\Big)\Big| \leq &\,C\tE_{\leq 2[\frac{k+1}{2}]} z^{-a_0-a_k -1}  + C \tE_{\leq 2[\frac{k}{2}]} z^{-(\frac{c}{2} +\frac23-\delta)k } \label{E:THETAZNONLINEAREST4}.
\end{align}
\end{lemma}


\begin{proof}
We begin by proving~\eqref{E:THETAZNONLINEAREST}. To compress notation, we set $x(z)=\frac{\theta_z}{\bzeta_z}$.
It is straightforward to check that
 $|f^{(i)}(x(z))|\leq C$ for each $i\in \{1,\ldots,2j\}$ using the pointwise estimate~\eqref{E:APRIORIMAINPT} to see $|x(z)|\leq \frac12$. In order to estimate the derivatives of $x(z)$ , we first expand 
\be
\pa_z^i x(z) = \frac{\pa_z^{i+1}\th}{\pa_z\bzeta}+\sum_{j=0}^{i-1}{i \choose j}\pa_z^{i-j}(\bzeta_z^{-1})\pa_z^{j+1}\theta.
\ee
Now from Lemma~\ref{L:ZETABAR}, we note that $\pa_z^{2+\ell}\bzeta = O(z^{-3-\ell})$ for $\ell\geq0$, so that, in particular, $\pa_z^\ell(\bzeta_z^{-1})=O(z^{-2-\ell})$. Thus, applying~\eqref{Linfty1} and~\eqref{E:AI1},
\be\label{Bound x^i 0}
| z^{a_{i+ 1}} \pa_z^i x|\leq C\Big( \frac{z^{a_{i+ 1}}|\pa_z^{i+1}\th|}{|\pa_z\bzeta|}+\sum_{j=0}^{i-1}z^{a_{i+1}-2-(i-j)}|\pa_z^{j+1}\theta|\Big) \leq C  \tE^{\frac12}_{\le 2[\frac{i+2}{2}]},
\quad 1\le i\le  2j-1 .
\ee
In the case $i=2j$,  we apply~\eqref{E:DAMPENINGTOPORDER} and~\eqref{E:AI1},
\begin{align}
|  \pa_z^{2j} x|\leq &\, C\Big(\frac{|\pa_z^{2j+1}\th|}{|\pa_z\bzeta|}+\sum_{j=0}^{2j-1}z^{-2-(i-j)}|\pa_z^{j+1}\theta|\Big)\leq C\Big(  \frac{|\pa_z^{2j+1}\th|}{|\pa_z\bzeta|}+\sum_{j=0}^{2j-1}z^{-a_{j+1}-2-(2j-j)}\tE_{\leq 2j}^{\frac12}\Big)\notag\\
  \leq &\, C\Big( z^{-\frac23(2j+1)}|\bD^{2j+1}\th| + z^{-a_{2j}-1}\tE_{\leq 2j}^\frac12\Big).\label{E:D2MX}
\end{align}
We now use these bounds to estimate  $\pa_z^\ell f(x(z))$. By the Faa di Bruno formula,
\be\label{f^(j) 0}
 \pa_z^\ell f (x(z)) = f'(x) \pa_z^\ell x +  \sum_{\substack{(\lambda_1, \dots, \lambda_\ell)\in M_\ell \\ \lambda_\ell  =0}} \frac{\ell !}{\lambda_1 ! \cdots \lambda_{\ell} !} f^{(\lambda_1+ \cdots+ \lambda_{\ell})} \prod_{i=1}^{\ell} (\frac{\pa_z^i x }{i !})^{\lambda_i},
\ee
where 
\[
M_\ell = \{(\lambda_1, \dots, \lambda_\ell)\in (\mathbb Z_{\ge 0})^\ell : \sum_{i=1}^\ell i \lambda_i = \ell \}.
\]
From the definitions, it is clear that if $\ell=1$, then the summation on the right of~\eqref{f^(j) 0} is empty, while, for $\ell\geq 2$,
for $(\lambda_1, \dots, \lambda_\ell)\in M_\ell$ with $\lambda_\ell=0$, we have $2 \le \sum_{i=1}^{\ell} \lambda_i\le \sum_{i=1}^{\ell} i \lambda_i = \ell $.
By using \eqref{Bound x^i 0} (note $i \le \ell-1\le  2j-1$), the second (summation) term in \eqref{f^(j) 0} is then bounded (recalling the a priori assumption~\eqref{E:APRIORIMAIN}) by 
\beqa\label{E:FDBDAMPENING}
&\left|   \sum_{\substack{(\lambda_1, \dots, \lambda_\ell)\in M_\ell \\ \lambda_\ell  =0}} \frac{\ell !}{\lambda_1 ! \cdots \lambda_{\ell} !} f^{(\lambda_1+ \cdots+ \lambda_{\ell})} \prod_{i=1}^{\ell} (\frac{\pa_z^i x }{i !})^{\lambda_i} \right| \leq C
\tE_{\leq 2[\frac{\ell+1}{2}]} z^{-\sum_{i=1}^\ell \lambda_i (a_{i+1})} \leq C \tE_{\leq 2[\frac{\ell+1}{2}]}
z^{-(\frac{c}{2} +\frac23)\ell + \ell \delta},
\eeqa
where we recall $\de= \frac{2}{3}- c + \frac{\alpha-1}{2}$ from~\eqref{def:delta}.

Estimating the first term in~\eqref{f^(j) 0}, we use~\eqref{Bound x^i 0} and~\eqref{E:D2MX} to see
\be\label{E:FPRIMEXEST}
|f'(x) \pa_z^\ell x| \leq C\begin{cases}   \tE_{\leq 2[\frac{\ell+2}{2}]}^\frac12 z^{-a_{\ell +1}}, & \ell\le 2j-1,\\
z^{-\frac23(2j+1)}|\bD^{2j+1}\th| + z^{-a_{2j}-1}\tE_{\leq 2j}^\frac12, & \ell=2j.
\end{cases}
\ee
Summing~\eqref{E:FPRIMEXEST} with~\eqref{E:FDBDAMPENING} yields~\eqref{E:THETAZNONLINEAREST}.

In the case that also $f'(0)=0$, then clearly $|f'(x(z))|\leq C|\theta_z|\leq Cz^{-a_1}\tE_2^{\frac12}$, so that~\eqref{E:FPRIMEXEST} improves to yield
\[
|f'(x) \pa_z^\ell x| \leq C\begin{cases} \tE_{\leq 2[\frac{\ell+2}{2}]} z^{-a_1-a_{\ell +1}}, & \ell\le 2j-1,\\
z^{-a_1-\frac23(2j+1)}|\bD^{2j+1}\th|\tE_2^{\frac12} + z^{-a_1-a_{2j}-1}\tE_{\leq 2j}, & \ell=2j,
\end{cases}
\]
which in turn proves~\eqref{E:THETAZNONLINEAREST3}.

Finally, to prove~\eqref{E:THETAZNONLINEAREST2}, we make a very similar argument, now defining $\tilde x(z)=\frac{\th}{\bzeta}$. An argument as above, as in the derivation of~\eqref{Bound x^i 0}, using Lemma~\ref{L:weightedthetaHardy}, yields
\be\label{Bound x^i}
| z^{a_i +1} \pa_z^i \tilde x|\leq C \tE_{\leq 2[\frac{i+1}{2}]}^\frac12, \quad 1\le i\le 2j. 
\ee
Again applying the Faa di Bruno formula, we arrive at
\be
|\pa_z^\ell f(\tilde x(z))| \leq C\Big(\tE_{\leq 2[\frac{\ell+1}{2}]}^{\frac12} z^{-a_\ell -1}  +  \tE_{2[\frac{\ell}{2}]} z^{-(\frac{c}{2} +\frac23-\delta)\ell }  \Big).
\ee 
\end{proof}


\subsubsection{$L^2$- and $L^\infty$-type estimates for $\kappa$ weight}

\begin{lemma}\label{L:KAPPAQBOUND}
Let $Q\in \mathcal X_j$, $2\le j \le 2M+1$, where we recall the definitions~\eqref{E:XTWOJ}--\eqref{E:XTWOJPLUSONE}. Then, under the assumption~\eqref{E:STANDING}, 
there exists a constant $C>0$, independent of $Z_0>z_*$, such that
\begin{align}
\kappa \int_{Z_0}^\infty |Q\th|^2 z^{-2} \diff z \le C \tE_{\leq \JDN}.  
\end{align}
 Also, if $Q\in \X_1$, 
\begin{align}
\kappa \int_{Z_0}^\infty |Q\th|^2 z^{-2-\frac23} \diff z \le C  \tE_{0}.
\end{align}
\end{lemma}


\begin{proof}
The proof is analogous to the proof of Lemma~\ref{L:QBOUNDS} and relies on~\eqref{theta_bound} with $\al=0$. When $j$ is even, the bound can be
improved by replacing $z^{-2}$ by $z^{-\frac43}$.
\end{proof}


\begin{lemma}\label{L:KAPPAQINFTY}
Let $Q\in \mathcal X_j$, $2\le j \le 2M$, where we recall the definitions~\eqref{E:XTWOJ}--\eqref{E:XTWOJPLUSONE}. Then, under the assumption~\eqref{E:STANDING},  
there exists a constant $C>0$, independent of $Z>z_*$, such that
\beq
\kappa Z^{-1}|Q\theta(Z)|^2\leq C  \tE_{\leq \JUP}. 
\eeq
For $Q\in\X_1$, we have
\beq
\kappa Z^{-1-\frac23}|Q\theta(Z)|^2\leq C  \tE_{\leq 2} .
\eeq
\end{lemma}


\begin{proof}
We first consider the case $j\geq 2$. We apply Lemma~\ref{L:XJDECOMP} to write $Q\theta=\sum_{\ell=0}^j c_\ell z^{-\frac{j-\ell}{3}} \bD^\ell\theta $. When $\ell=0$, we note $j-\ell\geq 2$, so that we may apply~\eqref{E:HARDYREFINED} with $\beta=-\frac83$ and then again with $\beta=0$ to estimate
\beqas
\kappa Z^{-1-\frac43}|\theta(Z)|^2\leq C\kappa |\theta(1)|^2+C\kappa \int_1^Z(1+z)^{-\frac83}|\bd\theta|^2\leq C\kappa \int_0^Z(1+z)^{-\al}{|\bD^{1}\theta|^2}{(\bp\bzeta)^2}\leq C\tE_{0},
\eeqas
where we have  recalled that $\al<2c<\frac43$ in the final inequality.\\
For any even $\ell\in\{2,\ldots,j\}$, we apply~\eqref{E:HARDYREFINED} with $\beta=-\frac43$ and then again with $\beta=0$ to see
\beqa
\kappa Z^{-1}|\bD^\ell\theta(Z)|^2\leq C\kappa |\bD^\ell\theta(1)|^2+C\kappa \int_1^Zz^{-\frac43}|\bD^{\ell+1}\theta|^2\leq C\kappa \int_0^Z\frac{|\bD^{\ell+1}\theta|^2}{(\bp\bzeta)^2}\leq C\tE_{\ell}.
\eeqa
For odd $\ell=1,\ldots,j$, we instead apply~\eqref{E:HARDYinf3} with $\beta=-\frac43-\frac{2(j-\ell)}{3}$ and \eqref{E:theta} to obtain 
\beqa
&\kappa Z^{-1-\frac{2(j-\ell)}{3}}|\bD^\ell\theta(Z)|^2\leq  \kappa \int_Z^\infty|\bD^{\ell+1}\theta|^2z^{-\frac43-\frac{2(j-\ell)}{3}}+C\kappa \int_Z^\infty |\bD^\ell\theta|^2z^{-2-\frac{2(j-\ell)}{3}} \\
&\qquad \le  \kappa \int_Z^\infty|\bD^{\ell+2}\theta|^2 z^{-\frac23-\frac{2(j-\ell)}{3}}+C\kappa \int_Z^\infty |\bD^\ell\theta|^2z^{-2-\frac{2(j-\ell)}{3}}
\leq C(\tE_{\ell+1}+ \tE_{\ell-1}),
\eeqa 
where we have used that if $\ell=1$ then $j-\ell\geq 1$  and $\alpha < 2c < 4/3$  so that, in the final inequality, $z^{-2-\frac{2(j-\ell)}{3}}\leq C(\bp\bzeta)^{-2}z^{-\al}$, where $C$ is independent of $Z>z_*$. 

When $Q\in\X_1$, we  estimate exactly as in the previous inequality, using the additional factor of $z^{-\frac23}$ to compensate for the missing $z^{-\frac{2(j-\ell)}{3}}$.
\end{proof}


\begin{lemma}\label{L:NONLINEARPQEXPANSION}
Let $f\in C^{2M}([-\frac12,\frac12])$ and assume~\eqref{E:STANDING}. Then, for any $j=1,\ldots,2M$ and $P\in\Y_j$, there exists $C>0$, independent of $Z_0>z_*$, such that 
\beq\label{E:PJNL2}
\kappa\int_{Z_0}^\infty z^{-\frac23}\Big(\Big|P\Big(f\big(\frac{\theta_z}{\bzeta_z}\big)\Big)\Big|^2+\Big|P\Big(f\big(\frac{\theta}{\bzeta}\big)\Big)\Big|^2\Big)\leq C\tE_{\leq 2[\frac{j+1}{2}]},
\eeq
and, if $ j\leq M+1$,
\beq\label{E:PJNLINFTY}
\sqrt{\kappa}\Big\|P\Big(f\big(\frac{\theta_z}{\bzeta_z}\big)\Big)\Big\|_{L^\infty(Z_0,\infty)}+\sqrt{\kappa}\Big\|P\Big(f\big(\frac{\theta}{\bzeta}\big)\Big)\Big\|_{L^\infty(Z_0,\infty)}\leq CZ_0^{-\frac16}\tE_{\leq2[\frac{j+2}{2}]}^{\frac12}.
\eeq
If, in addition, $f'(0)=0$, then the $L^2$ type estimate improves to
\begin{align}
{}&\kappa\int_{Z_0}^\infty z^{-\frac23}\Big(\Big|P\Big(f\big(\frac{\theta_z}{\bzeta_z}\big)\Big)\Big|^2+\Big|P\Big(f\big(\frac{\theta}{\bzeta}\big)\Big)\Big|^2\Big)\leq C\tE_{\leq 2[\frac{j+1}{2}]}(C_*\sqrt{\eps_0}+Z_0^{-\frac13}\tE_{\leq 2[\frac{j}{4}+1]}).\label{E:PJNL2IMP}
\end{align}
Finally, for $Q\in\X_\ell$, $0\leq \ell\leq 2M-1$, we have
\beq\label{E:QKL2}
\kappa\int_{Z_0}^\infty |QK\theta|^2\leq C Z_0^{-\frac23}\tE_{\leq2[\frac{\ell+2}{2}]},
\eeq
and, for $\ell+2\leq M+1$,
\beq\label{E:QKLINFTY}
\Big\|z^{\frac13}QK\theta\Big\|_{L^\infty(Z_0,\infty)}\leq CZ_0^{-\frac12}\tE_{\leq 2[\frac{\ell+3}{2}]}^{\frac12}.
\eeq
\end{lemma}

\begin{proof}
As in the proof of Lemma~\ref{L:NBOUND}, an inductive argument based on Lemma~\ref{L:XYPRODCHAIN} shows that $P\big(f\big(\frac{\theta_z}{\bzeta_z}\big)\big)$ expands as a sum of terms of the form
\begin{align}\label{E:PFSUMMATION}
f^{(k)}(\sqrt{\bG} \bp\theta)  P_{j_1}(\sqrt{\bG} \bp\theta) P_{j_2}(\sqrt{\bG} \bp\theta) \dots P_{j_{k}}(\sqrt{\bG} \bp\theta)
\end{align}
for some $j_1,\dots,j_{k}\in\mathbb N$, where 
\[
1\le k \le j, \ \ P_{j_n}\in \mathcal Y_{j_n},  \ \ n\in\{1,\dots, k\}, \ \ \ j_1+j_2+\dots j_{k} = j.
\]
For the other term, $P\big(f\big(\frac{\theta}{\bzeta}\big)\big)$, we recall~\eqref{E:GTILDEDEF}, so that $\theta\bzeta^{-1}=\tilde G\frac{\th}{z^{\frac13}}$ and we have the expansion
\begin{align}
f^{(k)}(\tilde{G} \frac{\theta}{z^{\frac13}})  P_{j_1}(\tilde{G} \frac{\theta}{z^{\frac13}}) P_{j_2}(\tilde{G} \frac{\theta}{z^{\frac13}}) \dots P_{j_{k}}(\tilde{G} \frac{\theta}{z^{\frac13}}).
\end{align}
For any such $P_{j_n}$, from~\eqref{E:PEXP1} and~\eqref{E:PGBOUND}, we have the estimate
\beq\label{E:PJNEST}
\big|P_{j_{n}}(\sqrt{\bG} \bp\theta)(z)\big|+\big|P_{j_{n}}(\tilde{G} \frac{\theta}{z^{\frac13}})(z)\big|\leq C\sum_{i=0}^{j_n}\sum_{Q\in\X_{i+1}}(1+z)^{-\frac{2+j_n-i}{3}}|Q\theta|.
\eeq
Thus, for each $j_n\leq M+1$, we apply Lemma~\ref{L:KAPPAQINFTY}, noting $j_n\geq 1$ for all $n$, to see
\beq\label{E:PJNLINFTYPF}
\sqrt{\kappa}\big\|P_{j_n}(\sqrt{\bG} \bp\theta)\big\|_{L^\infty(Z_0,\infty)}+\sqrt{\kappa}\big\|P_{j_n}(\tilde{G} \frac{\theta}{z^{\frac13}})\big\|_{L^\infty(Z_0,\infty)}\leq CZ_0^{-\frac16}\tE_{\leq2[\frac{j_n+2}{2}]}^{\frac12},
\eeq
that is,~\eqref{E:PJNLINFTY}. On the other hand, applying Lemma~\ref{L:KAPPAQBOUND} to~\eqref{E:PJNEST}, we have
\beqa\label{E:PJNL2PF}
\kappa&\,\int_{Z_0}^\infty z^{-\frac23}\Big(|P_{j_n}(\sqrt{\bG}\bp\theta)|^2+|P_{j_n}(\tilde{G} \frac{\theta}{z^{\frac13}})|^2\Big)\\
\leq&\,C\int_{Z_0}^\infty \kappa\Big(\sum_{Q\in\X_1}(1+z)^{-2-\frac{2j_n}{3}}|Q\theta|^2+\sum_{i=1}^{j_n}\sum_{Q\in\X_{i+1}}(1+z)^{-2-\frac{j_n-i}{3}}|Q\theta|^2\Big)\leq C\tE_{\leq 2[\frac{j_n+1}{2}]}.
\eeqa
Noting that at most one index $j_n$ in~\eqref{E:PFSUMMATION} can be of size greater than $[\frac{j+1}{2}]$, this implies~\eqref{E:PJNL2}. 

In the case that $f'(0)=0$, note that either every contribution to~\eqref{E:PFSUMMATION} is at least quadratic, in which case we may combine the $L^\infty$ and $L^2$ bounds,~\eqref{E:PJNLINFTYPF}--\eqref{E:PJNL2PF}, to obtain, for example
\begin{align}
\kappa\int z^{-\frac23}\big|P_{j_1}(\sqrt{\bG} \bp\theta)\big|^2 \big|P_{j_2}(\sqrt{\bG} \bp\theta)\big|^2\leq&\, C  Z_0^{-\frac13}\tE_{\leq2[\frac{j_2+2}{2}]}\tE_{\leq 2[\frac{j_1+1}{2}]}\\
\leq &\, C  Z_0^{-\frac13}\tE_{\leq2[\frac{j}{4}+1]}\tE_{\leq 2[\frac{j+1}{2}]},
\end{align}
where we  used that, without loss of generality,  $j_2\leq \frac{j}{2}$,
 or we have an additional factor of 
$$|f'(\sqrt{\bG}\bp\th)|\leq C|\pa_z\th|\leq CC_*\sqrt{\eps_0}$$
for $z\geq Z_0$ by~\eqref{E:APRIORIMAINPT}. This proves~\eqref{E:PJNL2IMP}.

Now for $Q\in\X_\ell$, again applying an argument analogous to~\eqref{E:PEXP1} and Lemma~\ref{L:PGBOUND}, we obtain
\[
\big|QK\theta\big|=\big|(Q\bp)(\bG\bd\theta)\big|\leq C\sum_{i=1}^{\ell+2}\sum_{\substack{Q\in\X_{i} \\ P\in \Y_{\ell+2-i}}}|P\bG| |Q\theta|\leq \sum_{i=1}^{\ell+2}\sum_{Q\in\X_i}(1+z)^{-\frac{4+\ell+2-i}{3}}|Q\theta|.
\]
To show~\eqref{E:QKLINFTY}, for $\ell+2\leq M+1$, we again apply Lemma~\ref{L:KAPPAQINFTY} to estimate, when $z\geq Z_0$,
\[
z^{\frac13}\big|QK\theta\big|\leq \sum_{i=1}^{\ell+2}\sum_{Q\in\X_i}(1+z)^{-\frac{3+\ell+2-i}{3}}|Q\theta|\leq CZ_0^{-\frac12}\tE_{\leq 2[\frac{\ell+3}{2}]}^{\frac12}.
\]
Finally, to prove~\eqref{E:QKL2}, we apply Lemma~\ref{L:KAPPAQBOUND} to see
\beqas
\kappa\int_{Z_0}^\infty |QK\theta|^2\leq&\, C\int_{Z_0}^\infty \kappa\Big((1+z)^{-\frac{10+2\ell}{3}}\sum_{Q\in\X_1}|Q\theta|^2+\sum_{i=2}^{\ell+2}\sum_{Q\in\X_i}(1+z)^{-\frac{8+2(\ell+2-i)}{3}}|Q\theta|^2\Big)\\
\leq&\, CZ_0^{-\frac23}\tE_{\leq2[\frac{\ell+2}{2}]}.
\eeqas
\end{proof}


\subsection{Error term estimates: far-field flattening error}

The main goal of this section is to prove the following bound.

\begin{lemma}\label{L:CUTOFFBOUND}[Flattening term]
Let $\Phi=(\th,\phi)^\top$ satisfy the assumptions of Proposition~\ref{P:EE2}.
Then, for $0\leq j\leq M$, there exists $C>0$ such that
\begin{align}
&\Big|\int  \chi_{2j} \bD^{2j} \left(\frac{Cz-\tilde M}{\zeta^2}-\frac{\tilde g_z}{\tilde g}\frac1{\zeta_z}  \right)    \phi_{2j} \diff z\Big|\notag\\
&\quad \leq C\Big( \sqrt{\kappa}(r_*e^s)^{-\frac12-\frac{2j}{3}}\tE_{\leq 2j}^{\frac12}+ (r_\ast e^s)^{ - (\frac23-c)j - \frac{\alpha+1}{2}} \tE_{\leq 2j}^\frac12 + (r_*e^s)^{-\frac23}\tE_{\leq 2j}  +  (r_\ast e^s)^{-\frac{\alpha+1}{2} + 2\delta j }\tE_{\leq 2j}^\frac32\Big), \notag
\end{align}
where we recall $\delta>0$ is given in \eqref{def:delta}.
\end{lemma}

\begin{proof} 
We first recall from assumptions~\ref{item:g1}--\ref{item:g2} that the support of $\frac{\CLP z-\tilde M}{\zeta^2}  -\frac{\tilde g_z}{\tilde g}\frac1{\zeta_z}   $ lies in the set $\{z\ge r_\ast e^{s}\}$. Hence, recalling also the definition of the weight $\chi_{2j}$ from~\eqref{E:CHIJDEF}, it suffices to estimates  
\be\label{WTS1}
\int_{r_\ast e^s}^\infty \chi_{2j} \Big| \bD^{2j} \Big(\frac{\CLP z-\tilde M}{\zeta^2}  \Big)  \Big|^2dz, \quad   \int_{r_\ast e^s}^\infty \chi_{2j} \Big| \bD^{2j} \left(\frac{\tilde g_z}{\tilde g}\frac1{\zeta_z}  \right)  \Big|^2dz 
\ee
by the Cauchy-Schwartz inequality. We  first estimate  the second term arising here.

We begin with the case $j=0$. Since $\frac{1}{\zeta_z}=O(1)$ by the a priori bound~\eqref{E:APRIORIMAINPT} and $|\frac{\tilde g_z}{\tilde g}|\leq C z^{-1}$ for $z\gg 1$ by \ref{item:g2}, we recall the definition of $\chi_0$ from~\eqref{E:CHIJDEF} and estimate
\be
\int_{r_\ast e^s}^\infty  \kappa(1+z)^{-\al} \Big| \left(\frac{\tilde g_z}{\tilde g}\frac1{\zeta_z}  \right)  \Big|^2dz  \leq C \int_{r_\ast e^s}^\infty \frac{\kappa }{z^{2+\alpha}} dz \leq C \frac{\kappa }{(r_\ast e^s)^{1+\alpha}},
\ee
and so we are done for $j=0$.

When $j\geq 1$, we recall from~\eqref{E:CHIJDEF} that $\chi_{2j}=g_{2j}(z)(\kappa+(1+z)^{2cj-\al})$ and first estimate the factor arising from the term $(1+z)^{2cj-\al}$, recalling that $g_{2j}$ is uniformly bounded above and below for all $j\in\{1,\ldots,M\}$ by $M$-dependent constants.  To estimate the integrand in~\eqref{WTS1}, we first write
\be
\frac{\tilde g_z}{\tilde g}\frac{1}{\zeta_z} = h (z) f(x(z)), \quad h (z)= \frac{\tilde g_z}{\tilde g} \frac{1}{\bzeta_z}, \quad x(z) = \frac{\pa_z\theta}{\pa_z\bzeta}, \quad f(x) = \frac{1}{1+x},
\ee
so that  $f(x(z))$ satisfies~\eqref{E:THETAZNONLINEAREST}.
From~\ref{item:g2}, it is clear that $ 
\Big|\pa_z^k \big(\frac{\tilde g_z}{\tilde g} \big) \Big|\leq C  \frac{ e^{-(k+1)s}}{(1+ze^{-s})^{k+1}} \le C\frac{1}{z^{k+1}}$ 
and hence, using also Lemma~\ref{L:ZETABAR},
\be
\Big|\pa_z^k h \Big|\leq C  \frac{1}{z^{k+1}}, \quad k\ge 0.
\ee
To estimate  $\bD^{2j} (h(z) f(x(z)))$, we will use the representation of $\bD^{2j}$ via $z$-weighted normal  derivatives from Lemma~\ref{L:XJDECOMP}.
Putting this together with~\eqref{E:THETAZNONLINEAREST},
\begin{align}
\Big|\bD^{2j}  \Big( \frac{\tilde g_z}{\tilde g}\frac{1}{\zeta_z} \Big) \Big| &\leq C\Big( z^{\frac23(2M)}|h(z)||\pa_z^{2M}f(x(z))|\mathbf{1}_{j=M}+ \sum_{k=0}^{2j} z^{k-\frac{2j}{3}}\sum_{\substack{\ell=0 \\ \ell\neq 2M}}^{k}|\pa_z^{k-\ell}h||\pa_z^\ell f(x(z))|\Big)\notag \\
&\leq C z^{\frac23(2M)-1}\Big(z^{ -\frac23(2M+1)}  |\bD^{2M+1} \theta|+z^{-a_{2M}-1}\tE^{\frac12}+\tE z^{-(\frac c2+\frac23-\de)2M}\Big)  \cdot \mathbf{1}_{j=M}  \notag\\
&+ C\sum_{k=0}^{2j} z^{k-\frac{2j}{3}} \big[ z^{-(k +1)}  +\sum_{\substack{\ell=1 \\ \ell\neq 2M}}^k z^{-(k-\ell+1)} ( \tE^\frac12_{2[\frac{\ell+1}{2}]} z^{-a_{\ell +1}}  +  \tE_{\le 2[\frac{\ell+1}{2}]} z^{-(\frac{c}{2} +\frac23-\delta )\ell }   ) \big]\notag \\ 
&\leq C\Big( z^{-\frac{2j}{3} -1 }+ \tE^{\frac12}_{\le 2j}  z^{-c j -1 +\delta }+ \tE_{\le 2j}  z^{ (-c+2\delta) j -1  }  + |z^{ -1-\frac23}  \bD^{2M+1} \theta|  \cdot \mathbf{1}_{j=M}\Big), \notag
\end{align}
where we have applied~\eqref{E:AI2}.
Therefore,  
\begin{align}
&\int_{r_\ast e^s}^\infty (1+z)^{2cj -\alpha} \Big| \bD^{2j} \Big(\frac{\tilde g_z}{\tilde g}\frac{1}{\zeta_z}   \Big)  \Big|^2\notag\\
&\leq C \int_{r_\ast e^s}^\infty z^{-2(\frac23 -c) j -\alpha -2 }dz + C \tE_{\le 2j} \int_{r_\ast e^s}^\infty  z^{ -\alpha -2 + 2\delta} dz +  C\tE^2_{\le 2j}  \int_{r_\ast e^s}^\infty  z^{ -\alpha -2 + 4\delta j } dz\notag\\
&\quad + C \int_{r_\ast e^s}^\infty  (1+z)^{2cM -\alpha} \frac{1}{z^2} \frac{| \bD^{2M+1} \theta |^2}{|\bar\pa_z \zeta|^2} dz  \cdot \mathbf{1}_{j=M}  \notag  \\
&\leq C\Big( \frac{1}{(r_\ast e^s)^{ 2 (\frac23 -c) j +\alpha +1 }} +  \frac{1}{(r_\ast e^s)^{ \alpha +1- 2\delta }}\tE_{\le 2j } +  \frac{1}{(r_\ast e^s)^{  \alpha +1- 4\delta j }} \tE_{\le 2j}^2\Big) \label{E:FLAT1},
\end{align}
where we have used $-\al-2+4\de j<-1$ from assumption~\ref{item:a2} to see convergence of the integrals and also used $\al+1-2\de <2$ in the final bound.

To estimate the contribution from $\kappa$ to the weight $\chi_{2j}$, we observe from Lemma~\ref{L:XJDECOMP} that, for any $Q\in\X_{2j-\ell}$, we have
\[
|Qh|\leq C\sum_{i=0}^{2j-\ell}z^{i-\frac{2j-\ell}{3}}|\pa_z^ih|\leq Cz^{-1-\frac{2j-\ell}{3}}.
\]
Thus, from Lemma~\ref{L:NONLINEARPQEXPANSION},
\begin{align}
\kappa\int_{r_*e^s}^\infty|\bD^{2j}(f(x(z))h(z))|^2\leq&\, C\kappa\int_{r_*e^s}^\infty\sum_{\ell=0}^{2j}\sum_{\substack{P\in\Y_\ell \\ Q\in\X_{2j-\ell}}}|Pf(x(z))|^2|Qh|^2\notag\\
\leq&\,C\int_{r_*e^s}^\infty \kappa \Big(z^{-2-\frac{4j}{3}}+\sum_{\ell=1}^{2j}\sum_{P\in\Y_\ell}z^{-\frac43-2\frac{2j-\ell}{3}}z^{-\frac23}|Pf|^2\Big)\notag\\
\leq&\, C\kappa(r_*e^s)^{-1-\frac{4j}{3}}  +C(r_*e^s)^{-\frac43}\tE_{\le 2j}.\label{E:FLAT2}
\end{align}
Combining~\eqref{E:FLAT1} and~\eqref{E:FLAT2}, we conclude the estimate on the second term in~\eqref{WTS1}.

To control the first term in~\eqref{WTS1}, we use analogous arguments, this time writing
 $ \frac{Cz-\tilde M}{\zeta^2}$ as 
\be\label{dampening1}
 \frac{Cz-\tilde M}{\zeta^2}  =h_0(z) f_0 (x_0(z)), \quad h_0(z) =  \frac{Cz-\tilde M}{\bzeta^2},  \ \  x_0(z) = \frac{\theta}{\bzeta}, \ \  f_0(x)=\frac{1}{(1+x)^2},
\ee
so that $f_0(x_0(z))$ satisfies~\eqref{E:THETAZNONLINEAREST2}.
From~\eqref{E:MTILDEDEF}, we see $\pa_z \tilde M = 4\pi \tilde g $, and so using also assumption~\ref{item:g2} and Lemma \ref{L:ZETABAR}, 
it is easy to check 
\be
\pa_z^i h = O(\frac{1}{z^{i+1}}), \quad i \ge 0.
\ee
Following the argument above, we eventually find
\begin{align}
\Big|\bD^{2j}  \Big( \frac{Cz-\tilde M}{\zeta^2}\Big) \Big| &\leq C\Big( z^{-\frac{2j}{3} -1 }+ \tE^{\frac12}_{\le 2j}  z^{-c j -1 +\delta }+ \tE_{\le 2j}  z^{ (-c+2\delta) j -1  }\Big).  \label{bound (6.313)}
\end{align}
Thus, for $j\geq 1$,
\begin{align}
&\int_{r_\ast e^s}^\infty (1+z)^{2cj -\alpha} \Big| \bD^{2j} \Big(\frac{Cz-\tilde M}{\zeta^2}  \Big)  \Big|^2\notag\\
&\leq C\Big( \frac{1}{(r_\ast e^s)^{ 2 (\frac23 -c) j +\alpha +1 }} +  \frac{1}{(r_\ast e^s)^{ \alpha +1- 2\delta }}\tE_{\le 2j} +  \frac{1}{(r_\ast e^s)^{  \alpha +1- 4\delta j }} \tE_{\le 2j}^2 \Big).\label{E:FLAT3}
\end{align} 
Finally,
\beqa
&\int_{r_\ast e^s}^\infty \kappa\Big| \bD^{2j} \Big(\frac{Cz-\tilde M}{\zeta^2}  \Big)  \Big|^2\leq C\kappa(r_*e^s)^{-1-\frac{4j}{3}}  +C(r_*e^s)^{-\frac43}\tE_{\le 2j}\label{E:FLAT4}
\eeqa
also. Combining estimates~\eqref{E:FLAT1}--\eqref{E:FLAT2} and~\eqref{E:FLAT3}--\eqref{E:FLAT4} and noting $\frac23<\frac{\al+1}{2}-2\de$, we conclude.
\end{proof}


\subsection{Error term estimates: commutator bounds}


The purpose of the section is to prove bounds on the two commutator-type nonlinear 
remainder terms arising on the right hand side of~\eqref{Energy-high}, arising from $\mathcal{R}_{2j}$ and $\mathcal{N}_{2j}$. To this end, in each instance, we split the domain of integration into interior and exterior regions, for example as
\begin{align}
\int \chi_{2j} \frac{\bzeta_z^2}{\zeta_z^2} \mathcal R_{2j}\theta  \phi_{2j} \diff z = \int_0^{Z_0} \chi_{2j} \frac{\bzeta_z^2}{\zeta_z^2} \mathcal R_{2j}\theta  \phi_{2j}\diff z + \int_{Z_0}^\infty \chi_{2j} \frac{\bzeta_z^2}{\zeta_z^2} \mathcal R_{2j}\theta  \phi_{2j} \diff z.
\end{align}
We then use different 
strategies to estimate the interior and exterior contributions. For the interior bounds, we rely on the maximal accretivity estimate and the Hilbert spaces $\H^{2j}_{Z_0}$. In the exterior region, we crucially use  the largeness of $Z_0\gg1$ and the precise structure of the weights $\chi_j$ along with the technical estimates of Section~\ref{S:TECHNICAL}.

With a view also to proving Proposition~\ref{P:HMZDIFF} in the sequel, we also provide an estimate for the errors arising from a difference of solutions.


\begin{lemma}\label{L:COMMBOUND1}
Under the assumptions of Proposition~\ref{P:EE2},
for any $j\in\{1,\dots M\}$, the following bound holds
\beq
\Big| \int \chi_{2j} \frac{\bzeta_z^2}{\zeta_z^2} \mathcal R_{2j}\theta \phij \diff z\Big| \le C(Z_0)\|\Phi\|_{\mathcal{H}^{2j-1}_{Z_0}}\tE_{2j}^{\frac12} + CZ_0^{-\frac23}\tE_{\leq 2j},
\eeq
where we recall the commutator term $\mathcal R_{2j}$ from~\eqref{E:REMAINDER}.

Moreover, under the assumptions of Proposition~\ref{P:HMZDIFF}, setting $\vartheta=\theta_1-\theta_2$ and $\varphi=\phi_1-\phi_2$, for $j\leq \m$,
\beq\label{E:COMMBOUND1DIFF}
\Big| \int_0^{Z_0} g_{2j} \Big(\frac{\bzeta_z^2}{\zeta_{1,z}^2} \mathcal R_{2j}\theta_1-\frac{\bzeta_z^2}{\zeta_{2,z}^2} \mathcal R_{2j}\theta_2\Big) \varphi_{2j} \diff z\Big| \le C(Z_0)\|\Phi_1-\Phi_2\|_{\mathcal{H}^{2j-1}_{Z_0}}  \|\Phi_1-\Phi_2\|_{\mathcal{H}^{2j}_{Z_0}} .
\eeq
\end{lemma}


\begin{proof}
\textit{Step 1: Interior estimate.}
We observe that the term $\mathcal{R}_{2j}\theta$ is non-zero only for $j\geq 1$ and the weight $\chi_{2j}=g_{2j}(\kappa+(1+z)^{2cj -\al})\leq C(Z_0)$.

Using~\eqref{E:REMAINDER}, we have for the interior part
\begin{align}
\int_0^{Z_0} \chi_{2j} \frac{\bzeta_z^2}{\zeta_z^2} \mathcal R_{2j}\theta  \phi_{2j}  \diff z 
&= \sum_{\ell=1}^{2j} \sum_{Q\in \mathcal X_\ell, P\in \mathcal Y_{2j+2-\ell}} c^\ell_{PQ} 
\int_0^{Z_0} \chi_{2j} \frac{\bzeta_z^2}{\zeta_z^2}   |P\bG| |Q \theta|  |\phij| \diff z \notag\\
& \le C\sum_{\ell=1}^{2j} \sum_{Q,P}
\int_0^{Z_0} \chi_{2j}   |P\bG| |Q \theta|  |\phij| \diff z \notag\\
& \le C\sum_{\ell=1}^{2j} \sum_{Q,P} \int_0^{Z_0}\chi_{2j}^{\frac12} |\phij| \big(\sqrt{\kappa}+(1+z)^{cj-\frac\al2}\big)(1+z)^{-\frac{4+2j+2-\ell}{3}}|Q\th| \diff z \notag\\
& \le C(Z_0) \tE_{2j}^{\frac12} \sum_{\ell=1}^{2j} \sum_{Q\in\mathcal X_\ell}\Big(\int_0^{Z_0}|Q\th|^2\diff z\Big)^{\frac12}, \label{E:QUAD1}
\end{align}
where we have used the a priori bound $\big|\frac{\bzeta_z^2}{\zeta_z^2}\big|\leq C$, Lemma~\ref{L:PGBOUND}, $c<\frac23$, and $\ell\le2j$. 
From Lemma~\ref{L:EQUIVALENCE}, we estimate, using $1\leq\ell\leq 2j$,
\beq
\|Q\theta\|_{L^2(0,Z_0)}\leq C(Z_0)\|\Phi\|_{\H^{2j-1}_{Z_0}}.
\eeq
Using this in~\eqref{E:QUAD1} we conclude that for any $1\le j\le M $ we have
\beq
\Big|\int_0^{Z_0} \chi_{2j} \frac{\bzeta_z^2}{\zeta_z^2} \mathcal R_{2j}\theta  \phij  \diff z \Big|\leq C(Z_0)\|\Phi\|_{\H^{2j-1}_{Z_0}}\tE_{2j}^{\frac12}.\label{E:COMM1}
\eeq

In order to prove~\eqref{E:COMMBOUND1DIFF}, we first note from~\eqref{E:REMAINDER} that $\mathcal{R}_{2j}[\th_1]-\mathcal{R}_{2j}[\th_2]$ admits the very simple difference structure
\[\mathcal R_{2j} [\theta_1]-\mathcal R_{2j} [\theta_2] = \sum_{\ell=1}^{2j} \sum_{Q\in \mathcal X_\ell, P\in \mathcal Y_{2j+2-\ell}} c^\ell_{PQ} P\bG \,Q( \theta_1-\th_2).
\] Therefore a straightforward adaptation of the argument above yields the claimed estimate~\eqref{E:COMMBOUND1DIFF}.\\

\textit{Step 2: Exterior region.} In the exterior region $[Z_0,\infty)$, we first split the weight $\chi_{2j}\leq \kappa+(1+z)^{2cj-\al}$ and then, using the formula for $\mathcal R_{2j}$~\eqref{E:REMAINDER} as above,  we have
\begin{align}
\int_{Z_0}^{\infty}&\, (1+z)^{2cj-\al} \Big|\frac{\bzeta_z^2}{\zeta_z^2} \mathcal R_{2j}\theta  \phij \Big| \diff z 
 \le C\sum_{\ell=1}^{2j} \sum_{Q\in \mathcal X_\ell \atop P\in \mathcal Y_{2j+2-\ell}}
\int_{Z_0}^{\infty} \chi_{2j}   |P\bG| |Q \theta|  |\phij| \diff z \notag\\
& \le  C\sum_{\ell=1}^{2j}  \sum_{Q\in \mathcal X_\ell \atop P\in \mathcal Y_{2j+2-\ell}}\int_{Z_0}^\infty \chi_{2j}^{\frac12}\phij (1+z)^{(\frac{\ell}{2}-1)c-\frac{\al}{2}-\frac23}|Q\th| (1+z)^{(c-\frac23)(j-\frac\ell2+1)-\frac23}\diff z\notag\\
& \le C \sum_{\ell=1}^{2j}  \sum_{Q\in \mathcal X_\ell \atop P\in \mathcal Y_{2j+2-\ell}}(1+Z_0)^{(c-\frac23)(j-\frac {\ell}2+1)-\frac23} \tilde\E_{2j}^{\frac12}\tE_{\leq 2j}^{\frac12}\notag\\
& \le C Z_0^{-\frac 23}\tE_{\leq 2j},\label{E:COMM2}
\end{align}
where we have used Lemma~\ref{L:QBOUNDS}. 
By analogous bounds we show that 
\begin{align}
\kappa&\, \int_{Z_0}^{\infty}  \frac{\bzeta_z^2}{\zeta_z^2} \mathcal R_{2j}\theta  \phij  \diff z \le C \|\sqrt \kappa \phij\|_{L^2([Z_0,\infty))} \sum_{\ell=2}^{2j} \sum_{Q\in X_\ell}\|\sqrt\kappa Q\th z^{-1}\|_{L^2([Z_0,\infty))}\|z^{1-\frac{4+2j-\ell+2}{3}}\|_{L^\infty([Z_0,\infty))}\notag\\
&+ C \|\sqrt \kappa \phij\|_{L^2([Z_0,\infty))}\sum_{Q\in X_1}\|\sqrt\kappa Q\th z^{-\frac43}\|_{L^2([Z_0,\infty))}\|z^{-\frac{2j-1+2}{3}}\|_{L^\infty([Z_0,\infty))}\notag\\
& \le CZ_0^{-1}\tE_{\leq 2j},\label{E:COMM3}
\end{align}
where we have used Lemma~\ref{L:KAPPAQBOUND}. 
Adding up~\eqref{E:COMM1}--\eqref{E:COMM3} we obtain the claim of the lemma.
\end{proof}

\begin{lemma}\label{L:COMMBOUND2}
Under the assumptions of Proposition~\ref{P:EE2}, 
for any $j\in\{1,\dots M\}$ the following bound holds 
\begin{align}
\Big| \int  \chi_{2j} \mathcal N_{2j}[\theta] \phij \diff z\Big| \le C(Z_0)\tE_{\leq 2j} \tE_{\leq 2[\frac{j+3}{2}]}^{\frac12} +C   Z_0^{-\frac12}\tE_{\leq 2j}^{\frac32}+C Z_0^{-\frac{\al}{2}+(2M+1)\de-\frac{c}{2}-\frac16}\tE_{\leq 2j}^{2} ,
\end{align}
where we recall the error term $\mathcal N_{2j}$ from~\eqref{E:NTWOMDEF}.

Moreover, under the assumptions of Proposition~\ref{P:HMZDIFF}, setting $\vartheta=\theta_1-\theta_2$ and $\varphi=\phi_1-\phi_2$, for $j\leq \m$,
\begin{align}
\Big| \int_0^{Z_0}  g_{2j} \Big(\mathcal N_{2j}[\theta_1]-\mathcal N_{2j}[\theta_2]\Big) \varphi_{2j} \diff z\Big| \le C(Z_0)\|\Phi_1-\Phi_2\|_{\mathcal{H}^{2j}_{Z_0}}^2\big(\tE_{2j}[\Phi_1]^{\frac12}+\tE_{2j}[\Phi_2]^{\frac12}\big) .\label{E:COMMBOUND2DIFF}
\end{align}
\end{lemma}


\begin{proof}
For every term, we will estimate 
\[\Big| \int  \chi_{2j} \mathcal N_{2j}[\theta] \phij \diff z\Big| \le \bigg(\int \chi_{2j}|\mathcal N_{2j}[\theta]|^2\,\dif z\bigg)^{\frac12}\tE_{2j}^{\frac12},\]
and so we focus on the first factor on the right. Applying the product rule Lemma~\ref{L:XYPRODCHAIN} repeatedly, we may write
\begin{align}
\mathcal N_{2j}[\th] &= \sum_{\ell=0}^{2j-1} \sum_{P\in \Y_{2j-\ell} \atop Q\in \X_{\ell}}c_{PQ}  P \big(\frac{\bzeta_z^2}{\zeta_z^2}\big) Q K\th.\label{E:N2JEXPANSION}
\end{align}
Note that $\frac{\bzeta_z^2}{\zeta_z^2} = f(\frac{\bp\th}{\bp\bzeta})$ for $f(x)=\frac{1}{(1+x)^2}$.\\

\noindent \textit{Step 1: Interior region estimate}\\
As in the proof of Lemma~\ref{L:NBOUND}, an inductive argument based on Lemma~\ref{L:XYPRODCHAIN} shows that for any $P\in \Y_{2j-\ell}$, $0\le \ell\le2j-1$, $P \big(\frac{\bzeta_z^2}{\zeta_z^2}\big)$ is a sum of terms of the form
\begin{align}
f^{(k)}(\sqrt{\bG} \bp\theta)  P_{j_1}(\sqrt{\bG} \bp\theta) P_{j_2}(\sqrt{\bG} \bp\theta) \dots P_{j_{k}}(\sqrt{\bG} \bp\theta)
\end{align}
for some $j_1,\dots,j_{k+2}\in\mathbb N\cup\{0\}$, where 
\[
1\le k \le 2j-\ell, \ \ P_{j_n}\in \mathcal Y_{j_n},  \ \ n\in\{1,\dots, k\}, \ \ \ j_1+j_1+\dots j_{k} = 2j-\ell.
\]
By the a priori bound~\eqref{E:APRIORIMAINPT}, we have $|f^{(k)}(\sqrt{\bG}\bp\theta)|\leq C$ for all $k\geq 1$.
In the interior region, $z\in(0,Z_0)$, as in the proof of Lemma~\ref{L:NBOUND}, we easily find
\begin{align}
|P \big(\frac{\bzeta_z^2}{\zeta_z^2}\big)| \leq C(Z_0) \sum_{k=1}^{2j-\ell+1}\sum_{P'\in \X_k}|P'\th|,
\end{align}
where we used the bound $2j-\ell\ge1$.
Similarly to the proof of Lemma~\ref{L:KEYCOMM} we can show that 
\begin{align}
|Q K\th| \leq C(Z_0) \sum_{k=0}^{\ell+2}\sum_{Q'\in \X_k} |Q'\th|, \ \ Q\in \X_\ell, z\leq Z_0.
\end{align}
 Observe that for $0\leq\ell\leq 2j-1\leq 2M-1$, at most one of $2j-\ell+1$ and $\ell+2$ is greater than or equal to $j+2$, while the other is less than or equal to $j+1$. Applying  the Hardy-Sobolev embedding~\eqref{E:XKHARDY} and Lemma~\ref{L:EQUIVALENCE}, we deduce
\beqa\label{E:N2JINT}
\Big| \int_0^{Z_0} \chi_{2j}  \mathcal N_{2j}[\theta] \phij \diff z\Big|\leq &\, C(Z_0)\sqrt{\kappa}\Big(\int_0^{Z_0}\chi_{2j}\phi_{2j}^2\Big)^{\frac12}\Big(\int_0^{Z_0}\sum_{k=0}^{2j+1}\sum_{Q'\in \X_k} |Q'\th|^2\Big)^{\frac12}\Big\|\sum_{k=0}^{j+1}\sum_{P'\in \X_k} |P'\th|\Big\|_{L^\infty(0,Z_0)}^{\frac12}\\
\leq&\, C(Z_0)\sqrt{\kappa}\tE_{2j}^{\frac12}\|\Phi\|_{\mathcal{H}^{2j}_{Z_0}}\|\Phi\|_{\H^{j+2}_{Z_0}}\leq C(Z_0)\tE_{\leq 2j}  \tE_{\leq 2[\frac{j+3}{2}]}^{\frac12}. 
\eeqa

Before turning to the exterior region, we now note that a similar argument is sufficient to deduce the difference estimate~\eqref{E:COMMBOUND2DIFF}. In particular, it is straightforward to see that, expanding $\mathcal{N}_{2j}$ as in~\eqref{E:N2JEXPANSION}, any given term in the sum satisfies the difference identity
\[P \big(\frac{\bzeta_z^2}{\zeta_{1,z}^2}\big) Q K\th_1-P \big(\frac{\bzeta_z^2}{\zeta_{2,z}^2}\big) Q K\th_2=P \Big(\frac{\bzeta_z^2}{\zeta_{1,z}^2}\frac{\zeta^2_{2,z}-\zeta^2_{1,z}}{\zeta_{2,z}^2}\Big) Q K\th_1+P \big(\frac{\bzeta_z^2}{\zeta_{2,z}^2}\big) (Q K\th_1-QK\th_2),
\]
which is enough to apply a simple adaptation of the previous argument and conclude, as $\zeta_{2,z}^2-\zeta_{1,z}^2=(\th_{2,z}-\th_{1,z})(\zeta_{2,z}+\zeta_{1,z})$.\\

\noindent \textit{Step 2: Exterior region, growing weight.}\\
In the exterior region, we again split the weight $\chi_{2j}$ into the $\kappa$ term and the growing weight $(1+z)^{2cj-\al}$. We treat the contribution from the latter first. Return to~\eqref{E:N2JEXPANSION}. Setting $f(x)=(1+x)^{-2}$ and $x(z)=\frac{\pa_z\theta}{\pa_z\bzeta}$, clearly $f(x(z))$ satisfies the assumptions of Lemma~\ref{L:THETAZNONLINEAREST} and hence satisfies estimate~\eqref{E:THETAZNONLINEAREST}.
Thus, for $P\in \Y_{2j-\ell}$, as $Pf=\sum_{k=1}^{2j-\ell}c_{jk}z^{k-\frac{2j-\ell}{3}}\pa_z^k f$, we find
\beqa\label{E:N2JPTERMS}
\big|Pf(x(z))\big|\leq&\, C\sum_{\substack{ k=1 \\ k<2j}}^{2j-\ell}\Big(z^{k-\frac{2j-\ell}{3}-a_{k+1}}\tE^{\frac12}_{\le 2[\frac{k+2}{2}] }+z^{k-\frac{2j-\ell}{3}-(\frac{c}{2}+\frac23-\de)k}\tE_{\le 2[\frac{k+1}{2}] }\Big) + z^{-\frac23}|\bD^{2j+1}\theta|\mathbf{1}_{\ell=0}\\
\leq&\,C\Big(z^{-\frac{c}{2}(2j-\ell)+\de}\tE^{\frac12}_{\le 2j }+z^{-\frac{c}{2}(2j-\ell)+(2j-\ell)\de}\tE_{\le 2j } + z^{-\frac23}|\bD^{2j+1}\theta|\mathbf{1}_{\ell=0}\Big),
\eeqa
where we have used that $k-\frac{2j-\ell}{3}-a_{k+1}\leq -\frac{c}{2}(2j-\ell)+\de$ and $k-\frac{2j-\ell}{3}-(\frac{c}{2}+\frac23-\de)k\leq -\frac{c}{2}(2j-\ell)+k\de$ for $k\leq 2j-\ell$.

Turning to the second factor in each term of~\eqref{E:N2JEXPANSION}, we first consider $Q\in\X_\ell$ for $\ell\leq 2M-2$ and use that for any $Q\in \X_k$, we have $Qg=\sum_{i=0}^{k}c_{ki}z^{i-\frac{k}{3}}\pa_z^i g$ to estimate, from Lemmas~\ref{L:PGBOUND}--\ref{L:XYPRODCHAIN},
\beqa
|QK\theta|\leq&\, C\sum_{k=1}^{\ell+2}\sum_{Q\in\X_k}(1+z)^{-\frac{4+\ell+2-k}{3}}|Q\theta|\leq C\sum_{i=0}^{\ell+2}(1+z)^{i-\frac{\ell}{3}-2}|\pa_z^i\theta|\\
\leq&\, C\sum_{i=0}^{\ell+2}(1+z)^{i-\frac{\ell}{3}-2-a_i}\tE_{\le \IUP}^{\frac12} \leq C(1+z)^{-\frac{c}{2}\ell+\de -\frac23-\frac{c}{2}}\tE_{\leq 2[\frac{\ell+3}{2}]}^{\frac12},
\eeqa
where we have applied~\eqref{Linfty1} and  used $i-\frac{\ell}{3}-a_i\leq -\frac{c}{2}\ell+\de +\frac43-\frac{c}{2}$. If $\ell=2j-1$, we split the sum and  argue similarly to~\eqref{E:DAMPENINGTOPORDER} and obtain
\beqa\label{E:N2JQTERMS}
|QK\theta|\leq&\, C\sum_{k=1}^{2j+1}\sum_{Q\in\X_k}(1+z)^{-\frac{4+(2j-1)+2-k}{3}}|Q\theta|\leq C\sum_{i=0}^{2j+1}(1+z)^{i-\frac{2j-1}{3}-2}|\pa_z^i\theta|\\
\leq&\, C\sum_{i=0}^{2j}(1+z)^{i-\frac{2j-1}{3}-2-a_i}\tE_{\le \IUP}^{\frac12} + z^{\frac{2(2j-1)}{3}}|\pa_z^{2j+1}\theta|\\
\leq&\,C(1+z)^{-\frac{c}{2}(2j-1)+\de -\frac23-\frac{c}{2}}\tE_{\leq 2j}^{\frac12} +z^{-\frac{4}{3}}|\bD^{2j+1}\theta|,
\eeqa
where we have used $\frac{2(2j-1)}{3}-a_{2j}-1\leq -cj+\de-\frac23$ in the last step.

Taking the product of terms in~\eqref{E:N2JPTERMS}--\eqref{E:N2JQTERMS}, we therefore obtain 
\begin{align}
\int_{Z_0}^\infty&\, (1+z)^{2cj-\al}|\mathcal N_{2j}[\theta]|^2\leq C \int_{Z_0}^\infty (1+z)^{2cj-\al} \sum_{\ell=0}^{2j-1} \sum_{P\in \Y_{2j-\ell} \atop Q\in \X_{\ell}} \big| P \big(\frac{\bzeta_z^2}{\zeta_z^2}\big)\big|^2 |Q K\th|^2\notag\\
\leq C &\,\int_{Z_0}^\infty  \sum_{\ell=0}^{2j-1}(1+z)^{2cj-\al}\Big(z^{-c(2j-\ell)+2\de}\tE_{\leq 2j}+z^{-c(2j-\ell)+2(2j-\ell)\de}\tE_{\leq 2j}^2 + z^{-\frac43}|\bD^{2j+1}\theta|^2\mathbf{1}_{\ell=0}\Big)\notag\\
&\ \ \ \ \ \ \ \times\Big((1+z)^{-c\ell+2\de -\frac43-c}\tE_{\leq 2j} +z^{-\frac{8}{3}}|\bD^{2j+1}\theta|^2\mathbf{1}_{\ell=2j-1}\Big)\notag\\
\leq C &\, \sum_{\ell=0}^{2j-1}\int_{Z_0}^\infty\Big((1+z)^{-\al+4\de-c-\frac43}\tE_{\leq 2j}^2+(1+z)^{-\al+2(2j-\ell+1)\de-c-\frac43}\tE_{\leq 2j}^{3} \notag\\
&\ \ \ \ \ \ +\Big(z^{2cj-\al-c+2\de-\frac83}\tE_{\leq 2j}+z^{2cj-\al-c+2\de-\frac83}\tE_{\leq 2j}^2\Big)|\bD^{2j+1}\theta|^2\Big)
\notag\\
\leq C &\,Z_0^{-\al+4\de-c-\frac13}\tE_{\leq 2j}^2+Z_0^{-\al+2(2j+1)\de-c-\frac13}\tE_{\leq 2j}^{3} +  Z_0^{-c- \frac43+2\delta }  (\tE_{\leq 2j}^2+\tE_{\leq 2j}^3).\label{E:N2JEXT1}
\end{align}

\noindent\textit{Step 3: Exterior region, constant weight}\\
To control the contribution arising from the term $\kappa$  in the weight $\chi_{2j}$, we employ a different strategy. Returning to~\eqref{E:N2JEXPANSION}, we estimate from Lemma~\ref{L:NONLINEARPQEXPANSION}
\beqa\label{E:N2JEXT2}
\kappa&\,\int_{Z_0}^\infty |\mathcal N_{2j}[\theta]|^2\\
&\leq C\kappa\int_{Z_0}^\infty\Big(\sum_{\ell=0}^{j-1}\sum_{P\in \Y_{2j-\ell} \atop Q\in \X_{\ell}}|P(f(\frac{\bp\th}{\bp\bzeta}))|^2|QK\th|^2+\sum_{\ell=j}^{2j-1}\sum_{P\in \Y_{2j-\ell} \atop Q\in \X_{\ell}}|P(f(\frac{\bp\th}{\bp\bzeta}))|^2|QK\th|^2\Big)\\
&\leq C\sum_{\ell=0}^{j-1}\sum_{P\in \Y_{2j-\ell} \atop Q\in \X_{\ell}}\Big\|z^{\frac13}QK\theta\Big\|_{L^\infty(Z_0,\infty)}^2\kappa\int_{Z_0}^\infty z^{-\frac23}\Big|P\Big(f(\frac{\bp\th}{\bp\bzeta})\Big)\Big|^2\\
&\ \ \ \ +\sum_{\ell=j}^{2j-1}\sum_{P\in \Y_{2j-\ell} \atop Q\in \X_{\ell}}\Big\|P\Big(f\big(\frac{\theta_z}{\bzeta_z}\big)\Big)\Big\|_{L^\infty(Z_0,\infty)}^2\kappa\int_{Z_0}^\infty|QK\th|^2\\
&\leq CZ_0^{-1}\tE_{\leq 2j}^2.
\eeqa
Summing~\eqref{E:N2JINT},~\eqref{E:N2JEXT1} and~~\eqref{E:N2JEXT2}, and noting that $1<\al-4\de+c+\frac13$, we conclude.
\end{proof}


\subsection{The error term $\mathfrak R[\th]$}


In this section we estimate on the most complex error term $\int \chi_{2j} \bD^{2j}  \mathfrak R[\th] \phij \diff z$ appearing on the right-hand side of~\eqref{Energy-high}.
The error $\mathfrak R[\th]$ defined in~\eqref{E:MATHFRAKRDEF} is very sensitive to the different behaviour of the solution in the interior ($z\le Z_0$) and the exterior ($z\ge Z_0$). As a consequence,
we will have to rewrite $\mathfrak R[\th]$ in the interior region in a way that will allow us to close the estimates. In its current form~\eqref{E:MATHFRAKRDEF} we are facing a possible derivative loss issues
having to do the with the $\bzeta\sim_{z\to0}z^{\frac13}$-singularity at the comoving origin $z=0$.


\subsubsection{Exterior estimate}


We start  with the exterior estimate.

\begin{proposition}[Exterior bound]\label{P:EXT1}
Under the assumptions of Proposition~\ref{P:EE2}, 
for any $j\in\{1,\dots M\}$ the following bound holds: 
\begin{align}
\Big|\int_{Z_0}^\infty \chi_{2j} \bD^{2j}\mathfrak R[\th] \phij \diff z \Big| \le C\Big(Z_0^{-\frac{\al+1}{2}+\de}\tE_{\leq 2j} + Z_0^{-\frac{\al+1}{2}+2\de j}\tE_{\leq 2j}^{\frac32}\Big), 
\end{align}
where we recall definition~\eqref{E:MATHFRAKRDEF} of $\mathfrak R[\th]$.
\end{proposition}

\begin{proof} We recall~\eqref{E:MATHFRAKRDEF}:
\begin{align} 
\mathfrak R[\theta] &=  \frac{\bzeta_z^2}{\zeta_z^2} \mathcal V_1 \th  - \frac{2\th}{\zeta\bzeta}   + \frac{\CLP z\theta(2\bzeta+\th)}{\zeta^2\bzeta^2}   
 -  \frac{\bzeta_{zz}}{\bzeta_z^2 \zeta_z^2} (\pa_z \theta)^2.  \label{E:MATHFRAKRDEF2}
\end{align}
We present in detail the bound for the second term and the last term on the right-hand side of~\eqref{E:MATHFRAKRDEF2}, as the bounds on the remaining two terms  follow analogously. 

We begin with the last term, which contains a quadratic nonlinearity of the highest order, $(\pa_z\theta)^2$. Note from Lemma~\ref{L:ZETABAR} that 
\be
 \frac{\bzeta_{zz}}{\bzeta_z^2}=: h (z) = O(z^{-3}), \quad \pa_z^i h (z) = O(z^{-3-i}). 
\ee
Let $j=0$. We use \eqref{theta_bound} and \eqref{Linfty1} with $k=1$ to get  
\begin{align*}
\int_{Z_0}^\infty \kappa(1+z)^{-\alpha } \left( \frac{\bzeta_{zz}}{\bzeta_z^2} (\frac{\pa_z\theta}{\zeta_z})^2 \right)^2 \diff z
& \leq C Z_0^{-6 } \|\pa_z\th\|_{L^\infty_{\{z\ge Z\}}}^2 \int_{Z_0}^\infty \kappa(1+z)^{-\alpha} |\pa_z\th|^2 \diff z   \notag\\
& \leq C Z_0^{-6}   C_*^2\eps_0 \tE_0 
\end{align*}
where we have used the pointwise estimate~\eqref{E:APRIORIMAINPT} for $\th_z$ and hence also $\zeta_z$.
To handle $j\ge 1$,  we use the same strategy as in the proof of Lemma~\ref{L:CUTOFFBOUND}. To that end, 
we introduce further notations 
\begin{align*}
 \frac{\bzeta_{zz}}{\bzeta_z^2} (\frac{\pa_z\theta}{\zeta_z})^2 = h(z) f(x(z)), \quad x= \frac{\pa_z\theta}{\pa_z\bzeta}, \quad f(x) = (\frac{x}{1+x})^2,
\end{align*}
so that $f(x(z))$ satisfies the estimate~\eqref{E:THETAZNONLINEAREST3}.

Therefore, using the representation formula $\bD^{2j} (hf) = 
\sum_{k=0}^{2j} c_{jk} z^{k -\frac{2j}{3}}\pa_z^{k}(hf)$ and the product formula $\pa_z^k (h f ) =   \sum_{\ell =0}^k c_\ell \pa_z^{k-\ell} h \pa_z^\ell f $, we deduce that 
 \begin{align}
\Big|\bD^{2j} \left( \frac{\bzeta_{zz}}{\bzeta_z^2} (\frac{\pa_z\theta}{\zeta_z})^2 \right) \Big| &\leq C \sum_{k=0}^{2j} z^{k-\frac{2j}{3}}  \sum_{\substack{\ell=0 \\ \ell< 2j}}^k z^{-(k-\ell+3)} ( z^{-a_1-a_{\ell +1}}  +  z^{-(\frac{c}{2} +\frac23-\delta )\ell }   )   \tE_{\le 2j}
\notag \\ 
&\quad + Cz^{ -3-\frac23+\delta }  \tE^\frac12 |\bD^{2j+1} \theta|    \notag\\
&\leq C   \Big(\tE_{\le 2j}
z^{-c j -3 +2\delta j  } +  z^{ -3-\frac23+\delta }  \tE_2^\frac12 |\bD^{2j+1} \theta|\Big),
\end{align}
where we have used~\eqref{E:AI2}.
Hence 
\begin{align}
&\int_{Z_0}^\infty (1+z)^{2cj-\alpha }\left| \bD^{2j} \left( \frac{\bzeta_{zz}}{\bzeta_z^2} (\frac{\pa_z\theta}{\zeta_z})^2 \right)\right|^2 dz\notag \\
&\leq C    \tE^2_{\le 2j}
 \int_{Z_0}^\infty  z^{-6+4\delta j -\alpha } dz + C \tE_2 \int_{Z_0}^\infty (1+z)^{2cj -\alpha}z^{-6+2\delta } \frac{| \bD^{2j+1} \theta |^2}{|\bar\pa_z \zeta|^2} dz    \notag \\
&\leq C (\frac{1}{Z_0^{5+\alpha -4\delta j }} + \frac{1}{Z_0^{6-2\delta} })   \tE^2_{\le 2j}
\leq C Z_0^{-4} \tE^2_{\le 2j},
\end{align}
where we recall that $4\de M<\al+1$ from~\ref{item:a2}.

To control the contribution arising from the $\kappa$ term in the weight $\chi_{2j}$, we argue as in the proof of Lemma~\ref{L:CUTOFFBOUND}. First, we note that $|f(x(z))|\leq C\eps_0$. More precisely,
we observe that for any $Q\in\X_{2j-\ell}$, we have
\[
|Qh|\leq C\sum_{i=0}^{2j-\ell}z^{i-\frac{2j-\ell}{3}}|\pa_z^ih|\leq Cz^{-3-\frac{2j-\ell}{3}}.
\]
Thus, from the improved bound~\eqref{E:PJNL2IMP} from Lemma~\ref{L:NONLINEARPQEXPANSION},
\beqa
\kappa&\,\int_{Z_0}^\infty|\bD^{2j}(f(x(z))h(z))|^2\leq C\kappa\int_{Z_0}^\infty\sum_{\ell=0}^{2j}\sum_{\substack{P\in\Y_\ell \\ Q\in\X_{2j-\ell}}}|Pf(x(z))|^2|Qh|^2\\
\leq&\,C\int_{Z_0}^\infty \kappa \Big(z^{-6-\frac{4j}{3}}|\pa_z\th|^4+\sum_{\ell=1}^{2j}\sum_{P\in\Y_\ell}z^{-\frac{16}{3}-2\frac{2j-\ell}{3}}z^{-\frac23}|Pf|^2\Big)\\
\leq&\,C\Big( C_*^2\eps_0 Z_0^{-6-\frac{4}{3}+\al}\tE_0 +\tE_{\leq 2j}(C_*^2\eps_0+Z_0^{-\frac13}\tE_{\leq 2[\frac{j+3}{2}]})Z_0^{-\frac{16}{3}}\Big),
\eeqa
where we have also used~\eqref{E:APRIORIMAINPT} and~\eqref{theta_bound}.

Next we derive the bound for the second term on the right-hand side of~\eqref{E:MATHFRAKRDEF2}: 
\[
\int_{Z_0}^\infty (1+z)^{2cj -\alpha} \left| \bD^{2j} \left( \frac{\theta}{\zeta\bzeta } \right)\right|^2 dz.
\]
If $j=0$, using \eqref{theta_bound}, we deduce that 
\[
\int_{Z_0}^\infty (1+z)^{-\alpha} \left(  \frac{\theta}{\zeta\bzeta } \right)^2  dz \le \frac{1}{Z_0^2}\int_{Z_0}^\infty (1+z)^{-\alpha-2} \theta^2 dz \leq C Z_0^{-2} \mathcal E_0.
\]
To deal with $j\ge 1$, we write  
\[
\frac{\theta}{\zeta\bzeta }  = h(z) f(x(z)), \quad h(z) =  \frac{1}{\bzeta },  \ \  x(z) = \frac{\theta}{\bzeta}, \ \  f(x)=\frac{x}{1+x}.
\]
Applying~\eqref{E:THETAZNONLINEAREST4} and arguing as in the proof of Lemma~\ref{L:CUTOFFBOUND}, we deduce the following bound analogous to~\eqref{bound (6.313)} :
\be
 \left| \bD^{2j} \left( \frac{\theta}{\zeta\bzeta } \right)\right|  \leq C \Big(z^{-\frac{2j}{3} -1}|\theta| + \tE_{\leq 2j}^{\frac12}  z^{-c j -1 +\delta }+ \tE_{\leq 2j}  z^{ (-c+2\delta) j -1  } \Big),
\ee
which in turn leads to 
\be
\int_{Z_0}^\infty  (1+z)^{2cj -\alpha} \left| \bD^{2j} \left( \frac{\theta}{\zeta\bzeta } \right)\right|^2 dz \leq C\Big(  \kappa^{-1}\frac{1}{Z_0^{ 2 (\frac23 -c) j  }}\tE_{\leq 2j} +  \frac{1}{Z_0^{ \alpha +1- 2\delta }}\tE_{\leq 2j} +  \frac{1}{Z_0^{  \alpha +1- 4\delta j }} \tE_{\leq 2j}^2\Big) ,
\ee
where we have used~\eqref{E:theta} in the last line.

The first and third terms of~\eqref{E:MATHFRAKRDEF2} can be estimated analogously by writing them as the product of coefficients with the $z$-decay, and analytic functions of $\frac{\theta}{\bzeta}$ and $\frac{\pa_z\theta}{\pa_z\bzeta}$, and applying the same arguments. 
\end{proof}


\subsubsection{Interior estimate}


As written, term $\mathfrak R[\th]$ features several terms which each, individually scale like $z^{-\frac13}$ near $z=0$ assuming the expected behaviour $\th\sim_{z\to0}z^{\frac13}$; 
this is too singular to be bounded by our energy framework, in particular using the Hardy inequality. There is however a hidden cancellation between these terms, 
which requires some algebraic manipulation before we can display it clearly. We therefore first prove the following lemma.


\begin{lemma}\label{L:MATHFRAKRNEW}
Let $\th\in C^1([0,\infty)$ be given. There exist functions $g_1$ and $g_2$ with $g_1,\,z^{\frac23}g_2\in\DZodd$ such that the remainder term $\mathfrak R[\th]$, defined in~\eqref{E:MATHFRAKRDEF}, satisfies the following identity:
\begin{align}
\mathfrak R[\th] & = 
\mathcal S_1[ \th,\th_z] + \mathcal S_2[\th,\th_z] - \frac{\bG z^{\frac83}}{3(\bp\zeta)^2\zeta\bzeta^2}(\bzeta^3)_{zz} \th_z^2\th \notag\\ 
& \ \ \ \ 
 + \Big[\frac43 \frac{\bzeta_z^2}{\zeta_z^2}  z^{-\frac13} \bp^2 \bzeta \bG^\frac32 + \frac{2\CLP z}{\zeta^2\bzeta}\Big]\th
 + \frac{\CLP z\theta^2}{\zeta^2\bzeta^2},  \label{E:MATHFRAKR2}
\end{align}
where
\begin{align}
\mathcal S_1[\th,\th_z] &:= - 2 \frac{1}{(\bp\zeta)^2}  (z^\frac13 g_2)\frac{z^{\frac13}}{\zeta}\big(\bp\bzeta+\frac{\bzeta}{3z^{\frac13}}\big)\th, \label{E:SONEDEF}  \\
\mathcal S_2[\th,\th_z] & : = 2z^{\frac13}\frac{ z^{\frac13}}{(\bp\zeta)^2\zeta} \Big[\bp\big(\frac{\th}{z^{\frac13}}\big)\Big]^2+\frac{z^{\frac13}}{(\bp\zeta)^2\zeta}\Big(  \frac{\bzeta}{z^{\frac13}}g_1 (\bp\th)^2 - 4 z^{\frac13}g_2 \th \bp\th\Big).
\label{E:STWODEF}
\end{align}
\end{lemma}

\begin{proof}
We substitute~\eqref{E:KIDENTITY} into~\eqref{E:MATHFRAKRDEF} and rewrite $\mathfrak R[\th]$ in the form 
\begin{align}
\mathfrak R[\th] & =
 \frac29 z^{-\frac23}\frac{\bzeta_z^2}{\zeta_z^2}\bG \th - \frac2{\bzeta\zeta}\th -\frac{\bzeta_{zz}}{\bzeta_z^2 \zeta_z^2} (\pa_z \theta)^2
+ \Big(\frac43 \frac{\bzeta_z^2}{\zeta_z^2}  z^{-\frac13} \bp^2 \bzeta\bG^\frac32 + \frac{2\CLP z}{\zeta^2\bzeta} \Big)\th 
+ \frac{\CLP z\theta^2}{\zeta^2\bzeta^2}.
 \label{E:NLPHIAUX}
\end{align}
A direct calculation allows us to rewrite the first three terms on the right-hand side of~\eqref{E:NLPHIAUX} in the form
\begin{align}
 & \frac29 z^{-\frac23}\frac{\bzeta_z^2}{\zeta_z^2}\bG \th - \frac2{\bzeta\zeta}\th -\frac{\bzeta_{zz}}{\bzeta_z^2 \zeta_z^2} (\pa_z \theta)^2\notag\\
&=
\frac{\bG z^{\frac83}}{(\bp\zeta)^2\zeta\bzeta^2} \Big(\Big[\frac29 z^{-2}\bzeta^3\bzeta_z^2-2\bzeta_z^4\bzeta\Big]\th   + \Big[\frac29 z^{-2}\bzeta^2\bzeta_z^2\th^2 - \bzeta_{zz}\bzeta^3 \th_z^2 - 4\bzeta \bzeta_z^3\th\th_z \Big] 
- \frac13(\bzeta^3)_{zz} \th_z^2\th \Big),
\end{align}
and we set
\begin{align}
\mathcal S_1[\th] &= \frac{\bG z^{\frac83}}{(\bp\zeta)^2\zeta\bzeta^2}  \Big[\frac29 z^{-2}\bzeta^3\bzeta_z^2-2\bzeta_z^4\bzeta\Big]\th, \label{E:SONEDEF1}  \\
\mathcal S_2[\th,\th_z] & = \frac{\bG z^{\frac83}}{(\bp\zeta)^2\zeta\bzeta^2} \Big[\frac29 z^{-2}\bzeta^2\bzeta_z^2\th^2 - \bzeta_{zz}\bzeta^3 \th_z^2 - 4\bzeta \bzeta_z^3\th\th_z \Big].
\label{E:STWODEF1}
\end{align}
Upon regrouping, it remains to show that $\mathcal S_1$ and $\mathcal S_2$ may be re-written in the forms claimed.

We now set
\begin{align}
g_1= & -\frac{\bzeta_{zz}}{\bzeta_z^2}- \frac2{\bzeta},\qquad  g_2=\frac{\bzeta_z}{\bzeta} - \frac1{3z} ,\label{E:gONEgTWO}
\end{align}
so that clearly, from Lemma~\ref{L:ZETABAR}, $g_1\sim_{z\to0}z^{\frac13}$, $g_2\sim_{z\to0}z^{-\frac13}$ and the claimed properties of $g_1,g_2$ hold.

In particular, we have
\begin{align}\label{E:GTWO1}
\bp\bzeta = \frac{\bzeta}{3z^{\frac13}} + z^{\frac23}\bzeta g_2,
\end{align}
so that
\begin{align}
\mathcal S_1 & =2 \frac{1}{(\bp\zeta)^2\zeta\bzeta}\big[\big(\frac{\bzeta}{3z^{\frac13}}\big)^2-\big(\bp\bzeta\big)^2\big]  = - 2 \frac{1}{(\bp\zeta)^2}  (z^\frac13 g_2)\frac{z^{\frac13}}{\zeta}\big(\bp\bzeta+\frac{\bzeta}{3z^{\frac13}}\big), \label{E:SONESIMPLE}
\end{align}
which proves~\eqref{E:SONEDEF}.

Using~\eqref{E:gONEgTWO}, we rewrite $\mathcal S_2[\th,\th_z]$ from~\eqref{E:STWODEF1} as
\begin{align}
\mathcal S_2[\th,\th_z] & =  
\frac{\bG z^{\frac83}}{(\bp\zeta)^2\zeta\bzeta^2} \Big[\frac29 z^{-2}\bzeta^2\bzeta_z^2\th^2 + (\frac{2\bzeta_z^2}{\bzeta}+g_1\bzeta_z^2)\bzeta^3 \th_z^2 - 4\bzeta \bzeta_z^3\th\th_z \Big] \notag\\
& = \frac{\bG z^{\frac83}}{(\bp\zeta)^2\zeta} \bzeta_z^2 \Big[\frac29 \th^2z^{-2} +2 \th_z^2 + \bzeta g_1 \th_z^2 - 4 (\frac1{3z}+g_2)\th\th_z \Big] \label{E:S2REFORM}.
\end{align}
We next observe that 
\begin{align}
\frac{\bG z^{\frac83}\bzeta_z^2}{(\bp\zeta)^2\zeta}  \Big[\frac29 \th^2z^{-2} +2 \th_z^2  - \frac4{3z} \th \th_z \Big]  = 2\frac{\bG z^{\frac53}\bzeta_z^2}{(\bp\zeta)^2\zeta}z^{\frac53}\Big[\pa_z\big(\frac{\th}{z^{\frac13}}\big)\Big]^2 
= 2z^{\frac13}\frac{ z^{\frac13}}{(\bp\zeta)^2\zeta} \Big[\bp\big(\frac{\th}{z^{\frac13}}\big)\Big]^2  \label{E:CANCEL1}
\end{align}
and
\beqa
\frac{\bG z^{\frac83}\bzeta_z^2}{(\bp\zeta)^2\zeta}   \Big[ \bzeta g_1 \th_z^2 - 4g_2\th\th_z \Big]=\frac{z^{\frac13}}{(\bp\zeta)^2\zeta}\Big(  \frac{\bzeta}{z^{\frac13}}g_1 (\bp\th)^2 - 4 z^{\frac13}g_2 \th \bp\th\Big).\label{E:S22PF}
\eeqa
Combining~\eqref{E:CANCEL1}--\eqref{E:S22PF} in~\eqref{E:S2REFORM} gives~\eqref{E:STWODEF}.
\end{proof}


\begin{remark}
Terms $\mathcal S_1[\th,\th_z]$ and $\mathcal S_2[\th,\th_z]$ have been rearranged to eliminate the formal singularities appearing in expressions~\eqref{E:SONEDEF1}--\eqref{E:STWODEF1}, where through a formal identification $\zeta,\bzeta,\th\sim_{z\to0}z^{\frac13}$, one easily checks that individual 
terms scale like $z^{-\frac13}$. We have therefore established this cancellation structure in order to obtain useful bounds on $\bD^{2j}(\mathcal S_1[\th,\th_z])$ and $\bD^{2j}(\mathcal S_2[\th,\th_z])$.
\end{remark}


\begin{lemma}\label{L:LPPROP}
Under the assumptions of Proposition~\ref{P:EE2},
there exists a positive constant $C(Z_0)>0$ such that for any $j\in\{1,\dots M\}$ the following bounds hold:
\begin{align}
\Big|\int_{0}^{Z_0} \chi_{2j}\big|\bD^{2j}\big(\mathcal S_1[\th,\th_z]\big)\big|^2\diff z \Big| & \le C(Z_0)\|\Phi\|_{\mathcal{H}^{2j-1}_{Z_0}}\tE_{2j}^{\frac12} + \tE_{\leq 2j}^2, \label{E:SONEBOUND}\\
\Big|\int_{0}^{Z_0} \chi_{2j}\big|\bD^{2j}\big(\mathcal S_2[\th,\th_z] \big)\big|^2\diff z \Big| & \le C \tE_{\leq 2j}^2,  \label{E:STWOBOUND}\\
\Big|\int_{0}^{Z_0} \chi_{2j}\big|\bD^{2j}\big(\frac{\bG z^{\frac83}}{3(\bp\zeta)^2\zeta\bzeta^2}(\bzeta^3)_{zz} \th_z^2\th\ \big)\big|^2\diff z \Big| & \le C \tE_{\leq 2j}^3. \label{E:STHREEBOUND}
\end{align}
\end{lemma}


\begin{proof}
We start with the proof of~\eqref{E:STWOBOUND}.
We recognise in the first term of~\eqref{E:STWODEF} a term of the form $z^{\frac13} \tilde g(z) \bar X\th \bar X\th$, where 
\[
\bar X:=\bp\frac{\cdot}{z^{\frac13}}\in \X_2, \ \ \ \tilde g(z):= \frac{z^{\frac13}}{(\bp\zeta)^2\zeta}.
\] 
Note that $\tilde g$ is a bounded function of $\th$. Moreover, by writing
\[\tilde g =\frac{\bG}{(1+\sqrt{\bG}\bp\th)^2}\frac{z^{\frac13}}{\bzeta}\frac{1}{1+\tilde G\frac{\th}{z^{\frac13}}}, \]
an argument as in Lemma~\ref{L:NBOUND} shows that for any $R\in \Y_\ell$ with $\ell\leq 2j$,
  $\|R\tilde g\|_{L^2([0,Z_0])}\leq C(1+\tE_{\leq 2\ell})\leq C$ by our a priori assumptions~\eqref{E:APRIORIMAIN}--\eqref{E:APRIORIMAINPT}, while if $\ell\leq j+3$, then the Hardy-Sobolev type estimate~\eqref{E:XKHARDY} gives $\|R\tilde g\|_{L^\infty(0,Z_0)}\leq C$.
  
   The quadratic nonlinearity $\bar X\th \bar X\th$ 
contains the correctly weighted derivatives of $\th$ - this already entails a cancellation. However, applying $\bD^{2j}$ to this expression
leads to an apparent derivatives loss, as the leading order term $\bD^{2j}\bar X\th$ appears to scale like one derivative too many since $\bD^{2j}\bar X\in \mathcal X_{2j+2}$.
To mitigate this we observe that, with $\bar X=\bp\frac{\cdot}{z^{\frac13}}$
\begin{align}\label{E:CANCEL2}
\bd(z^{\frac13}\Big[\bar X\th\Big]^2 )= 2 \bar X\th \bp\bd\th   - \frac53   (\bar X\th)^2,
\end{align}
which can be checked by a direct calculation. This shows that the presence of the ``smoothing" factor $z^{\frac13}$ in~\eqref{E:CANCEL1} gets rid 
of the perceived derivatives loss, as the right-hand side of~\eqref{E:CANCEL2} is a quadratic function of entries of the form $X\th$, $X\in\X_2$.
Therefore
\begin{align}
\Big|\int_{0}^{Z_0} \chi_{2j}\big|\bD^{2j}\big(z^{\frac13}\tilde g (\bar X\th)^2\big)\big|^2\diff z \Big|  
\le&\,  \Big|\int_{0}^{Z_0} \chi_{2j}\big|\bD^{2j-1}\big(\tilde g \big(2 \bar X\th \bp\bD\th  - \frac53  (\bar X\th)^2\big)\big)\big|^2\diff z \Big|\notag\\
&+ \Big|\int_{0}^{Z_0} \chi_{2j}\big|\bD^{2j-1}\big(\bp\tilde g z^{\frac13}  (\bar X\th)^2 \big) \diff z \Big|.
\end{align}
Taking the first term on the right hand side, we argue from the product rule~\eqref{E:LEIBNIZ} analogously to the proof of interior bounds of Lemma~\ref{L:COMMBOUND1} in order to decompose 
\beqs
\bD^{2j-1}\big(\tilde g \big(2 \bar X\th \bp\bD\th\big)\big)=\sum_{\ell_1+\ell_2=0}^{2j-1}\sum_{\substack{R\in\Y_{2j-1-\ell_1-\ell_2} \\ Q_1\in\X_{\ell_1} \\ Q_2 \in \X_{\ell_2}}} c_{RQ}(R\tilde g)( Q_1\bar X\th)(Q_2\bD^2\th).
\eeqs
Clearly at most one of the operators $R$, $Q_1\bar X$ or $Q_2\bD^2$ can be of order  $\geq j+2$, the others are of order $\leq j+1$, and none of them is of order greater than $2j+1$. The other terms admit similar decompositions. We therefore use a standard $L^\infty$-$L^\infty$-$L^2$ splitting, applying the $L^\infty$ Hardy-Sobolev estimate~\eqref{E:XKHARDY}, in order  to infer the desired bound 
\be
\Big|\int_{0}^{Z_0} \chi_{2j}\big|\bD^{2j}\big(z^{\frac13}\tilde g (\bar X\th)^2\big)\big|^2\diff z \Big|    \leq C \tE_{\leq 2j}^2.
\ee
To bound the remainder of $\mathcal S_{2}[\th,\th_z]$, 
we use the same strategy, together with the estimates for $g_1$, $z^{\frac13}g_2$ provided by Lemma~\ref{L:MATHFRAKRNEW}, to obtain
\be
\Big|\int_{0}^{Z_0} \chi_{2j}\big|\bD^{2j}\big(\tilde g \frac{\bzeta}{z^{\frac13}}g_1 (\bp\th)^2 - 4\tilde g z^{\frac13}g_2 \th \bp\th\big)\big|^2\diff z \Big|  \leq C \tE_{\leq 2j}^2.
\ee
By Lemma~\ref{L:MATHFRAKRNEW} and a priori bounds~\eqref{E:APRIORIMAIN}--\eqref{E:APRIORIMAINPT} we now deduce by an analogous argument that
\beqa
{}&\Big|\int_{0}^{Z_0} \chi_{2j}\big|\bD^{2j}\big(\mathcal S_1\th\big)\big|^2\diff z \Big| 
&\leq C(Z_0)\|\Phi\|_{\mathcal{H}^{2j-1}_{Z_0}}^2 + C\tE_{\leq 2j}^2,
\eeqa
where 
we also used the Hardy-Sobolev~\eqref{E:XKHARDY}.

To prove~\eqref{E:STHREEBOUND} we rewrite
\begin{align}
\frac{\bG z^{\frac83}}{3(\bp\zeta)^2\zeta\bzeta^2}(\bzeta^3)_{zz} \th_z^2\th =\frac{\bG}{3(\bp\zeta)^2} \frac{z^{\frac13}}{\zeta}\frac{z^{\frac23}}{\bzeta^2}z^{\frac13}(\bzeta^3)_{zz} (\bp\th)^2\th. 
\end{align}
From Lemma~\ref{L:ZETABAR}, we observe the key cancellation $(\bzeta^3)_{zz}=z^{-\frac13}(1+O(z^{\frac23}))$. Therefore, a similar argument to those above yields~\eqref{E:STHREEBOUND} as claimed.
\end{proof}


\begin{proposition}[Interior bound]\label{P:INT1}
Under the assumptions of Proposition~\ref{P:EE2}
there exists a positive constant $C>0$ such that for any $j\in\{1,\dots M\}$ the following bound holds:
\begin{align}\label{E:FRAKREST}
\Big|\int_{0}^{Z_0} \chi_{2j}\bD^{2j}\mathfrak R[\th] \phij \diff z \Big| \le C(Z_0)\|\Phi\|_{\mathcal{H}^{2j-1}_{Z_0}}\tE_{2j}^{\frac12}+C \tE_{\leq 2j}^{\frac32} ,
\end{align}
where we recall that $\mathfrak R[\th]$ is given by~\eqref{E:MATHFRAKR2}.

Moreover, under the assumptions of Proposition~\ref{P:HMZDIFF}, setting $\vartheta=\theta_1-\theta_2$ and $\varphi=\phi_1-\phi_2$, for $j\leq \m$, 
\begin{align}
\Big|\int_{0}^{Z_0} g_{2j}\bD^{2j}\big(\mathfrak R[\th_1]-\mathfrak R[\th_2]\big) \varphi_{2j} \diff z \Big| \le C(Z_0)\|\Phi_1-\Phi_2\|_{\mathcal{H}^{2j-1}_{Z_0}}\|\Phi_1-\Phi_2\|_{\mathcal{H}^{2j}_{Z_0}}\big(\tE_{2j}[\Phi_1]^{\frac12}+\tE_{2j}[\Phi_2]^{\frac12}\big).\label{E:COMMBOUND3DIFF}
\end{align}
\end{proposition}


\begin{proof}
To complete the proof of~\eqref{E:FRAKREST}, it remains to estimate the second line on the right-hand side of~\eqref{E:MATHFRAKR2}, as the bounds for the first line follow from Lemma~\ref{L:LPPROP}. The linear term $\Big[\frac43 \frac{\bzeta_z^2}{\zeta_z^2}  z^{-\frac13} \bp^2 \bzeta \bG^\frac32 + \frac{2\CLP z}{\zeta^2\bzeta}\Big]\th$ is estimated analogously to the proof of~\eqref{E:SONEBOUND}, with the added simplification that the term in rectangular parenthesis is already regular at $z=0$ and no further cancellations are necessary. A similar comment applies to the last (quadratic) term on the right-hand side of~\eqref{E:MATHFRAKR2}. This is estimated analogously to~\eqref{E:STWOBOUND}, without any further cancellations necessary since 
\[
\frac{\CLP z\theta^2}{\zeta^2\bzeta^2}=\CLP z^{\frac13} \big(\frac{\theta}{z^{\frac13}}\big)^2 \big(\frac{z^{\frac13}}{\zeta}\big)^2\big(\frac{z^{\frac13}}{\bzeta}\big)^2.
\]
Therefore we arrive at 
\begin{align}
\int_{0}^{Z_0} \chi_{2j}\Big|\bD^{2j}\Big(\Big[\frac43 \frac{\bzeta_z^2}{\zeta_z^2}  z^{-\frac13} \bp^2 \bzeta \bG^\frac32 + \frac{2\CLP z}{\zeta^2\bzeta}\Big]\th\Big)\Big|^2 \diff z  &\le C(Z_0)\|\Phi\|_{\mathcal{H}^{2j-1}_Z}\tE_{2j}^{\frac12}+\tE_{\leq 2j}^2, \notag\\
\int_{0}^{Z_0} \chi_{2j}\Big|\bD^{2j}\Big(\frac{\CLP z\theta^2}{\zeta^2\bzeta^2}\Big)\Big|^2\diff z & \le C  \tE_{\leq 2j}^{2}.\notag
\end{align}
The claim now follows from Lemma~\ref{L:LPPROP}, the above two bounds, and the Cauchy-Schwarz inequality. 

To prove~\eqref{E:COMMBOUND3DIFF}, one observes that every term in $\mathfrak{R}[\th_1]-\mathfrak{R}[\th_2]$, after re-writing as in Lemma~\ref{L:MATHFRAKRNEW}, admits a decomposition in terms of the differences $\th_1-\th_2$ that respects the structure used in the proof of Lemma~\ref{L:LPPROP} and argues analogously. We omit the details.
\end{proof}


\subsection{Proofs of Proposition~\ref{P:EE2} and Proposition~\ref{P:HMZDIFF}}


\begin{proof}[Proof of Proposition~\ref{P:EE2}]
Proposition~\ref{P:EE2} is now a simple consequence of the energy inequalities stated in Propositions~\ref{P:ZEROENERGYGLOBAL}--\ref{P:ENERGYGLOBAL} and the nonlinear bounds stated in Lemmas~\ref{L:CUTOFFBOUND},~\ref{L:COMMBOUND1},~\ref{L:COMMBOUND2}, Propositions~\ref{P:EXT1} and~\ref{P:INT1}. Specifically, from Proposition~\ref{P:ZEROENERGYGLOBAL} and Lemma~\ref{L:CUTOFFBOUND}, we obtain
\begin{align}
&\frac12\pa_s\tE_{0}(s)  + c_4\tE_{0} (s) \le C \kappa\|\Phi\|_{\H^{0}_{2\m,Z_0}}^2
  + \sqrt{\kappa}(r_*e^s)^{-\frac12} .\label{E:Energy-zero}
\end{align}
For $j\geq 1$, we apply Lemmas~\ref{L:CUTOFFBOUND},~\ref{L:COMMBOUND1},~\ref{L:COMMBOUND2}, Propositions~\ref{P:EXT1} and~\ref{P:INT1} to estimate the right hand side of Proposition~\ref{P:ENERGYGLOBAL} and obtain
\begin{align}
&\frac12\pa_s\tE_{2j}(s)  + (c_3 j +c_4)\tE_{2j} (s) \le C \kappa\|\Phi\|_{\H^{2J}_{2\m,Z_0}}^2 \bm{1}_{j\le J}
 \notag\\
&\quad  +C\Big( \sqrt{\kappa}(r_*e^s)^{-\frac12-\frac{2j}{3}}\tE_{\leq 2j}^{\frac12}+ (r_\ast e^s)^{ - (\frac23-c)j - \frac{\alpha+1}{2}} \tE_{\leq 2j}^\frac12 + (r_*e^s)^{-\frac23}\tE_{\leq 2j}  +  (r_\ast e^s)^{-\frac{\alpha+1}{2} + 2\delta j }\tE_{\leq 2j}^\frac32\Big) \notag\\
&\quad + C(Z_0)\|\Phi\|_{\mathcal{H}^{2j-1}_{Z_0}}\tE_{2j}^{\frac12} + CZ_0^{-\frac23}\tE_{\leq 2j}+C(Z_0)\tE_{\leq 2j} \tE_{\leq 2[\frac{j+3}{2}]}^{\frac12} +C   Z_0^{-\frac12}\tE_{\leq 2j}^{\frac32}+C Z_0^{-\frac{\al}{2}+(2M+1)\de-\frac{c}{2}-\frac16}\tE_{\leq 2j}^{2}\notag\\
&\quad  +C\Big(Z_0^{-\frac{\al+1}{2}+\de}\tE_{\leq 2j} + Z_0^{-\frac{\al+1}{2}+2\de j}\tE_{\leq 2j}^{\frac32}\Big) +C(Z_0)\|\Phi\|_{\mathcal{H}^{2j-1}_{Z_0}}\tE_{2j}^{\frac12}+C \tE_{\leq 2j}^{\frac32} .\label{E:Energy-high1}
\end{align}
By recalling assumption~\ref{item:a2}, the right hand side easily simplifies to the form in the statement of the proposition.
\end{proof} 

 The proof of Proposition~\ref{P:HMZDIFF} now follows by the standard adaptation of the energy arguments above. We first have the following energy inequality, analogous to Lemma~\ref{L:HIGHENERGY0}. 

\begin{lemma}\label{L:DIFFENERGY}
Under the assumptions of Proposition~\ref{P:HMZDIFF}, let $\vartheta=\theta_1-\theta_2$, $\varphi=\phi_1-\phi_2$. As usual, we denote $\vartheta_{2\m}=\bD^{2\m}\vartheta$ and similarly for $\varphi$. We have the energy identity
\begin{align}
& \frac12\frac{d}{ds}\int_0^{Z_0} g_{2\m} \left[ \frac{(\bd\vartheta_{2\m})^2}{(\bp\zeta_1)^2} +\varphi_{2\m}^2 \right] \notag \\
&\leq -\frac{4\m-3}{6}\int_0^{Z_0} g_{2\m} \left[\frac{(\bd\vartheta_{2\m})^2}{(\bp\zeta_1)^2} + \varphi_{2\m}^2\right]+ 2\m \int_0^{Z_0} \frac{g_{2\m}\bp \bG}{\sqrt{\bG}} \frac{ \bd \vartheta_{2\m} }{\bp \zeta_1 }\varphi_{2\m}
-\int_0^{Z_0} g_{2\m} \frac{(\bd\vartheta_{2\m})^2}{(\bp\zeta_1)^2} \frac{\pa_s\bp\theta_1}{\bp\zeta_1}\notag \\
&\ \ \  \ 
  + \frac12\int_0^{Z_0} \Lambda g_{2\m}  \left[ \frac{(\bd\vartheta_{2\m})^2}{(\bp\zeta_1)^2} +\varphi_{2\m}^2 \right] 
   - \int_0^{Z_0}  \frac{\bp g_{2\m}}{(\bp\zeta_1)^2} \bd \vartheta_{2\m} \varphi_{2\m}  - \int_0^{Z_0} g_{2\m} \frac{(\bd\vartheta_{2\m})^2}{(\bp\zeta_1)^2}\frac{z\pa_{zz}\zeta_1}{\pa_z\zeta_1}  \notag\\
   &\ \ \ \ - \int_0^{Z_0} g_{2\m} \bp(\frac{\bzeta_z^2}{\zeta_{1,z}^2}) \bG  \bd \vartheta_{2\m} \varphi_{2\m} -2\m \int_0^{Z_0} g_{2\m} \frac{ \pa_z\theta_1 }{\zeta_{1,z}} \frac{\bp \bG }{\sqrt{\bG }} \frac{ \bd \vartheta_{2\m} }{\bp \zeta_1 }\varphi_{2\m}\notag\\
& \ \ \ \ +\int_0^{Z_0} g_{2\m} \varphi_{2\m}  \Big(\frac{\bzeta_z^2}{\zeta_{1,z}^2}-\frac{\bzeta_z^2}{\zeta_{2,z}^2}\Big)  K\theta_{2,2\m} \notag\\
& \ \ \ \ +2\m \int_0^{Z_0} g_{2\m}  \bp \bG\Big(\Big(\frac{\bzeta_z^2}{\zeta_{1,z}^2}-\frac{\bzeta_z^2}{\zeta_{2,z}^2}\Big)  \bd \theta_{2,2\m}\Big) \varphi_{2\m} + \int_0^{Z_0} g_{2\m} \Big(\frac{\bzeta_z^2}{\zeta_{1,z}^2} \mathcal R_{2\m}\theta_1-\frac{\bzeta_z^2}{\zeta_{2,z}^2} \mathcal R_{2\m}\theta_2\Big)  \varphi_{2\m}   \notag\\ 
& \ \ \ \ + \int_0^{Z_0} g_{2\m} \big(\mathcal N_{2\m}[\theta_1]-\mathcal N_{2\m}[\theta_2]\big) \varphi_{2\m} +\int_0^{Z_0} g_{2\m} \bD^{2\m}\big(\RR[\th_1]-\RR[\th_2]\big) \varphi_{2\m}.
\end{align}
\end{lemma}

\begin{proof}
This follows by a similar argument to Lemma~\ref{L:HIGHENERGY0}, where we note that, as $r_*>Z_0$, the far-field flattening error does not appear, and we have the additional boundary term on the right hand side  
\[-\Big[\frac12g_{2\m} z(\bp\zeta_1)^{-2}|\bd\vartheta_{2\m}|^2+\frac12g_{2\m} z|\varphi_{2\m}|^2-g_{2\m} z^{\frac23}\varphi_{2\m}(\bp\zeta_1)^{-2} \bd\vartheta_{2\m} \Big]^{Z_0}_0\leq 0
\]  
by the same argument as in Proposition~\ref{P:DISSIP}  (cf. \eqref{E:D4})  due to $Z_0>z_*$.
\end{proof}

\begin{proof}[Proof of Proposition~\ref{P:HMZDIFF}]
By arguing as in the proof of Proposition~\ref{P:EE2} but now employing the energy inequality of Lemma~\ref{L:DIFFENERGY} with~\eqref{E:COMMBOUND1DIFF},~\eqref{E:COMMBOUND2DIFF},~\eqref{E:COMMBOUND3DIFF}, in order to obtain control over the top order semi-norm contribution, $\|\Phi_1-\Phi_2\|_{\dot{\mathcal{H}}^{2 \m}_{Z_0}}^2$, it remains only to estimate
\beqa
{}&\bigg|\int_0^{Z_0} g_{2\m} \varphi_{2\m}  \Big(\frac{\bzeta_z^2}{\zeta_{1,z}^2}-\frac{\bzeta_z^2}{\zeta_{2,z}^2}\Big)  K\theta_{2,2\m} +2\m \int_0^{Z_0} g_{2\m} \bp \bG\Big(\Big(\frac{\bzeta_z^2}{\zeta_{1,z}^2}-\frac{\bzeta_z^2}{\zeta_{2,z}^2}\Big)  \bd \theta_{2,2\m}\Big) \varphi_{2\m}\bigg|\\
&\leq C\m\|\Phi_1-\Phi_2\|_{\H^{2\m}_{Z_0}}\|\Phi_2\|_{\mathcal{H}^{2\m+1}_{Z_0}}\|\bp\th_1-\bp\th_2\|_{L^\infty([0,\infty])}\\
&\leq C(Z_0)\|\Phi_1-\Phi_2\|_{\H^{2\m}_{Z_0}}^2\|\Phi_2\|_{\mathcal{H}^{2\m+1}_{Z_0}},
\eeqa
which yields the estimate  
\[
\frac12\frac{d}{ds}\|\Phi_1-\Phi_2\|_{\dot{\tilde{\mathcal{H}}}^{2 \m}_{Z_0}}^2\leq C(Z_0)\|\Phi_1-\Phi_2\|_{ \tilde{\mathcal{H}}^{2 \m}_{Z_0} }^2\big(\tE_{\leq 2\m+2}[\Phi_1]^{\frac12}+\tE_{\leq 2\m+2}[\Phi_2]^{\frac12}\big).
\]
where we have used ${\tilde{\mathcal{H}}^{2 \m}_{Z_0}}$ to denote the Hilbert space defined by replacing the weight $\bG$ by $(\bp \zeta_1)^{-2}$.   We observe this space is equivalent to $\HmZm$ due to the a priori estimate~\eqref{E:APRIORIMAINPT}.
 
To control the zero order contribution to the full norm, a similar (but simpler) integration by parts argument to Lemma~\ref{L:DIFFENERGY} yields
\begin{align}
& \frac12\frac{d}{ds}\int_0^{Z_0} g_{2\m} \left[ \frac{(\bd\vartheta)^2}{(\bp\zeta_1)^2} +\varphi^2 \right] \notag \\
&\leq \frac12\int_0^{Z_0} g_{2\m} \left[\frac{(\bd\vartheta)^2}{(\bp\zeta_1)^2} + \varphi^2\right]
-\int_0^{Z_0} g_{2\m} \frac{(\bd\vartheta)^2}{(\bp\zeta_1)^2} \frac{\pa_s\bp\theta_1}{\bp\zeta_1}\notag \\
&\ \ \  \ 
  + \frac12\int_0^{Z_0} \Lambda g_{2\m}  \left[ \frac{(\bd\vartheta)^2}{(\bp\zeta_1)^2} +\varphi^2 \right] 
   - \int_0^{Z_0}  \frac{\bp g_{2\m}}{(\bp\zeta_1)^2} \bd \vartheta \varphi  - \int_0^{Z_0} g_{2\m} \frac{(\bd\vartheta)^2}{(\bp\zeta_1)^2}\frac{z\pa_{zz}\zeta_1}{\pa_z\zeta_1}  \notag\\
   &\ \ \ \ - \int_0^{Z_0} g_{2\m} \bp(\frac{\bzeta_z^2}{\zeta_{1,z}^2}) \bG  \bd \vartheta \varphi +\int_0^{Z_0} g_{2\m} \varphi  \Big(\frac{\bzeta_z^2}{\zeta_{1,z}^2}-\frac{\bzeta_z^2}{\zeta_{2,z}^2}\Big)  K\theta_2  +\int_0^{Z_0} g_{2\m}  \big(\RR[\th_1]-\RR[\th_2]\big) \varphi   \notag.
\end{align}
From this, it is simple to use the a priori assumptions to deduce an estimate of the form  
\[  \frac12\frac{d}{ds}\|\Phi_1-\Phi_2\|^2_{\tilde{\H}^0_{2\m,Z_0}}\leq C\|\Phi_1-\Phi_2\|_{\tilde{\H}^0_{2\m,Z_0}}^2.\]
We integrate over $[s_1,s_2]$ and use the equivalence between $\|\cdot\|_{\HmZm}$ and $\|\cdot\|_{\tilde{\mathcal{H}}^{2 \m}_{Z_0}}$ to complete the proof. 
\end{proof}


\section{Pointwise control of the Lagrangian flow map}\label{S:POINTWISE}


In this section, we present the $L^\infty$-in-$z$ bounds on the quantity 
\[
\Pb(s)= \|\frac{\th}{\bzeta}\|_{L^\infty} +\|\frac{\bp\theta(s,\cdot)}{\bp\bzeta}\|_{L^\infty}+\|\frac{z\pa_z^2\theta(s,\cdot)}{\pa_z\bzeta}\|_{L^\infty} +\|\frac{\pa_s\bp\theta(s,\cdot)}{\bp\bzeta}\|_{L^\infty}
\]
introduced in~\eqref{E:PDEF}. This is the final step in justifying the main a priori assumptions~\eqref{E:APRIORIMAIN}--\eqref{E:APRIORIMAINPT}, which have been used heavily in the  high-order energy estimates of Section~\ref{S:ENERGYBOUNDS}. 

The essential idea is that  $L^\infty$ bounds on the finite interval $[0,Z]$ follow from the Hardy-Sobolev embedding and can thus be related to the norm $\HmZ$. 
However, the bounds on the counterpart $(Z,\infty)$ do not follow in the same way due to the inhomogeneity of the LP profile $\bzeta$ and the weight structure of our energy norm, which is not strong enough to control them directly, independent of time. To get around this difficulty, we propagate the initial control on $\Pb(0)$ via the finite speed of propagation property and the exponential decay of the total energy embedded in our definition~\eqref{E:TEDEF} of $\tE$ and the exponential decay with rate $\Omega$ in the a priori assumption~\eqref{E:APRIORIMAIN}. 

\begin{proposition} \label{P:PW}
Let $\Phi=(\th,\phi)^\top$ satisfy the assumptions of Proposition~\ref{P:EE2}.
Then there exists a constant $C_P =C(Z_0)\Omega^{-1}$ such that for any $s\in[0, S)$ the following bounds hold:
\begin{align}
\left\|\frac{\theta}{\bzeta}(s)\right\|_{L^\infty[0,\infty]} &\le   \left\|\frac{\theta_0}{\bzeta}\right\|_{L^\infty[0,\infty]} 
+ C_P \sqrt{\eps_0} \label{apriori1},\\ 
\left\|\frac{\bp\theta}{\bp\bzeta}(s)\right\|_{L^\infty[0,\infty]}& \le   \left\|\frac{\bp\theta_0}{\bp\bzeta}\right\|_{L^\infty[0,\infty]} 
+ C_P \sqrt{\eps_0} \label{apriori2},\\
\left\|\frac{z\pa_z^2\theta(s,\cdot)}{\pa_z\bzeta}\right\|_{L^\infty[0,\infty]}+\left\|\frac{\pa_s\bp\theta}{\bp\bzeta}(s)\right\|_{L^\infty[0,\infty]}& \le  \left\|\frac{z\pa_z^2\theta_0}{\pa_z\bzeta}\right\|_{L^\infty[0,\infty]}+ \left\|\frac{\pa_s\bp\theta_0}{\bp\bzeta}\right\|_{L^\infty[0,\infty]} 
+ C_P \sqrt{\eps_0}  \label{apriori3}.
\end{align}
\end{proposition} 

Below, in Lemma~\ref{L:STPROPERTIES}, we will take $C_*(Z_0)>6C_P+1$ and so successfully improve the a priori assumption~\eqref{E:APRIORIMAINPT}.

\begin{proof}
\textit{Step 1: Interior estimate.} Note that $\bzeta(z)\sim z^{\frac13}$ for $z\ll 1$ and $\bzeta (z)\sim z $ for $z\gg 1$, in particular for $z\geq Z_0$.
On $[0,Z_0]$, we demonstrate the estimate for the most complicated term, $\frac{\pa_s\bp\th}{\bp\bzeta}$, as the others follow by similar arguments. Recall that there exists a constant $C(Z_0)>0$ such that $\frac{1}{C(Z_0)}\leq \bp\bzeta\leq C(Z_0)$ for $z\in(0,Z_0)$. We now recall that $\theta_s +\Lambda\theta -\theta =\phi$, the commutation relation~\eqref{E:GAIN1}, and $\bp\th=\bd\th-\frac23 z^{-\frac13}\th$ to observe
\begin{align}
\pa_s\bp\theta=\bp(\phi-\Lambda\th+\th)=\bd(\phi+\frac23\th)-\frac23\frac{\phi+\frac23\th}{z^{\frac13}} -z^{\frac13}\bp\Big(\bd\th-\frac23\frac{\th}{z^{\frac13}}\Big).
\end{align}
It is now simple to estimate, using Lemma~\ref{L:LINFTYHARDYSOBOLEV} (specifically~\eqref{E:XKHARDY}),
\beqas
\big\|\pa_s\bp\theta\|_{L^\infty(0,Z_0)}\leq C(Z_0)\|\Phi\|_{\HmZm}.
\eeqas
In total, 
\begin{align}
&\left\|\frac{\theta}{\bzeta}(s)\right\|_{L^\infty[0,Z_0]}+\left\|\frac{z\pa_z^2\theta(s,\cdot)}{\pa_z\bzeta}\right\|_{L^\infty[0,\infty]}+\left\|\frac{\bp\theta}{\bp\bzeta}(s)\right\|_{L^\infty[0,Z_0]} +  \left\|\frac{\pa_s\bp\theta}{\bp\bzeta}(s)\right\|_{L^\infty[0,Z_0]} \notag\\
&\le C(Z_0) \|\Phi\|_{\HmZm} \le
 C(Z_0) \sqrt{\eps_0} e^{-\Om s },\label{0.8}
\end{align}
where we have used the a priori bound~\eqref{E:APRIORIMAIN} and \eqref{E:HMZTE} in the very last bound. It therefore remains only to establish~\eqref{apriori1}--\eqref{apriori3} for $z\geq Z_0$.

\textit{Step 2: Proof of~\eqref{apriori1}.} Recalling the relation $\theta_s +\Lambda\theta -\theta =\phi$ we integrate backwards in $s$ along the trajectory starting from $(s,z)$ to obtain, for $\tilde s\in[\sin,s]$,
\be\label{E:theta-phi}
\theta (s,z) =e^{s-\tilde s } \theta (\tilde s  ,e^{-(s-\tilde s )}z)  +\int_{\tilde s }^s e^{s-\sigma} \phi (\sigma , e^{-(s-\sigma)}z)\diff \sigma .
\ee
Then for $z>Z_0$, 
\be\label{E:FSP00}
\frac{\theta (s,z)}{z} =\frac{ \theta (\tilde s  ,e^{-(s-\tilde s )}z)}{e^{-(s-\tilde s )}z}  +\int_{\tilde s }^s  \left( \frac{\phi}{z} \right) (\sigma , e^{-(s-\sigma)}z)\diff \sigma.
\ee
We distinguish two cases. If $e^{-s}z\leq Z_0$, then we take $\tilde s$ such that $e^{-\tilde s}z=Z_0$. In the other case, we take $\tilde s=\sin$. Then we have found
\beqa\label{E:FSP0}
\Big|\frac{\theta (s,z)}{z}\Big| &\leq \big\|\frac{\theta_0}{z}\big\|_{L^\infty[Z_0,\infty]}  +\Big|\frac{\theta (\tilde s,Z_0)}{Z_0}\Big|+ \int_{\tilde s }^s  \left( \frac{\phi}{z} \right) (\sigma , e^{-(s-\sigma)}z)\diff \sigma\\
&\leq \big\|\frac{\theta_0}{z}\big\|_{L^\infty[Z_0,\infty]}  + C(Z_0) \sqrt{\eps_0}
+\int_{\tilde s }^s \left\|  \frac{\phi}{z} (\sigma)\right\|_{L^\infty(Z_0 ,\infty)} \diff\sigma,
\eeqa
where, in the last line, we have used that the trajectory along which we integrate is always contained in the region $z\geq Z_0$ by definition of $\tilde s$.

Now we employ H\"older's inequality and Lemma~\ref{L:LARGEZ}  to infer, for $z\geq Z_0$,
\begin{align}
\left|\frac{\phi}{z}\right| &= \left|\int_\infty^z \pa_z(\frac{\phi}{z}) \diff z\right| =  \left|\int_\infty^z \frac{\pa_z\phi}{z}- \frac{\phi}{z^2}\diff z\right|  \notag \\
&\le  (\int_z^\infty \frac{1}{\tilde z^{2c-\alpha +\frac83}} \diff \tilde z)^\frac12 (\int_z^\infty \tilde z^{2c-\alpha +\frac23 }(\pa_z\phi)^2 \diff \tilde z )^\frac12 + (\int_z^\infty \frac{1}{\tilde z^{4-\alpha}} \diff \tilde z)^\frac12 ( \int_z^\infty \tilde z^{-\alpha}\phi^2 \diff \tilde z )^\frac12 \notag\\
&\le CZ_0^{-(2c-\alpha +\frac53)} \bigg(\int_0^{Z_0} | \bd \phi |^2 (1+z)^{2c-\alpha-\frac23} \diff z + 
\int_{Z_0}^\infty |\bp \bd \phi|^2 (1+z)^{2c-\alpha} \diff z \bigg)^{\frac12}\notag\\
&\ \ \ \ +\frac{C}{\sqrt{\kappa}}Z_0^{-\frac{3-\al}{2}}\tE^{\frac12}\notag\\
&\leq C(Z_0)\tE^{\frac12},\label{E:PHIOVERZ}
\end{align} 
where in the last line we have  applied Lemma~\ref{L:EQUIVALENCE}. Using now the a priori estimate~\eqref{E:APRIORIMAIN}, we therefore find
\[\int_{\tilde s }^s \left\|  \frac{\phi}{z} (\sigma)\right\|_{L^\infty(Z_0 ,\infty)} \diff\sigma\leq C(Z_0)\sqrt{\eps_0}\int_0^s e^{-\Omega s}\,\dif s\leq C(Z_0)\Omega^{-1}\sqrt{\eps_0}. \]
Combining this with~\eqref{E:FSP0} yields~~\eqref{apriori1}.

\textit{Step 3: Proof of~\eqref{apriori2}.} Observe that $|\frac{\bp\theta}{\bp\bzeta} |= |\frac{\pa_z\th}{\pa_z\bzeta}| \approx |\pa_z\theta|$ on $(Z_0,\infty)$. From \eqref{E:theta-phi} we conclude,   
\[
\pa_z \theta (s,z) =\pa_z \theta (\tilde s,e^{-(s-\tilde s)}z)  +\int_{\tilde s}^s \pa_z \phi (\tilde s , e^{-(s-\sigma)}z)d\sigma.
\]
On the other hand, we have the simple identity $\pa_z \phi = \frac{ \bd \phi }{z^\frac23} - \frac23 \frac{\phi }{z}$. The second term on the right has already been estimated above by $C(Z_0)\tE^{\frac12}$, while for the first term we make the bound
\beqa
{}&\Big|\frac{ \bd \phi }{z^\frac23}\Big|^2\leq C Z_0^{-\frac43}|\bd \phi(Z_0)|^2+C\int_{Z_0}^\infty z^{-\frac53}|\bD^2\phi|^2\,\dif z\leq C C(Z_0)\tE^{\frac12},
\eeqa
where in the first inequality we have used~\eqref{E:HARDYREFINED2}, and the second estimate follows from Lemma~\eqref{L:EQUIVALENCE} and the simple inequality $z^{-\frac53}\leq (1+z)^{2c-\al}$ for $z\geq Z_0$.
We now apply an argument analogous to the one from the proof of~\eqref{apriori1} which then gives~\eqref{apriori2}. 

\textit{Step 4: Proof of~\eqref{apriori3}.} Using the equation  $\pa_z\theta_s + \Lambda \pa_z\theta = \pa_z\phi$, we note that 
\[
\frac{\pa_s\bp\theta}{\bp\bzeta}= \frac{\bp \phi - z^\frac23 \Lambda \pa_z\theta }{\bp\bzeta} = \frac{\pa_z\phi}{\pa_z\bzeta} - \frac{z \pa_z^2\theta}{\pa_z\bzeta},
\]
where we note that $\big|\frac{\pa_z\phi}{\pa_z\bzeta} \big|$ is bounded by $C(Z_0)\tE^{\frac12}$ by the above estimates. It is therefore sufficient to establish the estimate for $z\pa_z^2\th$.
We next note that 
\[
z\pa_z^2 \theta (s,z) =( z\pa_z^2 \theta) (\tilde s,e^{-(s-\tilde s)}z)  +\int_{\tilde s}^s (z\pa_z^2 \phi) (\sigma , e^{-(s-\sigma)}z)\diff \sigma.
\]
Since $z\pa_z^2\phi =  z\pa_z ( \frac{ \bd \phi }{z^\frac23}  - \frac23 \frac{\phi }{z}) = \frac{ \bp \bd \phi }{z^\frac13}   - \frac43 \frac{\bd \phi }{z^\frac23} + \frac{10}{9} \frac{\phi}{z}$, we deduce from similar arguments to the above that 
\[
\left\|  z\pa_z^2 \phi (\sigma)\right\|_{L^\infty(Z_0 ,\infty)}  
\leq C \tE (\sigma)^\frac12.
\]
By the same argument as in the  proof of~\eqref{apriori1}, \eqref{apriori3} follows. 
\end{proof}


We next provide a slightly sharpened bounds on the unknowns, which will be used in the proof of Theorem~\ref{T:EULERMAIN} in Section~\ref{S:MAINTHEOREM}.

\begin{lemma}[$L^\infty$-control of $\phi$]\label{L:PHIUPPER}
Let $\Phi=(\th,\phi)^\top$ satisfy the assumptions of Proposition~\ref{P:EE2}.
Then the following uniform upper bounds hold:
\begin{align}
\sup_{(s,r)\in[\sin,\infty)\times[0,\infty)}|\phi(s,e^sr)|&\le \|\phi_0\|_{L^\infty([0,\infty))} +C\sqrt{\eps_0}, \label{E:PHIUNIFORM}\\
 \sup_{(s,r)\in[\sin,\infty)\times[0,\infty)}|\Lambda\phi(s,e^sr)| &\le \|\Lambda\phi_0\|_{L^\infty([0,\infty))} +C\sqrt{\eps_0}.  \label{E:LAMBDAPHIUNIFORM}
\end{align}
\end{lemma}


\begin{proof}
We first observe from~\eqref{E:NLPSI0} that 
\begin{align}\label{E:PHIUPPER0}
|\pa_s\phi+\Lambda\phi| \le \frac{C\sqrt{\eps_0}}z, \ \ z\ge 1. 
\end{align}
In particular
\begin{align}\label{E:PHIFINITE}
\phi(s,z) = \phi(\tilde s, e^{-(s-\tilde s)}z) + \int_{\tilde s}^s \big[\pa_s\phi+\Lambda\phi\big](\sigma, e^{\sigma-s}z) \diff\sigma.
\end{align}
On the other hand, in the accretive region $0\le z\le Z_0$ we have the exponential decay bound
\begin{align}\label{E:PHIUPPER}
\|\phi(s,\cdot)\|_{L^\infty([0,Z_0])} \le C\sqrt{\eps_0}e^{-\Om s}.
\end{align}
From~\eqref{E:PHIFINITE} and~\eqref{E:PHIUPPER0} we conclude that for any $r\ge\frac1{1+T}$ we have
\begin{align}
\big|\phi(s,e^sr) - \phi(\tilde s,e^{\tilde s}r) \big| & =\big|  \int_{\tilde s}^s \big[\pa_s\phi+\Lambda\phi\big](\sigma, e^{\sigma}r) \diff\sigma \big| \notag \\
& \le 
 \frac{C\sqrt{\eps_0}}{r} \int_{\tilde s}^s e^{-\sigma} \diff \sigma   \le \frac{C\sqrt{\eps_0}}{r} e^{-\tilde s}. \label{E:PHIUPPEROUT}
\end{align}
If $r\le\frac1{1+T}$ we have
\begin{align}
|\phi(s,e^sr)| &\le C\sqrt{\eps_0} e^{-\Om s}\chi_{[\sin,-\log r]}(s) +C\frac{\sqrt{\eps_0}}{r} \int_{-\log r}^s e^{-\sigma} \diff \sigma \notag\\
& \le C\sqrt{\eps_0} e^{-\Om \sin} + C\sqrt{\eps_0}  \le C\sqrt{\eps_0}. \label{E:PHIUPPERIN}
\end{align}

Letting $\tilde s = \sin$ in~\eqref{E:PHIUPPEROUT} and using~\eqref{E:PHIUPPERIN} we obtain the bound
\begin{align}
|\phi(s,e^sr)| \le \|\phi_0\|_{L^\infty([0,\infty))} +C\sqrt{\eps_0},
\end{align}
which concludes~\eqref{E:PHIUNIFORM}. The proof of~\eqref{E:LAMBDAPHIUNIFORM} is analogous.
\end{proof}


\begin{lemma}[Sharpened $L^\infty$-control]\label{L:SHARPER}
Let $\Phi=(\th,\phi)^\top$ satisfy the assumptions of Proposition~\ref{P:EE2}.
Then for $h\in[1,Z_0]$ there exists a constant $C>0$ such that 
\begin{align}\label{E:SHARPER}
\big|\frac{\th}{\bzeta}(s,e^sr)\big|+\big|\frac{\th_z}{\bzeta_z}(s,e^sr)\big| 
+ \big|\frac{z\th_{zz}}{\bzeta_z}(s,e^sr)\big|\le C\sqrt{\eps_0}\Big(e^{-\Om s}\chi_{s\le\log\frac{h}{r}}(s) + r^{\Om}\chi_{s\ge\log\frac{h}{r}}(s)\Big).
\end{align}
Letting
\begin{align}\label{E:JDEF}
J(s,z) : = \big(\frac{\bzeta(s,z)}{\zeta(s,z)}\big)^{2} \frac{\pa_z\bzeta}{\pa_z\zeta},
\end{align}
we then in particular have
\begin{align}\label{E:JBOUND}
\big|J(s,e^sr)-1\big| \le C\sqrt{\eps_0}\Big(e^{-\Om s}\chi_{s\le\log\frac{h}{r}}(s) + r^{\Om}\chi_{s\ge\log\frac{h}{r}}(s)\Big).
\end{align}
\end{lemma}


\begin{proof}
We only prove the first bound in~\eqref{E:SHARPER}, as the second one follows by analogous arguments.
If $s\le  \log\frac{h}{r}$ we have $e^sr\le Z_0$ and therefore by~\eqref{0.8}
\begin{align}
|\frac{\th}{\bzeta}(s,e^sr)| \le C\sqrt{\eps_0}e^{-\Om s}.
\end{align} 
On the other hand, if  $s\ge  \log\frac{h}{r}$, we use~\eqref{E:FSP00} with $z = e^s r$ and $\tilde s =  \log\frac{h}{r}$ to obtain the identity
\begin{align}
\frac{\theta (s,e^sr)}{e^sr} =\frac{ \theta (\log\frac hr  ,h)}{h}  +\int_{\log\frac{h}{r}}^s  \left( \frac{\phi}{z} \right) (\sigma , e^{\sigma}r)\diff \sigma.
\end{align}
We now use~\eqref{E:PHIOVERZ},~\eqref{0.8}, and $\E^{\frac12}\le C\sqrt{\eps_0}e^{-\Om s}$ to conclude that 
\begin{align}
\big|\frac{\theta (s,e^sr)}{e^sr}\big| & \le C \sqrt{\eps_0} e^{-\Om\log\frac hr } + C \sqrt{\eps_0}\int_{\log\frac{h}{r}}^s  e^{-\Om\sigma} \diff \sigma \notag \\
& \le C \sqrt{\eps_0} r^\Om,
\end{align}
which completes the proof of the first claimed bound in~\eqref{E:SHARPER} The remaining two bounds follow analogously. Bound~\eqref{E:JBOUND} now follows easily from~\eqref{E:SHARPER} and
Proposition~\ref{P:PW}.
\end{proof}


\section{Proof of the Main Theorem}\label{S:MAINTHEOREM}

The goal of this final section is to put together the nonlinear estimates of Sections~\ref{S:DUHAMEL}--\ref{S:POINTWISE} in order to conclude the proofs of Theorem~\ref{T:MAIN} (the Main Theorem in Lagrangian form) and Theorem~\ref{T:EULERMAIN} (the Eulerian stability of the LP solution).


Our first objective is to prove that the a priori assumptions~\eqref{E:APRIORIMAIN}--\eqref{E:APRIORIMAINPT} may be propagated beyond the maximal time $S_T$. As part of the proof of this propagation, we determine the relation between the smallness constants $\tilde\eps_0$, $\eps_0$ and $T_0$. Recall that $T_0$ is defined through the smallness parameter $\tilde\eps$ as $T_0=C_0\sqrt{\tilde\eps}$ via~\eqref{E:T0DEF} and that the choice of $C_0$ determines a constant $\bar C$ through~\eqref{E:IDEST0}. In fact, we first take $\eps_0$ sufficiently small depending on $Z_0$, $\m$, and then take constants $\bar C(Z_0,\m)\ll\widehat{C}(Z_0,\m)$ such that
\[ \tilde \eps_0=\widehat{C}^{-1}\eps_0.\]


\begin{lemma}\label{L:STPROPERTIES}
There exist $\eps_0>0$ and constants $C_0,\widehat{C},C_*>0$ such that, setting $\tilde\eps_0=\widehat{C}^{-1}\eps_0$  and, for $\tilde\eps\in[0,\tilde\eps_0]$,  $T_0=C_0\sqrt{\tilde\eps}$, the following holds. For each  $T\in[-T_0,T_0]$, let $\Phi^T$ be the unique smooth solution to~\eqref{E:NLTHETA0}--\eqref{E:NLPHI0} with initial data as in Theorem~\ref{T:MAIN} and recall the definition of the maximal time~\eqref{E:STDEF}. Then there exists $\de(T)>0$ such that the a priori assumptions~\eqref{E:APRIORIMAIN}--\eqref{E:APRIORIMAINPT} hold for all $s\in[\sin,S_T+\de]$, where in the case $S_T=\infty$, we understand this to be the whole positive real line. Moreover,  the maximal time function $T\mapsto S_T$ is continuous for all $T\in[-T_0,T_0]$ such that $S_T<\infty$.
	\end{lemma}
	
	
	\begin{proof}
	For convenience, we define a modified `maximal time'
	\beq
	\tilde S_T:=\sup\{s>\sin\,|\, e^{2\nu\sg \tilde s}\|\Phi(\tilde s,\cdot)\|_{\HmZm}^2\leq  2\bar C\tilde\eps\text{ for all }\tilde s\in(\sin,s)\}
	\eeq
	and, for the purposes of a continuity argument, 
	\beq\label{E:STAST}
	S_T^*:=\sup\{s\in(\sin,\tilde S_T)\,|\,\text{\eqref{E:APRIORIMAIN}--\eqref{E:APRIORIMAINPT} hold for all }\tilde s\in[\sin,s]\}.
	\eeq
	Observe that $\tilde S_T>S_T$ by standard local well-posedness theory, so that once we show that $S_T^*=\tilde S_T$, then we may obtain the function $\de(T)>0$ as claimed by setting $\de(T)=\tilde S_T-S_T$.
	
	For any $S\in(\sin, S_T^*)$, we first apply Proposition~\ref{P:PW}, and, for any $s\in[\sin,S_T^*]$, we easily obtain
\begin{align}
\mathbb{P}(s) & \leq \left\|\frac{\theta_0}{\bzeta}\right\|_{L^\infty[0,\infty]} + \left\|\frac{\bp\theta_0}{\bp\bzeta}\right\|_{L^\infty[0,\infty]} + \left\|\frac{z\pa_z^2\theta_0}{\pa_z\bzeta}\right\|_{L^\infty}+ \left\|\frac{\pa_s\bp\theta_0}{\bp\bzeta_0}\right\|_{L^\infty[0,\infty]} +3C_P\sqrt{\eps_0} \notag\\
& \leq \frac12C_*\sqrt{\eps_0},\label{E:PBOUNDBETTER}
\end{align}
by assuming $C_*\geq 6 C_P +1$  and using  condition~\eqref{E:INITIALDATA} on initial data.

Next, we
 apply Proposition~\ref{P:EE2} and sum over $2j=0,\ldots,2(\m+1)$ to deduce, for all $s\in(\sin,S_T^*)$,
	\beqa
	\frac12\pa_s\tE_{\leq 2(\m+1)}+c_4\tE_{\leq 2(\m+1)}\leq &\, C(Z_0)\|\Phi\|_{\HmZm}^2+\sqrt{\kappa}(r_*e^{s})^{-\frac12}\tE_{\leq 2(\m+1)}^{\frac12}+((r_*e^{s})^{-\frac23}+Z_0^{-\frac12})\tE_{\leq 2(\m+1)}\\
	&+C(Z_0)\tE_{\leq 2(\m+1)}^\frac32
	+C(Z_0)\|\Phi\|_{\mathcal{H}^{2\m+1}_{Z_0}}\tE_{\leq 2(\m+1)}^{\frac12}.
	\eeqa
	Therefore, taking $Z_0$ large,  then $\eps_0$ small, we obtain
 (recall $\|\Phi\|_{\mathcal{H}^{2\m+1}_{Z_0}}\leq C\|\Phi\|_{\mathcal{H}^{2\m+2}_{Z_0}}^{\frac12}\|\Phi\|_{\mathcal{H}^{2\m}_{Z_0}}^{\frac12}$ by standard interpolation of Sobolev spaces and the equivalence~\eqref{E:HSOBOLEVEQUIV})
 \beq
 \frac12\pa_s\tE_{\leq 2\m+2}+\frac12 c_4\tE_{\leq 2\m+2}\leq  2 C(Z_0)\bar Ce^{-2\nu\sg s}\tilde \eps + C\kappa r_*^{-1}e^{-s},
 \eeq
 where we have used $S\leq \tilde S_T$ and recall that $\bar C$ is determined from $C_0$ via~\eqref{E:IDEST0}.  In particular, for $s\in(\sin, S_T^*)$, as $2\Om\leq c_4$, $\Om<2\nu\sg $ from~\eqref{E:OMEGABDS} and the definition $c_4=\frac{\al-1}{4}>0$ from Proposition~\ref{P:EE2},
 \beq
e^{2\Omega s} \tE_{\leq 2(\m+1)} \leq e^{2\Omega \sin}\tE_{\leq 2(\m+1)}(\sin)+ 2 C(Z_0)\bar C\tilde \eps + C\kappa r_*^{-1}\leq \frac12 \eps_0,
\label{E:EBOUNDBETTER}
 \eeq
  where we have used~\eqref{E:INITIALDATA}, taken $r_\ast$ chosen sufficiently large as a function of $\eps_0$ (note that $r_*$ is taken independent of $\tilde\eps$)  and we take $\widehat{C}$ and $ C_0$ such that $\tilde\eps_0\leq \frac{1}{\widehat C}\eps_0$  so that for any $\tilde\eps\leq\tilde\eps_0$, we have $ 2 C(Z_0)\bar C\tilde\eps<\frac{\eps_0}{16}$. 
Thus, in light of~\eqref{E:PBOUNDBETTER} and~\eqref{E:EBOUNDBETTER} a continuity argument shows that $S_T^*=\tilde S_T$.
 
 To show the continuity of $S_T$, we now apply Proposition~\ref{P:HMZDIFF} and a standard Gr\"onwall argument, using the fact that the a priori assumptions hold on an interval strictly larger than $[\sin,S_T]$. In particular, given $T_1,T_2\in [-T_0,T_0]$, we assume without loss of generality that $S(T_1)\leq S(T_2)$. Let $\eta>0$. If $S_{T_1}<\infty$, then we observe that there exists a time $s_c\in(S(T_1), S_{T_1}+\eta)$ such that $e^{2\nu\sg s_c}\|\Phi^{T_1}\|_{\HmZm}>\bar C\tilde \eps$.  Let $\tilde\delta= e^{2\nu\sg \sin}\|\Phi^{T_1}\|_{\HmZm} -\bar C \tilde \eps>0$.  We apply 
 Proposition~\ref{P:HMZDIFF} with $s_1=\sin$ and $s_2=s_c$ 
 to make the estimate 
 \beqa
e^{2\nu\sg s_c}\|\Phi^{T_2}(s_c,\cdot)\|_{\HmZm}^2\geq&\, e^{2\nu\sg s_c}\big(\|\Phi^{T_1}(s_c,\cdot)\|_{\HmZm}^2- \|(\Phi^{T_1}-\Phi^{T_2})(s_c,\cdot)\|_{\HmZm}^2\big)\\
\ge &\, \bar C\tilde \eps + \tilde\delta -e^{2\nu\sg s_c}e^{C (S(T_1)+\eta)}\|\Phi^{T_1}_{\textup{in}}-\Phi^{T_2}_{\textup{in}}\|_{\HmZm}^2>\bar C\tilde \eps,
 \eeqa 
 provided $|T_1-T_2|$ is sufficiently small, so that $S(T_2)\in(S(T_1),S(T_1)+\eta)$. 
\end{proof}
 
 Having shown that we may propagate the a priori assumptions beyond the maximal time (and the continuity of this time), we are now able to run a Brouwer fixed point argument in order to identify the true blow-up time $T_*$.

 \begin{proposition}\label{P:BROUWER}
Suppose $\Phi$ satisfies the assumptions of Lemma~\ref{L:STPROPERTIES} and  assume that $S_T<\infty$ for all $T\in[-T_0,T_0]$, where we recall $T_0$ is defined in~\eqref{E:T0DEF}. Then, possibly increasing $\widehat{C}$ and $\bar C$, there exists $T_*\in[-T_0,T_0]$ such that 
 \beq
 \bfP\bigg( \Phi^{T_*}_{in} +\int_{\sin}^{S_{T_*}} e^{\sin-\sigma} \bfN[\Phi](\sigma)\diff \sigma\bigg)=0.
 \eeq
 \end{proposition}
 
 \begin{proof}
Assume that for all $T\in[-T_0,T_0]$ we have $S_T<\infty$, where we recall~\eqref{E:STDEF}. We then consider the map
\[
T\mapsto p(T):= \bigg(\PhiT + \int_{\sin}^{S_T} e^{\sin-\sigma}\bfN[\Phi](\sigma)\,d\sigma\,,\, \Gamma\bigg)_{\H^0_{2\m,Z_0}},
\]
where we recall the initial data~\eqref{E:PROFILE1}, the growing mode~\eqref{E:GM}, and the space $\H^0_{2\m,Z_0}$ (Section~\ref{S:LOW}). 

For any given profile $\Xi=\begin{pmatrix} \zeta \\ \mu \end{pmatrix}$, we recall the $T$-modulated profile and decomposition~\eqref{E:XIDEF}:
\begin{align}
\PhiT = \Xi_0^T-\Xi_{\text{LP}}^T + \Xi_{\text{LP}}^T-\Xi_{\text{LP}}, \ \ \qquad \Xi_0:=\begin{pmatrix} \zeta_0 \\ \mu_0\end{pmatrix}, \ \  \Xi_{\text{LP}}:=\begin{pmatrix}\bzeta \\ \widehat\mu\end{pmatrix}.
\end{align}
Recalling~\eqref{E:GAMMAFORMULA}, we have $\frac d{dT}\left( \Xi_{\text{LP}}^T-\Xi_{\text{LP}}\right)\Big|_{T=0} =-\Gamma$ and thus 
\[
\left( \Xi_{\text{LP}}^T-\Xi_{\text{LP}}\,,\,\Gamma\right)_{\H^0_{2\m,Z_0}} =- T \|\Gamma\|_{\H^0_{2\m,Z_0}}^2 + O_{T\to0}(T^2).
\]
From here we infer that the equation $p(T)=0$ can be equivalently re-expressed in the form
\begin{align}
f(T) & =T, \\
 f(T) & : = \frac1{\|\Gamma\|_{\H^0_{2\m,Z_0}}^2} \left(\bigg(\Xi_0^T-\Xi_{\text{LP}}^T + \int_{\sin}^{S_T} e^{\sin-\sigma}\bfN[\Phi](\sigma)\,d\sigma\,,\, \Gamma\bigg)_{\H^0_{2\m,Z_0}}+ O(T^2)\right).
\end{align}
Therefore by Lemma~\ref{L:NBOUND} and Cauchy-Schwarz we conclude
\begin{align}
|f(T)| & \le  \|\Gamma\|_{\H^0_{2\m,Z_0}}^{-1} \Big(\|\Xi_0^T-\Xi_{\text{LP}}^T\|_{\H^0_{2\m,Z_0}} + C \sup_{\sigma\in[\sin,S_T)}\|\Phi(\sigma,\cdot)\|_{\mathcal{H}^{1}_{Z_0}}^2+CT^2\Big) \notag\\
& \le   \frac12 T_0 + C(Z_0)(\bar C\tilde\eps+ T^2).\label{E:F(T)EST}
\end{align}
by~\eqref{E:IDMODULATE} and~\eqref{E:T0DEF}.
In particular, for $\tilde \eps_0$ sufficiently small (recall this implies $|T|$ small since $T_0=C_0\sqrt{\tilde\eps}$), we find that $f$ maps $[-T_0,T_0]$ to itself. 

Due to Lemma~\ref{L:STPROPERTIES}, and as we have assumed $S_T$ always to be finite, it is straightforward to apply Proposition~\ref{P:HMZDIFF} and deduce that the map $f$ is also continuous on $[-T_0,T_0]$.

Then by Brouwer's theorem there exists a $T_\ast\in[-T_0,T_0]$ such that $p(T_\ast)=0$. Now, recalling the property $\textup{rg}\bfP=\langle\Gamma\rangle$ from Lemma~\ref{L:GROWINGPROJ}, this implies
\begin{align}
\bfP\bigg(\Phi_{\text{in}}^{T_\ast} +\int_{\sin}^{S_{T_\ast}} e^{\sin-\sigma}\bfN[\Phi](\sigma)\,d\sigma\bigg) =0.
\end{align}
 \end{proof}
 
 {\em Proof of Theorem~\ref{T:MAIN}.}
 Observe that if there exists $T_*\in[-T_0,T_0]$ such that $S_{T_*}=\infty$, then we are done. Suppose for a contradiction that $S_T<\infty$ for all $T\in[-T_0,T_0]$. Then there exists $T_*\in[-T_0,T_0]$ as in Proposition~\ref{P:BROUWER}.
 Now, applying Proposition~\ref{P:EE12}, we deduce, for $T=T_*$, that for $s\in(\sin,S_T)$, since $S^\ast_{T}\geq S_T$ (recall~\eqref{E:STAST}), 
\beqa
	e^{\nu\sg s}\|\Phi^{T_*}\|_{\HmZm}\leq &\,C\sqrt\eps_0\sup_{\sigma\in[s_{T_*},s]}(e^{\nu\sg \sigma}\|\Phi^{T_*}\|_{\HmZm})e^{-\Omega s} + e^{-(1-\nu)\sg s}\|\Phi^{T_*}_{\textup{in}}\|_{\HmZm},
	\eeqa
	which implies, for sufficiently small $\eps_0$, taking $\sup$ on the left hand side, 
	\beq
		\sup_{s\in [s_{T_*}, S_{T_*})}e^{\nu\sg s}\|\Phi^{T_*}\|_{\HmZm}\leq \frac32 
		\|\Phi^{T_*}_{\textup{in}}\|_{\HmZm},
	\eeq  
	so that $S_{T_*}=\infty$, a contradiction to our assumption that $S_T<\infty$ on $[-T_0,T_0]$. Thus there does indeed exist a $T_*\in[-T_0,T_0]$ such that $S_{T_*}=\infty$, and we conclude the proof of the main theorem.

\begin{remark}
One can determine the terminal
blow-up time $T_\ast$ implicitly as the solution of the  equation
\begin{align}
\bfP\Big(\Phi_{\text{in}}^{T_\ast} + \int_{s_{T_*}}^{\infty} e^{s_{T_*}-\sigma}\bfN[\Phi](\sigma)\,d\sigma\Big) =0.
\end{align}
To see this, one returns to the Duhamel identity of Section~\ref{S:DUHAMEL}. Assuming for a contradiction that the projection term does not vanish, one argues as in Proposition~\ref{P:EE12} to see that the contribution $\Phi_+$ as in~\eqref{E:STABLEMFD} grows exponentially in the $\HmZ$ norm, but is equal to a sum of decaying terms, leading to a contradiction. 
\end{remark}

\begin{remark}\label{R:ENHANCEDREG}
From the proof of Lemma~\ref{L:STPROPERTIES}, it appears that it ought to be possible to save one derivative by defining a variant of the energy $\tE_{k}$ for $k$ odd with suitably adapted weights. Indeed, performing an analysis analogous to that of Section~\ref{S:ENERGYBOUNDS}, one expects to be able to close all of the estimates of that section at the level $\tE_{2\m+1}$ rather than the $\tE_{2\m+2}$ that we actually estimate here. If one were to conduct this rather tedious procedure, it would of course lower the differentiability requirement on $(\zeta,\mu)^\top$ to require only 6 derivatives of $\zeta$ and 5 for $\mu$, and hence, at the level of the fluid variables, require only 5 derivatives.
\end{remark}


We finally turn to the proof of Theorem~\ref{T:EULERMAIN}, the Eulerian stability of the LP solution. We recall here that we only expect the leading order collapse behaviour to be described by the LP-behaviour in a zone generated by the data from the region $\{r\le Z_0\}$. 

\begin{lemma}[LP-dominated zone: $\{r\le Z_0\}$]\label{L:LPDOM}
Let $[-1,T)\ni t\mapsto \eta(t,\cdot)$ be a solution of Euler-Poisson generated by Theorem~\ref{T:MAIN}. 
Then, for any $h>0$ there
exists a constant $C_h>0$ such that if $e^sr\ge h$, then $C_h^{-1}R\le r \le C_h R$.
Conversely, there exists $\tilde C_h$ such that if $e^sR\le\frac{\bzeta(h)}{\tilde C_h}$ then $e^sr\le h$, and if $e^sR\ge\tilde C_h\bzeta(h)$ then $e^sr\ge h$. As a consequence, there exists a constant $c>0$, which may be taken independent of $Z_0$, such that if $R = \eta(t,r)\le c Z_0$, it follows that $r\le Z_0$.
\end{lemma}

\begin{proof}
Since $\eta(t,\cdot)=(T-t)\zeta(s,z)$ from~\eqref{E:SCALINGLAGR},
we have $\pa_r \eta = \pa_z\zeta$. Since $\|\frac{\pa_z\zeta(s,\cdot) - \pa_z \bzeta}{\pa_z\bzeta}\|_{L^\infty}\leq\bar C\sqrt{\eps_0}$ by the pointwise bound in~\eqref{E:MAINBOUND}, it follows that $\pa_z\zeta>0$ on $(0,\infty)$. In particular, the map $r\mapsto \eta(t,r)$ is strictly increasing.

Therefore, for any $h>0$ the set $\{e^sr\ge h\}$ is characterised by the condition $\zeta(s,e^s r)=e^s R\ge \zeta(s,h)$. 
Since $\frac{\zeta(s,e^sr)}{\bzeta(e^sr)}=1 + O(\sqrt{\eps_0})$ from the pointwise bound in Theorem~\ref{T:MAIN}, the first claim follows from the fact that there exists a constant $\tilde C_h>0$ such that
$\tilde C_h^{-1} z\le  \bzeta(z)\le \tilde C_h z$ for all $z\ge h$ by Lemma~\ref{L:ZETABAR}.

 The remaining claims follow from similar considerations, observing that for any $Z>0$, the set of Lagrangian labels $\{r\leq Z\}$ corresponds, at time $t$, to the Eulerian set of $R=\eta(t,r)$ such that $R\le \eta(t, Z)=e^{-s}\zeta(s,e^sZ)$, i.e.
\be\label{E:LPDOM1}
e^s R \le \zeta(s,e^s Z) = \bzeta(e^s Z)\Big(1+ \frac{\th(s,Z e^s)}{\bzeta (Z e^s)}\Big) =\bzeta(e^s Z)(1+O(\sqrt{\eps_0}))
\ee
from the pointwise bound in Theorem~\ref{T:MAIN}, while Lemma~\ref{L:ZETABAR} gives $\bzeta(z)= \tilde\zeta_{-1}z + O_{z\to\infty}(1)$ for some $\tilde\zeta_{-1}>0$.
\end{proof}


We now prove the existence of initial data in Eulerian variables that induces Lagrangian data satisfying the assumptions of our main Theorem~\ref{T:MAIN}. We recall that the LP flow map agrees with its self-similar representation at $t=-1$, so that $\eta_{\text{LP}}(-1,\cdot)=\bzeta(\cdot)$. For convenience, we recall from~\eqref{E:HSOBOLEVEQUIV} that we denote by $H^k$ the usual Sobolev space on $\R^3$, restricted to radial functions, i.e.~with the volume form $y^2\dif y$. For convenience, we also denote $\Delta_y$ the $\R^3$ Laplacian expressed in radial $y$ coordinate.

\begin{lemma}[Initial data in Eulerian variables]
Let $(\varrho_0,u_0)\in C^{2\m+2}_{loc}([0,\infty);\R^2)$. There exists $c_0>0$ such that if 
\begin{align}
&\big\|(\varrho_{\textup{LP}}-\varrho_{0},u_{\textup{LP}}-u_0)\big\|_{C^{2\m+2}([0,\bzeta(r_*)];\R^2)}^2\leq c_0\tilde\eps,\label{E:EULRHOINITBD}\\
&\sum_{j=0}^{\m+1} \bigg(\int_0^{\bzeta(Z_0)}\kappa\big|\Delta_y^j(u_{\textup{LP}}-u_0)\big|^2y^2\dif y+\int_{\bzeta(Z_0)}^\infty(1+y)^{2cj-\al}\big|(y^{\frac23}\pa_y^2 y^{\frac23})^j(u_{\textup{LP}}-u_0)\big|^2y^2\dif y\bigg)\leq c_0\eps_0,\label{E:EULUINITBD}
\end{align}
and, for $0\leq k\leq 2\m+2,$
\beq\label{E:EULRHODECAY}
\big|\pa_R^k\varrho_0(R)\big|\leq C(1+R)^{-(k+4)},\qquad R\geq \bzeta(r_*),
\eeq
then there exists an initial flow map $\eta_0:[0,\infty)\to[0,\infty)$ and gauge function $g$ so that $\eta_0$ is a strictly increasing bijection,
\begin{align}\label{E:ETASMALL}
\|\eta_{\textup{LP}}(-1,\cdot)\circ\eta_0^{-1} - \text{Id}\|_{C^{2}}\leq C \sqrt{\eps_0},
\end{align}
the pair $\tilde\Phi_0$ defined through
\begin{align}
\tilde\Phi_0 : = \begin{pmatrix} \eta_0-\eta_{\textup{LP}}(-1,\cdot) \\ u_0\circ\eta_0-u_{\textup{LP},0}\circ\eta_{\textup{LP}}(-1,\cdot)\end{pmatrix}
\end{align}
satisfies the smallness assumption~\eqref{IC0},
and $g$ satisfies the asymptotic flatness conditions~\ref{item:g1}--\ref{item:g2}.
\end{lemma}

\begin{proof}
From condition~\ref{item:g1}, we require $g(r)=\frac{\CLP}{4\pi}$ for $r\leq r_*$, and hence from~\eqref{E:DATAIN}, the flow map $\eta_0$ must satisfy
\beqa\label{E:ETA0ODE}
\pa_r\eta_0=&\,\frac{\CLP}{4\pi\eta_0^2\varrho_(\eta_0)}, \qquad r\leq r_*,\\
\eta_0(0)=&\,0.
\eeqa
From the regularity of $\varrho_0$ (observe that $\varrho_0\in C^2$ at least and $\varrho_0(0)>0$), it is simple to see that the solution to this ODE exists, satisfies the asymptotic relation $\eta_0(r)\simeq r^{\frac13}$ as $r\to0$, and  $\eta_0\in C^{2\m+3}(0,r_\ast]$. Moreover,
\beqa \label{E:CLOSETOLP}
\|\eta_0^3-\eta_{\text{LP}}^3(-1,\cdot)\|_{C^{2\m+3}[0,r_*]}\leq C\sqrt{\tilde\eps},\\
\|\eta_{\text{LP}}(-1,\cdot)\circ \eta_0^{-1}-\text{Id}\|_{C^{2\m+3}[0,r_*]}\leq C\sqrt{\tilde\eps},
\eeqa
where in these final inequalities, we use that $\eta_{\text{LP}}(-1,\cdot)$ satisfies~\eqref{E:DATAIN} with $\varrho_{\text{LP}}$.

Next, we observe that at $t=-1$, $r=z$, $\zeta_0=\eta_0$, and $\bzeta=\eta_{\text{LP}}(-1,\cdot)$,  so that we may move to working directly with $\zeta$ and $\bzeta$. Now we compute
\beqa
\bd (\zeta_0-\bzeta)=\pa_z\big(z^{\frac23}(\zeta_0-\bzeta)\big)=\pa_z(\zeta_0^3-\bzeta^3) + \pa_z\Big((\frac{z^{\frac23}}{\zeta_0^2}-1)\zeta_0^3\Big)-\pa_z\Big((\frac{z^{\frac23}}{\bzeta^2}-1)\bzeta^3\Big),
\eeqa
which enables a convenient calculation of derivatives in the Eulerian $y$ variable.
Now we see directly from~\eqref{E:ETA0ODE} that 
\beq
\frac{4\pi}{3\CLP}\pa_z(\zeta_0^3-\bzeta^3)=\frac{1}{\varrho_0(\zeta_0)}-\frac{1}{\varrho_{\text{LP}}(-1,\zeta)},
\eeq
and so, recalling from Definition~\ref{D:KEYOP} that $\bp$ and $\bd$ scale like the Eulerian gradient and divergence, respectively, it is straightforward to see that 
\begin{align}
\int_0^{Z_0}|\bD^{2\m+1}(\zeta_0-\bzeta)|^2\bG(z) g_{2\m}(z)\,\dif z \leq &\,C\sum_{j=0}^{\m}\int_0^{\bzeta(Z_0)}|\Delta_y^j\varrho_0(\zeta_0\circ\bzeta^{-1}(y))-\Delta_y^j\varrho_{\text{LP}}(-1,y)|^2 y^2\,\dif y,\notag\\
\leq&\, C\|\varrho_0-\varrho_{\text{LP}}\|_{H^{2\m+1}[0,\bzeta(r_*)]}^2,
\end{align}
where we have also used that the volume form $\dif z\simeq y^2\dif y$ and the boundedness of the weights $\bG$ and $g_{2\m}$. A similar argument for $\mu_0-\mu_{\text{LP}}$ shows that
\begin{align}
\int_0^{Z_0}|\bD^{2\m+1}(\mu_0-\widehat\mu)|^2 g_{2\m}(z)\,\dif z\leq C\|u_0-u_{\text{LP}}\|_{H^{2\m+1}[0,\bzeta(r_*)]}^2,
\end{align}
and hence, for $c_0$ sufficiently small, the first condition in~\eqref{IC0} is verified.

We now extend $\eta_0$ from $[0,r_*]$ to all of $[0,\infty)$. From the construction of $\eta_0$ on $[0,r_*]$ and the assumed regularity on $\varrho_0$, it is straightforward to see that $\eta_0\in C^{2\m+3}$ remains close in this space to $\eta_{\text{LP}}(-1,\cdot)$ as specified in~\eqref{E:CLOSETOLP}. Therefore  $\eta_0\in C^{2\m+3}$  may be extended as a $C^{2\m+3}$ function to all of $(0,\infty)$ such that, for $0\leq k\leq 2\m+3$, $\eta_0^{(k)}(r)$ satisfies the same asymptotics as $r\to\infty$ as $\pa_r^{k}\eta_{\text{LP}}(-1,r)$, see, e.g., Lemma~\ref{L:ZETABAR}, and so that the pointwise assumption of~\eqref{IC0} is satisfied. Moreover, using the assumption~\eqref{E:EULUINITBD} with $c_0$ sufficiently small, we perform this extension such that the final estimate of~\eqref{IC0} for the top order energies $\tE_{\leq 2(\m+1)}$ is also verified.

The gauge function $g$ is then determined, for $r\geq r_*$, by~\eqref{E:DATAIN}, that is through
\[g(r) = \varrho_0(\eta_0(r)) \eta_0^2 (r)\pa_r \eta_0(r) .\]
By distributing derivatives across this identity, exploiting the decay of $\eta_0^{(k)}$ by construction, and~\eqref{E:EULRHODECAY}, it is then trivial to verify that $g$ satisfies~\ref{item:g2}.
\end{proof}


\begin{proof}[Proof of Theorem~\ref{T:EULERMAIN}]
We note that for any $r\in[0,\infty)$ we have
\begin{align}
\varrho(t,\eta(t,r))& = (T-t)^{-2} \frac{g(r)}{\zeta(s,z)^2\pa_z\zeta(s,z)} = (T-t)^{-2} \frac{g(r)}{\bzeta(z)^2\pa_z\bzeta(z)} J(s,z), \label{E:RHOEX1}\\
\varrho_{\text{LP,T}}(t,\eta_{\text{LP},T}(t,r)) & = (T-t)^{-2} \frac{\CLP}{4\pi \bzeta(z)^2\pa_z\bzeta(z)},\label{E:RHOEX2}
\end{align}
where we recall~\eqref{E:JDEF}.
We define $f$ as the intertwining map
\begin{align}\label{E:FINTDEF}
f(t,\cdot): =\eta_{\text{LP},T}(t,\eta^{-1}(t,\cdot)).
\end{align}
We conclude from~\eqref{E:RHOEX1}--\eqref{E:RHOEX2} that
\begin{align}\label{E:RHOEX3}
\frac{\varrho(t,R)}{\varrho_{\text{LP,T}}\circ f(t,R)} =  \frac{4\pi g(r) J(s,z)}{\CLP}, \ \  r = \eta^{-1}(t,R).
\end{align}
From~\eqref{E:FINTDEF} it follows that 
\begin{align}
 \frac{e^s(f(t,R)-R)}{\bzeta(e^s r)}  = \frac{\bzeta(e^s r) - \zeta(s,e^s r)}{\bzeta(e^s r)} = -\frac{\th(s,e^s r)}{\bzeta(e^s r)}.
\end{align}
We recall now that by Proposition~\ref{P:PW} we have $R = e^{-s}\zeta(s,e^sr) = e^{-s}\bzeta(e^sr) (1+ O(\sqrt{\eps_0}))$, and therefore from the above
equation and Lemma~\ref{L:SHARPER} (with $h=1$) we conclude
\begin{align}\label{E:LAMBDAR0}
\big|\frac{f(t,R)}{R}-1\big| \le C\sqrt{\eps_0}\Big(e^{-\Om s}\chi_{s\le-\log r}(s) + r^{\Om}\chi_{s\ge - \log r}(s)\Big).
\end{align}
Moreover,
\begin{align}\label{E:LAMBDAR1}
\pa_R\big(\frac{f(t,R)}{R}\big) = \pa_R\big(\frac{\bzeta(s,e^sr)}{\zeta(s,e^sr)}\big) = \pa_z\big(\frac{\bzeta(s,z)}{\zeta(s,z)}\big)\big|_{z=e^s r} \frac{e^s}{\pa_z\zeta(s,e^sr)} ,
\end{align}
where we have used $\pa_rR =  \pa_z\zeta(s,e^sr)$. Rewriting $\frac{\bzeta(s,z)}{\zeta(s,z)}$ as $\frac{-\th(s,z)}{\zeta(s,z)}+1$, we easily see that 
\begin{align}
\Big|\pa_R\big(\frac{f(t,R)}{R}\big)\Big|
=\Big|\big(-\frac{\pa_z\th}{\pa_z\bzeta}\frac{\pa_z\bzeta}{\zeta} + \frac{\th}{\zeta}\frac{\pa_z\zeta}{\zeta}\big)\frac{e^s}{\pa_z\zeta}\Big|
\le C \big(\big|\frac{\pa_z\th}{\pa_z\bzeta}\big| + \big| \frac{\th}{\zeta}\big|\big) \frac{e^s}{\zeta} =\frac1R\big(\big|\frac{\pa_z\th}{\pa_z\bzeta}\big| + \big| \frac{\th}{\zeta}\big|\big)  ,
\end{align}
where we have used Proposition~\ref{P:PW}. Thus from~Lemma~\ref{L:SHARPER} we infer that 
\begin{align}\label{E:LAMBDAR3}
\big|R\pa_R\big(\frac{f(t,R)}{R}\big)\big|\le C\sqrt{\eps_0}\Big(e^{-\Om s}\chi_{s\le-\log r}(s) + r^{\Om}\chi_{s\ge - \log r}(s)\Big).
\end{align}
Using Lemma~\ref{L:LPDOM} it follows that the right-hand sides of~\eqref{E:LAMBDAR0} and~\eqref{E:LAMBDAR3}
are bounded by $C\sqrt{\eps_0}d_T(t,R)^\Om$, thus proving~\eqref{E:FBOUND}.

{\em Proof of part (a).} 
Bound~\eqref{E:SANDWICH} now follows easily from~\eqref{E:RHOEX3},~\eqref{E:JDEF},~\eqref{E:JBOUND}, and Proposition~\ref{P:PW}.

{\em Proof of part (b).} Since $R\le\frac{Z_0}{2}<\frac{r_\ast}{2}$ it follows from Lemma~\ref{L:LPDOM} that $r$ defined implicitly through the relation $R=\eta(t,r)$ 
satisfies $r\le r_\ast$. Therefore $g(r) = \frac{\CLP}{4\pi}$ and
from~\eqref{E:RHOEX3} we obtain the identity 
\begin{align}
\frac{\varrho(t,R)-\varrho_{\text{LP,T}}\circ f(t,R)}{\varrho_{\text{LP,T}}\circ f(t,R)} =  J(s,e^sr)-1, \ \  r = \eta^{-1}(t,R).
\end{align}
Together with~\eqref{E:JBOUND} (with $h=1$) this leads to 
\begin{align}
\big|\frac{\varrho(t,R)-\varrho_{\text{LP,T}}\circ f(t,R)}{\varrho_{\text{LP,T}}\circ f(t,R)}\big| 
&\le   C\sqrt{\eps_0}\Big(e^{-\Om s}\chi_{s\le-\log r}(s) + r^{\Om}\chi_{s\ge - \log r}(s)\Big).
\end{align}
Similarly to~\eqref{E:LAMBDAR1}--\eqref{E:LAMBDAR3} it is easy to see that for any $R\in[0,\frac{Z_0}{2}]$ we have (from $\frac{\rho}{\rhoLPT\circ f} = J(s,e^sr)$) 
\begin{align}
\big|R\pa_R\big(\frac{\varrho}{\rhoLPT\circ f}\big)\big| &\le C  \big(\big|z\frac{\pa_{zz}\th}{\pa_z\bzeta}\big|+\big|\frac{\pa_z\th}{\pa_z\bzeta}\big| + \big| \frac{\th}{\bzeta}\big|\big) \big(1 + \frac{\bzeta}{\Lambda\bzeta}\big) \notag\\
& \le C\sqrt{\eps_0}\Big(e^{-\Om s}\chi_{s\le-\log r}(s) + r^{\Om}\chi_{s\ge - \log r}(s)\Big),
\end{align}
where we have used Proposition~\ref{P:PW}, the uniform upper bound on $\frac{\bzeta}{\Lambda\bzeta} = \frac{1}{\omLP}$, and Lemma~\ref{L:SHARPER} in the last line.
Bound~\eqref{E:DOM} now follows by Lemma~\ref{L:LPDOM}. To prove~\eqref{E:DOM2} we note that for any fixed $R>0$ there exists a constant $c>0$ such that 
$e^sR>c$ implies $e^sr>1$, and thus the claim is a consequence of~\eqref{E:DOM}.

\medskip
{\em Proof of part (c).}
We observe that 
\begin{align}\label{E:UFORMULA}
u(t,R) = \pa_t\eta(t,r) = \mu(s,z), \ \ R = \eta(t,r),
\end{align}
where we recall that $\mu= -\zeta+\zeta_s+\Lambda \zeta$.
From~\eqref{E:UFORMULA} we infer that
\begin{align}\label{E:UFORMULA2}
u(t,R) &= -\bzeta (e^s r) + \Lambda\bzeta(e^sr) + \phi(s,e^sr) \notag\\
& = \uLPT \circ f + \phi(s,e^sr), \ \ R =\eta(t,r)
\end{align} 
and therefore the claim follows from Lemma~\ref{L:PHIUPPER}.
\end{proof}

\appendix
\numberwithin{equation}{section}

\section{Hardy-Sobolev bounds and interpolation estimates}


\subsection{Hardy at origin and norms}
We collect together some convenient Hardy-type inequalities. For convenience, we state these for $C^1$ functions, but they extend by simple density arguments whenever the right hand sides are finite.

\begin{lemma}\label{L:UPW}
Let $\beta\in \R$, $\beta\neq 1$. Let $0\leq \tilde Z< Z<\infty$ be given. Then there exists $C>0$, depending only on $\beta$, such that, for any $u\in C^1(\tilde Z,Z)$,
\begin{align}
\int_{\tilde{Z}}^Z u^2 z^{\beta-\frac{2}{3}}\,\dif z + Z^{\beta+\frac13}|u(Z)|^2  \leq C \bigg(\tilde Z^{\beta+\frac13}|u(\tilde Z)|^2 + \int_{\tilde{Z}}^Zz^{\beta}|\bd u|^2\,\dif z\bigg), \ \ \ \beta<1,\label{E:HARDYREFINED}\\
\int_{\tilde{Z}}^Z u^2 z^{\beta-\frac{2}{3}}\,\dif z +  \tilde Z^{\beta+\frac13}|u(\tilde Z)|^2 \leq C \bigg(Z^{\beta+\frac13}|u(Z)|^2 + \int_{\tilde{Z}}^Zz^{\beta}|\bd u|^2\,\dif z\bigg), \ \ \ \ \beta>1\label{E:HARDYREFINEDb}
\end{align}
and, 
\begin{align}
\int_{\tilde{Z}}^Z u^2 z^{\beta-2}\,\dif z + Z^{\beta-1}|u(Z)|^2  \leq C\bigg( \tilde Z^{\beta-1}|u(\tilde Z)|^2 + \int_{\tilde{Z}}^Zz^{\beta-\frac43}|\bp u|^2\,\dif z\bigg),\ \ \ \beta<1,\label{E:HARDYREFINED2}\\
\int_{\tilde{Z}}^Z u^2 z^{\beta-2}\,\dif z + \tilde Z^{\beta-1}|u(\tilde Z)|^2  \leq C\bigg( Z^{\beta-1}|u(Z)|^2 + \int_{\tilde{Z}}^Zz^{\beta-\frac43}|\bp u|^2\,\dif z\bigg),\ \ \ \beta>1\label{E:HARDYREFINED2b}
\end{align}
\end{lemma}

\begin{proof}
To show~\eqref{E:HARDYREFINED}, we use $0\le \|z^{\frac{\beta}{2}}(\al \bd u-  u z^{-\frac13})\|_{L^2(0,Z)}^2$, expand inside the integral, integrate-by-parts, and set $\al=\frac{1}{1-\beta}$. The other inequalities follow similarly.
\end{proof}


\begin{lemma}\label{L:LINFTYHARDYSOBOLEV}
Let $Z>0$ be given. Then there exists $C=C(Z)>0$, such that for all $u\in C^2(0,Z)$,
\begin{align}
\big\|u\|_{L^\infty(0,Z)}\leq C\Big(\|\bl u\|_{L^2(0,Z)}+\|\bp u\|_{L^2(0,Z)}+\|u\|_{L^2(0,Z)}\Big).\label{E:LINFTYHARDY2}
\end{align}
As a consequence, for $k\in\N$, $(\th,\phi)^\top\in\mathcal{H}^{k+1}_Z$, $Q_1\in\mathcal X_k$, $Q_2\in\mathcal{X}_{k-1}$ where we recall the definition~\eqref{E:XTWOJ}--\eqref{E:XTWOJPLUSONE}, we have the estimate
\begin{align}\label{E:XKHARDY}
\|Q_1\theta\|_{L^\infty(0,Z)}+\|Q_2\phi\|_{L^\infty(0,Z)} \le C \big\|\begin{pmatrix}\theta\\\phi\end{pmatrix}\big\|_{\mathcal H^{k+1}_{Z}}.
\end{align}
\end{lemma}

\begin{proof}
This follows directly from~\cite[Lemma C.2]{GHJ2021a}. More precisely, we observe that, away from the origin $z=0$, inequality~\eqref{E:LINFTYHARDY2} follows from the standard Sobolev embedding, and so it suffices to prove the estimates for functions supported on the interval $(0,\frac12)$. In this case, transferring back to Eulerian variables, we recall that the operators $\bp$ and $\bd$ become the Eulerian gradient and divergence, while the volume form $\dif z=r^2\dif r$. Thus~\eqref{E:LINFTYHARDY2} follows from~\cite[(C.425)]{GHJ2021a}.

The final estimate,~\eqref{E:XKHARDY}, now follows by induction on the order $k$ and combining~\eqref{E:LINFTYHARDY2} with~\eqref{E:HARDYREFINED}.
\end{proof}

\begin{lemma}[Norm equivalence]\label{L:EQUIVALENCE}
Let $m\in\mathbb N$ be given. Then for any $\begin{pmatrix} \theta \\ \phi \end{pmatrix}\in\H^m_Z$ we have
\begin{align}
\left\|\begin{pmatrix} \theta \\ \phi \end{pmatrix}\right\|_{\H^m_Z} 
\cong   \sum_{k=1}^{m+1} \sum_{Q\in \X_k}\left\|Q \theta \right\|_{L^2(0,Z)}+\sum_{k=0}^{m} \sum_{Q\in \X_k}\left\|Q\phi \right\|_{L^2(0,Z)},
\end{align}
where we recall the operator classes $\X_j$~\eqref{E:XTWOJ}--\eqref{E:XTWOJPLUSONE} and the constants of equivalence for the two norms depend on $Z$.
\end{lemma}


\begin{proof}
That the $\HmZ$ norm is bounded by the quantity on the right is trivial from the definition~\eqref{E:INNERPROD} of the norm and the algebras $\mathcal{X}_k$. The converse direction follows by induction on $m$ and application of the Hardy inequality~\eqref{E:HARDYREFINED}.
\end{proof}


\subsection{Hardy at infinity and embeddings}


\begin{lemma} Let $\beta\in\R$, $Z> 0$, and $v\in C^1(Z,\infty)$ be such that $\lim_{z\to\infty}z^{\beta+\frac13}v^2(z)=0$. Then the following identity holds:
\begin{align}
\int_Z^\infty (\bp v)^2 z^{\beta} \diff z + \frac23(\frac13-\beta) \int_Z^\infty v^2 z^{\beta-\frac23}\diff z = \frac23 Z^{\beta+\frac13} v^2(Z) + \int_Z^\infty (\bd v)^2 z^{\beta} \diff z.  \label{E:HARDYinf3}
\end{align}
In particular, given $\al>1$, there exists $C>0$, depending on $\al$, such that 
\be\label{theta_bound}
\int_Z^\infty z^{-\alpha} (\pa_z\theta)^2\diff z + \int_Z^\infty z^{-\alpha-2}\theta^2  \leq C \int_0^\infty (1+z)^{-\alpha} \frac{(\bd \theta)^2}{(\bp\bzeta)^2} \diff z.
\ee
\end{lemma}

\begin{proof}
To prove~\eqref{E:HARDYinf3}, we simply expand $(\bp v)^2 = (\bd v-\frac23 z^{-\frac13} v)^2$ inside the integral and integrate by parts in the resulting cross-term.
To show~\eqref{theta_bound}, we combine~\eqref{E:HARDYinf3} for $\beta=-\frac43-\al$ with \eqref{E:HARDYREFINED}.
\end{proof}

\begin{lemma}\label{L:LARGEZ} Let $c\in(0,\frac23)$, $\alpha>1$, $Z\geq 1$ be given such that $2c-\alpha + \frac43>1$ be given. Then there exists $C>0$, depending on $c,\al$, such that for all $\phi\in C^1(0,\infty)$ such that $\lim_{z\to\infty}z^{2c-\al+\frac13}|\phi(z)|^2=0$, 
\begin{equation}\label{E:phi_z}
\int_Z^\infty (\pa_z\phi)^2 z^{2c-\alpha+ \frac23} \diff z \leq C\bigg( \int_0^Z | \bd \phi |^2 (1+z)^{2c-\alpha-\frac23} \diff z + 
\int_Z^\infty |\bp \bd \phi|^2 (1+z)^{2c-\alpha} \diff z \bigg).
\end{equation}
Moreover, for any $\th\in C^1(0,\infty)$,
\begin{equation}\label{E:theta}
\int_Z^\infty\theta^2 z^{2cj -\alpha-2} \diff z \leq C
\int_0^\infty \frac{|\bd \theta|^2}{(\bp \bzeta)^2} (1+z)^{2cj -\alpha} \diff z, \ \text{ if } \  2cj -\alpha <\frac73 
\end{equation} 
and, if also $\lim_{z\to\infty}z^{2cj-\al-1}|\th(z)|^2=0$,
\begin{equation}\label{E:theta2}
\int_Z^\infty\theta^2 z^{2cj -\alpha-2} \diff z \leq C 
\int_Z^\infty \frac{|\bd \theta|^2}{(\bp \bzeta)^2} (1+z)^{2cj -\alpha} \diff z, \ \text{ if } \  2cj -\alpha >\frac73 .
\end{equation}
\end{lemma}

\begin{proof} Since $2c-\alpha + \frac43>1$, 
by \eqref{E:HARDYREFINED2b}, we first obtain (noting that the boundary term at infinity vanishes by assumption on $\phi$)
\[
\int_Z^\infty | \bd \phi|^2 (1+z)^{2c-\alpha-\frac23} \diff z  \leq C \int_Z^\infty |\bp \bd \phi|^2 (1+z)^{2c-\alpha} \diff z.
\]
On the other hand, an identity analogous to~\eqref{E:HARDYinf3} shows 
\[
\begin{split}
\int_Z^\infty  | \bd \phi|^2 (1+z)^{2c-\alpha-\frac23} \diff z = \int_Z^\infty  | \bp \phi|^2 (1+z)^{2c-\alpha-\frac23} \diff z  - \frac23 Z^\frac13 (1+Z)^{2c-\alpha-\frac23} \phi^2 (Z) \\
+ \int_Z^\infty (\tfrac29 z^{-\frac23} -\tfrac23(2c-\alpha-\tfrac23)\frac{z^\frac13}{1+z}) (1+z)^{2c-\alpha-\frac23} \phi^2 \diff z.
\end{split}
\]
Hence, as $\al>1$ and $c<\frac23$, so that $2c-\al-\frac23<0$,
\beqas
{}&\int_Z^\infty | \pa_z \phi|^2 z^{2c-\alpha+\frac23} \diff z \leq C \bigg(Z^{\frac13} (1+Z)^{2c-\alpha-\frac23} |\phi(Z)|^2 + \int_Z^\infty | \bd \phi|^2 z^{2c-\alpha-\frac23} \diff z\bigg)\\
& \leq  C\bigg(|\phi(1)|^2+ \int_1^Z | \bd \phi|^2 z^{2c-\alpha-\frac23} \diff z+ \int_Z^\infty | \bd \phi|^2 z^{2c-\alpha-\frac23} \diff z\bigg),
\eeqas
where we have applied~\eqref{E:HARDYREFINED}.
By further using \eqref{E:HARDYREFINED}, we obtain~\eqref{E:phi_z}. 

To prove \eqref{E:theta}, we apply~\eqref{E:HARDYREFINED} twice to obtain
\beqa
{}&\int_Z^\infty\theta_j^2 z^{2cj -\alpha-2} \diff z \leq C\bigg( Z^{2cj -\alpha-\frac43+\frac13} \theta_j^{2}(Z) +  \int_Z^\infty (\bd \theta_j)^2z^{2cj -\alpha-\frac43} \diff z\bigg)\\
&\leq C \bigg(\theta_j^{2}(1)+ \int_1^\infty  \frac{(\bd \theta_j)^2}{(\bp\bzeta)^2}z^{2cj -\alpha} \diff z\bigg)\leq C \int_0^\infty  \frac{(\bd \theta_j)^2}{(\bp\bzeta)^2}(1+z)^{2cj -\alpha} \diff z,
\eeqa
where we have used that $z^{-\frac43}\simeq (\bp\bzeta)^2$ on $(1,\infty)$ and $(\bp\zeta)^{-2}$ is bounded above and below on $(0,1)$ by Lemma~\ref{L:GBAR}. Finally, \eqref{E:theta2} follows from~\eqref{E:HARDYREFINEDb}. 
\end{proof}


\section{Semi-group theory}


For the convenience of the reader, we state here a version of the spectral theorem on compact perturbations of semigroups from~\cite[Appendix B]{Glogic22}. We first recall the definition of the Riesz projection associated to an isolated point $\lambda_0\in\sigma(A)$ for some closed, linear operator $A$ on a Banach space $X$, as in~\eqref{E:RIESZDEF}:
\[\bfP_{\lambda_0}:=\frac{1}{2\pi i}\int_{\gamma}(\lambda I-A)^{-1}\dif\lambda,\]
where $\gamma$ is a positively oriented circle centred at $\lambda_0$ contained in the resolvent set $\rho(A)$ and containing no other spectral point of $A$ in its interior.
\begin{theorem}[{\cite[Theorem B.1]{Glogic22}}]\label{THM:COMPACTPERT}
Let $X$ be a Banach space and suppose $A_0:D(A_0)\subset X\to X$ is a closed, densely defined linear operator generating a strongly continuous semigroup $\mathcal S_0=\big(S_0(\tau)\big)_{\tau\geq 0}\subset\mathcal B(X)$ satisfying
$$\|S_0(\tau)u\|\leq Me^{\varpi_0\tau}\|u\|$$
for all $u\in X$, $\tau\geq 0$ and some $M>0$ and $\varpi_0\in\R$. Let $B_0:X\to X$ be compact. Then the operator
$$A:=A_0+B_0:D(A):=D(A_0)\subset X\to X$$
generates a strongly continuous semigroup $\mathcal S=\big(S(\tau)\big)_{\tau\geq 0}\subset\mathcal B(X)$ such that, for all $\eps>0$, the following properties hold:
\begin{itemize}
\item[(i)] The spectrum of $A$ satisfies
$$\sigma_\eps:=\sigma(A)\cap\{\lambda\in\C\,|\,\Re\lambda\geq \varpi_0+\eps\}$$
is a finite set of discrete eigenvalues with finite algebraic multiplicity.
\item[(ii)] The spectral mapping
$$\sigma(S(\tau))\setminus\mathbb D_{e^{\tau(\varpi_0+\eps)}}=\{e^{\tau\lambda}\,|\,\lambda\in \sigma_\eps\},$$
for all $\tau>0$, where $\mathbb D_r$ is the open disc of radius $r$ centred at $0$.
\item[(iii)]  Let $\bfP=\sum_{\lambda\in\sigma_\eps}\bfP_\lambda$ be the Riesz projection onto the set $\sigma_\eps$ and write $r_\eps:=\sup\{\Re\lambda\,|\,\lambda\in\sigma(A)\setminus\sigma_\eps\}$. Then, for all $\sg> \max\{\varpi_0+\eps,r_\eps\}$, there exists $C\geq 1$ such that
$$\|S(\tau)(1-\bfP)u\|\leq Ce^{\sg\tau}\|(1-\bfP)u\|$$
for all $u\in X$.
\end{itemize}
\end{theorem}


\section{Eulerian and Lagrangian formulations and equivalence}\label{A:EULLAGEQUIV}


\subsection{Steady state properties}\label{SS:SSP}

We recall that the Larson-Penston solution is a self-similar solution of the Euler-Poisson system, that is, a steady solution $(\rhoLP,\vLP)=(\rhoLP,y\omLP)$ of system~\eqref{E:CONTSS}--\eqref{E:MOMSS}.
 As shown in~\cite{GHJ2021b}, the functions $\rhoLP$ and $\omLP$ satisfy the system
\begin{align}
\rhoLP'=&\,-\frac{2y\rhoLP\omLP(\rhoLP-\omLP)}{1-y^2\omLP^2},\label{E:RHOLP}\\
\omLP'=&\,\frac{1-3\omLP}{y}+\frac{2y\omLP^2(\rhoLP-\omLP)}{1-y^2\omLP^2}.\label{E:OMLP}
\end{align}
The Larson-Penston solution is an analytic solution of~\eqref{E:RHOLP}--\eqref{E:OMLP}, see Theorem~\ref{T:LP}. It is simple to show that $\rhoLP$ and $\omLP$ are both even in $y$ (due to the invariance of~\eqref{E:RHOLP}--\eqref{E:OMLP} under the change $y\mapsto-y$, see also Appendix~\ref{S:ORIGINTAYLOR}).

In order to employ the Lagrangian formulation, we collect here some crucial properties of the Lagrangian flow map. Recall from Lemma~\ref{L:LPLAGRANGIAN} that the initial particle labelling satisfies~\eqref{E:DATAIN}, with gauge function $g_{\textup{LP}}$ constant.
The following lemma is very simple to prove, but points to a crucial shift of perspective when working with Lagrangian coordinates.
\begin{lemma}[Flow map regularity at $r=0$]\label{L:ONETHIRD}
Assume that $\varrho_0:[0,\infty)\to\mathbb (0,\infty)$ and $g:[0,\infty)\to (0,\infty)$ are $C^1$-functions and let $\eta_0$ be an initial labelling satisfying~\eqref{E:DATAIN} and such that $\eta_0(0)=0$. Then $\eta_0$ satisfies the asymptotic relationship
\begin{align}
\eta_0(r)& \sim_{r\to0^+} r^{\frac13} \label{E:ONETHIRDREG}.
\end{align}
\end{lemma}


\begin{proof}
Claim~\eqref{E:ONETHIRDREG} is a simple consequence of~\eqref{E:DATAIN} and the strict positivity of $g(0),\varrho_0(0)$. 
\end{proof}

As a consequence, the Lagrangian flow map associated with positive smooth $\varrho_0$ and $g$, in particular with the LP-collapse 
is not smooth as a function of the Lagrangian label $r$ at $r=0$. In the self-similar coordinates, a similar feature appears, and we include
its description in the next lemma. In order to state and prove the expansions of $\bzeta$ at $z=0$ and $z=\infty$, we first recall from Lemma~\ref{L:LPLAGRANGIAN} the defining relation
\be\label{relation00}
z\pa_z\bzeta(z) = \bzeta(z) \omLP(\bzeta(z)).
\ee 
Differentiating this identity, we further find
\begin{align}
\frac{z\pa_{zz}\bzeta}{\pa_z\bzeta}=\bzeta\omLP'(\zeta(z))+\omLP(\bzeta(z))-1.\label{eq:zdzzZeta}
\end{align}

\begin{lemma}\label{L:ZETABAR} 
The self-similar flow map, $\bzeta$, expands analytically in powers of $z^{\frac13}$ near the sonic point. More precisely, there exist constants $(\zeta_k)_{k\in\N_0}$, $\zeta_0=1$, such that, for $z\ll1$,
\begin{align}\label{E:BARZETAREGULARITY}
\bzeta(z) = z^{\frac13}\sum_{k=0}^\infty\zeta_k z^{\frac{2k}{3}}.
\end{align}
In the far-field, there exist constants $\tilde{\zeta}_{-1}>0$ and $(\tilde\zeta_k)_{k\in\N_0}$ such that, for $z\gg1$, 
\beq
\bzeta(z) = \tilde{\zeta}_{-1} z + \sum_{k=0}^\infty \tilde\zeta_k z^{-k}.
\eeq
\end{lemma}

\begin{proof} The expansion at the origin follows from the real-analyticity of $(\rhoLP,\omLP)$ at $y=0$, the property $\omLP(0)=\frac13$, and \eqref{relation00}. The far-field expansion follows similarly from the analytic properties of $\omLP$, in particular, $\omLP(z)=1+\sum_{k=1}^\infty \tom_kz^{-k}$,  which can be derived as in~\cite[Theorem 7.4]{GHJ2023}.
\end{proof}

The final properties of the LP solution that we require concern the behaviour of the flow map $\bzeta$ around the sonic point $z_*$, defined in Definition~\ref{def:sonic}. We recall the definition of the key weight function $\bG(z)$ from~\eqref{E:GGDEF}. We characterise the sub/supersonicity of the LP solution via the sign of the function $z\mapsto z^{\frac23}-G(z)$. Namely $z=z_\ast$ if and only if $G(z_\ast)=z_\ast^{\frac23}$ 
and
\begin{align} \label{E:GRELATION}
z\lessgtr z_\ast \Longleftrightarrow z^{\frac23}-G(z) \lessgtr 0.
\end{align}
In Lemma~\ref{L:GBAR}, we collected some key properties of the function $z\mapsto G(z)$, which we now prove here.

\begin{proof}[Proof of Lemma~\ref{L:GBAR}]
The asymptotic behaviour $\bzeta^{-2} =z^{-\frac23}\big( 1+ O_{z\to\infty}(z^{\frac23})\big)$ follows directly from~\eqref{E:BARZETAREGULARITY} and hence the expansion
$\bG=9 + O_{z\to0}(z^{\frac23})$ follows easily. This implies $\bG\in\Dinfeven$. Similarly, the boundedness $\frac1C\leq \pa_z\bzeta \leq  C$ for $z$ large from Lemma~\ref{L:ZETABAR} implies the second estimate in~\eqref{E:GBDS}. 
 
To show~\eqref{E:DZG}--\eqref{E:DGSQRTG}, we observe that since $\bG = (\bp \bzeta)^{-2}$ we have
\begin{align}
\frac{\bp \bG}{\sqrt{\bG}} & = \bp \left((\bp \bzeta)^{-2}\right) \bp \bzeta = - 2 \frac{\bp^2\bzeta}{(\bp \bzeta)^2}  = - 2 \frac{z^2\pa_{zz}\bzeta+\frac23 z \pa_z\bzeta}{(z\pa_z \bzeta)^2}.
\end{align}
From~\eqref{eq:zdzzZeta}, we obtain
\[
z^2\pa_{zz}\bzeta = z\pa_z\bzeta \left(\omLP-1+\bzeta \omLP'\right).
\]
Therefore
\begin{align}
\frac{\bp \bG}{\sqrt{\bG}} 
& = -2\frac{\omLP+y\omLP'-\frac13}{\vLP},\label{E:dbarGidentity}
\end{align}
where we recall $y=\bzeta$, $\bar v = y \om$. From the key monotonicity property $\om'(y)\geq 0$ and $\om(y)\geq \frac13$, we deduce~\eqref{E:DZG}. From the asymptotics $\omLP(\bzeta(z))=1+O(\frac1z)$ and Lemma~\ref{L:ZETABAR}, we obtain \eqref{E:DZGBARFF}.
\end{proof}


\subsection{Stability problem in the Eulerian variables}

In this subsection, we identify the linearised operator in Eulerian variables. We recall that the Eulerian EP-system in self-similar variables was stated above in~\eqref{E:CONTSS}--\eqref{E:MOMSS}. 
To formulate the stability problem it is convenient to work with the momentum variable
\begin{align}\label{E:PIDEF}
\Pi(s,y): = \rho(s,y) v(s,y)
\end{align}
and the dynamically accessible variable $\psi$, defined as the unique solution to the equation
\begin{align}\label{E:PHIDEF}
D_y \psi(s,y) = \rho(s,y)-\rhoLP(y),\qquad \psi(s,0)=0.
\end{align}
We note that the existence of such a $\psi$ is guaranteed by the ellipticity of the divergence in radial coordinates, provided $\rho$ is sufficiently regular.
 This allows us to reformulate and linearise the self-similar problem~\eqref{E:CONTSS}--\eqref{E:MOMSS}. 

\begin{lemma}[Eulerian linearisation]\label{L:EULERIANLIN}
Let $(\rho,v)$ be a smooth solution of~\eqref{E:CONTSS}--\eqref{E:MOMSS} and let 
\[
\rho=\rhoLP+ D_y\psi, \ \  
\Pi = \PLP+ P,
\]
and assume that $(\rho,\Pi)$ solves the system~\eqref{E:CONTSS1}--\eqref{E:MOMSS1}. Then the pair $(\psi,P)$ solves the system
\begin{align}
\pa_s\begin{pmatrix} \psi \\ P \end{pmatrix} = \bfLEul \begin{pmatrix} \psi \\ P \end{pmatrix}+\begin{pmatrix} 0 \\ N^{\textup{Eul}}[D_y\psi,P] \end{pmatrix},\label{E:EULPERTURB}
\end{align}
where the Eulerian linearisation takes the form
\begin{align}
\bfLEul \begin{pmatrix} \psi \\ P \end{pmatrix}
= \begin{pmatrix} - P+\psi \\ - \pa_y \left(w D_y \psi \right) -2\pa_y(\vLP P) + 2\vLP\left(\omLP - \rhoLP\right) D_y \psi - 2\rhoLP \psi + (2-4\omLP) P \end{pmatrix}, \label{E:BFLEULDEF}
\end{align}
where we recall $w=1-\vLP^2$ from~\eqref{E:WDEF} and the nonlinearity is given by
\beq
N^{\textup{Eul}}[D_y\psi,P]: = -D_y\left(\frac{(\vLP D_y\psi-P)^2}{\rhoLP + D_y\psi} \right) - \frac{2D_y\psi M[D_y\psi]}{y^2}.
\eeq
If the perturbation $(\psi,P)$  solves the linearised problem obtained by neglecting the nonlinearity in~\eqref{E:EULPERTURB}, then equivalently $\psi$ solves the second-order linear partial differential equation
\begin{align}
&\pa_{ss}\psi +2\pa_sD_y(\vLP \psi) + D_y\left(\vLP^2 D_y \psi\right) -  \pa_yD_y \psi \notag\\
&  -2\PLP D_y \psi-2D_y(\vLP \psi)-3\pa_s\psi+(2- 2\rhoLP) \psi =0.\label{E:LIN4}
\end{align}
\end{lemma}

\begin{proof}
As $(\rho,v)$ is a smooth solution of~\eqref{E:CONTSS}--\eqref{E:MOMSS}, clearly $(\rho,\Pi)$ solve
\begin{align}
\pa_s\rho + D_y(\Pi) -\rho & = 0 ,\label{E:CONTSS1}\\
\pa_s\Pi + D_y(\frac{\Pi^2}{\rho}) + \pa_y\rho + 2\frac{\rho M}{y^2} -2\Pi & =0.   \label{E:MOMSS1}
\end{align}
As~\eqref{E:CONTSS1} is a linear equation, the first equation in~\eqref{E:EULPERTURB} follows easily.
Writing $R=D_y\psi$ for convenience, to expand the momentum equation around the LP solution, we first consider
\begin{align}
\frac{\Pi^2}{\rho}
& = (\PLP^2+2\PLP P+P^2)\left(\frac1{\rhoLP} - \frac{R}{\rhoLP^2}+ \frac{R^2}{\rhoLP^2(\rhoLP+R)} \right) \notag\\
& =  \frac{\PLP^2}{\rhoLP} - \vLP^2 R + \frac{\vLP^2 R^2}{\rhoLP+R} + 2\vLP P - \frac{2\vLP R P }{\rhoLP} + \frac{2\vLP P R^2}{\rhoLP (\rhoLP+R)}
+ \frac{P^2}{\rhoLP} - \frac{P^2 R}{\rhoLP^2} + \frac{P^2 R^2}{\rhoLP^2 (\rhoLP+R)}\notag\\
& = \frac{\PLP^2}{\rhoLP}+ 2\vLP P  - \vLP^2 R+ \frac{(\vLP R-P)^2}{\rhoLP + R}. 
\end{align}
The second equation in~\eqref{E:EULPERTURB} then follows directly. Equation~\eqref{E:LIN4} is a direct consequence.
\end{proof}


\subsection{Equivalence of Eulerian and Lagrangian linearisation}  


In this section, we show that the Eulerian and Lagrangian linearisation are equivalent. As it is simplest to prove this for the second order linearised operators, we state the equivalence for these formulations of the linearised problem and recall from Lemmas~\ref{L:GPROPERTIES} and~\ref{L:EULERIANLIN} that the first order formulations are then also equivalent. To this end, therefore, we recall from the proof of Lemma~\ref{L:GPROPERTIES} (specifically equations~\eqref{E:NONLINEARFLOW} and~\eqref{E:INTERIORFLOW}) that the Lagrangian linearized perturbation problem in second order form may be written, for $z\in(0,r_*e^s)$,
\begin{align}\label{E:THETADYNAMICS}
\theta_{ss} +2z\theta_{sz} - \pa_z\Big(\big(\frac{1}{\bzeta_z^2}-z^2\big) \theta_{z}\Big) - \theta_s - 2z\theta_z +\frac{2}{\bzeta}\Big(1-\frac{\CLP z}{\bzeta}\Big)\theta = 0.
\end{align}


\begin{lemma}[Equivalence between the Lagrangian and the Eulerian linearisation] \label{L:LAGEUL}
The Lagrangian and Eulerian linearised operators $\bfL$ and $\bfLEul$ satisfy the following equivalence properties.
\begin{itemize}
\item[(i)] Let $\theta\in\DZodd$ be a  solution to the linearised problem~\eqref{E:THETADYNAMICS}.
Then the 
function 
\begin{align}\label{Def:phi_y}
\psi(s,y) : = \bar\rho(y) \theta(s,z), \ \ z=\bzeta^{-1}(y)
\end{align}
solves the Eulerian linearisation~\eqref{E:LIN4}. Similarly, if $\psi$ is a smooth solution of~\eqref{E:LIN4}, then the function 
\begin{align}
\theta(s,z) : = \frac{\psi(s,y)}{\bar\rho(y)}, \ \ y=\bzeta(z)
\end{align}
solves the Lagrangian linearisation~\eqref{E:THETADYNAMICS}.
\item[(ii)] For any given $\l\in\mathbb C$ and $ \begin{pmatrix} f_1 \\ f_2 \end{pmatrix}\in \HmZ$, if 
\begin{align}\label{E:RESLAG}
(\bfL - \la) \begin{pmatrix} \th \\ \phi \end{pmatrix} = \begin{pmatrix} f_1 \\ f_2 \end{pmatrix},
\end{align}
then
\begin{align}
(\bfLEul - \la) \begin{pmatrix} \psi \\ P \end{pmatrix} = \begin{pmatrix} h_1 \\ h_2 \end{pmatrix}, \label{E:RESEUL}
\end{align}
where $\psi$ is given by~\eqref{Def:phi_y}, $P=(1-\la)\psi+h_1$, and 
\begin{align}\label{E:H12}
h_1 =  -\frac12 \rhoLP f_1\circ\bzeta^{-1}, \ \ h_2 = \rhoLP f_2\circ\bzeta^{-1} - \la h_1.
\end{align}
\end{itemize}
\end{lemma}


\begin{proof}
The results directly follow from the change of variables and unknowns via \eqref{Def:phi_y}, and by using the LP equations \eqref{eq:v} and \eqref{E:RHOSS}. We sketch the proof for readers' convenience. From the the LP properties \eqref{E:LPPROP} and  $z\pa_z = z\bzeta_z\pa_y = \vLP \pa_y $, we first obtain 
\[
\pa_z ((\bzeta_z^{-2}-z^2) \pa_z \theta ) =\frac{\vLP}{z} \pa_ y \Big(  (1-\vLP^2) \frac{z}{\vLP} \pa_y\theta\Big) =  \pa_ y\big( (1-\vLP^2) \pa_y\theta \big) -  \frac{\vLP' -1}{\vLP} (1-\vLP^2) \pa_y \theta.
\]
Now~\eqref{E:THETADYNAMICS} can be written in $(s,y)$ coordinates: 
\[
\pa_s^2\theta + 2 \vLP \pa_{y s}^2 \theta - \pa_ y\big( (1-\vLP^2) \pa_y\theta \big) + \frac{\vLP' -1}{\vLP} (1-\vLP^2) \pa_y \theta - \pa_s \theta  -2 \vLP \pa_y \theta  + \frac{2}{y^2} \left( 1- 2 y\bar\rho\vLP \right) \theta=0.
\]
The equation for $\psi =\bar\rho \theta$ now reads 
\be\label{linear_psi}
\pa_s^2\psi +2 \bar\rho \vLP \pa_{y s}^2 \theta- \pa_s \psi  - \bar\rho \pa_ y\big( (1-\vLP^2) \pa_y\theta \big) + \frac{\vLP' -1}{\vLP} (1-\vLP^2)\bar\rho  \pa_y \theta  -2 \vLP\bar\rho  \pa_y \theta  + \frac{2}{y^2} \left( 1-2 y\bar\rho\vLP \right) \psi =0.
\ee
Using the LP identity $(y^2\bar\rho\vLP)'=y^2\bar\rho$, which is derived easily from~\eqref{eq:v}--\eqref{E:RHOSS}, it is easy to verify  
\begin{align*}
2 \bar\rho \vLP \pa_{y s}^2 \theta- \pa_s \psi&=  2  \pa_s D_y  (\vLP\psi)  - 3\pa_s \psi. 
\end{align*} 
We note further that $\pa_y\theta = \frac{\pa_y \psi}{\bar\rho} - \frac{\bar\rho '}{\bar\rho } \frac{ \psi}{\bar\rho}$. Using this and the LP equations~\eqref{eq:v} and~\eqref{E:RHOSS} to express  terms in terms of $\psi$, after a calculation we conclude that 
\begin{align}
&- \bar\rho \pa_ y\big( (1-\vLP^2) \pa_y\theta \big) + \frac{\vLP' -1}{\vLP} (1-\vLP^2)\bar\rho  \pa_y \theta  -2 \vLP\bar\rho  \pa_y \theta \notag\\
& =  - \pa_ y\big( (1-\vLP^2)  D_y\psi  \big) + 2\vLP (\bar\omega-\bar\rho -1 )  \pa_y \psi  + \big(  4\bar\omega (\bar\omega -1) + 2(1-\bar\rho -\vLP' ) - \frac{2}{y^2}\big)\psi.
\end{align}
Using $\pa_y = D_y - \frac{2}{y}$, we see that \eqref{linear_psi} can be equivalently written as \eqref{E:LIN4}. The proof that
solutions of the Eulerian linearisation~\eqref{E:LIN4} map back to solutions of the Lagrangian linearisation~\eqref{E:THETADYNAMICS} is now straightforward
by unwinding the above changes of variables.

If~\eqref{E:RESLAG} is true, it is easy to check that~\eqref{E:THETADYNAMICS} holds with $\th$ replaced by $e^{\la s}\th$ and a right-hand side given
by $(-\Lambda f-\l f-g)e^{\la s}$. On the other hand equation~\eqref{E:RESEUL} is equivalent to~\eqref{E:LIN4} with $\psi$ replaced by $e^{\la s}\psi$ and
a right-hand side given by  $(-2\pa_y(\vLP h_1)-4\omLP h_1+(2-\l) h_1+h_2) e^{\la s}$. By the above calculation, the two wave equations are equivalent if we let $\psi$ be defined as in~\eqref{Def:phi_y} and if we can make the right-hand sides the same.
Choosing $h_1$, $h_2$ as in~\eqref{E:H12} and using the relation $\pa_y(\rhoLP\vLP)=(1-2\omLP)\rhoLP$ we easily see that the right-hand sides are indeed the same.
\end{proof}

Finally, we collect some useful properties of the time-translation mode in the Eulerian variables.
\begin{lemma}\label{L:GROWINGMODEEUL}
The time-translation mode $\Gamma$ associated to eigenvalue $1$ for $\bfL$ corresponds to the eigenfunction
\beq\label{E:GROWINGMODEEUL}
g_1(y)=y\rhoLP(1-\omLP).
\eeq
associated to $\l=1$ for $\bfLEul$. This eigenfunction satisfies the property
\beq\label{E:GROWINGMODEPROP}
\vLP\pa_y g_1(y)+g_1(y)> 0, \ \ \ y\in(0,y_*).
\eeq
\end{lemma}

\begin{proof}
The identity~\eqref{E:GROWINGMODEEUL} follows directly from~\eqref{E:GM} and the Eulerian-Lagrangian change of variables~\eqref{Def:phi_y}.

To prove~\eqref{E:GROWINGMODEPROP}, we first rewrite $g_1$ by recalling from~\cite{GHJ2021b} the identity $M[\rhoLP]=y^2\rhoLP\vLP$ and rewriting~\eqref{E:GROWINGMODEEUL} as
\beq\label{E:GROWINGMODEEULALT}
g_1(y)=\frac{1}{y^2}\int_0^y(\rhoLP+(y\rhoLP)')\tilde y^2\,\dif \tilde y.
\eeq
We now recall from~\cite[Lemma 4.2]{GHJ2021b} that $(y\rhoLP)'> 0$ for $y<y_*$, and then differentiate~\eqref{E:GROWINGMODEEULALT} to see
\begin{align*}
\vLP\pa_y g_1(y)+g_1(y)=&\,\vLP(\rhoLP+(y\rhoLP)')-\frac{2\vLP}{y}g_1+g_1\geq y\rhoLP\omLP+(1-2\omLP)y\rhoLP(1-\omLP)=y\rhoLP(1-2\omLP+2\omLP^2)>0,
\end{align*}
where we have substituted~\eqref{E:GROWINGMODEEUL} after the first inequality.
\end{proof}


\section{Implementation of Interval Arithmetic}\label{APP:IA}

As described in the introduction, we employ interval arithmetic in order to give validated bounds on the solution to certain ODE systems and algebraic quantities. The primary tool we use for this purpose is the package \verb!VNODE-LP!~\cite{Nedialkov10a,Nedialkov10b}. \verb!VNODE-LP! is a rigorous interval arithmetic solver for initial value problems of the form $\dot{u}(t,\la)=F(t,u,\la)$ for a vectorial function $u$ depending on time $t$ and (vectorial) parameters $\la$. When the nonlinearity is regular, \verb!VNODE-LP! takes an interval of initial time $[\underline{t_0},\overline{t_0}]$ with initial data $[\underline{u_0},\overline{u_0}]$ and parameters $[\underline{\la},\overline{\la}]$ and solves the ODE up to time interval $[\underline{t_1},\overline{t_1}]$, giving a validated enclosure $[\underline{u_1},\overline{u_1}]$ for the solution on this final time interval.

\verb!VNODE-LP! is not adapted to handling ODE problems with singularities, and so we are forced to build a rigorous approximation close to the singular points (both origin and sonic point) in order to start the ODE solver. To perform this approximation, we construct Taylor series for the LP solution and the solutions of the eigenfunction ODEs with explicit bounds on the growth rate of the coefficients at both singular points. By standard estimates on the difference between the Taylor polynomial of order $N$ and the full Taylor series, we then obtain enclosures for the solutions.

The code for the interval arithmetic implementation is available at \url{https://github.com/mrischrecker/Larson-Penston-Stability}. It was run on a MacBook Air 2013. The details of which functions within the code are used to prove each lemma are detailed below through this appendix. A table of contents for the code is contained in Appendix~\ref{APP:CODE} with further details to aid the reader to navigate the code.

\subsection{Constructing the LP solution}\label{A:LPCONSTRUCTION}

In order to implement the interval arithmetic ODE solver VNODE-LP for the ODE system~\eqref{E:RHOLP}--\eqref{E:OMLP}, we need to be able to pose data strictly away from the singular points of the ODE system, in particular, away from $y=0,y_*$. To this end, we first show that the Taylor series for the LP solution (which we know to converge from \cite{GHJ2021b}) has a precise growth rate for its coefficients, and hence a precise rate of convergence.

\subsubsection{Taylor expansion for LP near the sonic point}\label{S:SONICTAYLOR}
We begin by considering the expansion for the LP solution around the sonic point, $y_*$.
Recall that, given any $y_*\in[2,3]$, \cite{GHJ2021b} constructed a local analytic solution to~\eqref{E:RHOLP}--\eqref{E:OMLP} on an interval $[y_*-\nu,y_*+\nu]$ by
\beq\label{eq:Euleriansonicexpansions}
\rho(\frac{y}{y_*};y_*)=\sum_{k=0}^\infty\rho_k(y_*)(\frac{y}{y_*}-1)^k,\quad \om(\frac{y}{y_*};y_*)=\sum_{k=0}^\infty\om_k(y_*)(\frac{y}{y_*}-1)^k,
\eeq
where 
\beqa\label{eq:order012coeffs}
&\rho_0 = \omega_0 = \frac1{y_\ast}, \qquad (\rho_1,\omega_1) = (-\omega_0, 1-2\omega_0),\\
&(\rho_2,\omega_2) = \left(\frac{-y_\ast^2+6y_\ast-7}{2y_\ast (2y_\ast-3)}, \frac{-5y_\ast^2+19y_\ast-17}{2y_\ast(2y_\ast-3)}\right).
\eeqa
The main goal of this subsection is to establish the precise growth bounds on these coefficients contained in the following proposition.

\begin{proposition}\label{P:RHONOMNBDS}
The coefficients of the expansions~\eqref{eq:Euleriansonicexpansions} satisfy the growth bounds, for $j\geq 2$,
\begin{align}
\lv \rho_j \rv  \le \frac{C^{j-\alpha}}{j^2}, \qquad \lv \omega_j \rv  \le \frac{C^{j-\alpha}}{j^2},
\end{align}
where $C$ and $\al$ may be taken as
\beq
(C,\al)=\begin{cases}
(8.25,1.95), & y_*\in[2,3],\\
(7.2,1.98), & y_*\in[2.34,2.342].
\end{cases}
\eeq
\end{proposition}

The precise numerical bounds on the growth rate of the coefficients will arise from making the strategy of~\cite{GHJ2021b} produce explicit numerical bounds.  Observe that we have sharpened the constants $(C,\al)$ to $(7.2,1.98)$ in the neighbourhood of the sonic point $\S\in[2.34,2.342]$.  To this end, we therefore first establish the recurrence relation for the coefficients $\rho_j$ and $\om_j$.

\begin{lemma}\label{L:FNGN}
The coefficients $\rho_N$ and $\om_N$ satisfy
\begin{align}
\rho_N & = \frac1{2\left(N(1-\frac1{\omega_0})+1\right)} \mathcal F_N, \label{E:RHONFN}\\
\omega_N & = \frac1{2N(1-\frac1{\omega_0})} \mathcal G_N
+ \frac1{2N(1-\frac1{\omega_0})\left(N(1-\frac1{\omega_0})+1\right)} \mathcal F_N, \label{E:OMEGANFNGN}
\end{align}
where
the functions $\mathcal{F}_N$ and $\mathcal{G}_N$ depend on $y_*$ and coefficients $(\rho_k,\om_k)$ of order up to $N-1$ and may be expressed as
\begin{align}
\mathcal F_N = &\,\rho_{N-1}\Big((N-1)\S^2\big((\om^2)_2+2(\om^2)_1+(\om^2)_0\big)+2\Big)+\om_{N-1}\frac{2(2-3\S+\S^2)}{2\S-3}\notag\\
&+2\S^2(\rho_2+\rho_1)\sum_{k+\ell=N-1 \atop 0<k<N-1}\om_k\om_\ell  
+ \S^2\rho_1\sum_{m+n=N \atop 1<m<N-1}\omega_m\omega_n\notag\\
 &+\S^2 \bigg(\sum_{k+\ell = N \atop 1<k<N-2}(k+1)\rho_{k+1}(\omega^2)_\ell 
  + 2\sum_{k+\ell = N-1 \atop 0<k<N-2}(k+1)\rho_{k+1}(\omega^2)_\ell
 + \sum_{k+\ell = N-2 \atop k<N-2}(k+1)\rho_{k+1}(\omega^2)_\ell\bigg) \notag \\
 &  -2\S^2 \bigg((\rho_1-\om_1)\sum_{k+\ell=N-1 \atop 0<k<N-1}\om_k\rho_\ell  +  \sum_{k+\ell+n=N \atop 1<n<N-1 }\omega_k\rho_\ell(\rho_n-\omega_n)+\sum_{k+\ell+n=N-1 \atop 0<n< N-1}\omega_k\rho_\ell(\rho_n-\omega_n) \bigg)\label{E:FNREFORMED}
\end{align}
and
\begin{align}
\mathcal{G}_N=&(4\S-6)\rho_{N-1}+\om_{N-1}\Big(\S^2 (N-1)\big((\om^2)_2+2(\om^2)_1+(\om^2)_0\big) + \frac{2 (4 - 7 \S + 3 \S^2)}{2\S - 3 }\Big)  \notag\\
-&\frac{\S(\S^2-3\S+2)}{2\S-3}\sum_{k+\ell=N-1 \atop 0<k<N-1}\om_k\om_\ell +\S^2\sum_{k+\ell = N \atop 1<k<N-2}(k+1)\om_{k+1}(\omega^2)_\ell\notag\\
& +2\S^2\sum_{k+\ell = N-1 \atop 0<k<N-2}(k+1)\om_{k+1}(\omega^2)_\ell+\S^2\sum_{k+\ell = N-2 \atop k<N-2}(k+1)\om_{k+1}(\omega^2)_\ell+ \S\sum_{m+n=N \atop 1<m<N-1}\omega_m\omega_n \notag\\
& +2\S^2\sum_{k+\ell+n=N \atop 1<n<N-1 }\omega_k\om_\ell(\rho_n-\omega_n)+2\S^2\sum_{k+\ell+n=N-1 \atop 0<n<N-1}\omega_k\om_\ell(\rho_n-\omega_n)\notag\\
&+  \S^2(-1)^{N}(1-\frac{3}{y_*}+3\om_1)(\om^2)_0\notag\\
&+3\S^2 \bigg(-(\om^2)_0\sum_{k+m=N-2\atop k>1}(-1)^m\om_k+ \sum_{k+\ell=N-1 \atop 0<\ell<N-1}\omega_k(\omega^2)_\ell  +\sum_{\substack{k+\ell=N \\ 1<\ell <N-1}}\omega_k(\omega^2)_\ell\bigg). \label{E:GNREFORMED}
\end{align}
\end{lemma}

\begin{proof}
The proof is a direct calculation, following~\cite[Lemma 2.3]{GHJ2021b}, but fixing the sign errors arising in the final line of \cite[(2.42)]{GHJ2021b} and the second line of~\cite[(2.44)]{GHJ2021b}. This yields 
\begin{align*}
\mathcal F_N = & \S^2 \bigg(\sum_{k+\ell = N \atop 0<k<N-1}(k+1)\rho_{k+1}(\omega^2)_\ell 
+ \sum_{m+n=N \atop 0<m<N}\rho_1\omega_m\omega_n  + 2\sum_{k+\ell = N-1 \atop k<N-1}(k+1)\rho_{k+1}(\omega^2)_\ell
 \notag \\
 & + \sum_{k+\ell = N-2}(k+1)\rho_{k+1}(\omega^2)_\ell  - 2 \Big(\sum_{k+\ell+n=N \atop 0<n<N }\omega_k\rho_\ell(\rho_n-\omega_n)+\left(\omega\rho(\rho-\omega)\right)_{N-1} \Big)\bigg)
\end{align*}
and
\begin{align*}
\mathcal G_N = & \S^2 \bigg(\sum_{k+\ell = N \atop 0<k<N-1}(k+1)\omega_{k+1}(\omega^2)_\ell 
+ \sum_{m+n=N \atop 0<m<N}\omega_1\omega_m\omega_n  + 2\sum_{k+\ell = N-1 \atop k<N-1}(k+1)\omega_{k+1}(\omega^2)_\ell
  \notag \\
 &  + \sum_{k+\ell = N-2}(k+1)\omega_{k+1}(\omega^2)_\ell+ 2 \Big(\sum_{k+\ell+n=N \atop 0<n<N }\omega_k\omega_\ell(\rho_n-\omega_n)+\left(\omega^2(\rho-\omega)\right)_{N-1} \Big)\bigg) \notag \\
 &- \S^2\bigg((-1)^{N-1}(1-\frac{3}{y_*})(\om^2)_0+(\frac{3}{\S}-2)(\om^2)_{N-1}+ (1-\frac{3}{y_*} )\sum_{k+n=N\atop 0<k<N}\omega_k\omega_n\bigg) \notag\\
&+3\S^2 \bigg(-(\om^2)_0\sum_{k+m=N-2\atop k>0}(-1)^m\om_k+\big(2(\om^2)_0+(\om^2)_1\big)\om_{N-1} +  \sum_{k+\ell=N-1 \atop 0<\ell<N-1}\omega_k(\omega^2)_\ell  +\sum_{\substack{k+\ell=N \\ 1<\ell <N-1}}\omega_k(\omega^2)_\ell\bigg).
\end{align*}
 The final form of the coefficient source terms $\mathcal F_N$ and $\mathcal G_N$ then follow by simplifying these expressions.
\end{proof}

In order to control the growth of coefficients associated to quadratic functions of $\rho$ and $\om$, we introduce a constant
\beq\label{def:DCalpha}
D(C,\al,\S)=\frac{2}{y_*}+(1-\frac{1}{\S})\frac{16}{9C}+\frac{7}{4C^{\alpha}}
\eeq
and establish the following lemma.

\begin{lemma}\label{lemma:order23coeffs}
The coefficients at order 2 and 3 satisfy the estimates 
\beqa
|\rho_2|,\,|\om_2|\leq&\, \frac{C^{2-\al}}{4},\qquad |(\rho\om)_2|,\,|(\om^2)_2|\leq&\, D\frac{C^{2-\al}}{4},\\
|\rho_3|,\,|\om_3|\leq&\, \frac{C^{3-\al}}{9},\qquad |(\rho\om)_3|,\,|(\om^2)_3|\leq&\, D\frac{C^{3-\al}}{9},
\eeqa
where $D=D(C,\al,y_*)$ is as in \eqref{def:DCalpha} for the choices 
$(C,\al) =(8.25,1.95)$ whenever $\S\in[2,3]$ and for $(C,\al)= (7.2,1.98)$ whenever $\S\in[2.34,2.342]$.
\end{lemma}

\begin{proof}
The inequalities are established directly from  interval arithmetic. In the supplementary code, the function \verb!C_alpha_constraint_check_Sonic!  is employed. 
\end{proof}


This shows that Proposition~\ref{P:RHONOMNBDS} holds for $k=2,3$.


\begin{lemma}\label{L:PRELIMBOUNDS}
Let $\S\in[2,3]$ and $\alpha\in(1.9,2)$.  
Assume that 
\begin{align}
\lv \rho_k \rv,\ \ \lv\om_k\rv & \le \frac{C^{k-\alpha}}{k^2}, \ \ 2\le k \le N-1 \label{E:ASS1PRELIM}
\end{align}
for some $C\ge4$ and $N\ge4$. 
Then, for all $2\leq \ell \le N-1$,
\beq
|(\omega^2)_\ell|, \ \ \lv (\omega \rho)_\ell\rv 
\le D \frac{C^{\ell-\alpha}}{\ell^2}.
\eeq
\end{lemma}


\begin{proof}
The bound in the case $\ell=2,3$ is given in Lemma \ref{lemma:order23coeffs}.
We therefore prove the bounds for $|(\rho\omega)_\ell|$, $\ell\ge4$.

If $\ell\ge4$, we apply the identities \eqref{eq:order012coeffs} and \eqref{E:ASS1PRELIM} and have
\begin{align}
|(\rho\omega)_\ell| & \le \sum_{k=0}^\ell |\rho_k||\omega_{\ell-k}| \le |\omega_0||\rho_\ell|+|\rho_0||\omega_\ell| + |\omega_1||\rho_{\ell-1}|+|\rho_1||\omega_{\ell-1}| +\sum_{k=2}^{\ell-2}|\rho_k||\omega_{\ell-k}| \notag\\
& \le  \frac{2}{\S}\frac{C^{\ell-\alpha}}{\ell^2}  +  (1-\frac{1}{\S})\frac{C^{\ell-1-\alpha}}{(\ell-1)^2}
+ \sum_{k=2}^{\ell-2} \frac{C^{\ell-2\alpha}}{k^2(\ell-k)^2} \notag \\
& \le   \frac{C^{\ell-\alpha}}{\ell^2}\Big(\frac{2}{\S}+(1-\frac{1}{\S})\frac{\ell^2}{C(\ell-1)^2}+\frac{7}{4C^{\alpha}}\Big) \le D\frac{C^{\ell-\alpha}}{\ell^2}, \notag
\end{align}
where we have used that $|\rho_1|+|\om_1|=1-\om_0=1-\frac1{\S}$ by~\eqref{eq:order012coeffs}, $\frac{\ell^2}{(\ell-1)^2}\le\frac{16}{9}$ for all $\ell\ge4$, and
$$\sum_{k=2}^{\ell-2} \frac{1}{k^2(\ell-k)^2}\leq\frac{7}{4\ell^2},$$
the proof of which we temporarily defer to \eqref{ineq:comb8} in Lemma \ref{L:COMB} below.
Clearly, for $C\geq 4$, $D\leq 2$ as $\alpha>1$.
It is now clear that the estimates for $(\omega^2)_\ell$, 
 $\ell\ge4$ follow in the same way, as the only estimate that changes is to replace $|\rho_1|+|\om_1|=1-\frac{1}{\S}$ with $2\om_1=2-\frac{4}{\S}\leq 1-\frac{1}{\S}$. Other than that, we have used simple numerical bounds from~~\eqref{eq:order012coeffs} and $\S\in[2,3]$ and the inductive assumptions~\eqref{E:ASS1PRELIM}, which both depend only on the index, and are symmetric with respect to $\rho$ and $\omega$.
\end{proof}


The goal of this analysis is to establish the following lemma.

\begin{lemma}\label{lemma:inductivebounds}
Let $N\geq 4$ and assume that
\beqa
|\rho_k|\leq \frac{C^{k-\al}}{k^2}, \quad |\om_k|\leq\frac{C^{k-\al}}{k^2},\\
|(\rho\om)_k|\leq D\frac{C^{k-\al}}{k^2}, \quad |(\om^2)_k|\leq D\frac{C^{k-\al}}{k^2},\\
\eeqa
for all $2\leq k\leq N-1$, some constants $C\geq 4$, $D>0 $, $\al\in(1,2)$. Then, for $N\geq 4$,
\beqa
|\mathcal{F}_N|\leq \mathfrak{F}\frac{C^{N-\al}}{N},
\eeqa
where 
\beqa
\mathfrak{F}=&\S^2 \big|(\om^2)_2+2(\om^2)_1+(\om^2)_0\big|\frac{4}{3C}+\frac{8}{9C}+\frac{2-3\S+\S^2}{2\S-3}\frac{8}{9C}+2|\rho_2+\rho_1|\S^2\Big(\frac{2|\om_1|}{C^2}+ \frac{1}{3C^{1+\al}}\Big)\\
&+ \S\frac{5}{18}\frac{1}{C^{\al}}
+\S^2(0.506)\frac{D}{C^{\al-1}}+\S^2\frac{9}{5}\frac{D}{C^{\al}}+\S^2\Big(|\rho_1|\frac{D}{C^2}+|\om_0\om_1|\frac{4}{C^2}+\frac{D}{C^{1+\al}}\Big)\\
&+2\S^2|\om_1-\rho_1|\Big(\frac{|\rho_1|+|\om_1|}{C^2}+ \frac{1}{3C^{1+\al}}\Big)+\S^2\frac{10}{9}\frac{D}{C^\al} \\
&+2\S^2\Big(|\rho_1-\om_1|\frac{D}{C^2}+|\rho_1\om_0+\om_1\rho_0|\frac{2}{C^2}+\frac{2D}{3C^{1+\al}} \Big) 
\eeqa
and also
\beqa\label{ineq:GNBOUND}
|\mathcal{G}_N|\leq \mathfrak{G}\frac{C^{N-\al}}{N},
\eeqa
where
\beqa
\mathfrak{G}=&\,(4\S-6)\frac{4}{9C}+\S^2 \big|(\om^2)_2+2(\om^2)_1+(\om^2)_0\big|\frac{4}{3C} + \frac{8 (4 - 7 \S + 3 \S^2)}{9(2\S - 3) }\frac{1}{C}\\
&+\frac{\S(\S^2-3\S+2)}{2\S-3}\Big(\frac{2|\om_1|}{C^2}+ \frac{1}{3C^{1+\al}}\Big)+\S^2(0.506)\frac{D}{C^{\al-1}}\\
&+\S^2\Big(\frac{9}{5}\frac{D}{C^{\al}}+|\om_1|\frac{D}{C^2}+|\om_0\om_1|\frac{4}{C^2}+\frac{D}{C^{1+\al}}\Big)+\S \frac{5}{18}\frac{1}{C^{\al}}\\
&+2\S^2\Big(\frac{5}{9}\frac{D}{C^\al}+|\rho_1-\om_1|\frac{D}{C^2}+|\om_1\om_0|\frac{4}{C^2}+\frac{2D}{3C^{1+\al}} \Big)\\
&+\Big|1-\frac{3}{\S}+3\om_1\Big|\frac{4}{C^{4-\al}}+3\S^2\Big(\frac{1}{\S^2C^2}+\frac{2\om_0\om_1+D\om_1}{C^2}+\frac{D}{3C^{1+\al}}+\frac{5}{18}\frac{D}{C^\al}\Big). 
\eeqa
\end{lemma}

\begin{proof}
The proof follows directly but tediously by applying the summation estimates of Lemma~\ref{L:COMB} to the expressions given for $\mathcal{F}_N$ and $\mathcal{G}_N$ in Lemma~\ref{L:FNGN} and using also Lemma~\ref{L:PRELIMBOUNDS}.
\end{proof}

\begin{proof}[Proof of Proposition~\ref{P:RHONOMNBDS}]
We argue inductively. From Lemma \ref{lemma:order23coeffs}, we recall that the claimed estimates are indeed satisfied for $k=2,3$ provided $\al=1.95$, $C= 8.25$ and $\S\in[2,3]$ or $C= 7.2$, $\S\in[2.34,2.342]$. Now, making the inductive assumption that the estimates hold up to some order $N-1$, we begin from \eqref{E:RHONFN}--\eqref{E:OMEGANFNGN}:
\begin{align}
\rho_N & = \frac1{2\left(N(1-\frac1{\omega_0})+1\right)} \mathcal F_N, \\
\omega_N & = \frac1{2N(1-\frac1{\omega_0})} \mathcal G_N
+ \frac1{2N(1-\frac1{\omega_0})\left(N(1-\frac1{\omega_0})+1\right)} \mathcal F_N .
\end{align}
As the assumptions of Lemmas \ref{lemma:inductivebounds} and \ref{L:PRELIMBOUNDS} are satisfied, we make the estimates
\beqa
|\rho_N|\leq&\, \Big|\frac1{2\left(N(1-\frac1{\omega_0})+1\right)} \mathcal F_N\Big|\leq\frac{4}{2\left(4(\S-1)-1\right)}\mathfrak{F}\frac{C^{N-\al}}{N^2},
\eeqa
where we have used that the quantity
\beqs
\frac{N}{|2\left(N(1-\frac1{\omega_0})+1\right)|}
\eeqs
is maximised for $N=4$ and recalled $\frac{1}{\om_0}=\S$. 

Similarly, 
\beqa
|\omega_N| \leq &\Big|\frac1{2N(1-\frac1{\omega_0})} \mathcal G_N\Big|
+ \Big|\frac1{2N(1-\frac1{\omega_0})\left(N(1-\frac1{\omega_0})+1\right)} \mathcal F_N \Big|\\
\leq& \Big(\frac{1}{2(\S-1)}\mathfrak{G}+\frac{1}{2(1-\S)(4(1-\S)+1)}\mathfrak{F}\Big)\frac{C^{N-\al}}{N^2}. 
\eeqa
 Finally, we use interval arithmetic to conclude that, for $C= 8.25$, $\al=1.95$, all $\S\in[2,3]$, we have \beqa
&\frac{4}{2\left(4(\S-1)-1\right)}\mathfrak{F}<1,\\
&\frac{1}{2(\S-1)}\mathfrak{G}+\frac{1}{2(1-\S)(4(1-\S)+1)}\mathfrak{F}<1.
\eeqa
Moreover, in the case $\S\in[2.34,2.342]$, we obtain the same inequalities for $\al=1.98$ and $C\geq 7.2$.

More precisely, the function \verb!C_alpha_const_check_Sonic! is employed in the attached supplementary code to establish this.
\end{proof}

Next, we identify an accurate range for the Larson-Penston sonic point $\bar y_*$.
\begin{lemma}\label{L:LPYSTAR}
The sonic point of the Larson-Penston solution, $\bar y_*$, is bounded in the following interval,
\beq
\bar y_*\in \mathfrak{y}_*:=[2.3411172805,2.34111728062].
\eeq
\end{lemma}

\begin{proof}
We recall from~\cite{GHJS2022} that the Larson-Penston sonic point $\bar y_*$ has the property 
\beq
\bar y_*=\inf\big\{y_*\in(2,3)\,|\,\exists y_c(y_*)\text{ s.t.~}\om(y_c;y_*)=\frac13,\,\om'(y;y_*)>0\,\text{ for all~}y\in[y_c,y_*]\big\}.
\eeq 
To identify an accurate range for the sonic point $\bar y_*$, we therefore use the interval arithmetic ODE solver VNODE--LP (more precisely, the function \verb!y_bar_star_upper_bound! in the supplementary code) to verify that for all $y_*\in[2.34111728062,3]$, the corresponding solution $\om(y;y_*)$ passes below $\frac13$ at some $y_c\in(0,y_*)$ and $\om'(y;y_*)>0$ for all $y\in[y_c;y_*]$. This establishes the claimed upper bound. To prove the lower bound, we use the interval arithmetic solver \verb!LP_solver! to see that for $y_*=2.3411172805$, the solution admits a point $y_d=0.001$ such that $\om'(y_d;y_*)<0$. In each application of the ODE solver, initial data is given at the rescaled point $\frac{y}{y_*}=1-\de$ for $\de=0.08$. The initial data is determined from~\eqref{eq:Euleriansonicexpansions} as
\beqas
\rho(1-\de;y_*)=\sum_{k=0}^{200}\rho_k(y_*)(-\de)^k + \textup{err}[-1,1],\quad \om(1-\de;y_*)=\sum_{k=0}^{200}\om_k(y_*)(-\de)^k+ \textup{err}[-1,1],
\eeqas where the size of the error $\textup{err}$ is determined by the rate of convergence of the series using the coefficient estimates of Proposition~\ref{P:RHONOMNBDS} and a simple geometric series estimate.
\end{proof}

\subsubsection{Taylor expansion for LP near the origin}\label{S:ORIGINTAYLOR}
The final ingredient that we need to give an accurate construction of the LP solution is an estimate on $\rhoLP(0)$, as this characterises the LP  flow near $y=0$.

\begin{lemma}\label{L:ORIGINCOEFFS}
(i) The LP solution $(\rhoLP,\omLP)$ satisfies that $\rhoLP(0)\in[0.83290803,0.83290811]$.\\
(ii) The LP solution expands analytically, as a function of $y$ around the origin $y=0$ as 
\beq
\rhoLP(y)=\sum_{j=0}^\infty\tilde \rho_{2j} y^{2j},\qquad \omLP(y)=\sum_{j=0}^\infty\tom_{2j}y^{2j},
\eeq
where $\trho_{2j}$, $\tom_{2j}$ are smooth functions of $\rhoLP(0)=\trho_0$,
\begin{align}
\tr_2 &= -  \frac13\rho_0(\rho_0-\frac13) = - \frac13  \rho_0^2 + \frac19  \rho_0,  \label{eq:tr2} \\
\tom_2 & = \frac 2{45} (\rho_0-\frac13) = - \frac{2}{135}+   \frac{2}{45} \rho_0, \label{eq:tom2}
\end{align}
and, for $j\geq 2$, the coefficients satisfy the growth bounds
\begin{align}
\lv \trho_{2j} \rv  \le \frac{C^{2j-\alpha}}{(2j)^2}, \qquad \lv \tom_{2j} \rv  \le \frac{C^{2j-\alpha}}{(2j)^2},
\end{align}
where $C=2$ and $\al=1.95$.
\end{lemma}

\begin{proof}
(i) We first prove the claimed range for $\rhoLP(0)$. To obtain this, we first find from the implementation of VNODE--LP (using the function \verb!LP_solver!) that, at  $y=\frac{y_*}{100}$, for any $\S\in\mathfrak{y}_*$, we have $\rhoLP(\frac{y_*}{100})\in[0.832832036,0.832832045]=:[\underline{\rho},\bar\rho]$ and $\omLP(\frac{y_*}{100})\in[0.333318,0.333376]=:[\underline{\om},\bar\om]$.  We now recall firstly that $\rhoLP'(y)\leq 0\leq \omLP'(y)$, and bootstrap the estimate by supposing that we retain $\rhoLP(y)\leq 0.833$ for $y\in[\bar y,\frac{y_*}{100}]$ and integrating the ODE from any $y\geq \bar y$ to $\frac{y_*}{100}$:
\beq
\rhoLP(y)\leq \rhoLP(\frac{y_*}{100})+\int_{\bar y}^{\frac{y_*}{100}}\frac{2\tilde y\rhoLP\omLP(\rhoLP-\omLP)}{1-\tilde y^2\omLP^2}\,\dif \tilde y\leq \bar\rho+\frac{(0.833)\bar\om(0.833-\frac13)}{1-(\frac{y_*}{100})^2\bar\om^2}(\frac{y_*}{100})^2\leq 0.83290811,
\eeq
where  the precise number is established by the function \verb!rho0_bound!. This proves $\bar y=0$ and closes the claimed upper bound estimate on $\rhoLP(0)$. On the other hand, we similarly obtain a lower bound on $\rhoLP(0)$ by
\beqa
\rhoLP(0)=&\, \rhoLP(\frac{y_*}{100})+\int_{0}^{\frac{y_*}{100}}\frac{2\tilde y\rhoLP\omLP(\rhoLP-\omLP)}{1-\tilde y^2\omLP^2}\,\dif \tilde y\geq \underline{\rho}+\underline{\rho}\frac{1}{3}(\underline{\rho}-\bar{\om})(\frac{y_*}{100})^2
\geq 0.83290803,
\eeqa
where again the function \verb!rho0_bound! is used to verify the precise number.

(ii) In order to obtain the claimed estimate on the Taylor series, we again derive a recurrence relation for the coefficients, given $\trho_0=\rhoLP(0)$, observing that all odd coefficients vanish, to derive, for $N+1=2n$ even
\begin{align}
\tr_{N+1} & = \frac1{N+1} \tilde{\mathcal F}_{N+1},\qquad 
\tom_{N+1}  = \frac1{N+4} \tilde{\mathcal G}_{N+1},  \label{E:RHONOMNORIGIN}
\end{align}
where 
\begin{align}
\tilde{\mathcal F}_{N+1} & : =  \sum_{i+j = n-1}2i\tr_{2i}(\omega^2)_{2j} 
- 2  \left(\omega\rho(\rho-\omega)\right)_{2(n-1)}  \label{E:TILDEFNDEF2}\\
\tilde{\mathcal G}_{N+1} & : = \sum_{i+j = n-1}2i\tom_{2i}(\omega^2)_{2j} 
+3 \sum_{\substack{i+j = n-1 \\ i\neq 0}} \tom_{2i}(\omega^2)_{2j}
+ 2  \left(\omega^2(\rho-\omega)\right)_{2(n-1)}. \label{E:TILDEGNDEF2} 
\end{align}
Following an argument analogous to that above for the expansion around the sonic point (but significantly simpler), we define
$$D_0(C,\al,\tr_0)=\frac13+\tr_0+\frac{41}{36 C^{\al}}$$
in order to show, under the obvious inductive hypothesis,
\beq\label{E:ORIGINQUADRATICCOEFFS}
|(\tom^2)_{2n}|,\:|(\tr\tom)_{2n}|\leq D_0 \frac{C^{2n-\al}}{(2n)^2}.
\eeq
For the choices $C=2$, $\al=1.95$, we verify that these estimates hold for $n=1$ by interval arithmetic, in particular employing the function \verb!C_alpha_constraint_check_Origin!. Making the  obvious inductive assumptions, we obtain that
\beqa
\big|\tilde{\mathcal{F}}_{N+1}\big|\leq \frac{1}{C^2}\Big(&\frac{(N+1)^2}{9(N-1)}+\frac{3D_0(N+1)}{4 C^{\al}}+\frac{4\tr_0(N+1)^2}{3(N-1)^2}\\
&+\frac{2D_0(\tr_0-\frac13)(N+1)^2}{(N-1)^2}+\frac{9D_0}{C^{\al}}\Big)\frac{C^{N+1-\al}}{(N+1)^2}, \\
\big|\tilde{\mathcal{G}}_{N+1}\big|\leq \frac{1}{C^2}\Big(&\frac{(N+1)^2}{9(N-1)}+\frac{3D_0(N+1)}{4 C^{\al}}+\frac{4(N+1)^2}{9(N-1)^2}\\
&+\frac{2D_0(\tr_0-\frac13)(N+1)^2}{(N-1)^2}+\frac{9D_0}{C^{\al}}+\frac{(N+1)^2}{3(N-1)^2}+\frac{27D_0}{4C^{\al}}\Big)\frac{C^{N+1-\al}}{(N+1)^2}.  
\eeqa
Combining these estimates with~\eqref{E:RHONOMNORIGIN}, we close the inductive step using interval arithmetic, in particular employing the function \verb!C_alpha_constraint_check_Origin!.
\end{proof}


\subsection{Energy coefficients}

The purpose of this section is to establish identities for the coefficients used in the exclusion of eigenvalues proofs of Sections~\ref{S:LARGEIMAG}--\ref{S:SMALLIMAG}. Specifically, we demonstrate the coefficients used in the ODE formulae of Proposition~\ref{prop:A4B4} and~\eqref{eq:Pequation}.

\subsubsection{Coefficients for Proposition~\ref{prop:A4B4}}\label{S:Highbcoeffs}
Throughout this subsection, we follow the notation used in the proof of Proposition~\ref{prop:A4B4} and simply collect together the identities arising in that proof for use in the interval arithmetic code. The coefficient functions, expressed as functions of the LP solution and its derivatives up to order 5, are contained in the section of the code titled \verb!SEC: Large b Eigenfunction Coefficients!.

 First, solving for $A^{(1)}$ and $B^{(1)}$ using~\eqref{E:AUPPERZERO}--\eqref{def:b0} and~\eqref{E:A1B1DEF}, we find
\beq
A^{(1)}=\frac{2\wLP'+2a\vLP+2ib\vLP}{\wLP}-\frac{\rhoLP'}{\rhoLP}-d^{(0)},
\eeq
and
\begin{align*}
B^{(1)}=&\,\frac{b^{(1)}_1+ib^{(1)}_2}{\wLP},\\
b^{(1)}_1=&\,b^2-a^2-a+2\rhoLP+4a\vLP'+4a\omLP+\wLP''-4\omLP\vLP'-\wLP\big(\frac{\rhoLP'}{\rhoLP}\big)'-\wLP'\frac{\rhoLP'}{\rhoLP}-2\wLP\frac{\rhoLP'}{y\rhoLP}\\
&- (\wLP'+2a\vLP) d^{(0)}_1+2b\vLP d^{(0)}_2+\wLP d^{(0)}_1\frac{\rhoLP'}{\rhoLP},\\
b^{(1)}_2=&\,b\big(4\vLP'+4\omLP-2a-1\big)+\wLP d^{(0)}_2\frac{\rhoLP'}{\rhoLP}- (\wLP'+2a\vLP)  d^{(0)}_2-2b\vLP d^{(0)}_1,
\end{align*}
where
\beq
d^{(0)}=d^{(0)}_1+id^{(0)}_2=\frac{(b^{(0)})'}{b^{(0)}}.
\eeq
For convenience, we set $$D^{(1)}=d^{(0)}.$$
Next, using~\eqref{E:A2B2DEF} to solve for $A^{(2)}$ and $B^{(2)}$, we obtain
\beq
A^{(2)}=\frac{3\wLP'+2a\vLP+2ib\vLP}{\wLP}-\frac{\rhoLP'}{\rhoLP}-D^{(2)}
\eeq
and
\begin{align*}
{}&B^{(2)}=\frac{b^{(2)}_1+ib^{(2)}_2}{\wLP},\label{def:b2}\\
&b^{(2)}_1=b^2-a^2-a+2\rhoLP+6a\vLP'+3\wLP''+4\omLP\vLP'-2\big(\wLP\frac{\rhoLP'}{\rhoLP}\big)'
+\big( \wLP \frac{\rhoLP'}{\rhoLP} - \wLP'- 2a\vLP\big) d^{(0)}_1+2b\vLP d_2^{(0)}\\
&+2\wLP\frac{d^{(0)}_1}{y}-\wLP(d^{(0)}_1)'-\wLP'd^{(0)}_1-(2\wLP'+2a\vLP)d^{(1)}_1+2b\vLP d^{(1)}_2+\wLP\frac{\rhoLP'}{\rhoLP}d^{(1)}_1+\wLP d^{(0)}_1d^{(1)}_1-\wLP d^{(0)}_2d^{(1)}_2,\\
&b^{(2)}_2=b\big(6\vLP'-2a-1\big)+\wLP d^{(0)}_2\frac{\rhoLP'}{\rhoLP}+2\wLP \frac{d^{(0)}_2}{y}-\wLP(d^{(0)}_2)'-\wLP'd^{(0)}_2 - (\wLP'+2a\vLP) d_2^{(0)} - 2b \vLP d_1^{(0)} \\
&-2b\vLP d^{(1)}_1-(2\wLP'+2a\vLP)d^{(1)}_2+\wLP d^{(1)}_2\frac{\rhoLP'}{\rhoLP}+\wLP d^{(0)}_2d^{(1)}_1+\wLP d^{(0)}_1d^{(1)}_2,
\end{align*} 
where
\beq
D^{(2)}=d^{(0)}+d^{(1)},\qquad d^{(1)}=d^{(1)}_1+id^{(1)}_2=\frac{(b^{(1)})'}{b^{(1)}}.
\eeq
Continuing the procedure,
\beq
A^{(3)}=\frac{4\wLP'+2a\vLP+2ib\vLP}{\wLP}-\frac{\rhoLP'}{\rhoLP}-D^{(3)}
\eeq
and
\begin{align*}
B^{(3)}=&\,\frac{b^{(3)}_1+ib^{(3)}_2}{\wLP},\\
b^{(3)}_1 = &\,b^{(2)}_1 + \frac{2}{y} \big( 3\wLP'  + 2a\vLP - w\frac{\rhoLP'}{\rhoLP} - \wLP d_1^{(0)} - \wLP d_1^{(1)} \big) + 3\wLP'' + 2a\vLP' - \big(\wLP\frac{\rhoLP'}{\rhoLP} \big)'- (\wLP d_1^{(0)})' - (\wLP d_1^{(1)})' \\
& - d_1^{(2)} \Big( 3\wLP' +2a\vLP - \wLP\frac{\rhoLP'}{\rhoLP} - \wLP d_1^{(0)} - \wLP d_1^{(1)} \Big) + d_2^{(2)}\big( 2b \vLP - \wLP d_2^{(0)}- \wLP d_2^{(1)} \big),\\
b^{(3)}_2 = &\,b^{(2)}_2 + \frac{2}{y} \big( 2b \vLP - \wLP d_2^{(0)}- \wLP d_2^{(1)} \big) + 2b \vLP' - (\wLP d_2^{(0)})' - (\wLP d_2^{(1)})' \\
& - d_2^{(2)} \Big( 3\wLP' +2a\vLP - \wLP\frac{\rhoLP'}{\rhoLP} - \wLP d_1^{(0)} - \wLP d_1^{(1)} \Big) - d_1^{(2)}\big( 2b \vLP - \wLP d_2^{(0)}- \wLP d_2^{(1)} \big),
\end{align*}
where
\beq
D^{(3)}=d^{(0)}+d^{(1)}+d^{(2)},\qquad  d^{(2)}=d^{(2)}_1+id^{(2)}_2=\frac{(b^{(2)})'}{b^{(2)}}.
\eeq
Finally, after the application of the final derivative, $A^{(4)}$ takes the claimed form
\beq\label{E:A4IDENTITY}
A^{(4)}=\frac{4\wLP'+2a\vLP+2ib\vLP}{\wLP}-\frac{\rhoLP'}{\rhoLP}-D^{(4)}
\eeq
 and 
\beqa
b^{(4)}_1=&\,b^{(3)}_1-\frac{2}{y}\big(4\wLP'+2a\vLP-\wLP\frac{\rhoLP'}{\rhoLP}-\wLP(d^{(0)}_1+d^{(1)}_1+d^{(2)}_1)\big)+(4\wLP''+2a\vLP')-\big(\wLP\frac{\rhoLP'}{\rhoLP}\big)'\\
&-\big(\wLP(d^{(0)}_1+d^{(1)}_1+d^{(2)}_1)\big)'-d^{(3)}_1(4\wLP'+2a\vLP)+2b\vLP d^{(3)}_2\\
&+\wLP d^{(3)}_1\big(\frac{\rhoLP'}{\rhoLP}+d^{(0)}_1+d^{(1)}_1+d^{(2)}_1\big)-\wLP d^{(3)}_2\big(d^{(0)}_2+d^{(1)}_2+d^{(2)}_2),
\eeqa
and
\beqa
b^{(4)}_2=&\,b^{(3)}_2-\frac{2}{y}\big(2b\vLP-\wLP(d^{(0)}_2+d^{(1)}_2+d^{(2)}_2)\big)+2b\vLP'-\big(\wLP(d^{(0)}_2+d^{(1)}_2+d^{(2)}_2)\big)'\\
&-d^{(3)}_2\big(4\wLP'+2a\vLP-w\frac{\rhoLP'}{\rhoLP}-\wLP(d^{(0)}_1+d^{(1)}_1+d^{(2)}_1)\big)-d^{(3)}_1\big(2b\vLP-\wLP(d^{(0)}_2+d^{(1)}_2+d^{(2)}_2)\big),
\eeqa
where
\[
D^{(4)}=d^{(0)}+d^{(1)}+d^{(2)}+d^{(3)}, \qquad d^{(3)}=d^{(3)}_1+id^{(3)}_2=\frac{(b^{(3)})'}{b^{(3)}}.
\]

 \subsubsection{{Coefficients for~\eqref{eq:Pequation}}}\label{S:PEQCOEFFS}

Throughout this subsection, we follow the notation used in the proof of Proposition~\ref{P:NOSMALLIMAG} and simply collect together the identities arising in that proof for use in the interval arithmetic code. The coefficient functions, expressed as functions of the LP solution and its derivatives up to order 3, are contained in the section of the code titled \verb!SEC: Small b Eigenfunction Coefficients!.

 We begin by recalling the equation~\eqref{eq:Qlambda} for $Q=Q_\la$: 
\beq\label{eq:Q}
Q''+\widetilde{A}^{(0)}Q'+\widetilde{B}^{(0)}Q=0,
\eeq
where
\beq\label{def:A0Q}
\widetilde{A}^{(0)}=\frac{4}{y}+\frac{\wLP'+2a\vLP+2ib\vLP}{\wLP}+\widetilde{a}^{(0)},
\eeq
where 
\beq\label{def:a0Q}
\widetilde{a}^{(0)}=\frac{\rhoLP'}{\rhoLP}-\frac{2\omLP'}{1-\omLP},
\eeq
and
\beqa\label{def:B0Q}
\widetilde{B}^{(0)}=&\,\frac{\widetilde{b}^{(0)}_1+i\widetilde{b}^{(0)}_2}{\wLP},\\
\widetilde{b}^{(0)}_1=&\,a(1-a)+b^2+2a\omLP\frac{1-\vLP'}{1-\omLP},\qquad\widetilde{b}^{(0)}_2=(1-2a)b+2b\omLP\frac{1-\vLP'}{1-\omLP}.
\eeqa
 The above coefficients are easily read off from~\eqref{def:W1W2}.
Note that $B^{(0)}$ may equivalently be written as
\beq\label{eq:B0Qalt}
\widetilde{B}^{(0)}=\frac{a+ib}{\wLP}\Big(1-a-ib+2\vLP\frac{1-\vLP'}{y-\vLP}\Big).
\eeq


By differentiating~\eqref{eq:Q} directly, we derive that the first derivative quantity $Q'$ satisfies
\beq\label{eq:Q'}
(Q')''+\widetilde{A}^{(1)}(Q')'+\widetilde{B}^{(1)}Q'=0,
\eeq
where the coefficients are defined by
\beq\label{def:A1Q}
\widetilde{A}^{(1)}=\frac{4}{y}+\frac{2\wLP'+2a\vLP+2ib\vLP}{\wLP}+\widetilde{a}^{(0)}-\widetilde{d}^{(0)},
\eeq
where 
\beqa\label{def:d0Q}
\widetilde{d}^{(0)}=&\,\widetilde{d}^{(0)}_1+i\widetilde{d}^{(0)}_2,\\
\widetilde{d}^{(0)}_1=&\,\frac{\big(1-a+2\vLP\frac{1-\vLP'}{y-\vLP}\big)\big(2\vLP\frac{1-\vLP'}{y-\vLP}\big)'}{(1-a+2\vLP\frac{1-\vLP'}{y-\vLP})^2+b^2},\quad 
\widetilde{d}^{(0)}_2=\frac{b\big(2\vLP\frac{1-\vLP'}{y-\vLP}\big)'}{(1-a+2\vLP\frac{1-\vLP'}{y-\vLP})^2+b^2},
\eeqa
and
\beqa\label{def:B1Q}
\widetilde{B}^{(1)}=&\,-\frac{4}{y^2}+\frac{\widetilde{b}^{(1)}_1+i \widetilde{b}^{(1)}_2}{\wLP},\\
\widetilde{b}^{(1)}_1=&\,\widetilde{b}^{(0)}_1+\wLP''+2a\vLP'+(\wLP\frac{\rhoLP'}{\rhoLP}-2\frac{\wLP\omLP'}{1-\omLP})'-8\omLP\vLP'\\
&-(\wLP'+2a\vLP)\widetilde{d}^{(0)}_1+2b\vLP \widetilde{d}^{(0)}_2-\wLP\widetilde{d}^{(0)}_1\big(\frac{4}{y}+\widetilde{a}^{(0)}\big),\\
\widetilde{b}^{(1)}_2=&\,\widetilde{b}^{(0)}_2+2b\vLP'-(\wLP'+2a\vLP)\widetilde{d}^{(0)}_2-2b\vLP \widetilde{d}^{(0)}_1-\wLP\widetilde{d}^{(0)}_2\big(\frac{4}{y}+\widetilde{a}^{(0)}\big).
\eeqa


Now, differentiating equation \eqref{eq:Q'} and recalling $P=Q''$ from~\eqref{E:PLADEF} gives
\beqa
P''+\widetilde{A}^{(2)}P'+\widetilde{B}^{(2)}P=0,
\eeqa
\beqa
\widetilde{A}^{(2)}=&\,\frac{6}{y}+\frac{3\wLP'+2a\vLP+2ib\vLP}{\wLP}+\widetilde{a}^{(0)}-\widetilde{d}^{(0)}-\widetilde{d}^{(1)},
\eeqa
where we define
\beqa\label{def:d1Q}
\widetilde{d}^{(1)}=&\,\frac{(-4\wLP+y^2( \widetilde{b}^{(1)}_1+i\widetilde{b}^{(1)}_2))'}{-4\wLP+y^2(\widetilde{b}^{(1)}_1+i\widetilde{b}^{(1)}_2)}.
\eeqa
By recalling that $U$ is defined in~\eqref{def:U} such that
$$\frac{\wLP'}{\wLP}=\frac{-2a_*\vLP}{\wLP}+U,$$
we get
\beq\label{def:A2Q}
\widetilde{A}^{(2)}=\frac{6}{y}+\frac{\wLP'+2(a-2a_*)\vLP+2ib\vLP}{\wLP}+\widetilde{a}^{(2)},
\eeq
with
\beqa\label{def:a2Q}
\widetilde{a}^{(2)}=& \widetilde{a}^{(2)}_1+i\widetilde{a}_2^{(2)},\\
\widetilde{a}_1^{(2)}=&\,\frac{\rhoLP'}{\rhoLP}-\frac{2\omLP'}{1-\omLP}+2U-\widetilde{d}^{(0)}_1-\widetilde{d}^{(1)}_1,\qquad \widetilde{a}_2^{(2)}=-\widetilde{d}^{(0)}_2-\widetilde{d}^{(1)}_2,
\eeqa
Moreover,
\beqa
\widetilde{B}^{(2)}=&\,\frac{1}{\wLP}\Big(\widetilde{b}^{(1)}_1+2\wLP''+2a\vLP'+(\wLP\widetilde{a}^{(0)})'-\wLP(\widetilde{d}^{(0)}_1)'-16\omLP\vLP'+4a\omLP+\frac{2\wLP}{y}\big(\widetilde{a}^{(0)}-\widetilde{d}^{(0)}_1\big)\\
&\quad-\wLP'\widetilde{d}^{(0)}_1-(2\wLP'+2a\vLP)\widetilde{d}^{(1)}_1+2b\vLP \widetilde{d}^{(1)}_2-\wLP\widetilde{d}^{(1)}_1\big(\frac{4}{y}+\widetilde{a}^{(0)}-\widetilde{d}^{(0)}_1\big)-\wLP\widetilde{d}^{(1)}_2\widetilde{d}^{(0)}_2\Big)\\
&+i\frac{1}{\wLP}\Big(\widetilde{b}^{(1)}_2+2b\vLP'-\wLP(\widetilde{d}^{(0)}_2)'+4b\omLP-\frac{2}{y}\widetilde{d}^{(0)}_2-\wLP'\widetilde{d}^{(0)}_2-(2\wLP'+2a\vLP)\widetilde{d}^{(1)}_2-2b\vLP \widetilde{d}^{(1)}_1\\
&\quad -\wLP\widetilde{d}^{(1)}_2\big(\frac{4}{y}+\widetilde{a}^{(0)}-\widetilde{d}^{(0)}_1\big)+\wLP\widetilde{d}^{(1)}_1\widetilde{d}^{(0)}_2\Big),
\eeqa
as claimed in \eqref{eq:Pequation}.

\subsubsection{Energy coefficients}\label{S:ENERGYIA}

\begin{lemma}\label{L:HLAMBDASIGN}
The coefficients ${H_\l}(y)$ defined in~\eqref{E:HPSI} and $D^{(4)}_2$ as in~\eqref{E:AFOURDEF} (compare~\eqref{E:A4IDENTITY}) satisfy, for all $a\in[0,1]$, $b\geq 8$, 
\beq
{H_\l}(y)+\frac{1+\frac{b^2}{\al_*^2}}{4}\wLP|D^{(4)}_2|^2<0,\ \ \ \ y\in[0,y_*].
\eeq
\end{lemma}

\begin{proof}
The proof of this inequality directly employs the interval arithmetic package VNODE-LP. Having described above in Appendix~\ref{A:LPCONSTRUCTION} how we construct the LP solution, we now employ the identities of Section~\ref{S:Highbcoeffs} to derive expressions for this coefficient $H_\l$ in terms of the LP solution and its derivatives up to order 5. Due to the singular points at the origin and at the sonic point, we employ the Taylor expansions near these two points to construct the LP quantities, and use the VNODE-LP ODE solver to obtain the LP solution in the intermediate region between them. The interval arithmetic functions \verb!large_b_Sonic_Prover!, \verb!large_b_Origin_Prover!, and \verb!large_b_Intermediate_Prover! then establish the claimed sign condition for $b\geq 8$.

In order to compute the coefficient functions arising in $H_\l$ and as derived in Section~\ref{S:Highbcoeffs}, we require derivatives of the LP solution of order up to 5, as well as quotients of such functions with $y$ and their derivatives. All of these derivatives are computed and stored using functions in the section of the code titled \verb!SEC: LP Derivatives!. Due to the singularities at the origin and sonic point, we derive these various quantities using Taylor expansions (with suitable errors arising from truncation) in \verb!SUBSEC: Derivatives near sonic point! and \verb!SUBSEC: Derivatives near origin!, and rely on the structure of the LP ODE system~\eqref{E:RHOLP}--\eqref{E:OMLP} in the intermediate regions in the section \verb!SUBSEC: Derivatives in intermediate region!.
\end{proof}

\begin{lemma}\label{L:HLAMBDATILDE}
The coefficients ${\widetilde{H}_\l}(y)$ defined in~\eqref{def:HP} and $\widetilde{a}^{(2)}_2$ as in~\eqref{def:a2Q} satisfy, for all $a\in[0,1]$, $|b|\leq \frac15$, 
\beq
\widetilde{H}_\l + \frac{1+\frac{b^2}{\al_*^2}}{4}\wLP|a^{(2)}_2|^2\leq -c_1<0,\ \ \ y\in[0,y_*].
\eeq
\end{lemma}

\begin{proof}
The proof again uses interval arithmetic, and follows the same lines as the previous lemma, now using the identities of Appendix~\ref{S:PEQCOEFFS}.
\end{proof}


\subsection{Excluding intermediate eigenvalues}
In this appendix, we follow a strategy analogous to that employed above in sections~\ref{S:SONICTAYLOR}--\ref{S:ORIGINTAYLOR} in order to construct the Taylor series for potential eigenfunctions for $\bfL$ close to the sonic point $y_*$ and the origin $y=0$.  Throughout, we write $$1-\la=a+i b.$$
The goal of this appendix is to establish the following coefficient growth estimates.
 
 \begin{proposition}\label{L:EFUNCTIONTAYLOR}
Let $\la=1-a - ib$ where  $a\in[0,1]$, $b\in[\frac15,8]$, be an eigenvalue of $\bfL$ and suppose $\psi_\l$ is an eigenfunction solving~\eqref{eq:phila}. \\
(i) Then the eigenfunction $\psi_\la$ admits the expansion near the sonic point
 \beq\label{E:PSITILDEEXPSONIC}
 y^2\psi_\la(y)=\sum_{k=0}^\infty \psi_k(\l)\big(\frac{y}{y_*}-1\big)^k
 \eeq
 where for $k\geq 2$,
 \beq
 \Big|\frac{\psi_{k}}{\psi_0}\Big|\leq 4y_*^2\frac{\mathfrak{C}^{k+1-\al}}{k^3},
 \eeq
 where $\mathfrak{C}=8+\frac{3b}{2}$ and $\al=1.98$, so that, for $|\frac{y}{y_*}-1|\leq \de<\frac{1}{\mathfrak{C}}$, the normalised eigenfunction with $\psi_0=1$ satisfies
 \beq
 \Big|y^2\psi_\la(y)-\sum_{k=0}^N\psi_k(\frac{y}{y_*}-1)^k\Big|\leq 4y_*^2\mathfrak{C}^{1-\al}\frac{(\mathfrak{C}\de)^{N+1}}{(N+1)^3(1-\mathfrak{C}\de)}.
 \eeq
 (ii) Near the origin, the eigenfunction expands as 
 \beq
y^2\psi_\la(y)=\sum_{k=1}^\infty \tilde\psi_{2k+1}(\l) y^{2k+1},\label{E:PSITILDEEXPORIGIN}
\eeq
where, normalising to $\tilde \psi_3=1$,
\beq
|\tilde \psi_{2\ell+1}|\leq 7\frac{{\mathfrak{C}}_0^{2\ell-\al_0}}{(2\ell+1)^3}\qquad \text{ for }\ell\geq 4,
\eeq
for
 $\mathfrak{C}_0=2+\frac{b}{4}$, $\al_0=1.95$, so that for $0\leq y\leq \de <\frac{1}{\mathfrak{C}_0}$
\beq
 \Big|\psi_\l(y)-\sum_{k=1}^n \tilde\psi_{2k+1} y^{2k+1}\Big|\leq \frac{7}{\mathfrak{C}_0^{1+\al_0}}\frac{(\mathfrak{C}_0\de)^{2n+3}}{(2n+3)^3(1-(\mathfrak{C}_0\de)^2)}.
 \eeq
 \end{proposition}
 
 As is clear from the statement of the proposition above, we will work for convenience with the rescaled function
 \beq\label{E:PSITILDEDEF}
 \tpsi_\l(y):=y^2\psi_\l,
 \eeq
 which satisfies the ODE
  \beq\label{eq:psila}
 \tpsi_{\la}''+\Big(-\frac{2}{y}-\frac{\rhoLP'}{\rhoLP}+\frac{\wLP'}{\wLP}+\frac{2(1-\la)\vLP}{\wLP}\Big)\tpsi_\la'+\Big(\frac{2\rhoLP}{\wLP}+\frac{2(1-\la)\vLP'}{\wLP}-\frac{(2-\la)(1-\la)}{\wLP}\Big)\tpsi_\la=0.
 \eeq
 The advantage of this formulation is that the coefficients $W_1$ and $W_2$ are somewhat simpler, and so we hope to get tighter bounds on their Taylor coefficients.
 
 Moreover, for reasons of scaling for the coefficients near the sonic point, it is advantageous to work with the rescaled variable $z=\frac{y}{y_*}$ which fixes the sonic point to $z=1$.  We warn the reader here that this $z$ is used only in this appendix and should not be confused with the Lagrangian self-similar variable used throughout this paper. We commonly do not distinguish between functions defined in $z$ and in $y$, making it clear in each instance in which variable we are working. 

\subsubsection{Expansions for $\wLP(y)$}
From~\eqref{eq:psila}, it is clear that we will require expansions for the Taylor coefficients of the weight function $\wLP(y)=1-\vLP^2(y)$ in order to close the expansions for potential eigenfunctions below. We therefore collect here the growth rates of the coefficients for this function. To fix notation, we write
\beq
\wLP(y)=\Big(\frac{y}{y_*}-1\Big)\tilde w(y).
\eeq
 \begin{lemma}\label{L:WEXPANSIONSONIC}
 Let $C =7.2$, $\al=1.98$. \\
(i) The derivative $\frac{\dif}{\dif z}\wLP(z)$ can be expanded close to the sonic point $z=1$ as
\beq
\wLP'(z)=\sum_{k=0}^\infty (w')_k(z-1)^k,
\eeq
where 
\beqa\label{ineq:w'1expression}
(w')_0=&\,-2(y_*-1),\\
(w')_1=&\,-2y_*^2(\om_1^2+2\om_2\om_0+4\om_1\om_0+\om_0^2)=-2\frac{2y_*^3-8y_*^2+13y_*-8}{(2y_*-3)},
\eeqa and
\beq\label{ineq:w'bd}
|(w')_k|\leq 1.4166y_*^2\frac{C^{k+1-\al}}{k+1}\text{ for }k\geq 2.
\eeq
(ii) The quantities $\tilde w$ and $\frac{1}{\tilde w}$ expand as
\beq
\tilde w(z)=\sum_{k=0}^\infty \tilde w_j(z-1)^k\quad \text{ and }\quad  \frac{1}{\tilde w(z)}=\sum_{k=0}^\infty \bar w_k(z-1)^k
\eeq
with $\tilde w_0=-2(y_*-1)$ and $\bar w_0=(-2(y_*-1))^{-1}$, 
\beq\label{def:wbar1}
\bar w_1=2y_*^2\frac{\om_1^2+2\om_2\om_0+4\om_1\om_0+\om_0^2}{2\tilde w_0^2},
\eeq
and
 \beq\label{ineq:wtildebd}
 |\tilde w_k|\leq 1.4166y_*^2\frac{{C}^{k+1-\al}}{(k+1)^2}\text{ for }k\geq 2,
 \eeq
 and
 \beq\label{ineq:wbarbd}
 |\bar w_k|\leq (0.506)y_*^2\frac{{C}^{k+1-\al}}{(k+1)^2}\text{ for }k\geq 2.
 \eeq
\end{lemma}


\begin{proof}
We begin by proving (i). Recall
 \beqa\label{eq:w'expansion}
 \wLP'(z)=&\,-2y_*^2z\omLP(z\omLP'(z)+\omLP(z))\\
 =&\,-2y_*^2\sum_{k=0}^\infty\Big(\sum_{j=0}^k(j+1)\om_{j+1}\om_{k-j}+2\sum_{j=0}^{k-1}(j+1)\om_{j+1}\om_{k-1-j}+\sum_{j=0}^{k-2}(j+1)\om_{j+1}\om_{k-2-j}\Big)(z-1)^k\\
 &-2y_*^2\sum_{k=0}^\infty\big((\om^2)_k+(\om^2)_{k-1}\big)(z-1)^k. 
 \eeqa
 From Lemma~\ref{L:PRELIMBOUNDS}, we have that $|(\om^2)_k|\leq D\frac{C^{k-\al}}{k^2}$ for $k\geq 2$ already, with constants $C=7.2$ and $\al=1.98$.  We check directly with interval arithmetic (see Lemma~\ref{L:WSONICIA}(i)) that $|(w')_k|\leq 1.4166y_*^2\frac{C^{k+1-\al}}{k+1}$ for $k=2,3$  and the identities for $(w')_0$ and $(w')_1$ are a direct computation.
 
We now consider $k\geq 4$. We expand the coefficient in the first line on the right in \eqref{eq:w'expansion} as
 \beqa\label{E:D76}
{}& \sum_{j=0}^k(j+1)\om_{j+1}\om_{k-j}+2\sum_{j=0}^{k-1}(j+1)\om_{j+1}\om_{k-1-j}+\sum_{j=0}^{k-2}(j+1)\om_{j+1}\om_{k-2-j}\\
 &=(k+1)\om_{k+1}\om_0+\om_k\big((k+1)\om_1+2k\om_0\big)+\om_{k-1}\big((k+1)\om_2+2k\om_1+(k-1)\om_0\big)\\
 &\ \ \ +\om_{k-2}(k-1)\om_1+\sum_{j=2}^{k-3}(j+1)\om_{j+1}\om_{k-j}+\sum_{j=1}^{k-3}(j+1)\om_{j+1}\om_{k-1-j}+\sum_{j=1}^{k-4}(j+1)\om_{j+1}\om_{k-2-j}. 
 \eeqa
 Recalling the definitions of $\om_j$, $j=0,1,2$, from~\eqref{eq:order012coeffs}, the first line on the right hand side of~\eqref{E:D76} simplifies to
 \footnotesize
 \beqas
{}&\bigg|\frac{(k+1)}{y_*}\om_{k+1}+\om_k(k+1-\frac{2}{y_*})+\om_{k-1}\Big((k-1)\frac{1 - 5 y_* + 3 y_*^2}{4 y_*^2-6 y_*}+\frac{5 - 5 y_* + y_*^2}{3 y_* - 2 y_*^2}\Big)+\om_{k-2}(k-1)(1-\frac{2}{y_*})\bigg|\\
&\leq \frac{C^{k+1-\al}}{(k+1)}\Big(\frac{1}{y_*}+\frac{13}{10C}+\frac{1.214}{C^2}+\frac{7}{40C^2}+\frac{0.548}{C^3}\Big),
 \eeqas
 \small
 where we have used Lemma~\ref{L:WSONICIA}(ii). 
 
 We then see, from bounds \eqref{ineq:comb3}--\eqref{ineq:comb5},  that the remainder of~\eqref{E:D76} is bounded, for $k\geq 4$, i.e.
 \beqa
 \Big|\sum_{j=2}^{k-3}&(j+1)\om_{j+1}\om_{k-j}+\sum_{j=1}^{k-3}(j+1)\om_{j+1}\om_{k-1-j}+\sum_{j=1}^{k-4}(j+1)\om_{j+1}\om_{k-2-j}\Big|\\
 \leq&\,C^{k+1-2\al}\Big(\sum_{j=2}^{k-3}\frac{1}{(j+1)(k-j)^2}+\frac{1}{C}\sum_{j=1}^{k-3}\frac{1}{(j+1)(k-1-j)^2}+\frac{1}{C^2}\sum_{j=1}^{k-4}\frac{1}{(j+1)(k-2-j)^2}\Big)\\
 \leq&\,\frac{C^{k+1-\al}}{k}\Big(\frac{0.506}{C^\al}+\frac{9}{10C^{\al+1}}+\frac{1}{C^{\al+2}}\Big) \\
 \leq&\,\frac{C^{k+1-\al}}{(k+1)}\frac{5}{4}\Big(\frac{0.506}{C^\al}+\frac{9}{10C^{\al+1}}+\frac{1}{C^{\al+2}}\Big). 
 \eeqa
 We also observe the bound
 \beqa
\Big| (\om^2)_k&+(\om^2)_{k-1}\Big|\leq D\Big(\frac{C^{k-\al}}{k^2}+\frac{C^{k-1-\al}}{(k-1)^2}\Big)
\leq \frac{C^{k+1-\al}}{k+1}\Big(\frac{5D}{16C}+\frac{5D}{9C^2}\Big)
 \eeqa
 for $k\geq 4$.
 So the $k$-th coefficient of $\wLP'(z)$ is bounded as 
 \footnotesize
 \beqas
 |(w'&)_k|\\
 \leq&\, \frac{C^{k+1-\al}}{(k+1)}2y_*^2\Big(\frac{1}{y_*}+\frac{13}{10C}+\frac{1.214}{C^2}+\frac{7}{40C^2}+\frac{0.548}{C^3}+\frac{5}{4}\big(\frac{0.506}{C^\al}+\frac{9}{10C^{\al+1}}+\frac{1}{C^{\al+2}}\big)+\frac{5D}{16C}+\frac{5D}{9C^2}\Big)\\
 \leq&\,1.4166y_*^2\frac{C^{k+1-\al}}{(k+1)} 
 \eeqas
 \small
 for $C=7.2$, $\al=1.98$, $y_*\in[2.34,2.342]$ by Lemma~\ref{L:WSONICIA}(iii).   This concludes the proof of part (i).\\
 
 (ii) First, from the relation 
 \beq
 \sum_{k=0}^\infty(w')_k(z-1)^k=\wLP'(z)=\tilde w + \tilde w'(z-1)=\sum_{k=0}^\infty\tilde w_k(z-1)^k+\sum_{k=0}^\infty k\tilde w_{k}(z-1)^{k},
 \eeq
 we deduce
 \beq\label{eq:w'towtilde}
 (w')_{k}=(k+1)\tilde w_k\text{ for } k\geq 1,
 \eeq
which implies~\eqref{ineq:wtildebd}, while 
 \beq
 \tilde w_1= -\frac{2y_*^3-8y_*^2+13y_*-8}{(2y_*-3)}.
 \eeq
To address $\frac{1}{\tilde w}$, we first expand the trivial identity $\frac{\tilde w}{\tilde w}=1$ to find 
 \beqa
 \bar w_k=-\frac{1}{\tilde w_0}\sum_{j=1}^k\tilde w_j\bar w_{k-j}=-\frac{\tilde w_k}{\tilde w_0^2} -\frac{1}{\tilde w_0} \sum_{j=1}^{k-1}\tilde w_j\bar w_{k-j},
 \eeqa
 where we have used that $\bar w_0=\frac{1}{\tilde w_0}$.  
 
 We use the relation \eqref{eq:w'towtilde} and the identities \eqref{ineq:w'1expression} to obtain
 \beqa
 |\bar w_1|
 \leq \frac{y_*^2}{4}\frac{C^{2-\al}}{4}, \qquad |\bar w_2| \leq0.506y_*^2\frac{C^{3-\al}}{9},
 \eeqa
 where we use $C=7.2$, $\al=1.98$, and Lemma~\ref{L:WSONICIA}(iv).\\
 Then 
 we assume for an induction that $$\bar w_j\leq \beta y_*^2\frac{C^{j+1-\al}}{(j+1)^2},\quad j=2,\ldots,k-1,$$
 and estimate for $k\geq 3$
 \begin{align*}
 | \bar w_k|\leq&\,\big|\frac{\tilde w_k}{\tilde w_0^2}\big| +\frac{|\tilde w_{k-1}||\bar w_1|}{|\tilde w_0|} +\frac{1}{|\tilde w_0|} \sum_{j=1}^{k-2}|\tilde w_j\bar w_{k-j}|\\
 \leq&\,\frac{C^{k+1-\al}}{(k+1)^2}\frac{1.4166y_*^2}{\tilde w_0^2}+\frac{C^{k-\al}}{k^2}\frac{1.4166y_*^2|\bar w_1|}{|\tilde w_0|}+1.4166\beta y_*^4\frac{C^{k+2-2\al}}{|\tilde w_0|}\sum_{j=1}^{k-2}\frac{1}{(j+1)^2(k+1-j)^2}\\
 \leq&\, y_*^2\frac{C^{k+1-\al}}{(k+1)^2}\Big(\frac{1.4166}{\tilde w_0^2}+\frac{1.4166(k+1)^2|\bar w_1|}{Ck^2|\tilde w_0|}+ (1.242)\beta\frac{1.4166y_*^2}{C^{\al-1}|\tilde w_0|}\Big)\\
 \leq&\,\beta y_*^2\frac{C^{k+1-\al}}{(k+1)^2}
 \end{align*}
provided $\beta= 0.506$, where we have used~\eqref{ineq:wtildebd}, $\sum_{j=1}^{k-2}\frac{1}{(j+1)^2(k+1-j)^2}\leq \frac{1.242}{(k+1)^2}$, $C=7.2$, $\alpha=1.98$  and Lemma~\ref{L:WSONICIA}(v).
This concludes the proof.
 \end{proof}

 Close to the origin, we also require estimates for the coefficients of $\wLP$, which we obtain in the following lemma.
 \begin{lemma}
 Let $C_0=2$, $\al_0=1.95$. 
Expand $\wLP$ and $\frac{1}{\wLP}$ near the origin in $y$ coordinates as
\beq
\wLP(y)=\sum_{k=0}^\infty \hat w_k y^k,\qquad \frac{1}{\wLP}(y)=\sum_{k=0}^\infty \check{w}_ky^k.
\eeq
Then the coefficients satisfy
\beqa\label{eq:worigincoeffs}
&\hat w_0=\check{w}_0=1,\\
&\hat w_2=-\check{w}_2=-\frac19,\\
&\hat w_4=-\frac{4}{135}(\tr_0-\frac13),\qquad \check{w}_4=\frac{1}{81}+\frac{4}{135}(\tr_0-\frac13),
\eeqa
and, for all $k\geq 2$,
\beq\label{eq:whatwcheck2kbds}
|\hat w_{2k}|\leq \frac{C_0^{2(k-1)-\al_0}}{(2(k-1))^2},\qquad |\check{w}_{2k}|\leq(1.35)\frac{C_0^{2(k-1)-\al_0}}{(2(k-1))^2}.
\eeq
\end{lemma}
\begin{proof}
It is obvious that $\what_0=1$ and $\wLP$ is even in $y$. Expanding  $\wLP=1-\vLP^2=1-y^2\omLP^2$, we see easily that for $k\geq 1$,
\beq
\what_{2k}=-\sum_{j=0}^{k-1}\tom_{2j}\tom_{2(k-1-j)},
\eeq
 where we recall that only the even coefficients are non-vanishing. So we easily see also
\beqa
\what_2=-\tom_0^2=-\frac19,\qquad \what_4=-2\tom_0\tom_2=-\frac{4}{135}(\tr_0-\frac13).
\eeqa
We see from Lemma~\ref{L:WORIGINIA}(i) that $|\what_4|\leq \frac{C_0^{2-\al_0}}{4}$ with $C_0=2$, $\al_0=1.95$, $\hat\rho(0)=\tilde\rho_0\in[0.83,0.84]$.

Now, for $k\geq 3$, we assume the bound by induction for $j\leq 2\leq k-1$, and estimate
\beqa
|\what_{2k}|=&\,\Big|2\tom_0\tom_{2(k-1)}+\sum_{j=1}^{k-2}\tom_{2j}\tom_{2(k-1-j)}\Big|\\
\leq&\,\frac23\frac{C_0^{2(k-1)-\al_0}}{(2(k-1))^2}+\frac{C_0^{2(k-1)-\al_0}}{C_0^{\al_0}}\sum_{j=1}^{k-2}\frac{1}{(2j)^2(2(k-j-1)^2)}\\
\leq&\,\frac{C_0^{2(k-1)-\al_0}}{(2(k-1))^2}\Big(\frac23+\frac{41}{36C_0^{\al_0}}\Big)\leq\frac{C_0^{2(k-1)-\al_0}}{(2(k-1))^2},
\eeqa
where we have bounded the sum using \eqref{ineq:comb9} and, in the last inequality, we used that $C_0=2$, $\al_0=1.95$ to deduce $\frac23+\frac{41}{36C_0^{\al_0}}<1$ Lemma~\ref{L:WORIGINIA}(ii). This concludes the first estimate in~\eqref{eq:whatwcheck2kbds}.

Now to estimate $\check{w}_{2k}$, we note
 $\check{w}_0=1$ and, for $k\geq 1$,
\beq
\check{w}_{2k}=-\sum_{j=1}^k\what_{2j}\check{w}_{2k-2j}.
\eeq
So 
\beqa
\check{w}_2=&\,-\what_2\check{w}_0=\frac19,\qquad \check{w}_4=-\what_2\check{w}_2-\what_4\check{w}_0=\frac{1}{81}+\frac{4}{135}(\tr_0-\frac13),
\eeqa
and we check from Lemma~\ref{L:WORIGINIA}(iii) that 
\beqa
|\check{w}_4|\leq 0.11\frac{C_0^{2-\al_0}}{4},\qquad |\check{w}_6|\leq 1.02\frac{C_0^{4-\al_0}}{16}.
\eeqa
Now for $k\geq 4$,
\beq
\check{w}_{2k}=-\what_2\check{w}_{2(k-1)}-\what_{2k}\check{w}_0-\what_{2(k-1)}\check{w}_{2}-\sum_{j=2}^{k-2}\what_{2j}\check{w}_{2k-2j}.
\eeq
Assuming for an induction that $|\check{w}_{2j}|\leq \beta \frac{C_0^{2(j-1)-\al_0}}{(2(j-1))^2}$ for all $2\leq j\leq k-1$, we have, for $k\geq 4$,
\beqa
|\check{w}_{2k}|\leq&\,|\what_2\check{w}_{2(k-1)}|+|\what_{2k}|+|\what_{2(k-1)}\check{w}_{2}|+\sum_{j=2}^{k-2}|\what_{2j}||\check{w}_{2k-2j}|\\
\leq&\,\frac{C_0^{2(k-1)-\al_0}}{(2(k-1))^2}\Big((\beta |\what_2|+|\check{w}_2|)\frac{(k-1)^2}{C_0^2(k-2)^2}+1\Big)+\beta\sum_{j=2}^{k-2}\frac{C_0^{2(j-1)-\al_0}}{(2(j-1))^2}\frac{C_0^{2(k-j-1)-\al_0}}{(2(k-j-1))^2}\\
\leq&\,\frac{C_0^{2(k-1)-\al_0}}{(2(k-1))^2}\Big(\frac{\beta+1}{9}\frac{(k-1)^2}{C_0^2(k-2)^2}+1+\frac{\beta}{C_0^{2+\al_0}}\sum_{j=2}^{k-2}\frac{1}{(2(j-1))^2(2(k-j-1))^2}\Big)\\
\leq&\,\frac{C_0^{2(k-1)-\al_0}}{(2(k-1))^2}\Big(\frac{\beta+1}{9}\frac{(k-1)^2}{C_0^2(k-2)^2}+1+\frac{9\beta}{4C_0^{2+\al_0}}\Big)\\
\leq&\,(1.35)\frac{C_0^{2(k-1)-\al_0}}{(2(k-1))^2}
\eeqa
for $\beta =1.35$, where we have used \eqref{ineq:comb11} to bound the sum and employed $C_0=2$, $\al_0=1.95$, and Lemma~\ref{L:WORIGINIA}(iv). This concludes the proof.
\end{proof}


\subsubsection{Expanding eigenfunctions near the sonic point}
The goal of this subsection is to establish the coefficient expansions and growth rates for the eigenfunction $\psi_\l$ around the sonic point, as claimed in Proposition~\ref{L:EFUNCTIONTAYLOR}(i). As advertised above, we work in the coordinate $z=\frac{y}{y_*}$ (not to be confused with the Lagrangian label $z$ used elsewhere in this paper). We recall the function $\tpsi_\l$ defined in~\eqref{E:PSITILDEDEF}, so that, in $z$ coordinates, $\tpsi(z)=(y_*z)^2\psi_\l(z).$ 
 Now, shifting to the $z$ variable in~\eqref{eq:psila} and recalling the factorisation $\wLP(z)=(1-z)\tilde w(z)$, we find that \eqref{eq:psila} is of the form
 \beq\label{eq:efunctionsimpleform}
 \tpsi_{\la}''(z)+\frac{\tilde a(z)}{z-1}\tpsi_\la'(z)+\frac{\tilde b(z)}{z-1}\tpsi_\la(z)=0,
 \eeq
 where $\tilde{a}=\sum_{j=0}^\infty \tilde a_j(z-1)^j$ and $\tilde b=\sum_{j=0}^\infty\tilde b_j(z-1)^j$ are analytic functions, defined by 
 \beqa\label{eq:atilde}
 \tilde a(z)=&\,-\frac{2(z-1)}{z}+\frac{A(z)}{\tilde w(z)},
 \eeqa
 where we have defined
 \beq\label{eq:A}
 A(z)=2y_*^2z\omLP(\rhoLP-\omLP)+\wLP'+2(1-\la)y_*^2z\omLP, 
 \eeq
 and 
 \beq\label{E:BTILDE}
\tilde b= \frac{y_*^2\big(2\rhoLP+2(1-\la)(z\omLP'+\omLP)-(2-\la)(1-\la)\big)}{\tilde w}.
 \eeq
 We emphasise once more that the $'$ notation refers to $z$ derivatives.
 

  In order to close the coefficient estimates for the eigenfunction $\tpsi_\l$, we require growth estimates on $\tilde a_j$ and $\tilde b_j$. This is the content of the next lemma.
  
  \begin{lemma}\label{L:D13}
  Let $1-\la=a+ib$, $a\in[0,1]$, $b\geq 0$. The coefficients $\tilde a$ and $\tilde b$ expand as
  \beq
  \tilde a(z)=\sum_{k=0}^\infty \tilde a_k(z-1)^k,\qquad  \tilde b(z)=\sum_{k=0}^\infty \tilde b_k(z-1)^k
  \eeq
  where we have explicit formulae for $k=0,1,2$, and, for $k\geq 3$,
  \begin{align}
  |\tilde a_k|\leq&\,(1.5819+0.692b)y_*^2\frac{C^{k+1-\al}}{k+1},\label{ineq:akbd}\\
  |\tilde b_k|\leq&\, \big(1.8932+2.251 b+\frac{(0.506)b^2y_*^2}{4}\big)y_*^2\frac{C^{k+1-\al}}{k+1},\label{ineq:bkbd}
  \end{align}
   for all $\S\in[2.34,2.342]$, where $C=7.2$, $\al=1.98$. 
  \end{lemma}
  

  \begin{proof}
 \textbf{Step 1: Expansion for $A$}\\
 To obtain an expansion for $A$, we first expand
 \beqas
 2y_*^2z\omLP(\rhoLP-\omLP)
 =&\,2y_*^2\sum_{k=1}^\infty \big((\rho\om)_k-(\om^2)_k+(\rho\om)_{k-1}-(\om^2)_{k-1}\big)(z-1)^k,\\
  2(1-\la)y_*^2z\omLP=&\,2(1-\la)y_*^2\sum_{k=0}^\infty(\om_k+\om_{k-1})(z-1)^k,
 \eeqas
 where we note that the zero order term has vanished due to $(\rho\om)_0=(\om^2)_0$. 
 
Substituting these identities into~\eqref{eq:A}, we then apply the bounds of Proposition~\ref{P:RHONOMNBDS} and Lemma~\ref{L:PRELIMBOUNDS} for the LP coefficients, along with~\eqref{ineq:w'bd}, to see that, for $k\geq 3$, 
 \begin{align}
 |A_k|\leq&\, 2y_*^2 \big|(\rho\om)_k-(\om^2)_k+(\rho\om)_{k-1}-(\om^2)_{k-1}\big|+|(w')_k|+\big|2(1-\la)y_*^2(\om_k+\om_{k-1})\big| \notag\\
 \leq&\, |(w')_k|+4Dy_*^2\big(\frac{C^{k-\al}}{k^2}+\frac{C^{k-1-\al}}{(k-1)^2}\big)+ 2y_*^2|1-\la|\big(\frac{C^{k-\al}}{k^2}+\frac{C^{k-1-\al}}{(k-1)^2}\big)\notag\\
 \leq&\,y_*^2\frac{C^{k+1-\al}}{k+1}\Big(1.4166+4D\big(\frac{k+1}{Ck^2}+\frac{k+1}{C^2(k-1)^2}\big)+ 2|1-\la|\big(\frac{k+1}{Ck^2}+\frac{k+1}{C^2(k-1)^2}\big)\Big)\notag\\
 \leq&\,y_*^2\frac{C^{k+1-\al}}{k+1}\Big(1.4166+(4D+2(1+b))\big(\frac{4}{9C}+\frac{1}{C^2}\big)\Big)\notag\\
 \leq&\,(1.913+0.1621b)y_*^2\frac{C^{k+1-\al}}{k+1}\label{ineq:Akbds}
 \end{align}
 for $k\geq 3$, where we use  Lemma~\ref{L:ABSONICIA}(i).

 To enable more accurate estimates, we give exact representations for the first few coefficients of $A$. These are computed as 
 \beqa\label{eq:A0A1A2}
 A_0=&\,(w')_0+2(a+ib)y_*=-2(y_*-1)+2(a+ib)y_*, \\
 A_1=&\,(w')_1+2y_*^2((\rho\om)_1-(\om^2)_1)+2(a+ib)y_*^2(\om_1+\om_0) , \\
 A_2=&\,(w')_2+2y_*^2\big((\rho\om)_2-(\om^2)_2+(\rho\om)_1-(\om^2)_1\big)+2(a+ib)y_*^2(\om_2+\om_1),
 \eeqa
 which can be used more conveniently for bounds.
 
\textbf{Step 2: Expansion for $\tilde a$}\\
 Now to obtain an expansion for the whole of $\tilde a$, we note that, for the given function $A(z)$, we have
 \beq
 \frac{A(z)}{\tilde w(z)}=\sum_{k=0}^\infty \Big(\sum_{j=0}^kA_j\bar w_{k-j}\Big)(z-1)^k .
 \eeq
 For our $A_j$, we have obtained explicit expressions for $j=0,1,2$ and bounds for $j\geq 3$, while $\bar w_j$ is bounded nicely for $j\geq 2$ and has an explicit expression for $j=0,1,2$ from Lemma~\ref{L:WEXPANSIONSONIC}(ii). We therefore treat the coefficients $\big( \frac{A}{\tilde w}\big)_j$ explicitly for $j=0,1,2$ for the purposes of estimates. 
 For the third order coefficient, we apply Lemma~\ref{L:ABSONICIA}(ii) to see
 \beqa\label{ineq:A/w3}
 \Big| \big( \frac{A}{\tilde w}\big)_3\Big| \leq&\,(1.057+(0.6626)b)y_*^2\frac{C^{4-\al}}{4},\qquad \Big|\big(\frac{A}{\tilde w}\big)_4\Big|\leq (1.301 + (0.692)b)y_*^2 \frac{C^{5-\al}}{5}.
 \eeqa
 For $k\geq 5$,
 \beq\label{eq:atilde4plus}
 \sum_{j=0}^kA_j\bar w_{k-j}=A_0\bar w_k + A_1\bar w_{k-1}+A_2\bar w_{k-2}+A_k\bar w_0 + A_{k-1}\bar w_1+ \sum_{j=3}^{k-2}A_j\bar w_{k-j}. 
 \eeq
To estimate the sum in the last term of \eqref{eq:atilde4plus}, recalling the estimates \eqref{ineq:wbarbd} and \eqref{ineq:Akbds}, we bound
 \beqa
 \sum_{j=3}^{k-2}|A_j\bar w_{k-j}|\leq&\,(0.506)(1.913+0.1621b)y_*^4\sum_{j=3}^{k-2}\frac{C^{j+1-\al}}{j+1}\frac{C^{k-j+1-\al}}{(k-j+1)^2} \\
 \leq &\,(0.506)(1.913+0.1621b)\frac{y_*^4}{C^{\al-1}}0.49\frac{C^{k+1-\al}}{k+1}, 
 \eeqa
 where we have used~\eqref{ineq:combextra} to control the summation. Thus, for $k\geq 5$, we may bound all of \eqref{eq:atilde4plus} using also Lemma~\ref{L:ABSONICIA}(iii) as 
 \beqa
 \big|A_0&\bar w_k + A_1\bar w_{k-1}+A_2\bar w_{k-2}+A_k\bar w_0+A_{k-1}\bar w_1\big|+ \sum_{j=3}^{k-2}|A_j\bar w_{k-j}|\\
 \leq&\,y_*^2\frac{C^{k+1-\al}}{k+1}\bigg((0.506)\Big(\frac{|A_0|}{6}+\frac{6|A_1|}{25C}+\frac{6|A_2|}{16C^2}\Big)\\
 &+(1.913+0.1621b)\Big(|\bar w_1|\frac{6}{5C}+|\bar w_0|+(0.506)(0.49)\frac{y_*^2}{C^{\al-1}}\Big)\bigg)\\
 \leq&\,(1.5771+(0.6091)b)y_*^2\frac{C^{k+1-\al}}{k+1}.
 \eeqa
 Then, for all $k\geq 4$, we take the representation for $\tilde a$ given by \eqref{eq:atilde} and note 
\beq
 -\frac{2(z-1)}{z}=\sum_{k=1}^\infty 2(-1)^k(z-1)^k
 \eeq
 in order to apply Lemma~\ref{L:ABSONICIA}(iv) and bound
 \beqa
 |\tilde a_k|\leq&\, |2(-1)^k|+(1.5771+0.692b)y_*^2\frac{C^{k+1-\al}}{k+1}\leq y_*^2\frac{C^{k+1-\al}}{k+1}\Big(2\frac{5}{y_*^2(7.2)^{5-\al}}+(1.5771+0.692b)\Big)\\
 \leq&\,(1.5819+0.692b)y_*^2\frac{C^{k+1-\al}}{k+1}.
 \eeqa
 For $k=0,1,2$, we have explicit representations of $\tilde a_k$ from \eqref{eq:A0A1A2},
 while for $k=3$, we apply  Lemma~\ref{L:ABSONICIA}(v) to verify 
 \beqa\label{ineq:a3tilde}
 |\tilde a_3|
 \leq&\,(1.0841+(0.6626)b)y_*^2\frac{C^{4-\al}}{4}.
 \eeqa
 This concludes the proof of~\eqref{ineq:akbd} for $\tilde a$.\\
 
\textbf{Step 3: Estimate for $B$}\\
 To estimate $\tilde b$, we begin by setting
 $$B=y_*^2\big(2\rhoLP+2(a+ib)(z\omLP'+\omLP)-(2-\la)(1-\la)\big),$$
  so that 
$ \tilde b=\frac{B}{\tilde w}.$
It is simple to see that 
 \begin{align}
 B_0=&\,y_*^2\big(2\rho_0+2(a+ib)(\om_1+\om_0)-(1+a)a+b^2-i(1+2a)b\big),\\ 
 B_j=&\,y_*^2\big(2\rho_j+2(a+ib)(j+1)(\om_{j+1}+\om_j)\big),\qquad j\geq 1.
 \end{align}
To estimate $|B_k|$, $k\geq 3$, from Proposition~\ref{P:RHONOMNBDS} and Lemma~\ref{L:ABSONICIA}(vi),
\beqa\label{ineq:Bkbd}
|B_k|\leq&\,y_*^2\big(2\frac{C^{k-\al}}{k^2}+2(a+b)\frac{C^{k+1-\al}}{k+1}+2(a+b)\frac{C^{k-\al}(k+1)}{k^2}\big)\\
\leq&\,y_*^2\frac{C^{k+1-\al}}{k+1}\Big(\frac{2(k+1)}{Ck^2}+2a+2b+\frac{(2a+2b)(k+1)}{k^2C}\Big)
\leq(2.247+2.124b)y_*^2\frac{C^{k+1-\al}}{k+1}.
\eeqa
 \textbf{Step 4: Expanding $\tilde b$}\\
As before, we leave $\big( \frac{B}{\tilde w}\big)_j$ explicit for $j=0,1,2$ for the purposes of estimates, and will need to control
 \beqa
\big( \frac{B}{\tilde w}\big)_3=&\,B_0\bar w_3+B_1\bar w_2+B_2\bar w_1+B_3\bar w_0,
 \eeqa
 and, for $k\geq 4$,
 \beq
 \sum_{j=0}^kB_j\bar w_{k-j}=B_0\bar w_k + B_1\bar w_{k-1}+B_2\bar w_{k-2}+B_{k-1}\bar w_1+B_k\bar w_0 + \sum_{j=3}^{k-2}B_j\bar w_{k-j}. 
 \eeq
We bound, using Lemma~\ref{L:ABSONICIA}(vii)--(viii),
 \begin{align}\label{ineq:btilde3}
& \Big| \big( \frac{B}{\tilde w}\big)_3\Big|\leq(1.4569+(2.091)b+\frac{0.506y_*^2b^2}{4})y_*^2\frac{C^{4-\al}}{4},\\
 &\Big| \big( \frac{B}{\tilde w}\big)_4\Big|\leq(1.5638+(2.0285)b+\frac{0.506y_*^2b^2}{5})y_*^2\frac{C^{5-\al}}{5}.
 \end{align}
 For $k\geq 5$, we estimate, using \eqref{ineq:wbarbd} and \eqref{ineq:Bkbd},
 \footnotesize
 \beqas
{}& \big| \sum_{j=0}^kB_j\bar w_{k-j}\big|=|B_0||\bar w_k| + |B_1||\bar w_{k-1}|+|B_2||\bar w_{k-2}|+|B_k||\bar w_0| + |B_{k-1}||\bar w_1|+ \sum_{j=3}^{k-2}|B_j||\bar w_{k-j}|\\
&\leq y_*^2\frac{C^{k+1-\al}}{k+1}\Big(0.506\big(\frac{|B_0|}{k+1}+\frac{|B_1|(k+1)}{Ck^2}+\frac{|B_2|(k+1)}{C^2(k-1)^2}\big)+(2.247+2.124b)\big(|\bar w_0|+\frac{(k+1)|\bar w_1|}{Ck})\Big)\\
&+(2.247+2.124b)(0.506)y_*^4C^{k+2-2\al}{\sum_{j=3}^{k-2}\frac{1}{(j+1)(k-j+1)^2}}\\
&\leq y_*^2\frac{C^{k+1-\al}}{k+1}\bigg(0.506\big(\frac{|B_0|}{6}+\frac{6|B_1|}{25C}+\frac{6|B_2|}{16C^2}\big)+(2.247+2.124b)\Big(|\bar w_0|+\frac{6|\bar w_1|}{5C}+\frac{(0.506)(0.49)y_*^2}{C^{\al-1}}\Big)\bigg)\\
&\leq(1.8932+2.251 b+\frac{(0.506)b^2y_*^2}{6})y_*^2\frac{C^{k+1-\al}}{k+1}, 
 \eeqas
 \small
 where we have finally used Lemma~\ref{L:ABSONICIA}(ix).\\
 Combining these various estimates, we conclude, for $k\geq 3$,
 \beqa
 |\tilde b_k|\leq y_*^2\frac{C^{k+1-\al}}{k+1}\big(1.8932+2.251 b+\frac{(0.506)b^2y_*^2}{4}\big).
 \eeqa
 \end{proof}


We are now finally in a position to estimate the growth rate of the coefficients of the eigenfunction $\psi_\l$ near the sonic point.

\begin{proof}[Proof of Proposition~\ref{L:EFUNCTIONTAYLOR}(i)]
  Starting from the Taylor expansion~\eqref{E:PSITILDEEXPSONIC}, we substitute into \eqref{eq:efunctionsimpleform} (recall from Lemma~\ref{lemma:Frobenius} that, for the regular solution, we have $\tilde a_0\tpsi'(1)+\tilde b_0\tpsi(1)=0$) and group terms to obtain
  \beqa
  0=&\,\sum_{k=0}^\infty\Big((k+1)(k+2)\psi_{k+2}+\sum_{j=0}^{k+1}\tilde a_{k-j+1}(j+1)\psi_{j+1}+\sum_{j=0}^{k+1}\tilde b_{k-j+1}\psi_j\Big)(z-1)^k. 
  \eeqa
 Then we have, for $k\geq 0$,
  \beq\label{eq:psik+2exp}
  \psi_{k+2}=-\frac{1}{(k+1)(k+2)+(k+2)\tilde a_0}\Big(\sum_{j=0}^k\tilde a_{k-j+1}(j+1)\psi_{j+1}+\sum_{j=0}^{k+1}\tilde b_{k-j+1}\psi_j\Big). 
  \eeq
  Without loss of generality, suppose $\psi_0=1$ (this is achieved by a simple scaling as the ODE is linear; one also checks easily that if $\psi_0=0$, then the regular solution is identically zero). We then compute
  \beqa
  \psi_1=&\,-\frac{\tilde b_0}{\tilde a_0}=-\frac{2y_*+2(1-\la)(y_*^2\om_1+y_*^2\om_0)-(2-\la)(1-\la)y_*^2}{(w')_0+2(1-\la)y_*}\\
  =&\,-\frac{2y_*+2(1-\la)y_*(y_*-1)-(2-\la)(1-\la)y_*^2}{-2(y_*-1)+2(1-\la)y_*}, 
  \eeqa
where we use~\eqref{eq:atilde} and~\eqref{E:BTILDE}. 
  We check directly from Lemma~\ref{L:PSISONICIA}(i) that 
  \beqa
  |\psi_2|\leq 4y_*^2\frac{\mathfrak{C}^{3-\al}}{2^3},\quad
    |\psi_3|\leq 4y_*^2\frac{\mathfrak{C}^{4-\al}}{3^3},\quad
  |\psi_4|\leq 4y_*^2\frac{\mathfrak{C}^{5-\al}}{4^3} ,\quad |\psi_{5}|\leq 4\frac{y_*^2\mathfrak{C}^{6-\al}}{5^3},
  \eeqa 
  for $a\in[0,1]$, $b\in[0.2,8]$.
  For $k\geq 4$, we shorten notation by writing 
  \beq\label{E:AFRAKBFRAK}
  \mathfrak{a}=1.5819+0.692b,\qquad \mathfrak{b} = 1.8932+2.251 b+\frac{(0.506)b^2y_*^2}{4}.
  \eeq For computational convenience, we assume for an induction that, for $j= 2,...,k+1$, we have 
  \beq\label{ineq:psikinductive}
  |\psi_j|\leq Py_*^2\frac{\mathfrak{C}^{j+1-\al}}{j^3},
  \eeq
  where $\mathfrak{C}=8+\frac{3b}{2}$ as in the statement of Proposition~\ref{L:EFUNCTIONTAYLOR} and $P=4$.
  We now separate terms in \eqref{eq:psik+2exp} and apply estimates~\eqref{ineq:akbd}--\eqref{ineq:bkbd} for $\tilde a_k$ and $\tilde b_k$, along with the inductive hypothesis \eqref{ineq:psikinductive}
  \beqa\label{ineq:psik+2init}
  (k+2)&|(k+1)+\tilde a_0||\psi_{k+2}|\leq \Big(|\psi_{1}||\tilde a_{k+1}|+(k+1)|\psi_{k+1}||\tilde a_{1}|+k|\psi_{k}||\tilde a_{2}|\\
  &\hspace{30mm}+|\psi_0||\tilde b_{k+1}|+|\psi_1||\tilde b_{k}|+ |\psi_{k+1}||\tilde b_0|+ |\psi_{k}||\tilde b_1|+|\psi_{k-1}||\tilde b_{2}|\\
  &\hspace{30mm}+\sum_{j=1}^{k-2}|\tilde a_{k-j+1}|(j+1)|\psi_{j+1}|+\sum_{j=2}^{k-2}|\tilde b_{k-j+1}||\psi_j|\Big) \\
  \leq&\Big(|\psi_1|\mathfrak{a} y_*^2\frac{C^{k+2-\al}}{k+2} + P|\tilde a_1|y_*^2\frac{\mathfrak{C}^{k+2-\al}}{(k+1)^2}+P|\tilde a_2|y_*^2\frac{\mathfrak{C}^{k+1-\al}}{k^2}  \\
  &+ \mathfrak{b} y_*^2\frac{C^{k+2-\al}}{k+2}+|\psi_1|\mathfrak{b} y_*^2\frac{C^{k+1-\al}}{k+1}+ P|\tilde b_0|y_*^2\frac{\mathfrak{C}^{k+2-\al}}{(k+1)^3}+ P|\tilde b_1|y_*^2\frac{\mathfrak{C}^{k+1-\al}}{k^3}+P|\tilde b_2|y_*^2\frac{\mathfrak{C}^{k-\al}}{(k-1)^3}  \\
  &+P\mathfrak{a} y_*^4\sum_{j=1}^{k-2}\frac{C^{k-j+2-\al}}{k-j+2}\frac{\mathfrak{C}^{j+2-\al}}{(j+1)^2}+P\mathfrak{b} y_*^4\sum_{j=2}^{k-2}\frac{C^{k-j+2-\al}}{k-j+2}\frac{\mathfrak{C}^{j+1-\al}}{(j+1)^3}\Big) \\
  \leq&\frac{y_*^2\mathfrak{C}^{k+3-\al}}{(k+2)}\bigg(\big(\frac{C}{\mathfrak{C}}\big)^{k+2-\al}\Big(\frac{|\psi_1|\mathfrak{a}+\mathfrak{b}}{\mathfrak{C}} +\frac{|\psi_1|\mathfrak{b} (k+2)}{(k+1)C\mathfrak{C}}\Big) \\
  &+P|\tilde a_1|\frac{(k+2)}{\mathfrak{C}(k+1)^2}+P|\tilde a_2|\frac{k+2}{\mathfrak{C}^2k^2}+ P|\tilde b_0|\frac{(k+2)}{\mathfrak{C}(k+1)^3}+ P|\tilde b_1|\frac{(k+2)}{\mathfrak{C}^2k^3}+P|\tilde b_2|\frac{(k+2)}{\mathfrak{C}^3(k-1)^3} \\
  &+(k+2)P\mathfrak{a} y_*^2 \mathfrak{C}^{1-\al}{\sum_{j=1}^{k-2}\frac{\big(\frac{C}{\mathfrak{C}}\big)^{k-j+2-\al}}{(k-j+2)(j+1)^2}}+(k+2)\frac{P\mathfrak{b} y_*^2}{\mathfrak{C}^\al}{\sum_{j=2}^{k-2}\frac{\big(\frac{C}{\mathfrak{C}}\big)^{k-j+2-\al}}{(k-j+2)(j+1)^3}}\bigg).
  \eeqa
   Here we recall $C=7.2$, $\al=1.98$. 
    To get a bound for $k\geq 4$, we first make the bound, from Lemma~\ref{L:PSISONICIA}(ii),
  \beqa
  \frac{(k+2)^2}{(k+2)|k+1+\tilde a_0|}\leq 1.41,  
  \eeqa
for $a\in[0,1]$, $b\in[0.2,8]$. Then, from Lemma~\ref{L:PSISONICIA}(iii), for $k\geq 4$, also
  \beqas
  \sum_{j=1}^{k-2}\frac{\big(\frac{C}{\mathfrak{C}}\big)^{k-j+2-\al}}{(k-j+2)(j+1)^2}\leq&\,\frac{0.465}{(k+2)(1+b)}, \qquad\sum_{j=2}^{k-2}\frac{\big(\frac{C}{\mathfrak{C}}\big)^{k-j+2-\al}}{(k-j+2)(j+1)^3}\leq&\,\frac{0.08}{(k+2)(1+b)}. 
  \eeqas
Substituting these bounds into \eqref{ineq:psik+2init}, we make the estimate, using now $k\geq 4$,
  \beqa
  |\psi_{k+2}|\leq&\, y_*^2\frac{\mathfrak{C}^{k+3-\al}}{(k+2)^3}(1.41)\bigg(\big(\frac{C}{\mathfrak{C}}\big)^{6-\al}\Big(\frac{|\psi_1|\mathfrak{a}+\mathfrak{b}}{\mathfrak{C}} +\frac{6|\psi_1|\mathfrak{b} }{5C\mathfrak{C}}\Big)+P|\tilde a_1|\frac{6}{25\mathfrak{C}}+P|\tilde a_2|\frac{6}{16\mathfrak{C}^2}\\
  &+ P|\tilde b_0|\frac{6}{5^3\mathfrak{C}}+ P|\tilde b_1|\frac{6}{4^3\mathfrak{C}^2}+P|\tilde b_2|\frac{6}{3^3\mathfrak{C}^3} +0.465\frac{P\mathfrak{a} y_*^2}{1+b} \mathfrak{C}^{1-\al}+0.08\frac{P\mathfrak{b} y_*^2}{(1+b)\mathfrak{C}^\al}\bigg)\\
  \leq&\,Py_*^2\frac{\mathfrak{C}^{k+3-\al}}{(k+2)^3}
  \eeqa
as required, where we have used Lemma~\ref{L:PSISONICIA}(iv).
\end{proof}


\subsubsection{Expanding eigenfunctions near the origin}
The goal of this subsection is to establish the coefficient expansions and growth rates for the eigenfunction $\psi_\l$ around the origin, as claimed in Proposition~\ref{L:EFUNCTIONTAYLOR}(ii). We continue to work in the $y$ variable and observe from Lemma~\ref{lemma:Frobenius} and~\eqref{E:PSITILDEDEF} that $\tpsi(y)$ is an odd, analytic function. From~\eqref{eq:psila}, we write the eigenfunction ODE in the form
\beq\label{eq:psilatilde}
\tpsi_\l''(y)+\frac1y\Big(-2+\frac{X}{\wLP}\Big)\tpsi_\l'(y)+\frac{Y}{\wLP}\tpsi_\l(y)=0,
\eeq
where
\begin{align}
 X(y) =&\,2y^2\rhoLP\omLP-4y^2\omLP^2-2y^3\omLP\omLP'+2(a+ib)y^2\omLP,\label{def:X}\\
 Y(y)=&\,2\rhoLP+2(a+ib)(\omLP+y\omLP')-a(1+a)+b^2-i(2a+1)b. \label{def:Y}
\end{align}
We first estimate the coefficients of $\frac{X}{\wLP}$ and $\frac{Y}{\wLP}$ with the following lemma.

\begin{lemma}\label{L:D14}
  Let $1-\la=a+ib$, $a\in[0,1]$, $b\geq 0$. The coefficients $\frac{X}{\wLP}$ and $\frac{Y}{\wLP}$ expand as
  \beq
\frac{X}{\wLP}=\sum_{k=0}^\infty \big(\frac{X}{\wLP}\big)_{2k}y^{2k},\qquad \frac{Y}{\wLP}=\sum_{k=0}^\infty \big(\frac{Y}{\wLP}\big)_{2k}y^{2k}
  \eeq
  where $(\frac{X}{\wLP})_0=0$, we have explicit formulae for $k=0,1,2$, and the estimates
 \begin{align}
 \big|\big(\frac{X}{\wLP}\big)_{2k}\big|\leq (4.156 + 0.638 b )\frac{C_0^{2(k-1)-\al_0}}{2(k-1)} \quad k\geq 3,\label{ineq:Xwbds}\\
 \big|\big(\frac{Y}{\wLP}\big)_{2k}\big|\leq \big(3.41 + 3.239 b+ 0.1266 b^2 \big)\frac{C_0^{2k-\al_0}}{2k}\quad k\geq 2,\label{ineq:Ywbds}
\end{align}
 where $C_0=2$, $\al_0=1.95$. 
  \end{lemma}


\begin{proof}
 \textbf{Step 1: Expansion for $X$}\\
 First, from \eqref{def:X}, the coefficients $X_{2k}$ of $X$ satisfy
 \beqa\label{eq:X2kidentity}
 X_{2k}
 =&\,2(\tr\tom)_{2(k-1)}-4(\tom^2)_{2(k-1)}+2(a+ib)\tom_{2(k-1)}-2\sum_{j=0}^{k-1}2j\tom_{2j}\tom_{2(k-1)-2j},
 \eeqa
which immediately yields formulae when $k=0,1,2$, including $X_0=0$ and
\beqa\label{eq:X2X4}
 X_{2}=&\,\frac23\tr_0-\frac49+\frac23(a+ib), \quad X_4=\frac{2(\tr_0-\frac13)}{45}\big(-3\tr_0-4+2(a+ib)).
\eeqa
To bound $|X_{2k}|$ for $k\geq 3$, we use~\eqref{E:ORIGINQUADRATICCOEFFS} and Lemma~\ref{L:ORIGINCOEFFS} to estimate
\begin{align}
|X_{2k}|\leq&\,(6D_0+2(1+b))\frac{C_0^{2(k-1)-\al_0}}{(2(k-1))^2}+ 4(k-1)|\tom_{2(k-1)}|\tom_0+2\frac{C_0^{2(k-1)-\al_0}}{C_0^{\al_0}}\sum_{j=1}^{k-2}\frac{1}{2j(2(k-1)-2j)^2}\notag\\
\leq&\,\frac{C_0^{2(k-1)-\al_0}}{2(k-1)}\Big(\frac{3D_0+(1+b)}{2}+\frac23+\frac{2(0.57)}{C_0^{\al_0}}\Big),\label{E:X2KBD}
\end{align}
 where we used~\eqref{ineq:combextra2} to bound the sum.\\
\textbf{Step 2: Expansion for $\frac{X}{\wLP}$}\\
 Therefore, to expand $\frac{X}{\wLP}$, we use
 \beq
 \big(\frac{X}{\wLP}\big)_{2k}=\sum_{j=0}^kX_{2j}\check{w}_{2k-2j}.
 \eeq
 As usual, we work with the closed formulae for the low order coefficients $(\frac{X}{w})_{2k}$, combining~\eqref{eq:worigincoeffs} and~\eqref{eq:X2X4} which yields
\beqa\label{X/w012}
 \big|\big(\frac{X}{\wLP}\big)_{0}\big|=&\,0,\qquad \big|\big(\frac{X}{\wLP}\big)_{2}\big|=|\frac23\rho_0-\frac49+\frac23(a+ib)|,\\
 \big|\big(\frac{X}{\wLP}\big)_{4}\big|=&\,\big|\frac19\big(\frac23\rho_0-\frac49+\frac23(a+ib)\big)+\frac{2(\rho_0-\frac13)}{45}\big(-3\rho_0-4+2(a+ib))\big|.
 \eeqa
We apply Lemma~\ref{L:XYORIGINIA}(i) to see
 \beqa\label{ineq:Xw6bound}
  \big|\big(\frac{X}{\wLP}\big)_{6}\big|  \leq(3.677+0.523 b)\frac{C_0^{4-\al_0}}{4}.
 \eeqa
 Now using $|\check{w}_{2k}|\leq1.35\frac{C_0^{2(k-1)-\al_0}}{(2(k-1))^2}$ from \eqref{eq:whatwcheck2kbds} along with \eqref{eq:worigincoeffs} and~\eqref{eq:X2X4}--\eqref{E:X2KBD}, we have, for $k\geq 4$,
 \footnotesize
 \beqa
  \big|\big(\frac{X}{\wLP}\big)_{2k}\big|\leq&\,\Big|X_2\check{w}_{2(k-1)}+X_{2(k-1)}\check{w}_2+X_{2k}\check{w}_0+\sum_{j=2}^{k-2}X_{2j}\check{w}_{2k-2j}\Big|\\
  \leq&\,|\frac23\tr_0-\frac49+\frac23(a+ib)|(1.35)\frac{C_0^{2(k-2)-\al_0}}{(2(k-2))^2}+\frac19\frac{C_0^{2(k-2)-\al_0}}{2(k-2)}\Big(\frac{3D_0+(1+b)}{2}+\frac23+\frac{2(0.57)}{C_0^{\al_0}}\Big)\\
  &+\frac{C_0^{2(k-1)-\al_0}}{2(k-1)}\Big(\frac{3D_0+(1+b)}{2}+\frac23+\frac{2(0.57)}{C_0^{\al_0}}\Big)\\
  &+(1.35)\Big(\frac{3D_0+(1+b)}{2}+\frac23+\frac{2(0.57)}{C_0^{\al_0}}\Big)\sum_{j=2}^{k-2}\frac{C_0^{2(j-1)-\al_0}}{2(j-1)}\frac{C_0^{2(k-j-1)-\al_0}}{(2(k-j-1))^2}\\
  \leq&\,\frac{C_0^{2(k-1)-\al_0}}{2(k-1)}\bigg((1.35)|\frac23\tr_0-\frac49+\frac23(a+ib)|\frac{2(k-1)}{C_0^2(2(k-2))^2}\\
  &\hspace{6mm}+\Big(\frac{3D_0+(1+b)}{2}+\frac23+\frac{2(0.57)}{C_0^{\al_0}}\Big)\Big(\frac{(k-1)}{9(k-2)C_0^2}+1+(1.35)\frac{3}{4}\frac{1}{C_0^{2+\al_0}}\Big)\bigg),
 \eeqa
 \small
 where we have used~\eqref{ineq:comb10} to bound the sum.

This gives us the bound, from Lemma~\ref{L:XYORIGINIA}(ii), for $k\geq 4$,
\beq
 \big|\big(\frac{X}{\wLP}\big)_{2k}\big|\leq (4.156 + 0.638 b )\frac{C_0^{2(k-1)-\al_0}}{2(k-1)}. 
\eeq
\textbf{Step 3: Expansion for $Y$}\\
Moving on to $Y$, we recall from \eqref{def:Y} that this is 
$$Y=2\rhoLP+2(a+ib)(\omLP+y\omLP')-(a+1+ib)(a+ib).$$ 
Expanding $Y$ for low order coefficients, we substitute the expressions \eqref{eq:tr2}--\eqref{eq:tom2} and find we have
\beqa\label{eq:Y0Y2}
Y_0=&\,2\tr_0+\frac{2(a+ib)}{3}-a(1+a)+b^2 -i(1+2a)b, \quad Y_2
=-\frac23\tr_0(\tr_0-\frac13)+\frac{4(a+ib)}{15}(\tr_0-\frac13), \\
Y_{2k}=&\,2\tr_{2k}+2(a+ib)(2k+1)\tom_{2k}, \ \ k\ge2.
\eeqa
So, from Lemma~\ref{L:ORIGINCOEFFS}, for $k\geq 2$, we have
\beqa\label{ineq:Y2kbd}
|Y_{2k}|\leq \frac{C_0^{2k-\al_0}}{(2k)^2}\big(2+2(2k+1)(a+b)\big)\leq \big(3+\frac{5b}{2}\big)\frac{C_0^{2k-\al_0}}{2k}. 
\eeqa
\textbf{Step 4: Expanding $\frac{Y}{\wLP}$}\\
As usual, we leave $(\frac{Y}{\wLP})_0$ and $(\frac{Y}{\wLP})_2$ as explicit functions of $\rhoLP(0)$, giving
 \beqa\label{Y/w012}
\big|\big(\frac{Y}{\wLP}\big)_0\big|=&\,|Y_0|=\big|2\rhoLP_0+\frac{2(a+ib)}{3}-(1+a+ib)(a+ib)\big|,\\
\big|\big(\frac{Y}{\wLP}\big)_2\big|
=&\,\Big|\frac19\Big(2\rhoLP_0+\frac{2(a+ib)}{3}-(1+a+ib)(a+ib)\Big)+(\rhoLP_0-\frac13)\big(-\frac23\rhoLP_0+\frac{4(a+ib)}{15}\big)\Big|.
\eeqa
Estimating  $(\frac{Y}{\wLP})_4$ and $(\frac{Y}{\wLP})_6$, we apply Lemma~\ref{L:XYORIGINIA}(iii) to get the bounds
\begin{align}\label{ineq:Yw4bound}
\big|\big(\frac{Y}{\wLP}\big)_4\big|
\leq&\,\big(3.0143 + 2.548 b + 0.0265 b^2)\frac{C_0^{4-\al_0}}{4},\\
\big|\big(\frac{Y}{\wLP}\big)_6\big|\leq&\,\big(3.341 + 3.239 b+ 0.1266 b^2 \big)\frac{C_0^{6-\al_0}}{6}.
\end{align}
Then, for $k\geq 4$, we use the representations \eqref{eq:worigincoeffs} and \eqref{eq:Y0Y2} as well as the bounds \eqref{eq:whatwcheck2kbds} and \eqref{ineq:Y2kbd} to estimate
\footnotesize
\beqa
{}&\big|\big(\frac{Y}{\wLP}\big)_{2k}\big|=\Big|Y_0\check{w}_{2k}+Y_2\check{w}_{2(k-1)}+Y_{2(k-1)}\check{w}_2+Y_{2k}+\sum_{j=2}^{k-2}Y_{2j}\check{w}_{2(k-j)}\Big|\\
&\leq\big|2\tr_0+\frac{2(a+ib)}{3}-(1+a+ib)(a+ib)\big|(1.35)\frac{C_0^{2(k-1)-\al_0}}{(2(k-1))^2}\\
&+\Big|(\tr_0-\frac13)\big(-\frac23\tr_0+\frac{4(a+ib)}{15}\big)\Big|(1.35)\frac{C_0^{2(k-2)-\al_0}}{(2(k-2))^2}\\
&+\frac19\big(3+\frac{5b}{2}\big)\frac{C_0^{2(k-1)-\al_0}}{2(k-1)}+\big(3+\frac{5b}{2}\big)\frac{C_0^{2k-\al_0}}{2k}+(1.35)\big(3+\frac{5b}{2}\big)\frac{C_0^{2k-\al_0}}{C_0^{2+\al_0}}\sum_{j=2}^{k-2}\frac{1}{2j(2(k-1-j))^2}\\
&\leq\frac{C_0^{2k-\al_0}}{2k}\Big(\big|2\tr_0+\frac{2(a+ib)}{3}-(1+a+ib)(a+ib)\big|(1.35)\frac{1}{C_0^2}\frac{2}{9} \\
&+\Big|(\tr_0-\frac13)\big(-\frac23\tr_0+\frac{4(a+ib)}{15}\big)\Big|(1.35)\frac{1}{C_0^{4}}\frac{1}{2}+\big(3+\frac{5b}{2}\big)\big(1+\frac{4}{27C_0^2}\big)+ (1.35)\frac{7}{12C_0^{2+\al_0}}\big(3+\frac{5b}{2}\big)\Big),
\eeqa
\small
where we have used~\eqref{ineq:combextra3} to bound the sum.
This yields the bound, by Lemma~\ref{L:XYORIGINIA}(iv), 
\beq
\big|\big(\frac{Y}{\wLP}\big)_{2k}\big|\leq\big(3.41 + 2.91 b + 0.075 b^2\big)\frac{C_0^{2k-\al_0}}{2k}. 
\eeq
\end{proof}

  
  We are now finally in a position to estimate the growth rate of the coefficients of the eigenfunction $\psi_\l$ near the origin.

\begin{proof}[Proof of Proposition~\ref{L:EFUNCTIONTAYLOR}(ii)]
Starting from the Taylor expansion~\eqref{E:PSITILDEEXPORIGIN} for $\tpsi_\l$, we substitute this into~\eqref{eq:psilatilde} and group terms (recalling $\tilde\psi_1=0$, $(\frac{X}{\wLP})_0=0$) to yield, for $\ell\geq 1$,
\beq
2\ell(2\ell+3)\tilde\psi_{2\ell+3}=-\sum_{j=1}^\ell(2j+1)\tilde\psi_{2j+1}\big(\frac{X}{\wLP}\big)_{2\ell+2-2j}-\sum_{j=1}^\ell\tilde\psi_{2j +1}\big(\frac{Y}{\wLP}\big)_{2\ell-2j.} 
\eeq
We will assume without loss of generality that $\tilde\psi_3=1$ as this can be ensured by scaling. \\
For $\ell=1,2,3$, we easily obtain $\tilde\psi_{2\ell+3}$ from this identity.
For $\ell=4,5$, we employ Lemma~\ref{L:PSIORIGINIA}(i)  to see
\beqa
|\tilde\psi_{2\ell+1}|\leq Q\frac{\mathfrak{C}_0^{2\ell-\al_0}}{(2\ell+1)^3},
\eeqa
where $Q=7$, $\mathfrak{C}_0=2+\frac{b}{4}$.

We assume for an induction that, for $\ell\geq 4$, we have the bound
\beq
|\tilde\psi_{2\ell+1}|\leq Q\frac{\mathfrak{C}_0^{2\ell-\al_0}}{(2\ell+1)^3}.
\eeq
Now, for $\ell\geq 5$, we have the general identity (where the sums in the last line are empty if $\ell=5$)
\beqa
\tilde\psi_{2\ell+3}=&-\frac{1}{2\ell(2\ell+3)}\bigg(3\big(\frac{X}{\wLP}\big)_{2\ell}+5\tilde\psi_5\big(\frac{X}{\wLP}\big)_{2\ell-2}+7\tilde\psi_7\big(\frac{X}{\wLP}\big)_{2\ell-4}+(2\ell-1)\tilde\psi_{2\ell-1}\big(\frac{X}{\wLP}\big)_4\\
&+(2\ell+1)\tilde\psi_{2\ell+1}\big(\frac{X}{\wLP}\big)_2+(\frac{Y}{\wLP}\big)_{2\ell-2}+\tilde\psi_5(\frac{Y}{\wLP}\big)_{2\ell-4}+\tilde\psi_7(\frac{Y}{\wLP}\big)_{2\ell-6}+\tilde\psi_{2\ell-1}(\frac{Y}{\wLP}\big)_2\\
&+\tilde\psi_{2\ell+1}(\frac{Y}{\wLP}\big)_0+\sum_{j=4}^{\ell-2}(2j+1)\tilde\psi_{2j+1}\big(\frac{X}{\wLP}\big)_{2\ell+2-2j}+\sum_{j=4}^{\ell-2}\tilde\psi_{2j+1}\big(\frac{Y}{\wLP}\big)_{2\ell-2j}\bigg). 
\eeqa
To compress notation, we set 
\beq\label{E:AFRAK0BFRAK0}
\mathfrak{a}_0=4.156 + 0.638 b,\qquad \mathfrak{b}_0 = 3.41 + 3.239 b+ 0.1266 b^2.
\eeq
 Using the estimates  \eqref{ineq:Xwbds}--\eqref{ineq:Ywbds} and the inductive hypothesis, we obtain
\beqa
|\tilde\psi_{2\ell+3}|\leq\frac{1}{2\ell(2\ell+3)}\bigg(&\mathfrak{a}_0 \frac{C_0^{2\ell-\al_0}}{2\ell}\Big(\frac{3\ell}{C_0^2(\ell-1)}+5|\tilde\psi_5|\frac{2\ell}{2(\ell-2)C_0^4}+7|\tilde\psi_7|\frac{2\ell}{2(\ell-3)C_0^6}\Big)\\
&+Q\frac{\mathfrak{C}_0^{2\ell-\al_0}}{(2\ell+1)^2}\Big(\frac{(2\ell+1)^2}{\mathfrak{C}_0^2(2\ell-1)^2}\big|\big(\frac{X}{\wLP}\big)_4\big|+\big|\big(\frac{X}{\wLP}\big)_2\big|\Big)\\
&+\mathfrak{b}_0\frac{C_0^{2\ell-\al_0}}{2\ell}\Big(\frac{2\ell}{C_0^2(2(\ell-1))}+|\tilde\psi_5|\frac{2\ell}{2(\ell-2)C_0^4}+|\tilde\psi_7|\frac{2\ell}{2(\ell-3)C_0^6}\Big)\\
&+Q\frac{\mathfrak{C}_0^{2\ell-\al_0}}{(2\ell+1)^2}\Big(\frac{(2\ell+1)^2}{\mathfrak{C}_0^2(2\ell-1)^3}\big|\big(\frac{Y}{\wLP}\big)_2\big|+\frac{1}{2\ell+1}\big|\big(\frac{Y}{\wLP}\big)_0\big|\Big)\\
&+\mathfrak{a}_0 Q\sum_{j=4}^{\ell-2}\frac{\mathfrak{C}_0^{2j-\al_0}}{(2j+1)^2}\frac{C_0^{2\ell-2j-\al_0}}{2(\ell-j)}+\mathfrak{b}_0 Q\sum_{j=4}^{\ell-2}\frac{\mathfrak{C}_0^{2j-\al_0}}{(2j+1)^3}\frac{C_0^{2\ell-2j-\al_0}}{2(\ell-j)}\bigg).
\eeqa
This is then bounded by
\beqa
|\tilde\psi_{2\ell+3}|\leq&\,\frac{\mathfrak{C}_0^{2\ell+2-\al_0}}{(2\ell+3)^3}\frac{(2\ell+3)^2}{2\ell}\\
\times\bigg(& \frac{\big(\frac{C_0}{\mathfrak{C}_0}\big)^{2\ell-\al_0}}{2\ell\mathfrak{C}_0^2}\Big(\mathfrak{a}_0\big(\frac{3\ell}{C_0^2(\ell-1)}+5|\tilde\psi_5|\frac{2\ell}{2(\ell-2)C_0^4}+7|\tilde\psi_7|\frac{2\ell}{2(\ell-3)C_0^6}\big)\\
&\hspace{10mm}+\mathfrak{b}_0\big(\frac{2\ell}{C_0^2(2(\ell-1))}+|\tilde\psi_5|\frac{2\ell}{2(\ell-2)C_0^4}+|\tilde\psi_7|\frac{2\ell}{2(\ell-3)C_0^6}\big)\Big)\\
&+\frac{Q}{\mathfrak{C}_0^2}\Big(\frac{\big|\big(\frac{X}{w}\big)_4\big|}{\mathfrak{C}_0^2(2\ell-1)^2}+\frac{\big|\big(\frac{Y}{w}\big)_2\big|}{\mathfrak{C}_0^2(2\ell-1)^3}+\frac{\big|\big(\frac{X}{w}\big)_2\big|}{(2\ell+1)^2}+\frac{\big|\big(\frac{Y}{w}\big)_0\big|}{(2\ell+1)^3}\Big)\\
&+\mathfrak{a}_0 \frac{Q}{\mathfrak{C}_0^{2+\al_0}}\sum_{j=4}^{\ell-2}\frac{\big(\frac{C_0}{\mathfrak{C}_0}\big)^{2\ell-2j-\al_0}}{(2j+1)^2(2(\ell-j))}+\mathfrak{b}_0 \frac{Q}{\mathfrak{C}_0^{2+\al_0}}\sum_{j=4}^{\ell-2}\frac{\big(\frac{C_0}{\mathfrak{C}_0}\big)^{2\ell-2j-\al_0}}{(2j+1)^32(\ell-j)}\bigg).
\eeqa
We apply Lemma~\ref{L:PSIORIGINIA}(ii) to estimate the summation terms and so obtain the bound
\beqa
|\tilde\psi_{2\ell+3}|\leq&\,\frac{\mathfrak{C}_0^{2\ell+2-\al_0}}{(2\ell+3)^3}\frac{(2\ell+3)^2}{2\ell}\\
\times\bigg(& \frac{\big(\frac{C_0}{\mathfrak{C}_0}\big)^{2\ell-\al_0}}{2\ell\mathfrak{C}_0^2}\Big(\mathfrak{a}_0\big(\frac{3\ell}{C_0^2(\ell-1)}+5|\tilde\psi_5|\frac{2\ell}{2(\ell-2)C_0^4}+7|\tilde\psi_7|\frac{2\ell}{2(\ell-3)C_0^6}\big)\\
&\hspace{10mm}+\mathfrak{b}_0\big(\frac{2\ell}{C_0^2(2(\ell-1))}+|\tilde\psi_5|\frac{2\ell}{2(\ell-2)C_0^4}+|\tilde\psi_7|\frac{2\ell}{2(\ell-3)C_0^6}\big)\Big)\\
&+\frac{Q}{\mathfrak{C}_0^2}\Big(\frac{\big|\big(\frac{X}{\wLP}\big)_4\big|}{\mathfrak{C}_0^2(2\ell-1)^2}+\frac{\big|\big(\frac{Y}{\wLP}\big)_2\big|}{\mathfrak{C}_0^2(2\ell-1)^3}+\frac{\big|\big(\frac{X}{\wLP}\big)_2\big|}{(2\ell+1)^2}+\frac{\big|\big(\frac{Y}{\wLP}\big)_0\big|}{(2\ell+1)^3}\Big)\\
&+\mathfrak{a}_0 \frac{Q}{\mathfrak{C}_0^{2+\al_0}}\frac{0.093}{2\ell(1+b)}+\mathfrak{b}_0 \frac{Q}{\mathfrak{C}_0^{2+\al_0}}\frac{0.01}{2\ell(1+b)}\bigg)
\eeqa
By employing Lemma~\ref{L:PSIORIGINIA}(iii), we deduce
\beq
|\tilde\psi_{2\ell+3}|\leq 7\frac{\mathfrak{C}_0^{2(\ell+1)-\al_0}}{(2\ell+3)^3},
\eeq
as required.
\end{proof}


\subsection{Auxiliary lemmas}

 \begin{lemma}\label{L:WSONICIA}
 The following estimates hold, for $y_*\in[2.341,2.342]$, $C=7.2$, $\alpha=1.98$.
 \begin{itemize}
 \item[(i)] $|(w')_k|\leq 1.4166y_*^2\frac{C^{k+1-\al}}{k+1}$ for $k=2,3$;
 \item[(ii)] For $k\geq 4$, $y_*\in[2.34,2.342]$,
 \begin{align*}
{}&(k+1)\frac{k+1-\frac{2}{y_*}}{k^2}\leq \frac{5(5-\frac{2}{y_*})}{16}\leq \frac{13}{10},\\
&\big|\frac{1 - 5 y_* + 3 y_*^2}{4 y_*^2-6 y_*}\big|\frac{k+1}{k-1}\leq \frac{1 - 5 y_* + 3 y_*^2}{4 y_*^2-6 y_*}\frac{5}{3}\leq1.214,\\
&\big|\frac{5 - 5 y_* + y_*^2}{3 y_* - 2 y_*^2}\big|\frac{k+1}{(k-1)^2}\leq \frac{5 - 5 y_* + y_*^2}{3 y_* - 2 y_*^2}\frac{5}{9}\leq\frac{7}{40},\\
&\frac{(k-1)(k+1)}{(k-2)^2}(1-\frac{2}{y_*})\leq\frac{15}{4}(1-\frac{2}{y_*})\leq 0.548;
\end{align*}
\item[(iii)] $2\Big(\frac{1}{y_*}+\frac{13}{10C}+\frac{1.214}{C^2}+\frac{7}{40C^2}+\frac{0.548}{C^3}+\frac{5}{4}\big(\frac{0.506}{C^\al}+\frac{9}{10C^{\al+1}}+\frac{1}{C^{\al+2}}\big)+\frac{5D}{16C}+\frac{5D}{9C^2}\Big)
 \leq 1.4166$;
 \item[(iv)] $ |\bar w_1|\leq \frac{y_*^2}{4}\frac{C^{2-\al}}{4}$ and $ |\bar w_2| \leq0.0458y_*^2\frac{C^{3-\al}}{9}$;
 \item[(v)] $\frac{1.4166}{\tilde w_0^2}+\frac{(1.4166)\cdot16|\bar w_1|}{9C|\tilde w_0|}+ (1.242)\cdot(0.506)\frac{1.4166y_*^2}{C^{\al-1}|\tilde w_0|}\leq 0.506$.
 \end{itemize}
 \end{lemma}
 
 \begin{proof}
 These estimates are proved using the function \verb!w_Sonic_Constraint!  in the attached code.
 \end{proof}
 

\begin{lemma}\label{L:WORIGINIA}
The following estimates hold for $\tilde\rho_0\in[0.83,0.84]$, $C_0=2$, $\al_0=1.95$.
\begin{itemize}
\item[(i)] $|\what_4|\leq \frac{C_0^{2-\al_0}}{4}$;
\item[(ii)] $\frac23+\frac{41}{36C_0^{\al_0}}<1$;
\item[(iii)] $|\check{w}_4|\leq 0.11\frac{C_0^{2-\al_0}}{4},\qquad |\check{w}_6|\leq 1.02\frac{C_0^{4-\al_0}}{16}$.
\item[(iv)] For $k\geq 4$, $\frac{2.35}{9}\frac{(k-1)^2}{C_0^2(k-2)^2}+1+\frac{9\cdot(1.35)}{4C_0^{2+\al_0}}\leq 1.35$.
\end{itemize}
\end{lemma}

 \begin{proof}
 These estimates are proved using the function \verb!w_Origin_Constraint!  in the attached code.
 \end{proof}


The following lemma contains the key bounds obtained via Interval Arithmetic used in the proof of Lemma~\ref{L:D13}.


\begin{lemma}\label{L:ABSONICIA}
 The following estimates hold, for $y_*\in[2.341,2.342]$, $C=7.2$, $\alpha=1.98$.
 \begin{itemize}
 \item[(i)] $1.4166+(4D+2(1+b))\big(\frac{4}{9C}+\frac{1}{C^2}\big)\leq 1.913+0.1621b$ for $k=2,3$;
 \item[(ii)] $ \Big| \big( \frac{A}{\tilde w}\big)_3\Big| \leq (1.057+(0.6626)b)y_*^2\frac{C^{4-\al}}{4}$ and $\Big|\big(\frac{A}{\tilde w}\big)_4\Big|\leq (1.301 + (0.692)b)y_*^2 \frac{C^{5-\al}}{5}$;
\item[(iii)] $(0.506)\Big(\frac{|A_0|}{6}+\frac{6|A_1|}{25C}+\frac{6|A_2|}{16C^2}\Big)+(1.913+0.1621b)\Big(|\bar w_1|\frac{6}{5C}+|\bar w_0|+(0.506)(0.49)\frac{y_*^2}{C^{\al-1}}\Big)\leq 1.5771+(0.6091)b$;
 \item[(iv)] $2\frac{5}{y_*^2(7.2)^{5-\al}}+(1.5771+0.692b)\leq 1.5819+0.692b$;
 \item[(v)] $ |\tilde a_3| \leq (1.0841+(0.6626)b)y_*^2\frac{C^{4-\al}}{4}$;
 \item[(vi)] For $k\geq 3$, $\frac{2(k+1)}{Ck^2}+2a+2b+\frac{(2a+2b)(k+1)}{k^2C}\leq 2.247+2.124b$;
 \item[(vii)] $ \Big| \big( \frac{B}{\tilde w}\big)_3\Big|\leq(1.4569+(2.091)b+\frac{0.506y_*^2b^2}{4})y_*^2\frac{C^{4-\al}}{4}$;
 \item[(viii)] $ \Big| \big( \frac{B}{\tilde w}\big)_4\Big|\leq(1.5638+(2.0285)b+\frac{0.506y_*^2b^2}{5})y_*^2\frac{C^{5-\al}}{5}$;
\item[(ix)] $0.506\big(\frac{|B_0|}{6}+\frac{6|B_1|}{25C}+\frac{6|B_2|}{16C^2}\big)+(2.247+2.124b)\Big(|\bar w_0|+\frac{6|\bar w_1|}{5C}+\frac{(0.506)(0.49)y_*^2}{C^{\al-1}}\Big)
\leq 1.8932+2.251 b+\frac{(0.506)b^2y_*^2}{6}$.
 \end{itemize}
 \end{lemma}
 
 \begin{proof}
These estimates are proved using the function \verb!efun_Coeff_Sonic_Constraint! in the attached code.
 \end{proof}


The following lemma contains the key bounds obtained via Interval Arithmetic used in the proof of Proposition~\ref{L:EFUNCTIONTAYLOR}(i), which can be found
right after the proof of Lemma~\ref{L:D13}.


 \begin{lemma}\label{L:PSISONICIA}
 The following estimates hold, for $y_*\in[2.341,2.342]$, $C=7.2$, $\alpha=1.98$, $\mathfrak{C}=8+\frac{3b}{2}$.
  Let $\psi_0 = 1$, $\psi_1 = -\frac{2y_*+2(1-\la)y_*(y_*-1)-(2-\la)(1-\la)y_*^2}{-2(y_*-1)+2(1-\la)y_*}$, and $\psi_j$, $j\ge2$, be given
 recursively by the relation~\eqref{eq:psik+2exp}. Then 
 \begin{itemize}
 \item[(i)] $  |\psi_2|\leq 4y_*^2\frac{\mathfrak{C}^{3-\al}}{2^3}$, \  
   $ |\psi_3|\leq 4y_*^2\frac{\mathfrak{C}^{4-\al}}{3^3}$, \
 $ |\psi_4|\leq 4y_*^2\frac{\mathfrak{C}^{5-\al}}{4^3}$,  \ $|\psi_{5}|\leq 4\frac{y_*^2\mathfrak{C}^{6-\al}}{5^3}$;
 \item[(ii)] $k\geq 4$, $ \frac{(k+2)^2}{(k+2)|k+1+\tilde a_0|}\leq 1.41$;
\item[(iii)] For $k\geq 4$, $\sum_{j=1}^{k-2}\frac{\big(\frac{C}{\mathfrak{C}}\big)^{k-j+2-\al}}{(k-j+2)(j+1)^2}\leq\frac{0.465}{(k+2)(1+b)}$ and $\sum_{j=2}^{k-2}\frac{\big(\frac{C}{\mathfrak{C}}\big)^{k-j+2-\al}}{(k-j+2)(j+1)^3}\leq\frac{0.08}{(k+2)(1+b)}$;
 \item[(iv)] For $P=4$, \beqas
 &\big(\frac{C}{\mathfrak{C}}\big)^{6-\al}\Big(\frac{|\psi_1|\mathfrak{a}+\mathfrak{b}}{\mathfrak{C}} +\frac{6|\psi_1|\mathfrak{b} }{5C\mathfrak{C}}\Big)+P|\tilde a_1|\frac{6}{25\mathfrak{C}}+P|\tilde a_2|\frac{6}{16\mathfrak{C}^2}+ P|\tilde b_0|\frac{6}{5^3\mathfrak{C}}+ P|\tilde b_1|\frac{6}{4^3\mathfrak{C}^2} \\
  &+P|\tilde b_2|\frac{6}{3^3\mathfrak{C}^3}+0.465\frac{P\mathfrak{a} y_*^2}{1+b} \mathfrak{C}^{1-\al}+0.08\frac{P\mathfrak{b} y_*^2}{(1+b)\mathfrak{C}^\al} \leq \frac{4}{1.41},
 \eeqas
 where we recall $\mathfrak{a}$ and $\mathfrak{b}$ from~\eqref{E:AFRAKBFRAK}.
 \end{itemize}
 \end{lemma}
 
 \begin{proof}
  The estimates in (i), (ii), and (iv) are proved using the function \verb!efun_Sonic_Constraint! in the attached code. (iii) requires more work, and so we give a more complete proof here. First, we note the estimates~\eqref{ineq:combapp1}--\eqref{ineq:combapp2} which give, for $k\geq 16$, 
 \beqas
 \sum_{j=1}^{\lfloor\frac{k}{2}\rfloor}\frac{k+2}{(k-j+2)(j+1)^2}\leq 0.662,\\
 \sum_{\lceil\frac{k}{2}\rceil}^{k-2}\frac{k+2}{(k-j+2)(j+1)^2}\leq 0.135.
 \eeqas
 Moreover, for $b\in[0.2,8]$, we verify $(1+b)(C/\mathfrak{C})^2\leq 1.33$ and $(1+b)(C/\mathfrak{C})^8\leq 0.385$. Thus, splitting the summation in the first estimate and using the simple estimates $(C/\mathfrak{C})^{k-j+2-\al}\leq (C/\mathfrak{C})^2$ for $j\leq k-2$ and $(C/\mathfrak{C})^{k-j+2-\al}\leq (C/\mathfrak{C})^8$ for $j\leq \lfloor\frac{k}{2}\rfloor$ and $k\geq 16$, we conclude the first estimate in (iii) when $k\geq 16$ as $(0.662)\cdot(0.385) + (0.135)\cdot(1.33)\leq 0.44$. We then exhaust over $k=4,\ldots,15$ in the function \verb!efun_Sonic_Sum_Constraint! to conclude the claimed estimate. For the second summation estimate, a similar argument shows that the estimate holds for all $k\geq 16$, and we again exhaust the remaining collection $k=4\ldots,15$ in \verb!efun_Sonic_Sum_Constraint!.
 \end{proof}

  
  The following lemma contains the key bounds obtained via Interval Arithmetic used in the proof of Lemma~\ref{L:D14}.


 \begin{lemma}\label{L:XYORIGINIA}
 The following estimates hold, for $\tilde\rho_0\in[0.83,0.84]$, $C_0=2$, $\alpha_0=1.95$, $a\in[0,1]$, $b\in[0.2,8]$:
 \begin{itemize}
 \item[(i)]   $\big|\big(\frac{X}{\wLP}\big)_{6}\big|  \leq(3.677+0.523 b)\frac{C_0^{4-\al_0}}{4}$;
 \item[(ii)] For $k\geq 4$, $(1.35)|\frac23\tr_0-\frac49+\frac23(a+ib)|\frac{2(k-1)}{C_0^2(2(k-2))^2}  +\Big(\frac{3D_0+(1+b)}{2}+\frac23+\frac{2(0.57)}{C_0^{\al_0}}\Big)\Big(\frac{(k-1)}{9(k-2)C_0^2}+1+(1.35)\frac{3}{4}\frac{1}{C_0^{2+\al_0}}\Big)\leq 4.156 + 0.638 b$;
\item[(iii)] $\big|\big(\frac{Y}{\wLP}\big)_4\big|\leq\big(3.0143 + 2.548 b + 0.0265 b^2)\frac{C_0^{4-\al_0}}{4}$ and
$\big|\big(\frac{Y}{\wLP}\big)_6\big|\leq\big(3.341 + 3.239 b+ 0.1266 b^2 \big)\frac{C_0^{6-\al_0}}{6}$;
 \item[(iv)] \footnotesize\beqas
& \big|2\tr_0+\frac{2(a+ib)}{3}-(1+a+ib)(a+ib)\big|(1.35)\frac{1}{C_0^2}\frac{2}{9} +\Big|(\tr_0-\frac13)\big(-\frac23\tr_0+\frac{4(a+ib)}{15}\big)\Big|(1.35)\frac{1}{C_0^{4}}\frac{1}{2}\\
&+\big(3+\frac{5b}{2}\big)\big(1+\frac{4}{27C_0^2}\big)+ (1.35)\frac{7}{12C_0^{2+\al_0}}\big(3+\frac{5b}{2}\big)\leq3.41 + 2.91 b + 0.075 b^2.
 \eeqas
 \small
 \end{itemize}
 \end{lemma}
 
 \begin{proof}
These estimates are proved using the function \verb!efun_Coeff_Origin_Constraint! in the attached code.
 \end{proof}


The following lemma contains the key bounds obtained via Interval Arithmetic used in the proof of Proposition~\ref{L:EFUNCTIONTAYLOR}(ii), which can be found
right after the proof of Lemma~\ref{L:D14}.


 \begin{lemma}\label{L:PSIORIGINIA}
 The following estimates hold, for $\tilde\rho_0\in[0.83,0.84]$, $C_0=2$, $\alpha_0=1.95$, $\mathfrak{C}_0=2+\frac{b}{4}$.
 \begin{itemize}
 \item[(i)] For $\ell=4,5$, $ |\tilde\psi_{2\ell+1}|\leq 7\frac{\mathfrak{C}_0^{2\ell-\al_0}}{(2\ell+1)^3},$;
\item[(ii)] For $\ell\geq 6$, $\sum_{j=4}^{\ell-2}\frac{\big(\frac{C_0}{\mathfrak{C}_0}\big)^{2\ell-2j-\al_0}}{(2j+1)^2(2(\ell-j))}\leq \frac{0.093}{2\ell(1+b)}$ and $\sum_{j=4}^{\ell-2}\frac{\big(\frac{C_0}{\mathfrak{C}_0}\big)^{2\ell-2j-\al_0}}{(2j+1)^3(2(\ell-j))}\leq \frac{0.01}{2\ell(1+b)}$;
 \item[(iii)] For  $\ell\geq 5$,
  \beqas
 &\frac{(2\ell+3)^2}{2\ell}\bigg(\frac{\big(\frac{C_0}{\mathfrak{C}_0}\big)^{2\ell-\al_0}}{2\ell\mathfrak{C}_0^2}\Big(\mathfrak{a}_0\big(\frac{3\ell}{C_0^2(\ell-1)}+5|\tilde\psi_5|\frac{2\ell}{2(\ell-2)C_0^4}+7|\tilde\psi_7|\frac{2\ell}{2(\ell-3)C_0^6}\big)\\
&\hspace{10mm}+\mathfrak{b}_0\big(\frac{2\ell}{C_0^2(2(\ell-1))}+|\tilde\psi_5|\frac{2\ell}{2(\ell-2)C_0^4}+|\tilde\psi_7|\frac{2\ell}{2(\ell-3)C_0^6}\big)\Big)\\
&+\frac{7}{\mathfrak{C}_0^2}\Big(\frac{\big|\big(\frac{X}{\wLP}\big)_4\big|}{\mathfrak{C}_0^2(2\ell-1)^2}+\frac{\big|\big(\frac{Y}{\wLP}\big)_2\big|}{\mathfrak{C}_0^2(2\ell-1)^3}+\frac{\big|\big(\frac{X}{\wLP}\big)_2\big|}{(2\ell+1)^2}+\frac{\big|\big(\frac{Y}{\wLP}\big)_0\big|}{(2\ell+1)^3}\Big)\\
&+\mathfrak{a}_0 \frac{7}{\mathfrak{C}_0^{2+\al_0}}\frac{0.093}{2\ell(1+b)}+\mathfrak{b}_0 \frac{7}{\mathfrak{C}_0^{2+\al_0}}\frac{0.01}{2\ell(1+b)} \bigg)\leq 7,
 \eeqas
 where we recall $\mathfrak{a}_0$ and $\mathfrak{b}_0$ from~\eqref{E:AFRAK0BFRAK0}.
 \end{itemize}
 \end{lemma}


\begin{proof}
  The estimates in (i) and (iii) are proved using the function \verb!efun_Origin_Constraint! in the attached code. The proof of (ii) proceeds as in the proof of Lemma~\ref{L:PSISONICIA}(iii), and the exhaustion over the cases $\ell=6,\ldots,15$ is contained in the function \verb!efun_Origin_Sum_Constraint!.
 \end{proof}


\section{Summation estimates}
To control the functions $\mathcal{F}_N$ and $\mathcal{G}_N$,  we need some simple technical bounds which will be important in establishing convergence later on. These bounds are optimised for constant rather than for power of $N$. In particular, as we only want a single power of $N$ in the denominator, we only find the best constant $c/N$ in each case, even for sums that are $O(N^{-2})$ as $N\to\infty$.

\begin{lemma}\label{L:COMB}
There exists a constant $c>0$ such that for all $N\in\mathbb N$, the following bounds hold
\begin{align}
   \sum_{k=2}^{N-2} \frac{1}{k^2(N-1-k)^2} &\le  N^{-1} \label{ineq:comb1} \\
 \sum_{k=2}^{N-2} \frac{1}{k^2(N-k)^2} &\le  \frac{5}{18} N^{-1} \label{ineq:comb2} \\
  \sum_{k=2}^{N-3}\frac1{k+1}\frac1{(N-k)^2} &\le 0.506 N^{-1} \label{ineq:comb3} \\
  \sum_{k=1}^{N-3}\frac{1}{(k+1)(N-1-k)^2} & \le \frac{9}{10}N^{-1}\label{ineq:comb4}  \\
   \sum_{k=1}^{N-4} \frac1{(k+1)(N-2-k)^2} & \le N^{-1} \label{ineq:comb5} \\
  \sum_{k=2}^{N-3} \frac{1}{k^2(N-1-k)^2} &\le  \frac13 N^{-1} \label{ineq:comb6} \\
   \sum_{k=2}^{N-2} \frac{C^{k-\alpha}}{k^2} & \le \frac{C^{N-2-\al}}{N}, \quad \text{ for }C\geq4,\label{ineq:comb7} \\
\sum_{k=2}^{N-2} \frac{1}{k^2(N-k)^2} &\le  \frac74 N^{-2}, \label{ineq:comb8} \\
\sum_{k=3}^{N-2}\frac{1}{(k+1)(N-k+1)^2}&\leq 0.49\frac{1}{N+1}, \label{ineq:combextra}\\
\sum_{k=1}^{n-1} \frac{1}{(2k)^2(2n-2k)^2} &\le  \frac{41}{36} (2n)^{-2}. \label{ineq:comb9} \\
\sum_{k=1}^{n-2} \frac{1}{(2k)(2(n-1)-2k)^2} &\le  \frac{3}{4} (2n)^{-1}. \label{ineq:comb10} \\
\sum_{k=1}^{n-2} \frac{1}{(2k)^2(2(n-1)-2k)^2} &\le  \frac{9}{4} (2n)^{-2}, \label{ineq:comb11} \\
\sum_{k=1}^{n-2}\frac{1}{2k(2(n-1)-2k)^2}&\leq(0.57)\frac{1}{2(n-1)}, \label{ineq:combextra2}\\
\sum_{k=2}^{n-2}\frac{1}{2k(2(n-1-k))^2}&\leq\frac{7}{12}\frac{1}{2n}, \label{ineq:combextra3}\\
 \sum_{j=1}^{\lfloor\frac{k}{2}\rfloor}\frac{1}{(k-j+2)(j+1)^2}&\leq \frac{0.662}{k+2}\quad \text{for }k\geq16,\label{ineq:combapp1}\\
 \sum_{\lceil\frac{k}{2}\rceil}^{k-2}\frac{1}{(k-j+2)(j+1)^2}&\leq \frac{0.135}{k+2}\quad \text{for }k\geq16.\label{ineq:combapp2}
\end{align}
\end{lemma}


\begin{proof}
We first prove the inequalities for the quantities where the required bound $1/N$ is weaker than the true asymptotic decay. These are \eqref{ineq:comb1}, \eqref{ineq:comb2}, and \eqref{ineq:comb6}. For simplicity, we give the proof  only for \eqref{ineq:comb2}, as the others are similar.\\
We note that
\beqa
\sum_{k=2}^{N-2}\frac{1}{k^2(N-k)^2}=\sum_{k=2}^{N-2}\frac{1}{N^2}\Big(\frac{1}{k}+\frac{1}{N-k}\Big)^2\leq\frac{4}{N^2}\sum_{k=2}^\infty\frac{1}{k^2}=\frac{4}{N^2}\big(\frac{\pi^2}{6}-1\big).
\eeqa
Observing that $\frac{4}{N^2}\big(\frac{\pi^2}{6}-1\big)\leq \frac{5}{18}\frac{1}{N}$ provided $N\geq 10$, 
we simply maximise the quantity
$$N\sum_{k=2}^{N-2}\frac{1}{k^2(N-k)^2}\text{ over }N=4,\ldots,9,$$
and find the maximum is achieved at exactly $\frac{5}{18}$ for $N=5$, thus proving \eqref{ineq:comb2}.
 We note that~\eqref{ineq:comb1} is a consequence of~\eqref{ineq:comb2} and the obvious bound $\frac{N-1}{N}\ge\frac{5}{18}$ for all $N\ge2$. 

For those sums whose asymptotic decay is of the same order as the required bound, we are slightly more careful. We give the proof here for \eqref{ineq:comb3}. First, we again split the sum as
\beqa
  \sum_{k=2}^{N-3}\frac1{k+1}\frac1{(N-k)^2} =&\,\sum_{k=2}^{N-3}\bigg(\frac{1}{(N+1)^2}\Big(\frac{1}{k+1}+\frac{1}{N-k}\Big)+\frac{1}{N+1}\frac{1}{(N-k)^2}\bigg)\\
  =&\,\frac{2}{(N+1)^2}\sum_{k=3}^{N-2}\frac{1}{k}+\frac{1}{N+1}\sum_{k=3}^{N-2}\frac{1}{k^2}. 
  \eeqa
  For the first term, we bound it with the integral test as $\sum_{k=3}^{N-2}\frac{1}{k}\leq\int_2^{N-2}\frac1x\,\dif x=\log(\frac{N-2}{2})$, while the second sum is bounded by $\sum_{k=3}^\infty\frac{1}{k^2}=\frac{\pi^2}{6}-\frac54$, so we have
  \beqa
  \sum_{k=2}^{N-3}\frac1{k+1}\frac1{(N-k)^2}\leq \frac{2}{(N+1)^2}\log(N-2)+\frac{1}{N+1}(\frac{\pi^2}{6}-\frac54).
  \eeqa
  Noting that $\frac{\pi^2}{6}-\frac54<\frac{4}{10}$, we note also that, for $N\geq 100$,
  $$\frac{2N}{(N+1)^2}\log(N-2)\leq\frac{2\log N}{N}<\frac1{10}, $$ 
  so that, for $N\geq 100$, we conclude
   \beqa
  \sum_{k=2}^{N-3}\frac1{k+1}\frac1{(N-k)^2}\leq \frac{2}{(N+1)^2}\log(N-2)+\frac{1}{N+1}(\frac{\pi^2}{6}-\frac54)<\frac{1}{2N}. 
  \eeqa
 Then, maximising $N\sum_{k=2}^{N-3}\frac1{k+1}\frac1{(N-k)^2}$ over $N=5,\ldots,100$, we find this is maximised at $N=21$ and conclude the claimed bound. The bounds for \eqref{ineq:comb4}, \eqref{ineq:comb5}, and \eqref{ineq:comb8}--\eqref{ineq:comb11} follow similar lines, splitting the sum into a top order part with good constant plus a rapidly decaying remainder, and maximising over a finite set of $N$. 

The final inequality to show is  \eqref{ineq:comb7}, which is slightly different. We first split the sum, for $N\geq 12$, as
\beqa
\sum_{k=2}^{N-2}\frac{C^{k+2-N}}{k^2}=&\,\sum_{k=2}^{\lceil\frac{N}{2}\rceil}\frac{C^{k+2-N}}{k^2}+\sum_{k=\lceil\frac{N}{2}\rceil+1}^{N-2}\frac{C^{k+2-N}}{k^2}
\leq C^{3-\frac{N}{2}}\sum_{k=2}^\infty\frac{1}{k^2}+\sum_{k=\lceil\frac{N}{2}\rceil+1}^{N-2}\frac{1}{k^2}\\
\leq&\,4^{3-\frac{N}{2}}\big(\frac{\pi^2}{6}-1\big)+\int_{\frac{N}{2}}^{N-2}\frac{1}{x^2}\dif x
\leq \frac{2}{N(N-2)}+\frac{2}{N}-\frac{1}{N-2}=\frac{1}{N},
\eeqa
where we have used $C\geq 4$ and where the inequality $(\frac{\pi^2}{6}-1)4^{3-\frac{N}{2}}\leq\frac{2}{N(N-2)}$ follows from  checking that
$$\frac{\dif}{\dif x}\big(x(x-2)4^{3-\frac{x}{2}}\big)=-2^{5-x}\big((\log 4) x^2-2x(2+\log4)+4\big)<0$$
for $x\geq 12$ and, moreover, $\big(N(N-2)(\frac{\pi^2}{6}-1)4^{3-\frac{N}{2}}-2\big)\big|_{N=12}<0$.  
 It then remains to maximise the sum over $N=4\ldots,11$.
\end{proof}

For each of the Taylor series for $\rhoLP$ and $\omLP$ at the origin and sonic point, we will want error bounds for the difference between the truncated series and the full one for derivatives up to order 5. These are stated in the following lemma which, for convenience, is stated in a coordinate $x$ which can be considered either as $y$ near the origin or $z$ near 1.

\begin{lemma}
Let $F$ be an analytic function defined on an interval $x\in[x_0-\nu,x_0+\nu]$ with power series
$$\bar F(x)=\sum_{k=0}^\infty F_k(x-x_0)^k$$
such that the coefficients $F_k$ satisfy the estimate
$$|F_k|\leq\frac{C_*^{k-\al}}{k^2}\text{ for all }k\geq N_0 $$
for some $C_*\geq 1$ and $\al\in\R$.
Given $N\geq N_0$, we define the truncated Taylor series 
\beq
\bar F_N=\sum_{k=0}^NF_k(x-x_0)^k.
\eeq
The difference between $\bar F$  and its truncation satisfies, on the region $|x-x_0|\leq\frac{\epsilon}{C_*}$,
\begin{align}
|\bar F-\bar F_N|\leq&\, \frac{C_*^{-\al}}{(N+1)^2}\frac{\epsilon^{N+1}}{1-\epsilon},  \qquad \big|\pa_{x}\big(\bar F-\bar F_N\big)\big|\leq \frac{C_*^{1-\al}}{N+1}\frac{\epsilon^N}{1-\epsilon},   \\
\big|\pa^2_{x}\big(\bar F-\bar F_N\big)\big|\leq&\, C_*^{2-\al}\frac{\epsilon^{N-1}}{1-\epsilon},   \qquad \big|\pa^3_{x}\big(\bar F-\bar F_N\big)\big|\leq C_*^{3-\al}\frac{\epsilon^{N-2}}{(1-\epsilon)^2}\Big(N(1-\epsilon)+2\epsilon-1\Big),\\
\big|\pa^4_{x}\big(\bar F-\bar F_N\big)\big|\leq&\, C_*^{4-\al}\frac{\epsilon^{N-3}}{(1-\epsilon)^3}\Big( N^2 (1-\epsilon)^2 -  N (5\epsilon-3)(\epsilon-1) + 6 \epsilon^2 - 6 \epsilon+ 2\Big),\\
\big|\pa^5_{x}\big(\bar F-\bar F_N\big)\big|\leq&\, C_*^{5-\al}\frac{\epsilon^{N-4}}{(1-\epsilon)^4}\Big( N^3 (1-\epsilon)^3 + 3N^2(1-\epsilon)^2(3\epsilon-2)\notag\\
&\hspace{18mm} + N(1-\epsilon)(26\epsilon^2-31\epsilon+11) + 6(2\epsilon-1)(2\epsilon^2-2\epsilon+1)\Big).
\end{align}
\end{lemma}

\begin{proof}
The proof follows easily from the identity
\begin{align}\label{eq:powersum}
\sum_{k=n-m}^\infty(k+m)\cdots(k+1)r^k = \frac{\textup{d}^m}{\textup{d}r^m}\Big(\frac{r^n}{1-r}\Big).
\end{align}
\end{proof}

Recall that we have found $\rhoLP$ near the origin as the series
\beq
\rhoLP(y)=\sum_{i=0}^\infty \rho_{2i}y^{2i}.
\eeq
The following lemma provides  convenient expressions for the derivatives of $\rhoLP$ and its ratios with $y$.
\begin{lemma}
The series for the derivative $\rhoLP^{(j)}(y)$, $j\geq 1$, can be written as
\begin{align}
\rhoLP^{(j)}(y)=&\,j!\rho_{2(\frac{j}{2})}+\sum_{i=1}^\infty (2i+j)\cdots(2i+1)\rho_{2(i+\frac{j}{2})} y^{2i} && \text{ if $j$ is even}, \\
\rhoLP^{(j)}(y)=&\,(j+1)!\rho_{2(\frac{j+1}{2})} y + \sum_{i=1}^\infty \Big(\prod_{k=2i+2}^{2i+j+1} k\Big)\rho_{2(i+\frac{j+1}{2})} y^{2i + 1} && \text{ if $j$ is odd}. 
\end{align}
Moreover, dividing by $y$ after differentiating once,
\begin{align}
\frac{\rhoLP'}{y}=&\,
 2\rho_2+\sum_{i=1}^{n-1} 2(i+1)\rho_{2(i+1)}y^{2i} 
+ \textup{err}_1\\
\Big(\frac{\rhoLP'}{y}\Big)'=&\,
 8\rho_4 y + \sum_{i=1}^{n-2} 2(i+2)2(i+1)\rho_{2(i+2)} y^{2i+1} + \textup{err}_2\\
\Big(\frac{\rhoLP'}{y}\Big)''=&\,
8\rho_4 + \sum_{i=1}^{n-2} 2(i+2)2(i+1)(2i+1)\rho_{2(i+2)} y^{2i} + \textup{err}_3\\
\Big(\frac{\rhoLP'}{y}\Big)^{(3)}=&\,
 144\rho_6y + \sum_{i=1}^{n-3} 2(i+3)2(i+2)(2i+3)2(i+1)\rho_{2(i+3)} y^{2i+1}+ \textup{err}_4,\\
\Big(\frac{\rhoLP'}{y}\Big)'\frac{1}{y}=&\,
  8\rho_4  + \sum_{i=1}^{n-2} 2(i+2)2(i+1)\rho_{2(i+2)} y^{2i} + \textup{err}_5,\\
 \Big(\Big(\frac{\rhoLP'}{y}\Big)'\frac{1}{y}\Big)'=&\, \sum_{i=1}^{n-2} 2(i+2)2(i+1)2i\rho_{2(i+2)} y^{2i-1}  + \textup{err}_6,\\
\frac{\rhoLP^{(3)}}{y}=&\,
 24\rho_4+\sum_{i=1}^{n-2}(2i+4)(2i+3)(2i+2)\rho_{2(i+2)} y^{2i}+\textup{err}_7\\
\Big(\frac{\rhoLP^{(3)}}{y}\Big)'=&\,
 \sum_{i=1}^{n-2}(2i+4)(2i+3)(2i+2)2i\rho_{2(i+2)} y^{2i-1}+\textup{err}_8,
\end{align}
where the error functions satisfy the bounds
\begin{align}
|\textup{err}_1|\leq&\, \frac{C_0^{2-\al}}{2(n+1)}\frac{(C_0\de)^{2n}}{1-(C_0\de)^2},\\
|\textup{err}_2|\leq&\, C_0^{4-\al}\de\frac{(C_0\de)^{2(n-1)}}{1-(C_0\de)^2},\\
|\textup{err}_3|\leq&\, C_0^{4-\al}\frac{(C_0\de)^{2(n-1)}\big(2n(1-(C_0\de)^2)+3(C_0\de)^2-1\big)}{(1-(C_0\de)^2)^2},\\
|\textup{err}_4|\leq&\, 2C_0^{5-\al}(C_0\de)^{2n-3}\frac{\big(2n^2(1-(C_0\de)^2)^2 +n(-7(C_0\de)^4+10(C_0\de)^2-3)+6(C_0\de)^4-3(C_0\de)^2+1\big)}{(1-(C_0\de)^2)^3}, \\
|\textup{err}_5|\leq&\, C_0^{4-\al}\frac{(C_0\de)^{2(n-1)}}{1-(C_0\de)^2}, \\
|\textup{err}_6|\leq&\, 2C_0^{5-\al}(C_0\de)^{2n-3}\frac{n(1-(C_0\de)^2)+2(C_0\de)^2-1}{(1-(C_0\de)^2)^2}, \\
|\textup{err}_7|\leq&\, 2C_0^{4-\al}(C_0\de)^{2n-2}\frac{n(1-(C_0\de)^2)+(C_0\de)^2}{(1-(C_0\de)^2)^2}, \\
|\textup{err}_8|\leq&\, 4C_0^{5-\al}(C_0\de)^{2n-3}\frac{n^2(1-(C_0\de)^2)^2+n(-3(C_0\de)^4+4(C_0\de)^2-1)+2(C_0\de)^4}{(1-(C_0\de)^2)^3}. 
\end{align}
\end{lemma}

\begin{proof}
The identities and error bounds are all straightforward.
\end{proof}

\section{Interval Arithmetic code -- a guide}\label{APP:CODE}

For the convenience of the reader, we provide here an outline table of contents for the attached interval arithmetic code (\url{https://github.com/mrischrecker/Collapse/blob/main/sonic_Taylor.cc#L7341}) For each section/subsection, we provide a short description of the contents in order to aid the reader. The headings in bold and italics are written directly in comments in the attached code.\\

\noindent\textbf{SECTION: Auxiliary functions}\\
A collection of convenient auxiliary functions used throughout the code.\\

\noindent\textbf{SECTION: Larson-Penston solution}\\
\noindent\textit{SUBSECTION: LP Construction from the sonic point}\\
Functions used to prove the growth rate of the Taylor coefficients for the LP solution from the sonic point and to construct the Taylor series approximation of the LP solution close to the sonic poinr. Specifically used for the proof of Proposition~\ref{P:RHONOMNBDS}.\\
\noindent\textit{SUBSECTION: LP Construction from the origin}\\
Functions used to prove the growth rate of the Taylor coefficients for the LP solution from the origin and to construct the Taylor series approximation of the LP solution close to the origin. Specifically used for the proof of Lemma~\ref{L:ORIGINCOEFFS}(ii).\\
\noindent\textit{SUBSECTION: LP ODE functions}\\
This contains the functions implementing \verb!VNODE-LP! to solve the LP ODE backwards from the sonic point, taking initial data determined from the Taylor series.\\
\noindent\textit{SUBSECTION: Solvers for $y_*$ and $\rho_0$}\\
Functions to establish the enclosures for $y_*$ and $\rhoLP(0)$. Specifically used for the proof of Lemma~\ref{L:LPYSTAR} and~\ref{L:ORIGINCOEFFS}(i).\\

\noindent\textbf{SECTION: Eigenfunction solvers}\\
This section contains all the functions needed to construct solutions to the eigenfunction ODE~\eqref{eq:phila} via Taylor expansion close to the origin and sonic point and to verify the numerical bounds on the growth rates of the coefficients. It also contains the functions used in the exclusion of intermediate eigenvalues. It is split into several subsections.\\
\noindent\textit{SUBSECTION: Eigenfunction construction at sonic point}\\
This subsection contains the functions used to verify the constants in the growth rates of the eigenfunctions around the sonic point in Proposition~\ref{L:EFUNCTIONTAYLOR}(i), including the estimates of Lemmas~\ref{L:WSONICIA},~\ref{L:ABSONICIA} and~\ref{L:PSISONICIA}.\\
\noindent\textit{SUBSECTION: Eigenfunction ODE solvers}\\
Functions to implement \verb!VNODE-LP! and solve the eigenfunction ODE~\eqref{eq:phila} backwards from the sonic point towards the origin, given data determined by the Taylor series.\\
\noindent\textit{SUBSECTION: Eigenfunction construction at origin}\\
This subsection contains the functions used to verify the constants in the growth rates of the eigenfunctions around the sonic point in Proposition~\ref{L:EFUNCTIONTAYLOR}(ii), including the estimates of Lemmas~\ref{L:WORIGINIA} and~\ref{L:PSIORIGINIA}.\\
\noindent\textit{SUBSECTION: Eigenfunction exclusion}\\
This section contains the functions used to exclude intermediate eigenvalues, as described in Proposition~\ref{P:NOINTERIMAG}.\\

\noindent\textbf{SECTION: Energy}\\
This section contains all of the functions needed to verify the exclusion of eigenvalues in the high and low frequency regimes, stated in Propositions~\ref{P:NOLARGEIMAG} and~\ref{P:NOSMALLIMAG}. To implement this, we split the code into subsections, first to construct all of the relevant derivative quantities derived from the LP solution that are needed to prove these propositions, then to construct the ODE coefficients appearing in Proposition~\ref{prop:A4B4} and~\ref{P:NOSMALLIMAG}, and finally to construct the energy coefficients $H_\la$ and $\widetilde{H}_\la$ appearing in the identities of Lemma~\ref{L:LARGEBENERGY} and Proposition~\ref{P:NOSMALLIMAG}.\\
\noindent\textit{SUBSECTION: Derivatives near sonic point}\\
This section contains the functions needed to compute the derivatives of the LP solution on a region $\frac{y}{y_*}\in[1-\de_1,1]$ using the Taylor expansions.\\
\noindent\textit{SUBSECTION: Derivatives in intermediate region}\\
This section contains the functions needed to compute the derivatives of the LP solution on a region $\frac{y}{y_*}\in[\de_0,1-\de_1]$ via differentiating the LP ODE.\\
\noindent\textit{SUBSECTION: Derivatives near origin}\\
This section contains the functions needed to compute the derivatives of the LP solution on a region $\frac{y}{y_*}\in[0,\de_0]$ via the Taylor expansion.\\
\noindent\textit{SUBSECTION: Energy -- Large b coefficients}\\
Functions appearing as coefficients upon differentiating the eigenfunction ODE~\eqref{eq:phila}, as detailed in Appendix~\ref{S:Highbcoeffs}.\\
\noindent\textit{SUBSECTION: Energy -- Large $b$ energy coefficients}\\
Functions defining the energy coefficient $H_\la$ appearing in the proof of Proposition~\ref{P:NOLARGEIMAG}.\\
\noindent\textit{SUBSECTION: Energy -- Small $b$ coefficients}\\
Functions appearing as coefficients upon differentiating the eigenfunction ODE~\eqref{eq:phila}, as detailed in Appendix~\ref{S:PEQCOEFFS}, and  defining the energy coefficient $\widetilde{H_\la}$ appearing in the proof of Proposition~\ref{P:NOSMALLIMAG}.\\
\noindent\textit{SUBSECTION: Energy -- Coefficient sign functions}\\
Functions used to establish the required sign conditions for the energy coefficients, proving Propositions~\ref{P:NOLARGEIMAG} and~\ref{P:NOSMALLIMAG}.


The main function of the attached code simply runs the proving functions in the order contained in the following table, with approximate run-times included.

\begin{table}[h]
\centering
 \begin{tabular}{|c |c |c|} 
 \hline 
 Lemma/Proposition & Function & Approx.~run-time  \\ [0.5ex] 
 \hline\hline
Lemma D.3 & \verb!C_alpha_constraint_check_Sonic! & <1s \\ \hline
Proposition D.1 & \verb!C_alpha_const_check_Sonic! & <1s \\ \hline
Lemma D.6 & \begin{tabular}{@{}c@{}}\verb!y_bar_star_upper_bound! \\ \verb!LP_solver!\end{tabular}   & \ $\sim$ 50s\\ \hline
Lemma D.7 & \begin{tabular}{@{}c@{}}\verb!rho0_bound! \\ \verb!C_alpha_constraint_check_Origin!\end{tabular} & <1s\\ \hline
Lemma D.8 &  \begin{tabular}{@{}c@{}} \verb!large_b_Origin_Prover! \\ \verb!large_b_Sonic_Prover! \\ \verb!large_b_Sonic_Prover!\end{tabular}  & $\sim$2h\\ \hline
Lemma D.9 &  \begin{tabular}{@{}c@{}} \verb!small_b_Origin_Prover! \\ \verb!small_b_Sonic_Prover! \\ \verb!small_b_Sonic_Prover!\end{tabular} & $\sim$2h\\ \hline
Lemma D.15 & \verb!w_Sonic_Constraint! & <1s\\ \hline
Lemma D.16 & \verb!w_Origin_Constraint! & <1s\\ \hline
Lemma D.17 & \verb!efun_Coeff_Sonic_Constraint! & <1s\\ \hline
Lemma D.18 & \begin{tabular}{@{}c@{}}\verb!efun_Sonic_Constraint! \\ \verb!efun_Sonic_Sum_Constraint!\end{tabular} & $\sim$13m\\ \hline
Lemma D.19 & \verb!efun_Coeff_Origin_Constraint! & $\sim$7s\\ \hline
Lemma D.20 & \begin{tabular}{@{}c@{}}\verb!efun_Origin_Constraint! \\ \verb!efun_Origin_Sum_Constraint!\end{tabular} & <1s\\ \hline
Proposition 3.16 & \verb!intermediate_evalue_excluder! & $\sim$ 23 days\\ \hline
 \end{tabular}
\end{table}


\end{document}